\documentclass[a4paper, leqno, 10pt]{amsart}

\usepackage[T1]{fontenc}	
\usepackage{graphicx}
\usepackage{amssymb,centernot}
\usepackage{upgreek}
\usepackage{mathrsfs}
\usepackage{pdfsync}
\usepackage[usenames,dvipsnames]{xcolor}
\usepackage[inline,shortlabels]{enumitem}
\usepackage[sort,numbers]{natbib}						%bibliografia latex
\usepackage{verbatim}								%per commentare cose lunghe
\usepackage{dsfont}									%doublestroke \mathds
\usepackage{pdflscape}
\usepackage{afterpage}

\usepackage[bookmarks,
						colorlinks,				%link a colori invece che riquadrati
						breaklinks,unicode,
						hypertexnames=false,
						citecolor=OliveGreen,
						linkcolor=Maroon
						]{hyperref} %hyperref deve stare per ultimo

\usepackage{multirow}

\usepackage{xypic}
\usepackage{mathtools}
\usepackage{xspace}

\usepackage[a4paper,top=3.2cm,bottom=2.7cm,left=3cm,right=3cm, bindingoffset=0mm]{geometry}
\usepackage{booktabs}

\usepackage{tikz}
\usetikzlibrary{calc, positioning, shapes.geometric, patterns, cd}

\newlength{\hatchspread}
\newlength{\hatchthickness}
\newlength{\hatchshift}
\newcommand{\hatchcolor}{}
\tikzset{hatchspread/.code={\setlength{\hatchspread}{#1}},
         hatchthickness/.code={\setlength{\hatchthickness}{#1}},
         hatchshift/.code={\setlength{\hatchshift}{#1}},% must be >= 0
         hatchcolor/.code={\renewcommand{\hatchcolor}{#1}}}
\tikzset{hatchspread=3pt,
         hatchthickness=0.4pt,
         hatchshift=0pt,% must be >= 0
         hatchcolor=black}
\pgfdeclarepatternformonly[\hatchspread,\hatchthickness,\hatchshift,\hatchcolor]% variables
   {custom north west lines}% name
   {\pgfqpoint{\dimexpr-2\hatchthickness}{\dimexpr-2\hatchthickness}}% lower left corner
   {\pgfqpoint{\dimexpr\hatchspread+2\hatchthickness}{\dimexpr\hatchspread+2\hatchthickness}}% upper right corner
   {\pgfqpoint{\dimexpr\hatchspread}{\dimexpr\hatchspread}}% tile size
   {% shape description
    \pgfsetlinewidth{\hatchthickness}
    \pgfpathmoveto{\pgfqpoint{0pt}{\dimexpr\hatchspread+\hatchshift}}
    \pgfpathlineto{\pgfqpoint{\dimexpr\hatchspread+0.15pt+\hatchshift}{-0.15pt}}
    \ifdim \hatchshift > 0pt
      \pgfpathmoveto{\pgfqpoint{0pt}{\hatchshift}}
      \pgfpathlineto{\pgfqpoint{\dimexpr0.15pt+\hatchshift}{-0.15pt}}
    \fi
    \pgfsetstrokecolor{\hatchcolor}
    \pgfusepath{stroke}
   }

\pgfdeclarepatternformonly[\hatchspread,\hatchthickness,\hatchshift,\hatchcolor]% variables
   {custom north east lines}% name
   {\pgfqpoint{\dimexpr-2\hatchthickness}{\dimexpr-2\hatchthickness}}% lower left corner
   {\pgfqpoint{\dimexpr\hatchspread+2\hatchthickness}{\dimexpr\hatchspread+2\hatchthickness}}% upper right corner
   {\pgfqpoint{\dimexpr\hatchspread}{\dimexpr\hatchspread}}% tile size
   {% shape description
    \pgfsetlinewidth{\hatchthickness}
    \pgfpathmoveto{\pgfqpoint{\dimexpr\hatchshift-0.15pt}{-0.15pt}}
    \pgfpathlineto{\pgfqpoint{\dimexpr\hatchspread+0.15pt}{\dimexpr\hatchspread-\hatchshift+0.15pt}}
    \ifdim \hatchshift > 0pt
      \pgfpathmoveto{\pgfqpoint{-0.15pt}{\dimexpr\hatchspread-\hatchshift-0.15pt}}
      \pgfpathlineto{\pgfqpoint{\dimexpr\hatchshift+0.15pt}{\dimexpr\hatchspread+0.15pt}}
    \fi
    \pgfsetstrokecolor{\hatchcolor}
    \pgfusepath{stroke}
   }

\makeatletter
\newcommand*{\centerfloat}{%
  \parindent \z@
  \leftskip \z@ \@plus 1fil \@minus \textwidth
  \rightskip\leftskip
  \parfillskip \z@skip}
\makeatother

\newcommand{\boldSigma}{{\boldsymbol\A}}

\newcommand{\Kappa}{\mathrm{K}}

\newcommand{\Tau}{\mathrm{T}}

\newcommand{\boldpi}{{\boldsymbol \pi}}

\newcommand{\bmssd}{\boldsymbol\mssd}

\usepackage{xparse}

\ExplSyntaxOn
\NewDocumentCommand{\makeabbrev}{mmm}
 {
  \yoruk_makeabbrev:nnn { #1 } { #2 } { #3 }
 }

\cs_new_protected:Npn \yoruk_makeabbrev:nnn #1 #2 #3
 {
  \clist_map_inline:nn { #3 }
   {
    \cs_new_protected:cpn { #2 } { #1 { ##1 } }
   }
 }
 \ExplSyntaxOff

\makeabbrev{\textbf}{tbf#1}{a,b,c,d,e,f,g,h,i,j,k,l,m,n,o,p,q,r,s,t,u,v,w,x,y,z,A,B,C,D,E,F,G,H,I,J,K,L,M,N,O,P,Q,R,S,T,U,V,W,X,Y,Z}

\makeabbrev{\textbf}{bf#1}{a,b,c,d,e,f,g,h,i,j,k,l,m,n,o,p,q,r,s,t,u,v,w,x,y,z,A,B,C,D,E,F,G,H,I,J,K,L,M,N,O,P,Q,R,S,T,U,V,W,X,Y,Z}

\makeabbrev{\textsf}{tsf#1}{a,b,c,d,e,f,g,h,i,j,k,l,m,n,o,p,q,r,s,t,u,v,w,x,y,z,A,B,C,D,E,F,G,H,I,J,K,L,M,N,O,P,Q,R,S,T,U,V,W,X,Y,Z}

\makeabbrev{\mathsf}{mss#1}{a,b,c,d,e,f,g,h,i,j,k,l,m,n,o,p,q,r,s,t,u,v,w,x,y,z,A,B,C,D,E,F,G,H,I,J,K,L,M,N,O,P,Q,R,S,T,U,V,W,X,Y,Z}

\makeabbrev{\mathfrak}{mf#1}{a,b,c,d,e,f,g,h,i,j,k,l,m,n,o,p,q,r,s,t,u,v,w,x,y,z,A,B,C,D,E,F,G,H,I,J,K,L,M,N,O,P,Q,R,S,T,U,V,W,X,Y,Z}

\makeabbrev{\mathrm}{mrm#1}{a,b,c,d,e,f,g,h,i,j,k,l,m,n,o,p,q,r,s,t,u,v,w,x,y,z,A,B,C,D,E,F,G,H,I,J,K,L,M,N,O,P,Q,R,S,T,U,V,W,X,Y,Z}

\makeabbrev{\mathbf}{mbf#1}{a,b,c,d,e,f,g,h,i,j,k,l,m,n,o,p,q,r,s,t,u,v,w,x,y,z,A,B,C,D,E,F,G,H,I,J,K,L,M,N,O,P,Q,R,S,T,U,V,W,X,Y,Z}

\makeabbrev{\mathcal}{mc#1}{A,B,C,D,E,F,G,H,I,J,K,L,M,N,O,P,Q,R,S,T,U,V,W,X,Y,Z}

\makeabbrev{\mathbb}{mbb#1}{A,B,C,D,E,F,G,H,I,J,K,L,M,N,O,P,Q,R,S,T,U,V,W,X,Y,Z}

\makeabbrev{\mathscr}{ms#1}{A,B,C,D,E,F,G,H,I,J,K,L,M,N,O,P,Q,R,S,T,U,V,W,X,Y,Z}

\makeabbrev{\mathrm}{#1}{
Id,id,ran,rk,diag,stab,ann,conv,pr,ev,tr,End,Hom,sgn,im,op,can,fin,ext,red,tot,
rot,usc,lsc,Lip,LocLip,lip,bSymLip,osc,AC,loc,uloc,spec,coz,z,ul,
supp,Opt,Adm,Cpl,Geo,GeoOpt,GeoAdm,GeoCpl,reg,
bd,co,Ric,Exp,dExp,dist,seg,Seg,cut,fcut,Cut,SDiff,Iso,Isom,diam,cl,Homeo,Diff,Der,vol,dvol,inj,relint, Graph,sub,length,
var,law,Var,Poi,Gam,pa,so,iso,fs,inv,pqi,mix,
TestF,
}

\makeabbrev{\mathsf}{#1}{DP,CD,BE,MCP,Ent,wMTW,MTW,RCD,QCD,EVI,Irr,IH,SC,wFe,VA,UP,Curv,Alex,CAT}

\newcommand{\bLip}{\mathrm{Lip}_b}
\newcommand{\bsLip}{\mathrm{Lip}_{bs}}

\newcommand{\KSD}{D} %dense set in the KS convergence

\newcommand{\T}{\tau} %TOPOLOGY
\newcommand{\A}{\Sigma} %SIGMA-ALGEBRA
\newcommand{\Bo}[1]{\msB_{#1}} %BOREL SIGMA-ALGEBRA
\newcommand{\Ko}[1]{\msK_{#1}}
\newcommand{\rKo}[1]{{}_r\!\msK_{#1}}
\newcommand{\Bdd}[1]{\msO_{#1}}
\newcommand{\Ed}{\msE_\mssd}  

\newcommand{\eps}{\varepsilon}

\newcommand{\defeq}{\eqqcolon}
\renewcommand{\complement}{\mathrm{c}}

\newcommand{\emparg}{{\,\cdot\,}}

\newcommand{\slo}[2][]{\abs{\mathrm{D}#2}_{#1}}
\newcommand{\wslo}[2][]{\abs{\mathrm{D}#2}_{w,\, #1}}

\newcommand{\Sb}{\A_b}

\newcommand{\dint}[2][]{\sideset{^{#1}\!\!\!}{_{#2}^{\scriptstyle\oplus}}\int}

\newcommand{\Ch}[1][]{\mathsf{Ch}_{#1}}
\newcommand{\CE}[1][]{\mathsf{CE}_{#1}}

\newcommand{\forallae}[1]{{\textrm{\,for ${#1}$-a.e.\,}}}
\newcommand{\as}[1]{\quad #1\text{-a.e.}}
\newcommand{\Leb}{{\mathrm{Leb}}}

\renewcommand{\Cap}{\mathrm{cap}}

\newcommand{\dom}[1]{\msD(#1)}
\DeclareMathOperator{\Dom}{dom}

\newcommand{\domloc}[1]{\msD(#1)_{\loc}^\bullet}
\newcommand{\domext}[1]{\msD_{e}(#1)}
\newcommand{\DzLoc}[1]{\mbbL^{#1}_{\loc}}
\newcommand{\DzLocB}[1]{\mbbL^{#1}_{\loc,b}}
\newcommand{\DzLocBprime}[1]{\mbbL^{\prime\, #1}_{\loc,b}}
\newcommand{\DzB}[1]{\mbbL^{#1}_b}
\newcommand{\DzBprime}[1]{\mbbL^{\prime\, #1}_b}
\newcommand{\Lipu}{\mathrm{Lip}^1}
\newcommand{\dotloc}[1]{{#1}^\bullet_\loc}

\newcommand{\Rad}[2]{\mathsf{Rad}_{#1,#2}}
\newcommand{\dRad}[2]{{#1}\textrm{-}\mathsf{Rad}_{#2}}
\newcommand{\ScL}[3]{\mathsf{ScL}_{#1,#2,#3}}
\newcommand{\cSL}[3]{\mathsf{cSL}_{#1,#2,#3}}
\newcommand{\SL}[2]{\mathsf{SL}_{#1,#2}}
\newcommand{\dcSL}[3]{{#1}\textrm{-}\mathsf{cSL}_{#2,#3}}
\newcommand{\dSL}[2]{{#1}\textrm{-}\mathsf{SL}_{#2}}

\DeclareMathOperator{\eqdef}{\coloneqq}

\newcommand\doublecheck{\checkmark\!\!\!\checkmark}

\let\epsilon\varepsilon
\newcommand{\stigma}{\varsigma}
\let\temp\phi
\let\phi\varphi
\let\varphi\temp

\newcommand{\longrar}{\longrightarrow}
\newcommand{\rar}{\rightarrow}

\newcommand{\nlim}{\lim_{n}}								%omesso \rightarrow\infty

\newcommand{\klim}{\lim_{k }}

\newcommand{\nliminf}{\liminf_{n }}
\newcommand{\mliminf}{\liminf_{m }}

\newcommand{\nlimsup}{\limsup_{n }\,}

\newcommand{\diff}{\mathop{}\!\mathrm{d}}						%Differenziale esatto

\newcommand{\ttabs}[1]{\lvert#1\rvert}	
\newcommand{\tabs}[1]{\big\lvert#1\big\rvert}	
\newcommand{\abs}[1]{\left\lvert#1\right\rvert}						%Modulo
\newcommand{\tnorm}[1]{\big\lVert#1\big\rVert}						%Norma
\newcommand{\ttnorm}[1]{\lVert#1\rVert}
\newcommand{\norm}[1]{\left\lVert#1\right\rVert}					%Norma
\newcommand{\set}[1]{\left\{#1\right\}}							%Insieme, graffe
\newcommand{\tset}[1]{\big\{#1\big\}}							%Insieme, graffe
\newcommand{\ttset}[1]{\{#1\}}									%Insieme, graffe
\newcommand{\tonde}[1]{\left(#1\right)}							%Tonde
\newcommand{\ttonde}[1]{\big({#1}\big)}
\newcommand{\quadre}[1]{\left[#1\right]}							%Quadre
\newcommand{\class}[2][]{\left[#2\right]_{#1}}						%Measure classes
\newcommand{\tclass}[2][]{\big [#2\big]_{#1}}						%Measure classes
\newcommand{\ttclass}[2][]{[#2]_{#1}}							%Measure classes
\newcommand{\spclass}[2][]{#2_{#1}}

\newcommand{\Li}[2][]{\mathrm{L}_{#1}(#2)}						%Lipschitz
\newcommand{\Lia}[2][]{\mathrm{Lip}^a_{#1}[#2]}						%Lipschitz
\newcommand{\rep}[1]{\hat #1}									%Rappresentante
\newcommand{\reptwo}[1]{\tilde{#1}}							%Rappresentante
\newcommand{\scalar}[2]{\left\langle #1 \,\middle |\, #2\right\rangle}		%Prodotto scalare
\newcommand{\hotimes}{\widehat{\otimes}}

\newcommand{\asym}[1]{{\scriptscriptstyle{[#1]}}}
\newcommand{\sym}[1]{{\scriptscriptstyle{(#1)}}}
\newcommand{\tym}[1]{{\scriptscriptstyle{\times #1}}}
\newcommand{\otym}[1]{{\scriptscriptstyle{\otimes #1}}}

\newcommand{\hotym}[1]{{\scriptscriptstyle{\widehat{\otimes}{#1}}}}

\DeclareSymbolFont{symbolsC}{U}{pxsyc}{m}{n}
\SetSymbolFont{symbolsC}{bold}{U}{pxsyc}{bx}{n}
\DeclareFontSubstitution{U}{pxsyc}{m}{n}
\DeclareMathSymbol{\medcirc}{\mathbin}{symbolsC}{7}
\DeclareSymbolFont{symbolsZ}{OMS}{pxsy}{m}{n}
\SetSymbolFont{symbolsZ}{bold}{OMS}{pxsy}{bx}{n}
\DeclareFontSubstitution{OMS}{pxsy}{m}{n}

\DeclareMathOperator{\interior}{int}								%Interno
							
\newcommand{\seq}[1]{\tonde{#1}}								%Successione
\newcommand{\tseq}[1]{{\big(#1\big)}}

\newcommand{\Cb}{\mcC_b}									%funzioni continue e limitate 
\newcommand{\Cc}{\mcC_c}									%funzioni continue e limitate 
\newcommand{\Cz}{\mcC_0}									%funzioni continue e evanescenti
\newcommand{\Cbs}{{\mcC_{bs}}}							%funzioni limitate e a supporto limitato
\newcommand{\Cbinfty}{{\mcC_{b}^{\infty}}}

\newcommand{\Meas}{\mathscr M}

\newcommand{\pfwd}{\sharp}
\DeclareMathOperator*{\esssup}{esssup}
\DeclareMathOperator*{\essinf}{essinf}

\DeclareMathOperator{\car}{\mathds 1}

\DeclareMathOperator{\emp}{\varnothing} %Known Sets, Fields and so on...
\newcommand{\N}{{\mathbb N}}
\newcommand{\R}{{\mathbb R}}
\DeclareMathOperator{\Rex}{{\overline{\mathbb R}}}
\DeclareMathOperator{\Q}{{\mathbb Q}}
\DeclareMathOperator{\Z}{{\mathbb Z}}

\newcommand{\LDS}{\textsc{lds}\xspace}
\newcommand{\TLDS}{\textsc{tlds}\xspace}
\newcommand{\MLDS}{\textsc{mlds}\xspace}
\newcommand{\EMLDS}{\textsc{emlds}\xspace}
\newcommand{\parEMLDS}{\textsc{(e)mlds}\xspace}

\newcommand{\lb}{\mfl}
\newcommand{\llb}{\scriptstyle\lb}
\newcommand{\Lb}{\mathfrak{L}}

\newcommand{\restr}{\big\lvert}

\newcommand{\mrestr}[1]{\!\downharpoonright_{#1}}
\newcommand{\trid}{{\star}}

\allowdisplaybreaks

\usetikzlibrary{shapes.misc}
\tikzset{cross/.style={cross out, draw=black, minimum size=2*(#1-\pgflinewidth), inner sep=0pt, outer sep=0pt},
cross/.default={4pt}}

\newcommand{\iref}[1]{\ref{#1}}

\newcommand{\comma}{\,\,\mathrm{,}\;\,}
\newcommand{\semicolon}{\,\,\mathrm{;}\;\,}
\newcommand{\fstop}{\,\,\mathrm{.}}

\newcommand{\cdc}{\Gamma}
\newcommand{\LL}[2]{\mcL^{#1, #2}}
\newcommand{\TT}[2]{\mcT^{#1, #2}}
\newcommand{\hh}[2]{\mssh^{#1, #2}}

\newcommand{\EE}[2]{\mcE^{#1, #2}}
\newcommand{\repSF}[2]{\rep\cdc^{#1, #2}}

\newcommand{\SF}[2]{\cdc^{#1, #2}}

\renewcommand{\iint}{\int\!\!\!\!\int}

\DeclareMathOperator{\zero}{{\mathbf 0}}
\DeclareMathOperator{\uno}{{\mathbf 1}}

\usepackage{scrextend}						%KOMA: confligge quindi deve stare qui

\newcommand{\Lin}{L} 						%Linear operators

\DeclareMathOperator{\inter}{int}
\newcommand{\Cyl}[1]{\mcF^\dUpsilon\mcC^\infty_b(#1)}
\newcommand{\CylQP}[2]{\mcF^\dUpsilon\mcC^{\infty}_b(#2)_{#1}}

\newcommand{\CylSet}[1]{\mcE^{\!\times\!}\mcC(#1)}

\newcommand{\Dz}{\mcD} %domain
 
\newcommand{\preW}{\msW^{1,2}_{\mathrm{pre}}} %pre-Sobolev

\newcommand{\locfin}{\mathrm{lcf}}

\newcommand{\PP}{{\pi}}

\newcommand{\cpl}{q}
\newcommand{\QP}{{\mu}}

\newcommand{\hr}[1]{\bar\mssd_{#1}} 									%Hino-Ramirez d_A

\newcommand{\coK}[2]{\co_{#2}({#1})}
\newcommand{\cok}{\kappa}
\newcommand{\cokb}{\bar\kappa}
\newcommand{\coi}{\iota}

\newcommand{\dUpsilon}{{\boldsymbol\Upsilon}}

\numberwithin{equation}{section}
\theoremstyle{plain}
\newtheorem{thm}{Theorem}[section]
\newtheorem*{thm*}{Theorem}
\newtheorem*{mthm*}{Main Theorem}

\newtheorem{prop}[thm]{Proposition}%[section]
\newtheorem{lem}[thm]{Lemma}%[section]
\newtheorem{cor}[thm]{Corollary}%[section]

\theoremstyle{definition}
\newtheorem{defs}[thm]{Definition}%[section]
\newtheorem{notat}[thm]{Notation}%[section]
\newtheorem*{defs*}{Definition}%[section]
\theoremstyle{remark}
\newtheorem{rem}[thm]{Remark}%[section]
\newtheorem{ese}[thm]{Example}%[section]
\newtheorem{ass}[thm]{Assumption}%[section]

\renewcommand{\paragraph}[1]{\medskip\emph{#1}.\quad}

\begin{document}

\title[Configuration spaces over singular spaces I]{Configuration Spaces Over Singular Spaces
\vspace{2mm}\\ 
-- {\tiny I. Dirichlet-Form and Metric Measure Geometry --}}

\dedicatory{Dedicated to Professor Karl-Theodor Sturm on the occasion of his sixtieth birthday}

\thanks{The authors are grateful to Prof.s Luigi Ambrosio, Michael R\"ockner, Giuseppe Savar\'e, and Karl-Theodor Sturm for several fruitful discussions on the subject of this work.
Since part of this work was carried out during the \emph{$5^\text{th}$ Summer School in Stochastic and Geometric Analysis -- Piz Buin, 6--12 July 2019} and the \emph{Japanese--German Open Conference on Stochastic Analysis -- Fukuoka University 2--6 September 2019}, it is their pleasure to thank the organizers of both events for making those occasions possible.}

\author[L.~Dello Schiavo]{Lorenzo Dello Schiavo}
\address{IST Austria, Am Campus 1, 3400 Klosterneuburg, Austria}
\email{lorenzo.delloschiavo@ist.ac.at}
\thanks{A large part of this work was completed while L.D.S.\ was a member of the Institut f\"ur Angewandte Mathematik of the University of Bonn. He acknowledges funding of his position at that time by the Collaborative Research Center 1060 of the University of Bonn.
He also acknowledges funding of his current position by the Austrian Science Fund (FWF) grant F65, and by the European Research Council (ERC, grant No.~716117, awarded to Prof.\ Dr.~Jan Maas).
Finally, he gratefully acknowledges financial support for his stay in Fukuoka from JSPS Grant-in-Aid for Scientific Research (A) Grant Nr.~19H00643 awarded to Prof.\ Masanori Hino (Graduate School of Science, Kyoto University).}

\author[K.~Suzuki]{Kohei Suzuki}
\address{Fakult\"at f\"ur Mathematik, Universit\"at Bielefeld, D-33501, Bielefeld, Germany}
\email{ksuzuki@math.uni-bielefeld.de}
\thanks{K.S.~gratefully acknowledges funding by: the JSPS Overseas Research Fellowships, Grant Nr. 290142; World Premier International Research Center Initiative (WPI), MEXT, Japan; JSPS Grant-in-Aid for Scientific Research on Innovative Areas ``Discrete Geometric Analysis for Materials Design'', Grant Number 17H06465; and the Alexander von Humboldt Stiftung, Humboldt-Forschungsstipendium}



\renewcommand{\abstractname}{\normalsize Abstract}
\begin{abstract}
\normalsize\noindent This paper is the first in a series on configuration spaces over singular spaces.
Here, we construct a canonical differential structure on the configuration space~$\dUpsilon$ over a singular base space~$X$ and with a general invariant measure~$\QP$ on~$\dUpsilon$.

\medskip

We first present an \emph{analytic structure} on~$\dUpsilon$, constructing a strongly local Dirichlet form~$\mcE$ on~$L^2(\dUpsilon, \QP)$ for~$\QP$ in a large class of probability measures.
We then investigate the \emph{geometric structure} of the extended metric measure space~$\dUpsilon$ endowed with the $L^2$-transportation extended distance~$\mssd_\dUpsilon$ and with the measure~$\QP$.
By establishing various Rademacher- and Sobolev-to-Lipschitz-type properties for~$\mcE$, we finally provide a complete identification of the analytic and the geometric structure ---~the \emph{canonical differential structure} induced on~$\dUpsilon$ by~$X$ and~$\QP$~--- showing that $\mcE$ coincides with the Cheeger energy of $(\dUpsilon,\mssd_\dUpsilon,\QP)$ and that the intrinsic distance of~$\mcE$ coincides with~$\mssd_\dUpsilon$.

\medskip

The class of base spaces to which our results apply includes sub-Riemannian manifolds,~$\RCD$ spaces, locally doubling metric measure spaces satisfying a local Poincar\'e inequality, and path/loop spaces over Riemannian manifolds; as for~$\QP$ our results include Campbell measures and quasi-Gibbs measures, in particular: Poisson measures, canonical Gibbs measures, as well as some determinantal/permanental point processes ($\mathrm{sine}_\beta$, $\mathrm{Airy}_\beta$, $\mathrm{Bessel}_{\alpha, \beta}$, Ginibre).

\medskip

 A number of applications to interacting particle systems and infinite-dimension\-al metric measure geometry are also discussed.
In particular, we prove the quasi-regularity of~$\mcE$ under minimal assumptions and consequently the existence and uniqueness of a Markov diffusion associated to~$\mcE$ describing a $\QP$-invariant particle system.
We further show the universality of the $L^2$-transportation distance~$\mssd_\dUpsilon$ for the Varadhan short-time asymptotics for diffusions on~$\dUpsilon$ for wide classes of invariant measures~$\QP$. %in the class of measures satisfying our main assumptions.
Our approach does not rely on any relation between~$\QP$ and a possible smooth structure on~$X$.
In particular, we assume no quasi-invariance property of~$\QP$ w.r.t.\ actions of any group of transformations of~$X$.
Many of our results are new even in the case when~$X$ is the standard Euclidean spaces.
\end{abstract}

\maketitle

%TABLE OF CONTENTS
\setcounter{tocdepth}{2}
\makeatletter
\def\l@subsection{\@tocline{2}{0pt}{2.5pc}{5pc}{}}
\makeatother

\tableofcontents

\section{Introduction}\label{s:Intro}
We aim to develop the foundations of analysis, geometry, and stochastic analysis on configuration spaces over singular base spaces, for a wide class of reference measures.
We construct a canonical differential structure under which analysis, geometry, and stochastic analysis on configuration spaces can be consistently understood.
In fact, we consider two fundamental structures: the analytic structure lifted from the base space by means of Dirichlet form theory, and the geometric structure of extended metric measure space induced on configuration spaces by the $L^2$-transportation distance.
We fully investigate the interplay between these structure and eventually show that they coincide.
A number of applications to interacting-particle systems and infinite-dimensional metric measure geometry are discussed. 

\smallskip

For the sake of simplicity, throughout this introduction let~$(X,\mssd)$ be a proper complete and separable metric space.
The configuration space $\dUpsilon$ over~$X$ is the set of all locally finite point measures on $X$,
\begin{equation*} %\label{config}
\dUpsilon\eqdef \set{\gamma=\sum_{i=1}^N \delta_{x_i}: x_i\in X\comma N \in \N\cup \set{+\infty}\comma \gamma K<\infty \quad K \Subset X }\fstop
\end{equation*}
The space~$\dUpsilon$ is endowed with the \emph{vague topology}~$\T_\mrmv$, induced by duality with continuous compactly supported functions on~$X$, and with a reference Borel probability measure~$\QP$, occasionally understood as the law of a proper point process on~$X$.

We will focus on two structures on $\dUpsilon$:
\begin{itemize}
\item the Dirichlet form~$\EE{\dUpsilon}{\QP}$, the $L^2(\QP)$-closure of a certain pre-Dirichlet form on a class of cylinder functions, which we will refer to as the \emph{analytic structure};

\item the extended-metric topological measure space~$(\dUpsilon, \T_\mrmv, \mssd_\dUpsilon, \QP)$ induced by the $L^2$-trans\-portation \emph{extended} distance $\mssd_\dUpsilon$ and by~$\QP$, which we will refer to as the \emph{geometric structure}.
\end{itemize}

Our goal is to investigate these structures on~$\dUpsilon$ for significant classes of singular base spaces~$X$, and of invariant measures~$\QP$.
As a main result of this paper, we prove that these structures are equivalent:
\begin{align}\label{equi: main}
\ttonde{\EE{\dUpsilon}{\QP},\dom{\EE{\dUpsilon}{\QP}}} \iff \ttonde{\dUpsilon, \T_\mrmv, \mssd_\dUpsilon, \QP}\comma
 \end{align}
and discuss various applications.
Our results can be summarized as follows:
\begin{enumerate}[$(1)$]
\item we construct the strongly local Dirichlet form~$\EE{\dUpsilon}{\QP}$ for a wide class of measures~$\QP$ and base spaces~$X$ (Prop.~\ref{p:MRLifting}, Thm.~\ref{t:ClosabilitySecond});

\item we identify~$\EE{\dUpsilon}{\QP}$ with a form on the infinite-product space~$X^\tym{\infty}$ (Cor.~\ref{c:Isometry}), and study its finite-dimensional approximations (Cor.~\ref{c:KSUpsilon});

\item\label{i:IntroMainRes:3} for the form~$\EE{\dUpsilon}{\QP}$ we prove a Rademacher Theorem (Thm.~\ref{t:Rademacher}) and different types of Sobolev-to-Lipschitz properties (Thm.s~\ref{t:cSLUpsilon}, \ref{t:dcSLUpsilon});

\item\label{i:IntroMainRes:4} we identify~$\EE{\dUpsilon}{\QP}$ with the Cheeger energy of the extended metric measure space~$(\dUpsilon,\mssd_\dUpsilon,\QP)$ (Thm.~\ref{t:IdentificationCheeger}), and the intrinsic distance of~$\EE{\dUpsilon}{\QP}$ with the $L^2$-transportation distance~$\mssd_\dUpsilon$ (Thm.~\ref{t:StoL2});

\item we prove the $\T_\mrmv$-quasi-regularity of~$\ttonde{\EE{\dUpsilon}{\QP},\dom{\EE{\dUpsilon}{\QP}}}$ (Cor.~\ref{c:RadQReg}) and therefore the existence and the uniqueness (for quasi-every starting point) of a Markov diffusion process~$\mbfM$ with state space~$\dUpsilon$ associated to~$\ttonde{\EE{\dUpsilon}{\QP},\dom{\EE{\dUpsilon}{\QP}}}$;

\item we show that, regardless of the choice of~$\QP$, the distance~$\mssd_\dUpsilon$ universally describes the integral-type Varadhan short-time asymptotic for~$\mbfM$ (Thm.~\ref{t:VaradhanSecond}). %, and an integral upper Gaussian heat kernel estimate (Cor.~\ref{c:HeatKernelEstimateConfig}), both
%with respect to~$\mssd_\dUpsilon$ \purple{for a wide class of base spaces $X$ and invariant measures $\QP$}.
\end{enumerate}

For a large class of measures~$\QP$ many of our results are new even in the case of configuration spaces over standard Euclidean spaces.

\subsection{Setting} Beyond this introduction we will work in far greater generality, assuming~$X$ to be either:
\begin{itemize}
\item a countably generated separated measurable space, in~\S\ref{s:FundMeasure};
\item a separable metrizable Luzin space, with no distance assigned, in~\S\ref{s:Analytic};
\item a complete and separable metric space, in~\S\ref{s:Geometry}.
\end{itemize}
We never assume local compactness (nor ---~\emph{a fortiori}~--- properness), replacing the family of compact subsets of~$X$ with a localizing ring in the sense of Kallenberg (Dfn.~\ref{d:LS}).

\subsubsection{Base spaces}
A \emph{metric local diffusion space} (in short: \MLDS) is any atomless Radon metric measure space~$(X, \mssd, \mssm)$, endowed with a local square field operator~$\cdc$ defined on some algebra~$\Dz$ of continuous compactly supported functions and taking values in the space of $\mssm$-classes of measurable functions on~$X$.
In particular, as it is typical on singular spaces, the function~$\cdc(f)$ will be defined \emph{only $\mssm$-a.e.}, even for~$f$ in the core~$\Dz$.
This fact introduces various difficulties, as discussed in~\S\ref{sss:Intro:ConstructionWellposedness} below.
Since~$\cdc$ is local,~$\cdc(f)$ is $\mssm$-a.e.-vanishing outside the compact~$\supp f$, for every~$f\in\Dz$.

We further assume that the pre-Dirichlet form obtained by integrating~$(\cdc,\Dz)$ w.r.t.~$\mssm$ is closable, and that its closure~$\ttonde{\EE{X}{\mssm},\dom{\EE{X}{\mssm}}}$ is a quasi-regular Dirichlet form.
This assumption will mainly be used in our subsequent work to discuss the Markov diffusion process on $X$ properly associated with~$\EE{X}{\mssm}$, and its $\QP$-invariant counterpart on~$\dUpsilon$.

\subsubsection{Configuration spaces}
Let us now turn to the choice of a reference measure on~$\dUpsilon$.
For~$\QP$ we consider different assumptions, granting well-posedness and closability of the form~$\ttonde{\EE{\dUpsilon}{\QP},\dom{\EE{\dUpsilon}{\QP}}}$.

These assumptions reflect different approaches to the study of~$\EE{\dUpsilon}{\QP}$, namely
\begin{itemize}
\item a `\emph{transfer method}', relying on a Georgii--Nguyen--Zessin-type (GNZ) formula for~$\QP$ in order to `transfer' the closability of~$\EE{X}{\mssm}$ to that of~$\EE{\dUpsilon}{\QP}$;

\item `\emph{conditional closability}', relying on the construction of a consistent system of Dirichlet forms~$\EE{\dUpsilon(K)}{}$ on the spaces~$\dUpsilon(K)$ for~$K\Subset X$.
The latter is obtained by conditioning~$\QP$ on the complement of~$\dUpsilon(K)$ in~$\dUpsilon$, that is, by a \emph{specification of~$\QP$} in the sense of Preston~\cite{Pre76};

\item an `\emph{approximation method}', transfering the form~$\EE{\dUpsilon}{\QP}$ to a form on the infinite-product space~$X^\tym{\infty}$ and subsequently approximating it by forms on~$X^\tym{n}$ as~$n\to\infty$.
\end{itemize}

As we will thoroughly discuss in \S\ref{ss:ExamplesConfig}, these assumptions, together, include most of the relevant examples in the literature.

The `transfer method' has been a successful tool in establishing closability both on configuration spaces~\cite{MaRoe00}, and beyond, e.g., on spaces of measures~\cite{KonLytVer15}, and of probability measures~\cite{LzDS17+}.
In spite of its numerous applications to other related constructions (e.g., the computation of explicit expressions for the generator of~$\EE{\dUpsilon}{\QP}$), this method has limited applicability to measures satisfying a GNZ formula (see Assumption~\ref{ass:Closability}), as are for instance \emph{grand-canonical Gibbs measures} in the sense of~\cite{AlbKonRoe98b}.

`Conditional closability' has been verified mostly on a case-by-case basis, among others: when~$\QP$ is the law of an infinite singularly interacting Brownian system~\cite{Osa96}, the law of a determinantal/permanental point process~\cite{Yoo05}, and more generally for some particular quasi-Gibbs measures~\cite{Osa13}.
The strength of this technique lies in the fact that~$\dUpsilon(K)$ can be realized as the disjoint union of all symmetric products~$K^\sym{n}$, $n\in\N$, of the set~$K$.
Thus, the study of the forms~$\EE{\dUpsilon(K)}{\QP}$ may be treated by `finite-dimensional' techniques and is ---~in essence~--- only as difficult as that of the form~$\EE{X}{\mssm}$.

The `approximation method' is one main novelty of this work.
It relies on Assumption~\eqref{eq:Intro:AC} below, relating~$\QP$ to the product measures~$\mssm^\otym{n}$ on~$X^\tym{n}$, $n\in\N$, and on
the convergence of approximating forms on~$X^\tym{n}$ to a form on~$X^\tym{\infty}$ corresponding to~$\EE{\dUpsilon}{\QP}$.
Similarly to~\cite{LzDS17+}, the covergence of the approximating forms is phrased in terms of the Kuwae--Shioya topology on nets of Hilbert spaces~\cite{KuwShi03}.
We now proceed to explain Assumption~\eqref{eq:Intro:AC} in greater detail.

\subsubsection{Labeling maps and Assumption~$($\ref{eq:Intro:AC}$)$} \label{subsub:labelling}
Our main interest lies in the subspace~$\dUpsilon^\sym{\infty}$ of infinite particles configurations on an \MLDS~$X$.
The space~$\dUpsilon^\sym{\infty}$ may be regarded as a subset of the quotient of~$X^\tym{\infty}$ by the action of the (projective) infinite symmetric group~$\mfS^\infty$ permuting coordinates.
Any right inverse of the quotient projection restricts to a \emph{labeling map} (also cf.\ Fig.~\ref{fig:LabelingMaps} in~\S\ref{ss:TopConfig} below)
\begin{equation}\label{eq:Intro:LabelingMap}
\lb\colon \dUpsilon^\sym{\infty}\longrar X^\tym{\infty}\comma \qquad \gamma=\sum_i \delta_{x_i} \longmapsto \seq{x_i}_i\fstop
\end{equation}

Labeling maps will play a key role throughout our study, for both the analytic and the geometric structure. 
It is worth to stress here that we never rely on labeling maps to tranfer \emph{topological} properties from~$\dUpsilon$ to~$X^\tym{\infty}$ or viceversa, since these maps are never continuous (Rem.~\ref{r:DiscontinuityL}, Prop.~\ref{p:Discontinuity-l}).
Furthermore, labeling maps are instrumental in formulating one of our main assumptions for the reference measure~$\QP$. 
We say that a measure~$\QP$ on~$\dUpsilon$, concentrated on~$\dUpsilon^\sym{\infty}$, satisfies the (\emph{marginal}) \emph{absolute continuity} condition (see Ass.~\ref{d:ass:AC} for details) if
\begin{equation}\tag{$\mathsf{AC}$}\label{eq:Intro:AC}
(\tr^n\circ \lb)_\pfwd \QP \ll \mssm^\otym{n}\comma \qquad n\in\N \comma
\end{equation}
where~$\tr^n\colon \seq{x_i}_i\mapsto \seq{x_1,\dotsc, x_n}$, and~$f_\pfwd\mu$ denotes the push-forward of a measure~$\mu$ via a measurable map~$f$.
This assumption is satisfied by the laws of most point processes in the literature, including e.g.\ canonical and grand-canonical Gibbs measures and some determinantal/permanental measures (see~\S\ref{ss:Intro:ex} below).

\subsection{Analytic Structure}
Let us now turn to the construction of the main object of our interest, a strongly local Dirichlet form~$\ttonde{\EE{\dUpsilon}{\QP},\dom{\EE{\dUpsilon}{\QP}}}$ on $\dUpsilon$ with square field operator $\cdc^\dUpsilon$ lifted from the square-field operator~$\cdc$ on the base space $X$.

\subsubsection{Construction and well-posedness of~\texorpdfstring{$\cdc^\dUpsilon$}{CdcUpsilonEll}}\label{sss:Intro:ConstructionWellposedness}
In~\cite{MaRoe00}, Z.-M.~Ma and M.~R\"ockner introduced a general method to lift a square-field operator~$\cdc$ on a measurable space to a square field operator~$\cdc^\dUpsilon$ on the corresponding configuration space.
Whereas their construction works even without the need for a topology on~$X$, they assume~$\cdc$ to be \emph{pointwise defined}, i.e., taking values in the space of Borel measurable functions, rather than of $\mssm$-classes.
This is necessary in order for functions of the form
\begin{align*}
f^\trid\colon \gamma\longmapsto \int f\diff\gamma \qquad \text{and} \qquad \cdc(f,g)^\trid\colon\gamma\longmapsto \int \cdc(f,g)\diff\gamma\comma \qquad f,g\in\Dz\subset \Cc(X)
\end{align*}
to be well-defined on~$\dUpsilon$, since~$\gamma$ is a purely atomic measure and points in~$X$ are $\mssm$-negligible by assumption.
However, as already observed above, any natural square field~$\cdc(f,g)$ on a general non-smooth space is typically defined only \emph{almost everywhere}.
We will show how to adapt the construction of~\cite{MaRoe00} to a square field operator~$\cdc$ defined only almost everywhere, by resorting to the theory of \emph{liftings} (Prop.~\ref{p:MRLifting}).

The lifted square field~$\cdc^\dUpsilon$ is naturally defined, by chain rule, on the family~$\Cyl{\Dz}$ of \emph{cylinder functions} of the form
\begin{equation}\label{eq:Intro:CylinderF}
u\colon \gamma\longmapsto F\tonde{f_1^\trid\gamma,\dotsc, f_k^\trid\gamma}\comma \qquad f_i\in \Dz\comma i\leq k\in \N\comma F\in\Cb^\infty(\R^k) \fstop
\end{equation}
It is one cornerstone of the construction in~\cite{MaRoe00} that~$\cdc^\dUpsilon$ is well-defined in the following sense: for every~$u\in\Cyl{\Dz}$, the function~$\cdc^\dUpsilon(u)$ is independent of the choice of~$F$ and~$\seq{f_i}_{i\leq k}$ representing~$u$ as in~\eqref{eq:Intro:CylinderF}.

\subsubsection{Well-posedness and closability of~\texorpdfstring{$\EE{\dUpsilon}{\QP}$}{EE}}
A local pre-Dirichlet form is defined on~$\dUpsilon$ by
\begin{equation*}
\EE{\dUpsilon}{\QP}(u,v)\eqdef \int_\dUpsilon \cdc^\dUpsilon(u,v)\diff\QP \comma \qquad u,v\in\Cyl{\Dz} \fstop
\end{equation*}
As discussed in~\S\ref{s:FundMeasure}, its well-posedness as a bilinear form on $L^2(\QP)$-classes is \emph{not} obvious.
It was shown in~\cite{MaRoe00} together with its closability, for measures~$\QP$ satisfying a GNZ-type formula.
In~\S\ref{s:Analytic}, we will extend these results to measures satisfying~\eqref{eq:Intro:AC} and the conditional closability Assumption~\ref{ass:ConditionalClosability}, thus greatly expanding the class of examples for~$\QP$ in the non-smooth setting (Thm.~\ref{t:ClosabilitySecond}).

As it turns out, the closure~$\ttonde{\EE{\dUpsilon}{\QP},\dom{\EE{\dUpsilon}{\QP}}}$ of~$\ttonde{\EE{\dUpsilon}{\QP},\Cyl{\Dz}}$ is a symmetric quasi-regular strongly local conservative Dirichlet form on~$(\dUpsilon,\T_\mrmv,\QP)$ (Cor.~\ref{c:RadQReg}).
Thus, it is properly associated with a Markov diffusion process~$\mbfM$ with state space~$\dUpsilon$ and invariant measure~$\QP$.
Here, we will focus on the analytic aspects of the theory, postponing the study of~$\mbfM$ to future work in this series.

\subsubsection{Finite-dimensional approximation of \texorpdfstring{$\EE{\dUpsilon}{\QP}$}{EUpsilon}.}\label{sss:ApproxForm}
Each labeling map~$\lb$ as in~\eqref{eq:Intro:LabelingMap} allows us to transfer the form~$\EE{\dUpsilon}{\QP}$ to a form~$\EE{\infty}{\QP^\infty}$ on~$L^2(\QP^\infty)$ over~$X^\tym{\infty}$, where we let~$\QP^\infty\eqdef \lb_\pfwd\QP$ for some fixed labeling map~$\lb$.
Standard techniques in metric measure geometry are unsuitable to the study of the extended-metric measure space~$(X^\tym{\infty},\mssd_\tym{\infty},\QP^\infty)$.
Partly, this is because~$\QP^\infty$ is concentrated on the coset~$\lb(\dUpsilon)$ of the quotient projection, and the latter is `too small' in~$X^\tym{\infty}$ in any topological sense.
Nonetheless, one can approximate the forms~$\EE{\infty}{\lb_\pfwd\QP}$ by some `finite-dimensional' counterparts~$\EE{n}{\QP^n}$, where~$\QP^n\eqdef (\tr^n\circ\lb)_\pfwd\QP$.
The approximation will be shown to hold in the sense of Kuwae--Shioya convergence (Cor.~\ref{c:KSUpsilon}).

The study of~$\EE{n}{\QP^n}$ for a single labeling map~$\lb$ is not sufficient to characterize the form~$\EE{\dUpsilon}{\QP}$.
The power of this approximation ---~playing a crucial role in the proof of the Rademacher property~--- consists indeed in the possibility to consider all labeling maps~$\lb$ at once, thus recovering the full information of~$\EE{\dUpsilon}{\QP}$ (see Fig.~\ref{fig:KS} in~\S\ref{ss:AnalyticForms} below). 
This is possible because of Assumption~\eqref{eq:Intro:AC}, precisely granting that~$\QP^n$-representatives of functions may be addressed as $\mssm^\otym{n}$-representatives, consistently for every~$\lb$.

\subsection{Geometry}\label{ss:Intro:Geometry}
Beside the vague topology~$\T_\mrmv$, the configuration space~$\dUpsilon$ over~$(X,\mssd)$ inherits from it the \emph{$L^2$-transportation} (also: \emph{$L^2$-Wasserstein}) {distance}
\begin{equation*}
\mssd_\dUpsilon(\gamma, \eta) \eqdef \inf \tonde{\int_{X^\tym{2}} \mssd^2(x,y)\diff\cpl(x,y)}^{1/2}\comma 
\end{equation*}
where the infimum is taken over all measures~$\cpl$ on~$X^\tym{2}$ with marginals~$\gamma$ and~$\eta$.
It is readily verified that~$\mssd_\dUpsilon$ may attain the value~$+\infty$, and is thus an \emph{extended} distance.
We study the metric topological properties of~$(\dUpsilon,\T_\mrmv,\mssd_\dUpsilon)$ in great detail (Prop.~\ref{p:W2Upsilon}, Prop.s~\ref{p:W2Upsilon2}), extending results previously obtained for configuration spaces over manifolds by M.~R\"ockner and A.~Schied~\cite{RoeSch99}.
In light of these results, the space~$(\dUpsilon,\T_\mrmv,\mssd_\dUpsilon)$ is an extended-metric topological space in the sense of Ambrosio--Erbar--Savar\'e~\cite{AmbErbSav16} (see Prop.~\ref{p:ConfigEMTS}), thus fitting the general framework of~\cite{AmbErbSav16}.

Whereas~$(\dUpsilon,\T_\mrmv, \QP)$ is a Polish Radon probability space, the corresponding extended-metric topological measure space~$(\dUpsilon,\T_\mrmv, \mssd_\dUpsilon,\QP)$ is not an amenable object:
\begin{itemize}
\item $(\dUpsilon,\mssd_\dUpsilon)$ has uncountably many components at infinite distance from each other;
\item $\mssd_\dUpsilon$ is $\T_\mrmv$-lower semicontinuous, yet not continuous;
\item typically (depending on~$\QP$), every~$\mssd_\dUpsilon$-ball is $\QP$-negligible;
\item $\mssd_\dUpsilon$-Lipschitz functions are in general neither finite, nor $\T_\mrmv$-continuous, nor $\QP$-measurable.
\end{itemize}
These pathological properties cause various fundamental problems to develop non-trivial metric measure theory on $\dUpsilon$. We explain below the four difficulties regarding these issues.

\subsubsection{Measurability issues}
The lack of measurability of $\mssd_\dUpsilon$-Lipschitz functions affects the measurability of various important functions in metric geometry, in particular: the point-to-set $\mssd_\dUpsilon$-distance, and, consequently, metric cut-off functions, McShane extensions, and other functions constructed from it.
We provide a careful treatment of the several measurability issues arising, proving the universal measurability of the point-to-set distance induced by~$\mssd_\dUpsilon$ (Cor.~\ref{c:MeasurabilitydU}) and of the McShane extensions of (restrictions of) $\mssd_\dUpsilon$-Lipschitz functions (Prop.~\ref{p:BorelMcShane}).

\subsubsection{Lack of Lipschitz continuity of cylinder functions}
One main difficulty in establishing relations between Dirichlet-form theory applied to~$\ttonde{\EE{\dUpsilon}{\QP},\dom{\EE{\dUpsilon}{\QP}}}$ and metric measure theory applied to~$(\dUpsilon,\mssd_\dUpsilon,\QP)$
lies in the fact that their natural cores, respectively the space of cylinder functions~$\Cyl{\Dz}$ and the space of bounded $\mssd_\dUpsilon$-Lipschitz functions, have small intersection.
Indeed, cylinder functions whereas $\T_\mrmv$-continuous are typically not $\mssd_\dUpsilon$-Lipschitz (see Example~\ref{e:NonLipCyl}).
\emph{A~priori}, there is therefore no relation between these two theories.

In order to overcome this difficulty, we introduce a family of localizing subsets~$\Xi\subset \dUpsilon$, and show that each cylinder function is \emph{locally} Lipschitz with respect to this localization. 
In particular, one can construct an exhausting sequence~$\Xi_n$ of subsets of~$\dUpsilon$, depending on~$u$, such that the restriction~$u_n$ of~$u$ to~$\Xi_n$ is $\mssd_\dUpsilon$-Lipschitz. 
Furthermore, we show that the McShane extensions~$\bar u_n$ of the restrictions~$u_n$ form a sequence of \emph{measurable} $\mssd_\dUpsilon$-Lipschitz functions coinciding with~$u$ on~$\Xi_n$ (Prop.~\ref{p:BorelMcShane}) and thus approximating~$u$ in a strong sense.
We note that standard localization arguments involving~$\mssd_\dUpsilon$ do not work, since as noted above every~$\mssd_\dUpsilon$-ball is typically $\QP$-negligible.

\subsubsection{Lack of isometric labeling maps}
In order to transfer metric properties from the base space~$X$ to the configuration space $\dUpsilon$, one ought to use labeling maps consistent with the metric structure on $X$. However, there is no isometric labeling map from $(\dUpsilon,\mssd_\dUpsilon)$ to $(X^\tym{\infty},\mssd_\tym{\infty})$, which makes the situation complicated.

In order to settle this problem, we make full use of \emph{radially isometric labeling map} in combination with the marginal absolute continuity assumption~\eqref{eq:Intro:AC}. 
For every~$\gamma$ in~$\dUpsilon$ there exists a labeling map~$\lb_\mbfx$ \emph{radially isometric} around~$\gamma$  (Prop.~\ref{p:PoissonLabeling}), satisfying that is
\begin{align*}
\mssd_\tym{\infty}\ttonde{\lb_\mbfx(\gamma),\lb_\mbfx(\eta)}=\mssd_\dUpsilon(\gamma,\eta)\comma \qquad \mbfx=\lb_\mbfx(\gamma)\comma\qquad \eta\in\dUpsilon \fstop
\end{align*}
Since no radially isometric labeling map extends to a global isometry of~$(\dUpsilon,\mssd_\dUpsilon)$ into~$(X^\tym{\infty},\mssd_\tym{\infty})$, it is essential to \emph{simultaneously} treat all labeling maps in a consistent way, which will be possible under the marginal absolute continuity assumption~\eqref{eq:Intro:AC}.

\subsubsection{Several inequivalent Cheeger Energies}
The extended-metric measure space~$(\dUpsilon,\mssd_\dUpsilon,\QP)$ is endowed with several, a priori inequivalent, definitions of \emph{Cheeger energies}~$\Ch[\mssd_\dUpsilon,\QP]$, convex energy functionals on~$L^2(\QP)$ induced by~$\mssd_\dUpsilon$, introduced on general extended-metric measure spaces by L.~Ambrosio, N.~Gigli, and G.~Savar\'e in~\cite{AmbGigSav14}, L.~Ambrosio, M.~Erbar, G.~Savar\'e~\cite{AmbErbSav16}, and G.~Savar\'e~\cite{Sav19}.  
Making use of results in~\cite{AmbGigSav14, AmbGigSav14b, AmbErbSav16, Sav19}, we show that these Cheeger energies  
for~$(\dUpsilon,\mssd_\dUpsilon,\QP)$ coincide (Prop.~\ref{p:ConsistencyCheeger}), which will play a significant role for investigating interplays between the Dirichlet form $\EE{\dUpsilon}{\QP}$ and $\Ch[\mssd_\dUpsilon,\QP]$ on $\dUpsilon$. We stress that the Cheeger energy constructed in this abstract way could in general be trivial (e.g., identically vanishing) or non-quadratic. 
Therefore, the metric measure theory based on $\Ch[\mssd_\dUpsilon,\QP]$ will be informative only after we identify $\Ch[\mssd_\dUpsilon,\QP]$ with the non-trivial Dirichlet form~$\EE{\dUpsilon}{\QP}$.

\subsection{Main results: Interplay between analysis and geometry} 
Our identification of the analytic and geometric structures on~$\dUpsilon$ will consist of two different results:
\begin{itemize}
\item the identification of the Dirichlet form~$\EE{\dUpsilon}{\QP}$ with~$\Ch[\mssd_\dUpsilon,\QP]$;

\item the identification of~$\mssd_\dUpsilon$ with the \emph{intrinsic distance~$\mssd_\QP$} of~$\EE{\dUpsilon}{\QP}$, viz.
\begin{equation*}
\mssd_\QP(\gamma,\eta)\eqdef \inf\set{u(\gamma)-u(\eta) : u\in\dom{\EE{\dUpsilon}{\QP}}\cap\Cb(\T_\mrmv) \comma \cdc^\dUpsilon(u)\leq 1 \as{\QP}}\fstop
\end{equation*}
\end{itemize}

Neither of these identifications is a consequence of the other.
Both of them are established by a detailed study of Rademacher- and Sobolev-to-Lipschitz-type properties on~$\dUpsilon$.

\subsubsection{Rademacher- and Sobolev-to-Lipschitz-type properties}
Denote by~$\Li[\mssd_\dUpsilon]{u}$ the global Lipschitz constant of~$u\colon\dUpsilon\rar\R$, and by~$\bLip(\dUpsilon,\mssd_\dUpsilon)$ the space of all bounded \emph{measurable} $\mssd_\dUpsilon$-Lipschitz functions.
We say that the form~$\EE{\dUpsilon}{\QP}$ has
\begin{itemize}
\item the \emph{Rademacher property}, if~$\bLip(\dUpsilon,\mssd_\dUpsilon)\subset \dom{\EE{\dUpsilon}{\QP}}$ and
\begin{equation*}
\norm{\cdc^\dUpsilon(u)}_{L^\infty(\QP)}\leq \Li[\mssd_\dUpsilon]{u}^2\comma\qquad u\in \bLip(\dUpsilon,\mssd_\dUpsilon)\semicolon
\end{equation*}

\item the (\emph{$\T_\mrmv$-continuous-})\emph{Sobolev-to-Lipschitz property}, if
\begin{equation*}
\Li[\mssd_\dUpsilon]{u}^2\leq \norm{\cdc^\dUpsilon(u)}_{L^\infty(\QP)} \comma\qquad u\in \dom{\EE{\dUpsilon}{\QP}} \cap \Cb(\dUpsilon, \T_\mrmv)\fstop
\end{equation*}
\end{itemize}

We show a \emph{superposition principle} for the Rademacher and the Sobolev-to-Lipschitz property of configuration spaces endowed with~$\mssd_\dUpsilon$.
That is, we show that either property holds for~$\EE{\dUpsilon}{\QP}$ if it holds for the reference form~$\EE{X}{\mssm}$ on the base space.
For the Rademacher property, this is achieved by means of the finite-dimensional approximating forms~$\EE{n}{\QP^n}$ discussed in~\S\ref{sss:ApproxForm} (Thm.~\ref{t:Rademacher}).
A treatment of the Sobolev-to-Lipschitz property on~$\dUpsilon$ is more involved, and we present some variations of it (Thm.s~\ref{t:cSLUpsilon}, \ref{t:dcSLUpsilon}).

We stress that, while fixing the single distance~$\mssd_\dUpsilon$ as the reference distance, these properties hold for a variety of measures~$\QP$ on~$\dUpsilon$ mutually singular with respect to each other.
The same holds for the Varadhan short-time asymptotics for~$\EE{\dUpsilon}{\QP}$ which is a consequence of both.
This is in sharp contrast with the case of finite-dimensional spaces, and in particular of Euclidean spaces.
Indeed, as shown by G.~De Philippis and F.~Rindler in~\cite{DePRin16}, if~$\mssm$ is a measure on~$\R^n$ such that every Lipschitz function with respect to the Euclidean distance is $\mssm$-a.e.-differentiable, then $\mssm$ is absolutely continuous to the standard $n$-dimensional Lebesgue measure.

This lack of rigidity in the determination of the absolute-continuity class of~$\QP$ by, e.g., the Rademacher property is a characteristic feature of infinite-dimensional spaces, and has been already observed in loop spaces by M.~Hino and J.~A.~Ram\'irez~\cite{HinRam03}, and spaces of probability measures~\cite{LzDS19b}.

\subsubsection{Identification of forms}
We may thus proceed to study the interplay between~$\Ch[\mssd_\dUpsilon,\QP]$ and~$\EE{\dUpsilon}{\QP}$. 
Relying on~\eqref{eq:Intro:AC} and the Rademacher property for~$\dUpsilon$ we show, as one main result of the paper (Thm.~\ref{t:IdentificationCheeger}) that
\begin{equation}\label{eq:Intro:IdForms}
\ttonde{\EE{\dUpsilon}{\QP},\dom{\EE{\dUpsilon}{\QP}}}=\ttonde{\Ch[\mssd_\dUpsilon,\QP],\dom{\Ch[\mssd_\dUpsilon,\QP]}} \fstop
\end{equation}
In the same spirit as in the previous section, this result will be shown under an analogous assumption for the base space~$(X,\mssd,\mssm,\cdc)$.
Namely that, on each of its $n$-fold products~$X^\tym{n}$, the product square field operator~$\cdc^\otym{n}$ of a Lipschitz function~$f$ coincides with the slope of~$f$ w.r.t.\ the product distance~$\mssd_\tym{n}$ (see Ass.~\ref{d:ass:Tensorization}).
Importantly, in showing~\eqref{eq:Intro:IdForms} we do not assume~$\Ch[\mssd_\dUpsilon,\QP]$ to be a quadratic functional.

This identification brings together the analytic and the geometric aspects of the theory in a powerful way:
\begin{itemize}
\item it implies that~$\Ch[\mssd_\dUpsilon,\QP]$ is non-trivial and quadratic, allowing us to access a well-established set of tools from metric geometry;
\item we may now make use of those same tools in order to study the stochastic process~$\mbfM$ properly associated to~$\EE{\dUpsilon}{\QP}$;
\item $\mbfM$ can thus be regarded as both the $\QP$-invariant infinite-interacting-particles diffusion on~$X$ and the Brownian motion of the extended metric measure space~$(\dUpsilon,\mssd_\dUpsilon,\QP)$.
\end{itemize}

\subsubsection{Identification of distances}
Independently and in parallel to establishing~\eqref{eq:Intro:IdForms}, we show the coincidence of the extended distances (Thm.~\ref{t:StoL2})
\begin{equation}\label{eq:Intro:DistanceIdentification}
\mssd_\dUpsilon=\mssd_\QP \fstop
\end{equation}
This identification is obtained by separatly studying the inequalities~`$\leq$' and~`$\geq$', respectively corresponding to weak forms of the Rademacher and Sobolev-to-Lipschitz properties.
It implies that the intrinsic distance~$\mssd_\QP$ is a non-trivial complete extended distance on~$\dUpsilon$.

\medskip

As an application of all the above results, we conclude that (cf.~Fig.~\ref{fig:diagram})

\smallskip

\begin{center}
\noindent \textbf{
The analytic structure~$\EE{\dUpsilon}{\QP}$ and the geometric structure~$(\dUpsilon, \T_\mrmv, \mssd_\dUpsilon, \QP)$ coincide.
}
\end{center}

\begin{figure}[htb!]
\begin{equation*}
\xymatrix @C=1pc { 
\EE{\dUpsilon}{\QP} \ar@{->}[r]&  \ttonde{\dUpsilon, \mssd_{\EE{\dUpsilon}{\QP}},  \QP}  \ar@{~>}[r] \ar@{=}[d]& \Ch[\mssd_{\EE{\dUpsilon}{\QP}},\QP] \ar@{=}[r] \ar@{=}[d] &\EE{\dUpsilon}{\QP}
\\
&(\dUpsilon, \mssd_\dUpsilon,  \QP)  \ar@{~>}[r] &  \Ch[\mssd_\dUpsilon,\QP] \ar@{->}[r] &  \ttonde{\dUpsilon, \mssd_{\Ch[\mssd_\dUpsilon,\QP]},  \QP} \ar@{=}[r]&(\dUpsilon, \mssd_\dUpsilon,  \QP)
}
\end{equation*}
\caption{
A summary of the main identification results in this work, marked by~`$=$'.
For any Dirichlet form~$E$ on~$L^2(\QP)$ we write~$\mssd_E$ to indicate the instrinsic distance of the form (in particular,~$\mssd_{\EE{\dUpsilon}{\QP}}=\mssd_\QP$).
Solid arrows~`$\to$' read as `take the intrinsic distance', and curly arrows~`$\rightsquigarrow$' as `take the Cheeger energy'.
}
\label{fig:diagram}
\end{figure}

\subsection{Applications, byproducts, and examples}\label{ss:Intro:ex}
The Rademacher and Sobolev-to-Lipschitz property, and the above identifications all have a number of important consequences, some of which are listed below.

\subsubsection{Quasi-Regularity} 
The quasi-regularity of Dirichlet forms on~$\dUpsilon$ is of particular importance for the study of Markov processes with~$\dUpsilon$ as state space.
Indeed, by the general theory of Dirichlet forms, if~$\EE{\dUpsilon}{\QP}$ is a quasi-regular Dirichlet form for the vague topology~$\T_\mrmv$, then there exists a $\QP$-special standard Markov processes~$\mbfM$ that is \emph{properly associated} with~$\EE{\dUpsilon}{\QP}$, e.g.~\cite[Dfn.~IV.2.5(i)]{MaRoe92} and that is unique for quasi-every starting point.
The strong locality of~$\EE{\dUpsilon}{\QP}$ is then equivalent to the almost sure $\T_\mrmv$-continuity of the sample paths of~$\mbfM$.
As a consequence,~$\mbfM$ is a $\QP$-invariant Markov diffusion, describing the evolution of an infinite system of Brownian particles on~$X$, with interaction driven by~$\QP$.

Proposition~\ref{p:QRegSLoc} and Corollary~\ref{c:RadQReg} provide a systematic approach to the quasi-regularity of~$\EE{\dUpsilon}{\QP}$, showing that the Rademacher property for the base space implies both quasi-regularity and strong locality on the configuration space.
Firstly, this extends quasi-regularity results in the existing literature to a much wider class of base spaces~$X$ and invariant measures~$\QP$. 
Secondly, it provides a systematic proof of the quasi-regularity compared to the \emph{ad hoc} proofs in the existing literature.

We stress that our result provides a very mild sufficient condition ---~namely, only the closability~\ref{ass:CWP} of the pre-form~$\ttonde{\EE{\dUpsilon}{\QP}, \Cyl{\Dz}}$ and \eqref{eq:Intro:AC}~--- for the quasi-regularity of~$\EE{\dUpsilon}{\QP}$ when the base space $X$ satisfies the Rademacher property. This gives the existence and the uniqueness of weak solutions to a large class of stochastic differential equations of infinite interacting particle systems.  In particular, every closable quasi-Gibbs measure $\QP$ satisfies the quasi-regularity (Corollary~\ref{c:EucQReg}).

\subsubsection{Wasserstein universality of the Varadhan Short-Time Asymptotic} 
When the base space $X=\R^d$ is the standard Euclidean space and the invariant measure $\QP=\PP_d$ is the Poisson measure with instensity the $d$-dimensional Lebesgue measure~$\Leb^d$, it is natural to expect that the integral-type Varadhan short-time asymptotic~\eqref{eq:Varadhan} for the heat semigroup~$\TT{\dUpsilon}{\QP}_t$ holds w.r.t.~$\mssd_\dUpsilon$, since there is no correlation among particles, and since the short-time asymptotic of each one-particle motion is governed by the Euclidean distance.
Indeed, this was shown by T.~S.~Zhang~\cite{Zha01} in the case of the one-dimensional Euclidean space $X=\R$ for the Poisson measure $\QP=\PP_1$. 
However, when~$\QP$ is not the Poisson measure, i.e., when particles are correlated, determining what governs the short-time asymptotic of the (infinite) particle system is ---~in principle~--- highly non-trivial, and no such characterization seems to be known up to now.
Indeed, most of the usually considered invariant measures (canonical Gibbs, laws of determinantal point processes, etc.) are singular to the Poisson measure~$\PP$ and thus there is no obvious way to relate the short-time asymptotic in the case~$\QP=\PP$ with other general invariant measures.

In Theorem~\ref{t:VaradhanSecond}, we show that the Varadhan short-time asymptotic holds w.r.t.\ the $L^2$-trans\-porta\-tion distance~$\mssd_\dUpsilon$ under mild assumptions for the base space~$X$ (including e.g., RCD spaces and ideal sub-Riemannian manifolds) and the invariant measure~$\QP$ (e.g., in the case of $X=\R^n$, a large class of Gibbs measures, laws of determinantal point processes such as: $\mathrm{sine}_\beta$, $\mathrm{Airy}_\beta$, $\mathrm{Bessel}_{\alpha, \beta}$, Ginibre, etc.).
%This result shows that the $L^2$-transportation distance~$\mssd_\dUpsilon$ \emph{universally} describes the short-time asymptotic of~$\TT{\dUpsilon}{\QP}_t$, regardless of the choice of~$\QP$. 

It is notable that all the aforementioned classes of invariant measures, which are in general singular with respect to each other, belong to the same \emph{$\mssd_\dUpsilon$-universality class}, in the sense that the short-time asymptotic of the corresponding semigroups is universally governed by the same distance~$\mssd_\dUpsilon$, regardless of the choice of~$\QP$.
For singular measures, this universality phenomenon ---~a characteristic feature of infinite-dimensional spaces~--- was not known even in the case of~$X=\R^n$.

\subsubsection{Further applications}
We obtain a number of further applications ---~those marked by~$\circ$ will appear in forthcoming work:
\begin{itemize}
\item[$\bullet$] \emph{Gaussian Upper Bound}: An integral Gaussian upper heat-kernel estimate in terms of the set-to-set $\mssd_\dUpsilon$-distance holds for~$\EE{\dUpsilon}{\QP}$ (Cor.~\ref{c:HeatKernelEstimateConfig}).
Additionally, this upper bound is explicitly computable for relevant classes of sets.

\item[$\bullet$] \emph{Finiteness of $\mssd_\dUpsilon$}: When~$\EE{\dUpsilon}{\QP}$ is irreducible, the set-to-set $\mssd_\dUpsilon$-distance is always finite on sets of positive $\QP$-measure (Cor.~\ref{c:DistanceIrreducibile}).

\item[$\circ$] Various Sobolev spaces on~$\dUpsilon$ can be identified with~$\dom{\EE{\dUpsilon}{\QP}}$.

\item[$\circ$] General tightness criteria for interacting diffusions on~$\dUpsilon$.
\end{itemize}

Additionally, letting~$\QP=\PP_\mssm$ be the law of a \emph{Poisson point process} on~$X$ with intensity~$\mssm$, in forthcoming work we will obtain:
\begin{itemize}
\item[$\circ$] the essential self-adjointness and Markov uniqueness of the generator~$\ttonde{\LL{\dUpsilon}{\PP_\mssm},\dom{\LL{\dUpsilon}{\PP_\mssm}}}$ of~$\ttonde{\EE{\dUpsilon}{\PP_\mssm},\dom{\EE{\dUpsilon}{\PP_\mssm}}}$ on various families of cylinder functions over singular or infinite-dimensional spaces~$X$, including, e.g., path and loop spaces over manifolds;
\item[$\circ$] the $L^\infty(\PP_\mssm)$-to-$\bLip(\dUpsilon,\mssd_\dUpsilon)$ regularization property for the heat semigroup;
\item[$\circ$] synthetic Ricci-curvature lower bounds for the space~$(\dUpsilon,\mssd_\dUpsilon,\PP_\mssm)$ beyond the case that $X$ is a manifold.
\end{itemize}

\paragraph{A byproduct: Tensorization of the Rademacher property}
In the process of proving the Rade\-macher theorem on~$\dUpsilon$, we show the tensorization property of the Rademacher theorem on the base space~$X$ (Thm.~\ref{t:Tensor}).
This fact is of independent interest in metric measure geometry.
Up to now, it has been obtained when the Dirichlet form is the Cheeger energy: under synthetic Ricci-curvature lower bounds by L.~Ambrosio, N.~Gigli, and G.~Savar\'e in~\cite[Thm.~6.19]{AmbGigSav14b}, and under volume doubling and the weak $(1,2)$-Poincar\'e inequality by L.~Ambrosio, A.~Pinamonti, and G.~Speicht in~\cite[Thm.~3.4]{AmbPinSpe15}.
Our theorem (Thm.~\ref{t:Tensor}) deals with general Dirichlet forms, not necessary Cheeger energies, and also does not require any geometric assumption of~$X$.  

\subsubsection{Examples}
We extensively review examples for both the base spaces and the reference measures~$\QP$ in~\S\ref{s:Examples}.
Concerning base spaces, the theory developed in this work applies to:
\begin{itemize}
\item domains in Euclidean spaces, with arbitrary (e.g., Dirichlet, Neumann) boundary conditions;
\item complete weighted Riemannian manifolds;
\item ideal sub-Riemannian manifolds;
\item $\MCP(K,N)$ spaces with $K\in\R$ and~$N\in (2,\infty)$;
\item $\RCD(K,N)$ spaces with $K\in\R$ and~$N\in (2,\infty]$;
\item complete doubling metric measure spaces satisfying a weak Poincar\'e inequality;
\item path/loop spaces over Riemannian manifolds.
\end{itemize}

%\purple{
As for reference measures, if the base space $X$ satisfies the Rademacher property our main results can be applied to all \emph{quasi-Gibbs measures} (see Dfn.~\ref{d:QuasiGibbs}) satisfying Assumption~\ref{ass:dcSLConfig}.
Although we have explained our results in a very general setting, many of these results are novel even when~$X=\R^n$ for a wide class of invariant measures $\QP$, including $\mathrm{sine}_\beta$, $\mathrm{Airy}_\beta$, $\mathrm{Bessel}_{\alpha, \beta}$, and Ginibre point processes.
%In particular, \emph{all} our main results apply to many important examples of determinantal/permanental point processes such as $\mathrm{sine}_\beta$, $\mathrm{Airy}_\beta$, $\mathrm{Bessel}_{\alpha, \beta}$, and Ginibre point processes, for which the closability of the form~$\EE{\dUpsilon}{\QP}$ has been proved by H.~Osada~\cite{Osa96, Osa13}.
Some of our results, including proofs of the Rademacher property, quasi-regularity, and the coincidence of the canonical form with the Cheeger energy, apply as well to \emph{all} quasi-Gibbs measures. %in the sense of~\cite{Osa13}.
%, which, therefore, expand greatly the results in the existing literatures even in the case of $X=\R^n$. 
%\footnote{It is false. We did not prove such a great result.}} 
%in the sense of Osada~\cite{Osa13} (see Dfn.~\ref{d:QuasiGibbs}), and (in particular) by:
%\begin{itemize}
%\item Poisson and mixed Poisson measures;
%\item Ruelle Gibbs measures %(in the sense of~\cite{AlbKonRoe98b}); 
%\item (the laws of) some determinantal/permanental point processes, e.g.~\cite{ShiTak03};
%\item 
%\end{itemize}
%
We provide an overview of examples of base spaces and invariant measures in relation to our assumptions in Tables~\ref{tbl:2}--\ref{tbl:3}. 
%}
%\subsection{The case of the Euclidean spaces}
%\purple{Although we have explained our results in the setting far more general than the case that $X$ is the Euclidean spaces, many of our results are novel even in the Euclidean case for a wide class of invariant measures $\QP$.} The following list of our results holds true for arbitrary Borel probability measure $\QP$ on $\dUpsilon$ satisfying 

\subsection{Literature, motivations, and outlook}
The literature on configuration spaces is extensive.
With no ambition of giving a complete account, let us mention some of the relevant work.

\subsubsection{Analysis} \label{subsubsection: Ana}
In the case of $X=\R^n$, lifted square field operators~$\cdc^\dUpsilon$ on~$\dUpsilon$ have been considered under several different conditions of invariant measures by several authors, including M.~W.~Yoshida~\cite{Yos96}, H.~Tanemura~\cite{Tan96},  H.~Osada in~\cite{Osa96, Osa13},  V.~Choi and Y.-M.~Park, and H.~J.~Yoo~\cite{ChoParYoo98}, and H.~J.~Yoo~\cite{Yoo05}.
The case when~$X$ is a weighted Riemannian manifold~$(M,\rho\vol_g)$ with square field operator~$\cdc(f)\eqdef \abs{\nabla f}_g^2$ is studied in the framework of canonical Gibbs measures: 
In the case of (mixed) Poisson measures, the seminal work~\cite{AlbKonRoe98} by S.~A.~Albeverio, Yu.~G.~Kondrat'ev and M.~R\"ockner reaches a thorough understanding of the analytic structure on~$(\dUpsilon,\PP_{\rho\vol_g})$.
It includes the closability of the reference form~$\EE{\dUpsilon}{\PP}$, the computation of the corresponding generator and semigroup, and the identification of the latter operators as the second quantizations of the standard Laplace--Beltrami and heat semigroup operators on the base space.
The identification of the heat kernel measure of~$\EE{\dUpsilon}{\PP}$ and the study of the corresponding diffusion was subsequently achieved by Yu.~G.~Kondrat'ev, E.~Lytvynov, and M.~R\"ockner in~\cite{KonLytRoe02}.
Whereas the results in~\cite{AlbKonRoe98} partially rely on the known isomorphism between~$L^2(\PP_{\rho\vol_g})$ and the bosonic Fock space of~$L^2(\rho\vol_g)$ (see e.g.\ D.~Surgailis~\cite{Sur84}), many of them were subsequently extended to canonical and grand-canonical Gibbs measures~$\QP$ in~\cite{AlbKonRoe98b, KonLytRoe08, SKR01}.
A construction and study of~$\cdc^\dUpsilon$ and related objects in greater generality for the choice of the base space is the main result in the aforementioned work~\cite{MaRoe00} by Z.-M.~Ma and M.~R\"ockner.

\subsubsection{Geometry}
The geometric properties of~$(\dUpsilon,\mssd_\dUpsilon,\QP)$ were studied by M.~R\"ockner and A.~Schied in~\cite{RoeSch99} in the case when~$X$ is a Riemannian manifold and $\QP$ satisfies the integration by parts condition (often denoted by (IbP)) and the quasi-invariance w.r.t.\ the action on~$\dUpsilon$ of the diffeomorphism group~$\Diff_c(M)$, acting by push-forward of measures.  
In this setting, they establish both the Rademacher and \emph{a} Sobolev-to-Lipschitz property for~$\EE{\dUpsilon}{\QP}$ and the identification~\eqref{eq:Intro:DistanceIdentification} of~$\mssd_\dUpsilon$ with~$\mssd_{\QP}$.
The quasi-invariance of~$\QP$ is however not easy to verify in general, unless~$\QP$ is in the Ruelle class.
Furthermore, it is unavailable in the singular setting considered in the current paper, since it relies on the smooth structure of~$X$ in a delicate and essential way.
Since we do not rely on the quasi-invariance of~$\QP$, our geometric results ---~the Rademacher and Sobolev-to-Lipschitz properties, and the coincidence of forms and distances~--- are novel even when~$X=\R^n$, extending the results in~\cite{RoeSch99} to measures satisfying~\eqref{eq:Intro:AC}.

Further results on manifolds include the study~\cite{AlbDalLyt01,AlbDalKonLyt03} of de Rham cohomology on~$\dUpsilon$, and the study of the Ricci curvature of~$(\dUpsilon,\mssd_\dUpsilon,\PP_{\vol_g})$ by S.~Albeverio, A.~Daletskii, E.~Lytvynov~\cite{AlbDalLyt01b}, and by M.~Erbar and M.~Huesmann~\cite{ErbHue15}, on which we shall comment extensively in forthcoming work in this series.

\subsubsection{Stochastic analysis}
All the references mentioned in \S\ref{subsubsection: Ana} constructed quasi-regular local Dirichlet forms and, therefore, the corresponding diffusion processes~$\mbfM$ on $\dUpsilon$ exist.
These diffusions can be considered as interacting diffusion processes on the base space $X$.
When~$\QP=\PP$ is a Poisson measure, T.~Shiga and Y.~Takahashi~\cite{ShiTak74} used a sophisticated approach, by which the diffusion~$\mbfM$ on~$\dUpsilon$ can be realized as a Poisson point process on the configuration space over the path space of~$X$.
The pathwise analysis of the corresponding interacting diffusion processes as strong solutions of infinite-dimensional stochastic differential equations has been investigated by R.~Lang~\cite{Lan77, Lan77b} and J.~Fritz~\cite{Fri87} in the case when~$X=\R^n$ is the Euclidean space, and~$\QP$ is any grand-canonical Gibbs measure with compactly supported pair potential.
Under more general conditions for~$\QP$, this pathwise analysis has been further investigated by L.-C.~Tsai~\cite{Tsa15} in the case that $X=\R$ and for $\QP$ the law of a $\mathrm{sine}_\beta$ point process, and by H.~Osada~\cite{Osa13, Osa13b}, and H.~Osada and H.~Tanemura~\cite{OsaTan20} for $X=\R^n$ and $\QP$ in a class of quasi-Gibbs measures.
In the case~$\QP=\PP$, the corresponding small-time large deviation principles have been studied by T.~S.~Zhang~\cite{Zha01} (at the level of semigroups, when $X=\R$) and by M.~R\"ockner and T.~S.~Zhang~\cite{RoeZha04} (at the level of paths, when $X$ is a Riemannian manifold).
%We will comment extensively these results in a forthcoming paper in this series. 

For all the cases mentioned above, we prove quasi-regularity in a systematic way, and our identification of~$\EE{\dUpsilon}{\QP}$ with~$\Ch[\mssd_\dUpsilon,\QP]$ (Thm.~\ref{t:IdentificationCheeger}) implies that the corresponding interacting diffusion on the base space $X$ coincides with the Brownian motion on the infinite-dimensional metric measure space~$(\dUpsilon, \mssd_\dUpsilon, \QP)$.

\subsubsection{Motivation and outlook}
A treatment of $\dUpsilon$ for singular base spaces~$X$ and for general invariant measures~$\QP$ has many motivations in various related fields.

From the viewpoint of \emph{interacting particle systems}, in relation to statistical physics, material science, molecular biology, etc., particles move in spaces with a complicated structure, e.g.\ molecules inside cells, or electrons in composite or super-cooled media.
In these environments, the motion of particles is possibly highly degenerate and the corresponding base space~$X$ is far from being a smooth Riemannian manifold.
The degeneracy of the metric structure may consist of singularities of various type: obstacles, barriers, bottlenecks, etc.
Furthermore, particles could interact with each other in a complex manner induced by singular interactions beyond the standard treatment within the framework of Gibbs measures.

Stochastic geometry %, which dates back to the eighteenth-century Buffon needle problem,
has recently seen a thriving development in a wide range of areas, including the study of {point processes}, {integral geometry}, {random graphs}, {random convex geometries}, and many others.
These objects have provided useful mathematical frameworks for the description of random geometric objects in relation to biology, neuroscience, astronomy, computational geometry, communication networks, image analysis, material science, etc.; see, e.g.,~\cite{BBSW07, Cou19}.
Many important quantities and operations for the analysis of random geometries, e.g., contact distance functions of random point processes, the number of faces and the volume of random polytopes, thinning/superposition/clustering of point processes, can be described as maps on configuration spaces over appropriate ambient spaces, but not necessarily over smooth manifolds, nor equipped with Gibbs measures.
% As indicating the existences of graphs that can be isometrically embedded into a metric space, but not necessarily into manifolds, 
We expect the development of analysis and geometry of configuration spaces beyond smooth manifolds and the Gibbsian framework to have a considerable impact on the applicability of these fields.

From the viewpoint of \emph{infinite-dimensional metric measure geometry}, 
in light of the developments of metric measure geometry in the last two decades, various systematic treatments of the geometry of singular spaces have been developed based on e.g., volume doubling and weak Poincar\'e inequalities, or the synthetic Ricci curvature bounds.
However, many important infinite-dimensional spaces lie outside the scope of these theories, being extended-metric measure spaces and displaying patologies similar to those listed in the beginning of~\S\ref{ss:Intro:Geometry}.
This obstacle is partially overcome by the \emph{extended-metric measure theory} developed by L.~Ambrosio, N.~Gigli, G.~Savar\'e~\cite{AmbGigSav14b}, L.~Ambrosio, M.~Erbar, G.~Savar\'e~\cite{AmbErbSav16}, and G.~Savar\'e~\cite{Sav19}.
The applicability of this theory to concrete infinite-dimensional spaces remains however non-trivial.
Indeed, as indicated by measure-concentration phenomena, Lipschitz functions on infinite-dimensional spaces could be `approximately constant', and their Cheeger energies could vanish identically.
In the present case for instance, cylinder functions on~$\dUpsilon$ ---~on which we base our construction~--- are typically not $\mssd_\dUpsilon$-Lipschitz (Ex.~\ref{e:NonLipCyl}).
Thus, until one can show the non-triviality of Cheeger energies (Thm.~\ref{t:IdentificationCheeger}), the aforementioned abstract theory does not provide any concrete information.

Our motivation for this paper is to establish solid foundations for metric measure geometry on the configuration space~$\dUpsilon$ over general spaces~$X$ and invariant measures~$\QP$, aiming to include all the aforementioned singular settings appearing in the other fields of science, and to understand the corresponding statistical-physical diffusion structure in terms of infinite-dimensional metric measure theory.

\section{Measurable structure: local diffusion spaces}\label{s:FundMeasure}
In this section, we introduce the main terminology and notation about configuration spaces and Dirichlet forms thereupon.
We discuss such objects in the greatest possible generality, only relying on the measure-space structure.
In detail, closely following the exposition of Z.-M.~Ma and M.~R\"ockner~\cite{MaRoe00}, we introduce the class of cylinder functions, and show how to lift a square field operator on a base space~$X$ to a well-defined square field on the corresponding configuration space~$\dUpsilon$, defined on cylinder functions.
Subsequently, we recall the closability of the form~$\EE{\dUpsilon}{\QP}$ induced by any such square field via integration with respect to a probability measure~$\QP$ satisfying a Georgii--Nguyen--Zessin type formula (see Assumption~\ref{ass:Closability}), as established in~\cite{MaRoe00}.

Our purpose in this section is that to provide the reader with a gentle introduction to the issue of \emph{well-posedness} of such forms, which will be thoroughly discussed in~\S\ref{s:Analytic} under the additional assumption that~$X$ be endowed with a metrizable topology.
We therefore postpone to~\S\ref{s:Analytic} a different proof of (well-posedness and) closability of the form~$\EE{\dUpsilon}{\QP}$ under much more general assumptions than Assumption~\ref{ass:Closability} on the reference measure~$\QP$, thus effectively generalizing the closability statement in~\cite{MaRoe00} to a wider class of measures.

\subsection{Notation}
We collect here some general conventions on notation which we shall adhere to throughout this work.

\begin{notat}[General conventions]
\paragraph{Scalars} Write
\begin{equation*}
\N_0\eqdef\set{0,1,\dotsc}\comma \qquad \overline\N_0\eqdef \N_0\cup\set{+\infty}\comma \qquad \N_1\eqdef \set{1,2,\dotsc}\comma
\end{equation*}
and analogously for~$\overline\N_1$.
In general, elements of~$\overline\N_0$ are denoted by uppercase letters, whereas elements of~$\N_0$ are denoted by lowercase letters.
For instance, we write~$n\in \N_0$,~$N\in \overline\N_0$, and the expression~$n\leq N$ is taken to mean that \emph{either}~$N<\infty$ and~$n\leq N$ \emph{or}~$N=\infty$ and $n$ is an arbitrary non-negative integer.

\paragraph{Sets} We denote by the superscript~$\square^\complement$ the complement of a set inside some ambient space apparent from context.
For a property~$P$ we use the phrasing `co-$P$' to indicate that~$P$ holds on a complement set, e.g.~`cofinite' (i.e.\ with finite complement), `conegligible' (i.e.\ with negligible complement), etc..

\paragraph{Product objects} We shall adhere to the following conventions:
\begin{itemize}
\item the superscript~${\square}^\tym{N}$ (the subscript~$\square_\tym{N}$) denotes ($N$-fold) \emph{product objects};

\item the superscript~${\square}^\asym{N}$ denotes objects \emph{relative to products}, not necessarily in product form;

\item the superscript~${\square}^\sym{N}$ denotes objects \emph{relative to symmetric products}, not necessarily in symmetric product form;

\item for a permutation~$\sigma$ the subscript~${\square}_\sigma$ denotes that the objects in~$\square$ are permuted according to~$\sigma$, e.g., for~$\mbfx\eqdef\seq{x_i}_{i\leq N}$, we write~$\mbfx_\sigma$ in place of~$\seq{x_{\sigma(i)}}_{i\leq N}$;

\item upper case \textbf{boldface} letters always denote (subsets of) infinite-product spaces.
\end{itemize}

\paragraph{Configurations} We shall adhere to the following conventions:
\begin{itemize}
\item the upper case Greek letters~$\Upsilon$,~$\Lambda$, $\Xi$, $\Omega$ always denote (subsets of) configuration spaces.
As a general rule,~$\Omega$ will be used to denote sets of full measure,~$\Lambda$ and~$\Xi$ sets with special properties or generic sets.

\item the lower case Greek letters~$\QP$, $\PP$ always denote measures on configuration spaces;

\item the lower case Greek letters~$\alpha$, $\gamma$, $\eta$ always denote configurations;
\end{itemize}

\paragraph{Reference objects} \textsf{Sans-serif} letters denote reference objects, e.g.\ a reference measure~$\mssm$ or a reference distance~$\mssd$.

\paragraph{Representatives} For the better part of this work, we shall need to carefully distinguish a.e.-classes of functions from their representatives.
For the a.e.-class~$f$, we denote by~$\rep f$ any of its representatives.
We drop this distinction only for a.e.-classes~$f$ having a unique representative continuous w.r.t.\ a topology apparent from context, in which case we denote by~$f$ both the a.e.-class \emph{and} its unique continuous representative.
Further notation on a.e.-classes/representatives is discussed in Notation~\ref{n:Classes}.

\paragraph{Restrictions} For any family~$\msE$ of subsets~$E$ of a space~$X$, and for a fixed~$A\subset X$, we set $\msE_A\eqdef\set{E\cap A: E\in\msE}$.

\paragraph{Assumptions} We shall make use of a large variety of different assumptions.
For ease of notation, recollection, and reference, most of these assumptions will be marked by an abbreviation \emph{and} by a number.
For instance, the Closability Assumption~\ref{d:ass:CWP} will be referred to by~\ref{ass:CWP}.
\end{notat}

\subsection{Base spaces}
A \emph{configuration space}~$\dUpsilon$ is ---~informally~--- the set of non-negative integer-valued \emph{locally finite} measures over a measurable \emph{base space}~$(X,\A)$. In principle, no reference measure on~$(X,\A)$ is required in the definition. However, since we shall be interested in lifting additional structures from the base space to the configuration space, we shall require the former to be a \emph{measure space}, and satisfying to the following definition.

\begin{defs}\label{d:MS} A \emph{measure space} is a triple~$(X,\A,\mssm)$ so that
\begin{enumerate}[$(a)$]
\item\label{i:d:MS:1} $X$ is a \emph{non-empty} set;
\item\label{i:d:MS:2} $\A$ is a $\sigma$-algebra on~$X$ so that~$\set{x}\in \A$ for every~$x\in X$;
\item\label{i:d:MS:3} $\mssm$ is a \emph{$\sigma$-finite atomless} measure on~$X$;
\item\label{i:d:MS:4} $\A$ is \emph{$\mssm$-essentially countably generated}, i.e.\ there exists a countably generated $\sigma$-sub\-al\-ge\-bra~$\A_0$ of~$\A$ so that for every~$A\in \A$ there exists $A_0\in \A_0$ with~$\mssm(A\triangle A_0)=0$.
\end{enumerate}

Let~$(X,\A,\mssm)$ be a measure space. We denote by~$\A_\mssm\subset \A$ the algebra of sets of finite $\mssm$-measure, and by~$(X,\A^\mssm,\hat\mssm)$ the Carath\'eodory completion of~$(X,\A,\mssm)$ w.r.t.~$\mssm$.
\end{defs}

\subsubsection{Local structures}
The concept of `locality' in the informal designation of `locally finite measure' is captured by the notion of a \emph{local structure}.

\begin{defs}[Local structures]\label{d:LS}
A \emph{localizing ring} of~$(X,\A,\mssm)$ is any family~$\msE\subset \A$ so that
\begin{enumerate}[$(a)$]
\item\label{i:d:LS:1}$\msE$ is a ring ideal of~$\A_\mssm$, i.e.~$\msE\subset \A_\mssm$ is closed under finite unions and intersections, and~$A\cap E\in \msE$ for every~$A\in \A$ and~$E\in \msE$;
\item\label{i:d:LS:2}$\msE=\cup_{h\geq 0} (\A\cap E_h)$ for some sequence~$\seq{E_h}_h\subset \A_\mssm$ with~$E_h\uparrow_h X$, termed a \emph{localizing sequence}.
\end{enumerate}

A \emph{local structure} is a quadruple~$(X,\A,\mssm,\msE)$ with~$\msE$ a localizing ring of~$(X,\A,\mssm)$.
\end{defs}

Our definition of localizing ring is a modification of~\cite{Kal17}: in comparison with~\cite[p.~19]{Kal17}, in Definition~\ref{d:LS} we additionally require~$\mssm E<\infty$ for each~$E$ in~$\msE$. 
As noted in~\cite[p.~15]{Kal17}, the datum of a localizing ring is equivalent to that of a localizing sequence~$\seq{E_h}_h$, just by taking as a definition of~$\msE$ the ring ideal generated by the sequence in Definition~\ref{d:LS}\iref{i:d:LS:2}.
In the following, we shall denote a local structure~$(X,\A,\mssm,\msE)$ simply by~$\mcX$.

Let now~$\gamma$ be a measure, e.g.~a configuration, over~$(X,\A)$. Among the simplest functionals of~$\gamma$ are integral functionals of the form
\begin{align}\label{eq:Trid}
f^\trid\colon \gamma\longmapsto \gamma f \eqdef \int f\diff \gamma\comma
\end{align}
where~$f\colon X\rar [-\infty,\infty]$ is any function so that~$\gamma f$ is meaningful. It is natural to study functionals of the form~$f^\trid$ by means of the function spaces which~$f$ belongs to.

\begin{defs}[Function spaces]\label{d:FuncSp}
Let~$\mcX$ be a local structure.
\begin{enumerate}[$(a)$]
\item Write~$\mcL^\infty(\mssm)$ or~$\Sb(X)$ for the Banach lattice of real-valued \emph{bounded} (as opposed to: $\mssm$-essentially bounded) $\A$-measurable functions, and
\begin{align*}
\mcL^\infty(\mssm)\longrar L^\infty(\mssm)\colon f\longmapsto\class[\mssm]{f}
\end{align*}
for the quotient map to the corresponding quotient Banach algebra~$L^\infty(\mssm)$.
Here and elsewhere we indicate by~$\class[\mssm]{f}$ the class of a function~$f$ up to $\mssm$-a.e.\ equivalence.
For~$f_i\in L^\infty(\mssm)$, resp.~$\rep f_i\in \mcL^\infty(\mssm)$,~$i\leq k$, set
\begin{align*}
\mbff\eqdef \seq{f_1,\dotsc, f_k}\in L^\infty(\mssm;\R^k)\comma \qquad \text{resp.}\qquad \rep\mbff \eqdef \tseq{\rep f_1,\dotsc, \rep f_k}\in\mcL^\infty(\mssm;\R^k) \fstop
\end{align*}

\item Write $\Sb(\msE)$ for the space of bounded $\A$-measurable $\msE$-\emph{eventually vanishing} functions on~$X$, viz.
\begin{align}\label{eq:SbExhaustion}
\Sb\ttonde{\msE}\eqdef \set{f\in\Sb(X) : f\equiv 0 \text{~~on } E^\complement \text{~~for some }E\in\msE}  \semicolon
\end{align}

\item\label{i:d:FuncSp:5} Say that:
\begin{enumerate}[$(c_1)$]
\item a function~$f\colon (X,\A)\rar \R$ is \emph{virtually measurable} if there exists an $\mssm$-negligible set~$N$ such that~$f\colon N^\complement \rar \R$ is $\A_{N^\complement}$-measurable. In this case,~$f$ may be left undefined on~$N$;

\item a measurable function~$f\colon (X,\A)\rar \R$ is \emph{virtually} $\msE$-\emph{eventually vanishing} if there exists an $\mssm$-negligible set~$N$ such that~$f$ is identically vanishing on~$E^\complement\cap N^\complement$ for some~$E\in\msE$.
\end{enumerate}
Let \emph{virtually measurable virtually $\msE$-eventually vanishing functions} be defined in the obvious way. We denote the space of all \emph{bounded} such functions by~$\Sb(\msE,\mssm)$.
\end{enumerate}
\end{defs}

In general, a functional of the form~\eqref{eq:Trid} is not well-defined if~$f$ is merely a class of functions modulo $\mssm$-negligible sets. Thus, it will be of importance in the following to distinguish between measurable functions on~$\mcX$ and their $\mssm$-classes. We shall do so by the following notation.

\begin{notat}\label{n:Classes}
We denote $\mssm$-classes of functions by~$f$,~$g$, etc., measurable representatives by~$\rep f$,~$\rep g$, etc.
Whenever~$f$ is an $\mssm$-class,~$\rep f$ is taken to be a representative of~$f$. Whenever~$\rep f$ is a measurable function, possibly undefined on an $\mssm$-negligible set,~$f=\ttclass[\mssm]{\rep f}$ is taken to be the corresponding $\mssm$-class.
We shall adopt the same convention for other objects, including for instance algebras of functions, thus writing e.g.,~$\Dz\subset L^\infty(\mssm)$, resp.~$\rep\Dz\subset \mcL^\infty(\mssm)$.
\end{notat}

\subsubsection{Diffusion spaces}
The main object of our study will be the lifting of local \emph{square field operators} from the base space to the configuration space.

We start by recalling the main definitions.
We write~$\Cbinfty(\R^k)$ for the space of real-valued bounded smooth functions on~$\R^k$ with bounded derivatives of all orders.

\begin{defs}[Square field operators]\label{d:SF}
Let~$\mcX$ be a local structure. For every~$\phi\in \Cb^\infty(\R^k)$ set $\phi_0\eqdef \phi-\phi(\zero)$. A \emph{square field operator on $\mcX$} is a pair~$(\cdc, \Dz)$ so that, for every~$\mbff\in\Dz^\otym{k}$, every~$\phi\in \Cb^\infty(\R^k)$ and every~$k\in \N_0$,
\begin{enumerate}[$(a)$]
\item\label{i:d:SF:1} $\Dz$ is a subalgebra of~$L^\infty(\mssm)$ with~$\phi_0\circ\mbff\in \Dz$;

\item\label{i:d:SF:2} $\cdc\colon \Dz^\otym{2}\longrar L^\infty(\mssm)$ is a symmetric non-negative definite bilinear form;

\item\label{i:d:SF:3} $(\cdc,\Dz)$ satisfies the following \emph{diffusion property}
\begin{align}\label{eq:i:d:SF:3}
\cdc\ttonde{\phi_0\circ \mbff, \psi_0\circ \mbfg}=
\sum_{i,j}^k (\partial_i \phi)\circ \mbff \cdot (\partial_j\psi)\circ\mbfg \cdot \cdc(f_i, g_j) \as{\mssm}\fstop
\end{align}
\end{enumerate}

Let the analogous definition of a \emph{pointwise defined square field operator}~$(\rep\cdc, \rep\Dz)$ be given, with~$\rep\Dz\subset \mcL^\infty(\mssm)$ in place of~$\Dz\subset L^\infty(\mssm)$, and with~\eqref{eq:i:d:SF:3} to hold pointwise (as opposed to: $\mssm$-a.e.).
\end{defs}

As noted in~\cite[p.~282]{MaRoe00}, a pointwise defined square field operator~$(\rep\cdc, \rep\Dz)$ defines a square field operator~$(\cdc,\Dz)$ on $\mssm$-classes as soon as
\begin{align}\label{eq:Ss}
\rep\cdc(\rep f,\rep g)=0\comma \qquad \rep f,\rep g\in \rep\Dz\comma \rep f \equiv 0 \as{\mssm}\fstop
\end{align}

\begin{defs}[Local diffusion spaces]\label{d:DS}
A \emph{local diffusion space} (in short: \LDS) is a pair~$(\mcX,\rep\cdc)$ so that
\begin{enumerate}[$(a)$]
\item\label{i:d:DS:1}$\mcX$ is a local structure;

\item\label{i:d:DS:1.5} $\rep\Dz$ is a subalgebra of~$\Sb(\msE)$ $\mssm$-essentially generating~$\A$;

\item\label{i:d:DS:2}$\rep\cdc\colon \rep\Dz^\otym{2}\rar \Sb(\msE)$ is a pointwise defined square field operator satisfying~\eqref{eq:Ss};

\item\label{i:d:DS:3}the bilinear form~$(\EE{X}{\mssm},\Dz)$ induced by~$(\rep\cdc,\rep\Dz)$ on $\mssm$-classes, defined by
\begin{align}\label{eq:i:d:DS:3}
\EE{X}{\mssm}(f,g)\eqdef \int \cdc(f,g) \diff\mssm \comma \qquad f, g\in \Dz\comma
\end{align}
is closable and densely defined in~$L^2(\mssm)$.
\end{enumerate}
\end{defs}

We shall comment about Definition~\ref{d:DS}\iref{i:d:DS:1.5} in Remark~\ref{r:DensityQP} below.
As a consequence of~\eqref{eq:i:d:SF:3}, the closure of~$\ttonde{\EE{X}{\mssm},\Dz}$ is a Dirichlet form, which motivates to introduce the following standard notation.

\begin{notat}\label{n:Form} Let~$(\mcX,\rep\cdc)$ be an~\LDS.
\begin{enumerate*}[$(a)$]
\item We denote by~$\ttonde{\EE{X}{\mssm},\dom{\EE{X}{\mssm}}}$ the closure of~$\ttonde{\EE{X}{\mssm},\Dz}$.
\linebreak
\item\label{i:n:Form:1} It is readily verified that the latter form admits square field operator~$\ttonde{\SF{X}{\mssm},\dom{\SF{X}{\mssm}}}$ with $\dom{\SF{X}{\mssm}}=\dom{\EE{X}{\mssm}}\cap L^\infty(\mssm)$, and extending~$\ttonde{\cdc,\Dz}$.
Further let:
\item\label{i:n:Form:3} $\ttonde{\LL{X}{\mssm}, \dom{\LL{X}{\mssm}}}$ be the generator of~$\ttonde{\EE{X}{\mssm},\dom{\EE{X}{\mssm}}}$;
\item\label{i:n:Form:4} $\TT{X}{\mssm}_\bullet\eqdef \tseq{\TT{X}{\mssm}_t}_{t\geq 0}$ be the semigroup of~$\ttonde{\LL{X}{\mssm}, \dom{\LL{X}{\mssm}}}$, defined on~$L^p(\mssm)$ for every~$p\in [1,\infty)$.
\item by~$\hh{X}{\mssm}_\bullet\eqdef \tseq{\hh{X}{\mssm}_t(\emparg,\diff\emparg)}_{t\geq 0}$ the corresponding Markov kernel of measures, satisfying
\end{enumerate*}
\begin{align*}
\ttonde{\TT{X}{\mssm}_t f}(x)=\int f(y)\, \hh{X}{\mssm}_t(x,\diff y)\comma \qquad f\in L^2(\mssm)\comma t\geq 0 \semicolon
\end{align*}

Finally, let
\begin{enumerate*}[$(a)$]\setcounter{enumi}{5}
\item\label{i:n:Form:6}$\domext{\EE{X}{\mssm}}$ be the extended Dirichlet space of~$\ttonde{\EE{X}{\mssm},\dom{\EE{X}{\mssm}}}$, i.e.\ the space of $\mssm$-classes of functions~$f\colon X\rar \R$ so that there exists an $(\EE{X}{\mssm})^{1/2}$-fun\-da\-men\-tal sequence $\seq{f_n}_n\subset \dom{\EE{X}{\mssm}}$ with $\nlim f_n=f$ $\mssm$-a.e..
The form~$\EE{X}{\mssm}$ naturally extends to a quadratic form on~$\domext{\EE{X}{\mssm}}$, denoted by the same symbol~$\EE{X}{\mssm}$, and we always consider~$\domext{\EE{X}{\mssm}}$ as endowed with this extension;
\item\label{i:n:Form:7}$\domext{\SF{X}{\mssm}}\eqdef\domext{\EE{X}{\mssm}}\cap L^\infty(\mssm)$ be the extended space of~$\ttonde{\SF{X}{\mssm},\dom{\SF{X}{\mssm}}}$, endowed with the non-relabeled extension of~$\SF{X}{\mssm}$.
\end{enumerate*}

In this generality, the well-posedness of objects in~\iref{i:n:Form:6} and~\iref{i:n:Form:7}, that is, their independence from the approximating sequence~$\seq{f_n}_n$, follows by~\cite[Prop.~1]{Sch99}; cf., e.g.,~\cite{Kuw98} which we refer to for further properties of extended domains.
\end{notat}

\begin{defs}\label{d:PropEE}
We say that an~\LDS~$(\mcX,\rep\cdc)$ is:
\emph{irreducible} if, whenever~$\car_{A}\TT{X}{\mssm}_t f=\TT{X}{\mssm}_t (\car_A f)$ for every~$f\in L^2(\mssm)$, every~$t>0$, and some~$A\in \A$, then either~$\mssm A=0$ or~$\mssm A^\complement=0$;
\emph{stochastically complete}, if the Markov kernel~$\hh{X}{\mssm}_\bullet$ satisfies~$\hh{X}{\mssm}_t(x, X)=1$ for all~$t>0$ and $\mssm$-a.e.~$x\in X$.
\end{defs}

Let us collect some technical tools.

\begin{defs}[$L^p$-completions]\label{d:AbstractCompletion}
For every~$p\in [1,\infty]$ we define
\begin{enumerate}[$(a)$]
\item $\coK{\Dz}{p,\mssm}$ as the abstract linear completion of~$\Dz$ w.r.t.~the norm (cf.~\cite[p.~301]{MaRoe00})
\begin{align*}
\norm{\emparg}_{p,\mssm}\eqdef \EE{X}{\mssm}(\emparg)^{1/2} + \norm{\emparg}_{L^p(\mssm)} 
\comma \end{align*}
endowed with the unique (non-relabeled) continuous extension of~$\norm{\emparg}_{p,\mssm}$ to the completion $\coK{\Dz}{p,\mssm}$;

\item $\coi_{p,\mssm}\colon \Dz\hookrightarrow \coK{\Dz}{p,\mssm}$ the completion embedding;

\item $\cok_{p,\mssm}\colon\Dz\hookrightarrow L^p(\mssm)$ the canonical inclusion;

\item $\cokb_{p,\mssm} \colon \coK{\Dz}{p,\mssm} \rar L^p(\mssm)$ the unique continuous extension of~$\cok_{p,\mssm}\colon\Dz\hookrightarrow L^p(\mssm)$.
\end{enumerate}
\end{defs}

We note that the closability of the pre-form~$(\EE{X}{\mssm}, \Dz)$ is equivalent to the injectivity of~$\cokb_{2,\mssm}$; cf.~\cite[Rmk.~I.3.2(ii)]{MaRoe92}.
In the next lemma we show that the injectivity of~$\cokb_{2,\mssm}$ implies that of~$\cokb_{p,\mssm}$ for any~$p\in [1,2)$. The statement is, in some specific cases, quite standard; cf.~\cite[Rmk.~4.5]{MaRoe00}. We prefer to give a proof in full generality which will serve as a substitute for~\cite[Lem.s~4.2 and~4.4]{MaRoe00}.

\begin{lem}\label{l:AbstractCompletion}
Let~$(\mcX,\cdc)$ be an \LDS. For every~$p\in [1,2]$, the map~$\cokb_{p,\mssm}$ is injective, i.e.~$\coK{\Dz}{p,\mssm}\subset L^p(\mssm)$. Furthermore,~$\cdc$ uniquely extends to a (non-relabeled) continuous bilinear map $\coK{\Dz}{p,\mssm}^{\times 2}\rar L^p(\mssm)$.
\begin{proof}
Let~$\seq{f_n}_n\subset \Dz$ be $\norm{\emparg}_{p,\mssm}$-fundamental and with~$L^p(\mssm)$-$\nlim f_n=0$.
By the reverse triangle inequality,
\begin{align}\label{eq:l:AbstractCompletion:0}
\cdc(f_n-f_m)\geq \ttonde{\cdc(f_n)^{1/2}-\cdc(f_m)^{1/2}}^2
\comma \qquad m,n\in \N_1\comma
\end{align}
hence,~$\seq{\cdc(f_n)}_n$ is $L^1(\mssm)$-fundamental, thus converging to some~$g\in L^1(\mssm)$. It suffices to show that~$g=0$.
To this end, let~$\phi\in \Cb^\infty(\R)$ be satisfying~$\abs{\phi(t)}\leq 1\wedge \abs{t}$ for~$t\in \R$,~$\phi'(0)=1$, and~$0\leq \phi'\leq 1$. 
Since~$p\leq 2$, one has~$\abs{\phi(t)}\leq \abs{t}^{p/2}$ for~$t\in \R$.
Then,
\begin{equation}\label{eq:l:AbstractCompletion:1}
L^2(\mssm)\text{-}\nlim \phi\circ f_n=0
\end{equation}
by continuity of the Nemytskii Composition Operator~$N_\phi$; see e.g.~\cite[Thm.~3.6, p.~17]{BocCro13}.

By~\eqref{eq:i:d:SF:3}, up to passing to a suitable non-relabeled subsequence, we conclude that
\begin{align*}
\cdc(\phi\circ f_n-\phi\circ f_m)=&\ (\phi'\circ f_n)^2 \cdot \cdc(f_n)+(\phi'\circ f_m)^2\cdot \cdc(f_m)
\\
&+\phi'\circ f_n\cdot \phi'\circ f_m \cdot \cdc(f_n-f_m)
\\
&-\phi'\circ f_n \cdot \phi'\circ f_m \cdot \ttonde{\cdc(f_n)+\cdc(f_m)}
\end{align*}
converges to~$0$ $\mssm$-a.e.\ as~$n,m\rar \infty$.
By the Dominated Convergence Theorem with varying dominating function~$\cdc(f_n)+ \cdc(f_m)\to_{n,m} g\in L^1(\mssm)$ (e.g.~\cite[Thm.~10.1(c)]{DiB02}), we conclude that~$L^1(\mssm)$-$\lim_{n,m}\cdc(\phi\circ f_n-\phi\circ f_m)=0$.

By~\eqref{eq:l:AbstractCompletion:1}, closability of~$(\EE{X}{\mssm},\Dz)$ in~$L^2(\mssm)$ and~\eqref{eq:i:d:SF:3},
\begin{align*}
0=L^1(\mssm)\text{-}\nlim \cdc\ttonde{\phi\circ f_n}=\nlim (\phi'\circ f_n)^2 \cdot \cdc(f_n)=\phi'(0)^2 g=g \fstop & \qedhere
\end{align*}
\end{proof}
\end{lem}

\subsection{Configuration spaces}\label{ss:Upsilon}
Let~$\mcX$ be a local structure.
A (\emph{multiple}) \emph{configuration} on~$\mcX$ is any $\N_0$-valued measure~$\gamma$ on~$(X,\A)$, finite on~$E$ for every~$E\in\msE$. By assumption on~$\mcX$, cf.\ e.g.~\cite[Cor.~6.5]{LasPen18},
\begin{align*}
\gamma= \sum_{i=1}^N \delta_{x_i} \comma \qquad N\in \overline\N_0\comma \qquad \seq{x_i}_{i\leq N}\subset X \fstop
\end{align*}
In particular, we allow for~$x_i=x_j$ if~$i\neq j$.
Write~$\gamma_x\eqdef\gamma\!\set{x}$,~$x\in \gamma$ whenever~$\gamma_x>0$, and~$\gamma_A\eqdef \gamma\mrestr{A}$ for every~$A\in\A$.

\begin{defs}[Configuration spaces] The \emph{multiple configuration space}~$\dUpsilon=\dUpsilon(X,\msE)$ is the space of all (multiple) configurations over~$\mcX$. The \emph{configuration space} is the space
\begin{align*}
\Upsilon(\msE)=&\Upsilon(X,\msE)\eqdef\set{\gamma\in\dUpsilon: \gamma_x\in \set{0,1} \text{~~for all } x\in X} \fstop
\end{align*}
The $N$-particles multiple configuration spaces, resp.~configuration spaces, are
\begin{equation*}
\begin{aligned}
\dUpsilon^\sym{N}\ \eqdef&\ \set{\gamma\in \dUpsilon : \gamma X=N}\comma 
\\
\text{resp.} \qquad \Upsilon^\sym{N}\ \eqdef&\ \dUpsilon^\sym{N}\cap \Upsilon\comma 
\end{aligned}
\qquad N\in\overline\N_0 \fstop
\end{equation*}
Let the analogous definitions of~$\dUpsilon^\sym{\geq N}(\msE)$, resp.~$\Upsilon^\sym{\geq N}(\msE)$, be given.

For fixed~$A\in\A$ further set~$\msE_A\eqdef \set{E\cap A: E\in\msE}$ and
\begin{align}\label{eq:ProjUpsilon}
\pr^A\colon \dUpsilon\longrar \dUpsilon(\msE_A)\colon \gamma\longmapsto \gamma_A\comma \qquad A\in\A \fstop
\end{align}

Finally set~$\dUpsilon(E)\eqdef \dUpsilon(\msE_E)=\dUpsilon(\A_E)$ for all~$E\in\msE$, and analogously for~$\Upsilon(E)$.
\end{defs}

We endow~$\dUpsilon$ with the $\sigma$-algebra $\A_\mrmv(\msE)$ generated by the functions $\gamma\mapsto \gamma E$ with~$E\in\msE$.
This coincides with the $\sigma$-algebra on~$\dUpsilon$ given in~\cite[Rmk.~1.5]{MaRoe00} because of Definition~\ref{d:LS}\iref{i:d:LS:2}.

\subsubsection{Dirichlet forms}
Throughout this section,~$\QP$ will denote a probability measure on the space $\ttonde{\dUpsilon,\A_\mrmv(\msE)}$.
For any such~$\QP$, we denote by~$\mssm_\QP$ its  (\emph{mean}) \emph{intensity} (\emph{measure}), defined by
\begin{align*}
\mssm_\QP A \eqdef \int_\dUpsilon \gamma A \diff\QP(\gamma) \comma \qquad A\in\A_\mrmv(\msE)\fstop
\end{align*}
For many results here and later on, we shall make the following assumption.

\begin{ass}\label{d:ass:Mmu}
We say that~$\QP$ satisfies Assumption~\ref{ass:Mmu} if it has $\msE$-locally finite intensity, viz.\
\begin{equation}\tag*{$(\mssm_\QP)_{\ref{d:ass:Mmu}}$}\label{ass:Mmu}
\mssm_\QP E<\infty \comma \quad E\in\msE\fstop
\end{equation}
\end{ass}

This assumption is a very natural one:
from a mathematical point of view, it implies that we have sufficiently many functions in~$L^1(\QP)$;
from a physical point of view, it implies that any system of randomly $\QP$-distributed particles is $\msE$-locally finite in average.

Our main goal will be the study of a Dirichlet form
\begin{equation}\label{eq:DirichletForm}
\ttonde{\EE{\dUpsilon}{\QP},\dom{\EE{\dUpsilon}{\QP}}} \qquad \text{on} \qquad L^2(\QP)
\end{equation}
which we shall construct by `lifting' the diffusion structure of a base \LDS to a diffusion structure on the corresponding configurations space.

\paragraph{Cylinder functions} We shall start by defining a suitable core of \emph{cylinder} functions for the form \eqref{eq:DirichletForm}.
For~$\gamma\in \dUpsilon$ and~$\rep f\in\Sb(\msE)$ let~$\rep f^\trid\colon \dUpsilon\rar \R$ be defined as in~\eqref{eq:Trid} and set further
\begin{align*}
\rep\mbff^\trid\colon \gamma\longmapsto\ttonde{\rep f_1^\trid\gamma,\dotsc, \rep f_k^\trid\gamma}\in \R^k\comma \qquad \rep f_1,\dotsc, \rep f_k\in \Sb(\msE)\fstop
\end{align*}

\begin{defs}[Cylinder functions on~$\dUpsilon$] Let~$\mcX$ be a local structure.
Further let~$\rep D$ be a linear subspace of~$\Sb(\msE)$. We define the space of \emph{cylinder functions}
\begin{align*}
\Cyl{\rep D}\eqdef \set{\begin{matrix} \rep u\colon \dUpsilon\rar \R : \rep u=F\circ\rep\mbff^\trid \comma  F\in \mcC^\infty_b(\R^k)\comma \\ \rep f_1,\dotsc, \rep f_k\in \rep D\comma\quad k\in \N_0 \end{matrix}}\fstop
\end{align*}
\end{defs}

It is readily seen that cylinder functions of the form~$\Cyl{\Sb(\msE)}$ are $\A_\mrmv(\msE)$-measurable.
If~$\rep D$ generates the $\sigma$-algebra~$\A$ on~$X$, then~$\Cyl{\rep D}$ generates the $\sigma$-algebra~$\A_\mrmv(\msE)$ on~$\dUpsilon$. We stress that the representation of~$\rep u$ by~$\rep u=F\circ\rep\mbff$ is \emph{not} unique.

\paragraph{Lifted square field operators}
Let~$(\mcX,\rep\cdc)$ be an~\LDS. We may now lift a pointwise defined square field operator~$\rep\cdc$ on~$\mcX$ to a pointwise defined square field operator on~$\dUpsilon$, by setting
\begin{equation}\label{eq:d:LiftCdCRep}
\begin{gathered}
\rep\cdc^{\dUpsilon}(\rep u, \rep v)(\gamma)\eqdef \sum_{i,j=1}^{k,m} (\partial_i F)(\rep\mbff^\trid\gamma) \cdot (\partial_j G)(\rep\mbfg^\trid\gamma) \cdot \rep\cdc(\rep f_i, \rep g_j)^\trid \gamma \comma
\\
\rep u=F\circ \rep\mbff^\trid\in \Cyl{\rep\Dz} \comma\qquad \rep v=G\circ \rep\mbfg^\trid\in \Cyl{\rep\Dz}\fstop
\end{gathered}
\end{equation}

The next result clarifies why assumption~\ref{ass:Mmu} is crucial for our analysis.
\begin{lem}\label{l:MmuL1}
Let~$(\mcX,\rep\cdc)$ be an~\LDS, and~$\QP$ be a probability measure on~$\ttonde{\dUpsilon,\A_\mrmv(\msE)}$ satisfying Assumption~\ref{ass:Mmu}. Then,~$\ttonde{\rep\cdc^{\dUpsilon},\Cyl{\rep\Dz}}$ is $L^1(\QP)$-valued.
\begin{proof}
It suffices to integrate~\eqref{eq:d:LiftCdCRep} w.r.t.~$\QP$ and note that~$(\rep\cdc,\Dz)$ takes values in the space of uniformly bounded $\msE$-eventually vanishing functions, $F,G\in\Cb^\infty(\R^k)$, and apply~\ref{ass:Mmu}.
\end{proof}
\end{lem}
In the following, the above result will be applied mostly without explicit mention.

\medskip

Z.-M.~Ma and M.R\"ockner proved the following result, which will be of importance in establishing the well-definedness and closability of the form~\eqref{eq:DirichletForm}.

\begin{lem}[{\cite[Lem.~1.2]{MaRoe00}}]\label{l:MaRoeckner}
Let~$\rep u\in\Cyl{\rep\Dz}$, $\gamma\in\dUpsilon$, and~$x\in X$. Then,
\begin{align*}
X\ni y\longmapsto \rep u\ttonde{\car_{X\setminus\set{x}}\cdot\gamma + \gamma_x\delta_y}-\rep u\ttonde{\car_{X\setminus\set{x}}\cdot\gamma} \quad \in \quad \rep\Dz\fstop
\end{align*}

Furthermore, for all~$\rep u, \rep v\in\Cyl{\rep\Dz}$ and each~$\gamma\in\dUpsilon$,
\begin{equation}\label{eq:MaRoeckner}
\begin{aligned}
\rep\cdc^\dUpsilon(\rep u,\rep v)(\gamma)= \sum_{x\in\gamma} \gamma_x^{-1}\cdot \rep\cdc\Big( \rep u&\ttonde{\car_{X\setminus\set{x}}\cdot\gamma + \gamma_x\delta_\bullet}-\rep u\ttonde{\car_{X\setminus\set{x}}\cdot\gamma} ,
\\
&\rep v\ttonde{\car_{X\setminus\set{x}}\cdot\gamma + \gamma_x\delta_\bullet}-\rep v\ttonde{\car_{X\setminus\set{x}}\cdot\gamma} \Big)(x)\fstop 
\end{aligned}
\end{equation}
\end{lem}

As noted in~\cite{MaRoe00}, by the previous lemma, the bilinear form~$\rep\cdc^\dUpsilon$ is well-defined on~$\Cyl{\rep\Dz}^\tym{2}$, in the sense that~$\rep\cdc^\dUpsilon(\rep u, \rep v)$ does not depend on the choice of representatives $\rep u=F\circ\rep\mbff$ and $\rep v=G\circ\rep\mbfg$ for~$\rep u$ and~$\rep v$.

For a given probability measure~$\QP$ on~$\ttonde{\dUpsilon,\A_\mrmv(\msE)}$, one is now tempted to set
\begin{align}\label{eq:Temptation}
\EE{\dUpsilon}{\QP}(u,v)\eqdef \int_{\dUpsilon} \cdc^{\dUpsilon}(\rep u, \rep v) \diff\QP\comma \qquad u,v\in \Cyl{\rep\Dz} \comma
\end{align}
and consider~$\EE{\dUpsilon}{\QP}$ as a pre-Dirichlet form on~$L^2(\QP)$.
Two main difficulties arise:
\begin{enumerate}[$(a)$]
\item $\EE{\dUpsilon}{\QP}$ is not necessarily well-defined as a form on~$L^2(\QP)$.
Indeed, since $\rep\cdc^{\dUpsilon}(\rep u, \rep v)(\gamma)$ is everywhere well-defined in the sense above,~$\EE{\dUpsilon}{\QP}$ is a well-defined bilinear form on the space of representatives~$\mcL^2(\QP)$.
However, it is not clear that a $\QP$-class~$\cdc^{\dUpsilon}(u, v)$ is defined independently of the chosen representatives~$\rep u$,~$\rep v$ of the corresponding $\QP$-classes functions~$u$,~$v$.
To this end, it is necessary and sufficient that
\begin{align}\label{eq:WPCdC}
\rep\cdc^\dUpsilon(0,\rep v)\equiv 0 \as{\QP} \comma\qquad \rep v\in\Cyl{\rep\Dz}\comma
\end{align}
in which case, by symmetry,~$\rep\cdc^\dUpsilon$ descends to a bilinear symmetric functional~$\cdc^\dUpsilon$ on the space of $\QP$-classes $\CylQP{\QP}{\rep\Dz}$ of the cylinder functions $\Cyl{\rep\Dz}$.

\item Even assuming that~$\EE{\dUpsilon}{\QP}$ is a well-defined pre-Dirichlet form on~$L^2(\QP)$, the latter is, in general, not closable on~$L^2(\QP)$.
\end{enumerate}

\subsubsection{Ma--R\"ockner's construction}
Over the course of the next sections, we will prove that this assumption is verified under some general conditions: For instance, it is satisfied by the laws of most point processes explicitly studied in the literature.

For the moment ---~in order to keep the discussion of closability as simple as possible while concentrating on other properties of our constructions~--- we shall start by providing some sufficient conditions for the well-posedness and closability of~$\ttonde{\EE{\dUpsilon}{\QP},\CylQP{\QP}{\Dz}}$ already established in the seminal work~\cite{MaRoe00} by Z.-M.~Ma and M.~R\"ockner.
For the sake of simplicity, we collect here the main assumptions in~\cite{MaRoe00}.

\begin{ass}\label{ass:Closability} We assume that:
\begin{enumerate}[$(a)$]
\item\label{i:ass:Closability:1}
Assumption~\ref{ass:Mmu} holds for~$\QP$;
\item\label{i:ass:Closability:2} there exists a $\A_\mrmv(\msE)\otimes\A$-measurable function~$\rho\colon \dUpsilon\times X \rar \R_+$ such that
\begin{align}\tag*{$(\mathsf{GNZ})_{\ref{ass:Closability}}$}\label{ass:GNZ}
\iint_{\dUpsilon\times X} u(\gamma,x)\diff\gamma(x)\diff\QP(\gamma)=\iint_{\dUpsilon\times X} u(\gamma+\delta_x,x)\, \rho(\gamma,x) \diff\mssm(x)\diff \QP(\gamma)
\end{align}
for all semi-bounded $\A_\mrmv(\msE)$-measurable~$\rep u\colon \dUpsilon\times X\rar \R$.

\item\label{i:ass:Closability:5} $\rep\cdc(0,\rep g)=0$ $\mssm_\gamma$-a.e.~for all~$\rep g\in \rep\Dz$ for~$\QP$-a.e.~$\gamma\in\dUpsilon$, where~$\diff\mssm_\gamma\eqdef \rho(\gamma,\emparg) \diff\mssm$;

\item\label{i:ass:Closability:6} there exists an $\mssm_\QP$-integrable function~$w\colon X\rar (0,1]$ such that the form
\begin{align}\label{eq:t:closabilityEwmgamma}
\EE{X}{w\cdot \mssm_\gamma}(f,g)\eqdef \int_X \Gamma(f,g) \, w\diff\mssm_\gamma \comma\qquad f,g\in \Dz\comma
\end{align}
is closable on~$L^2(w\cdot \mssm_\gamma)$ for~$\QP$-a.e.~$\gamma\in \dUpsilon$.
\end{enumerate}
\end{ass}

\begin{rem}
As a consequence of~\ref{ass:GNZ}, the intensity
\begin{align}\label{eq:Campbell2}
\mssm_\QP(A)\eqdef \int_{\dUpsilon} \gamma A\, \diff\QP(\gamma) =\iint_{X\times \dUpsilon} \car_A(x) \, \rho(\gamma,x) \diff \QP(\gamma) \diff\mssm(x)
\end{align}
of~$\QP$ is absolutely continuous w.r.t.~$\mssm$.
It is as well $\msE$-locally finite by Assumption~\ref{ass:Closability}\iref{i:ass:Closability:1}.
\end{rem}

\begin{ese}
The \emph{Georgii--Nguyen--Zessin-}, or \emph{Mecke}-, or \emph{Campbell-}type formula~\ref{ass:GNZ} holds for all (\emph{tempered}) \emph{grand-canonical Gibbs measures}, see e.g.,~\cite[Dfn.~2.1 and Prop.~5.2]{DaSKonRoe01}.
\end{ese}

If we denote by~$\CylQP{\QP}{\rep\Dz}$ the family of $\QP$-classes of functions in~$\Cyl{\rep\Dz}$, then, analogously to~\eqref{eq:Ss}, the bilinear form~$\rep\cdc^\dUpsilon\colon \Cyl{\rep \Dz}^\otym{2}\rar \R$ descends to a bilinear form~$\cdc^\dUpsilon\colon \CylQP{\QP}{\rep\Dz}^\otym{2}\rar \R$ by~\ref{ass:Closability}\iref{i:ass:Closability:5}.

We recall here the main result in~\cite[\S2]{MaRoe00}, where ---~in our terminology~--- a closable quadratic form is defined on the configuration space over an~\LDS.

\begin{thm}[{\cite[Thm.~2.6]{MaRoe00}}]\label{t:Closability}
Let~$(\mcX,\rep\cdc)$ be an~\LDS. Further let~$\QP$ be a probability measure on $\ttonde{\dUpsilon,\A_\mrmv(\msE)}$ satisfying Assumption~\ref{ass:Closability}.
Then, the form
\begin{align*}
\EE{\dUpsilon}{\QP}(u,v)\eqdef \int_{\dUpsilon} \cdc^{\dUpsilon}(u,v) \diff\QP\comma \qquad u,v\in \CylQP{\QP}{\rep\Dz} \comma
\end{align*}
is well-defined, and densely defined and closable in~$L^2(\QP)$.
Its closure~$\ttonde{\EE{\dUpsilon}{\QP},\dom{\EE{\dUpsilon}{\QP}}}$ is a Dirichlet form with carr\'e du champ operator~$\ttonde{\SF{\dUpsilon}{\QP},\dom{\SF{\dUpsilon}{\QP}}}$ so that
\begin{equation*}
\SF{\dUpsilon}{\QP}(u,v)=\cdc^\dUpsilon(u,v) \as{\QP}\comma\qquad u,v\in \CylQP{\QP}{\rep\Dz}\fstop
\end{equation*}
\end{thm}

\begin{rem}\label{r:DensityQP}
\begin{enumerate*}[$(a)$]
\item\label{i:r:DensityQP:1} Note that~\cite[Thm.~2.6]{MaRoe00} proves that, \emph{if}~$\ttonde{\EE{\dUpsilon}{\QP},\CylQP{\QP}{\rep\Dz}}$ is densely defined, \emph{then} it is a Dirichlet form. In our setting, Definition~\ref{d:DS}\iref{i:d:DS:1.5} is necessary and sufficient for the density of~$\CylQP{\QP}{\rep\Dz}$ in $L^2(\QP)$ by a monotone-class argument.
\item\label{i:r:DensityQP:2} 
As noted in~\cite[Rmk.~2.7]{MaRoe00}, the results in~\cite{MaRoe00}, in particular:~\cite[Thm.~2.6]{MaRoe00}, extend to the case when~$(\rep\cdc, \rep\Dz)$ takes values in~$\Sb(\msE,\mssm)$ (as opposed to~$\Sb(\msE)$).
Furthermore, it is also shown in~\cite{MaRoe00} that~$u=\ttclass[\QP]{F\circ\rep\mbff^\trid}\in\CylQP{\QP}{\rep\Dz}$ does not, in fact, depend on the representatives~$\rep\mbff$, but rather only on the $\mssm_\QP$-classes~$\mbff$, or, equivalently by~\eqref{eq:Campbell2}, on the $\mssm$-classes~$\mbff$.
We recall the precise statements below.
\end{enumerate*}
\end{rem}

Analogously to Definition~\ref{d:AbstractCompletion}, let~$\coK{\Dz}{1,\mssm_\QP}$ denote the closure of~$\Dz$ w.r.t.~$\norm{\emparg}_{1,\mssm_\QP}$, and~$\cokb_{1,\mssm_\QP}$ be the unique continuous extension to~$\coK{\Dz}{1,\mssm_\QP}$ of the inclusion~$\cok_{1,\mssm_\QP}\colon\Dz\rar L^1(\mssm)$.
In the same fashion, let~$\coK{\CylQP{\QP}{\rep\Dz}}{1,\QP}$ be the closure of~$\CylQP{\QP}{\rep\Dz}$ w.r.t.~$\norm{\emparg}_{1,\QP}$, and~$\cokb_{1,\QP}$ be the unique continuous extension to~$\coK{\CylQP{\QP}{\rep\Dz}}{1,\QP}$ of the inclusion~$\cok_{1,\QP}\colon \CylQP{\QP}{\rep\Dz}\rar L^1(\QP)$.

\begin{defs}[{\cite[\S4.2, p.~300]{MaRoe00}}] A function~$\rep u\colon \dUpsilon\rar \R\cup\set{\pm\infty}$ is called \emph{extended cylinder} if there exist~$k\in \N_0$, functions~$\rep f_1,\dotsc, \rep f_k$ with~$f_1,\dotsc,f_k\in \coK{\Dz}{1,\mssm_\QP}$, and a function~$F\in \Cb^\infty(\R^k)$, so that~$\rep u=F\circ\rep\mbff$.
We denote by
\begin{equation*}
\Cyl{\coK{\Dz}{1,\mssm_\QP}}
\end{equation*}
the space of all extended cylinder functions, and by~$\CylQP{\QP}{\coK{\Dz}{1,\mssm_\QP}}$ the space of their $\QP$-representa\-tives.
\end{defs}

We state here without proof the adaptations of~\cite[Lem.~4.3 and Prop.~4.6]{MaRoe00} to our setting.
We will provide a proof under a different set of assumptions in \S\ref{ss:IdentificationFormsProduct} below.
We note that all the aforementioned results in~\cite{MaRoe00} are presented under the standing assumption of~\cite[\S4]{MaRoe00} that~$\mcX$ be endowed with a metrizable topology.
However, they do not, in fact, depend on the choice of a topology.

\begin{lem}[{\cite[Lem.~4.3, Eqn.~(4.2)]{MaRoe00}}]\label{l:MaRoe4.3}
Let~$(\mcX,\cdc)$ be an \LDS and~$\QP$ be a probability measure on~$\ttonde{\dUpsilon,\A_\mrmv(\msE)}$ satisfying Assumption~\ref{ass:Closability}.
Further let~$\rep f\in \rep\Dz$. Then,
\begin{enumerate}[$(i)$]
\item\label{i:l:MaRoe4.3:1}$\tclass[\QP]{\rep f^\trid}\in\coK{\CylQP{\QP}{\rep\Dz}}{1,\QP}$ and~$\SF{\dUpsilon}{\QP}(\rep f^\trid)=\rep\cdc(\rep f)^\trid$;

\item\label{i:l:MaRoe4.3:2}$\tclass[\QP]{\SF{\dUpsilon}{\QP}(\rep f^\trid)}$ is independent of the chosen $\mssm$-representative $\rep f$ of~$f$;

\item\label{i:l:MaRoe4.3:3}$\ttnorm{\rep f^\trid}_{1,\QP}=\norm{f}_{1,\mssm_\QP}$.
\end{enumerate}
\end{lem}

The next Proposition, also taken from~\cite{MaRoe00}, shows that ---~\emph{a posteriori}~--- the form~$\EE{\dUpsilon}{\QP}$ is independent of any choice of representatives.

\begin{prop}[{\cite[Prop.~4.6]{MaRoe00}}]\label{p:ExtDom}
Let~$(\mcX,\cdc)$ be an \LDS and~$\QP$ be a probability measure on $\ttonde{\dUpsilon,\A_\mrmv(\msE)}$ satisfying Assumption~\ref{ass:Closability}.
Then,
\begin{enumerate}[$(i)$]
\item\label{i:p:ExtDom:1}~$u=\tclass[\QP]{F\circ \rep\mbff^\trid}\in \CylQP{\QP}{\coK{\Dz}{1,\mssm_\QP}}$ is defined in~$\dom{\EE{\dUpsilon}{\QP}}$ and independent of the $\mssm$-rep\-re\-sen\-ta\-tives~$\rep f_i$;

\item\label{i:p:ExtDom:2} for~$\rep u=F\circ \rep\mbff^\trid$ and $\rep v=G\circ \rep\mbfg^\trid\in \Cyl{\coK{\Dz}{1,\mssm_\QP}}$,
\begin{align}\label{eq:d:LiftCdC}
\SF{\dUpsilon}{\QP}(u,v)(\gamma)=& \sum_{i,j=1}^{k,m} (\partial_i F)(\rep\mbff^\trid\gamma) \cdot (\partial_j G)(\rep\mbfg^\trid\gamma) \cdot \rep\cdc(\rep f_i, \rep g_j)^\trid \gamma \as{\QP} \semicolon
\end{align}

\item\label{i:p:ExtDom:3} for every~$\rep f$ with~$f \in \coK{\Dz}{1,\mssm_\QP}$ one has~$\ttclass[\QP]{\rep f^\trid}\in\CylQP{\QP}{\coK{\Dz}{1,\mssm_\QP}}$, and
\begin{align*}
\SF{\dUpsilon}{\QP}\ttonde{\ttclass[\QP]{\rep f^\trid}}=\class[\QP]{\rep\cdc(\rep f)^\trid}
\end{align*}
is independent of the chosen $\mssm$- (i.e.,~$\mssm_\QP$-)representative~$\rep f$ of~$f$.
\end{enumerate}
\end{prop}

In particular, as a consequence of the above proposition, we are allowed to write~$\Cyl{\Dz}$ in place of~$\CylQP{\QP}{\Dz}$ for every~$\QP$ satisfying Assumption~\ref{ass:Closability}.

In spite of Proposition~\ref{p:ExtDom}, the construction of the form~$\ttonde{\EE{\dUpsilon}{\QP},\dom{\EE{\dUpsilon}{\QP}}}$ relies on a pointwise defined square field operator~$(\rep\cdc,\rep\Dz)$. Indeed, in order for a function~$\rep f^\trid$ on~$\dUpsilon$ to be well-defined, $\rep f$~needs to be pointwise defined at least on an $\mssm$-conegligible set (as opposed to: up to $\mssm$-equivalence). As a consequence, the construction of~\cite{MaRoe00} might ---~at first sight~--- appear unsuitable to address the case when the base space~$X$ is, for example, a metric measure space; that is, when any square field operator~$(\cdc,\Dz)$ would be naturally defined only on $\mssm$-classes of functions and, more importantly, take values only in spaces of $\mssm$-classes of functions.

Under some additional assumptions on a local structure~$\mcX$, we shall show in the next sections how to adapt the construction of~\cite{MaRoe00} in order to start with a space~$(\mcX,\cdc)$ and~$(\cdc,\Dz)$ in place of~$(\rep\cdc,\rep\Dz)$. Eventually, this will be done in~\S\ref{ss:Liftings} by resorting to the theory of \emph{liftings}.

\subsubsection{Poisson random measures}
As a main example of measures satisfying Assumption~\ref{ass:Closability}, we recall here the definition and main properties of Poisson random measures.
Since these measures will be central to our investigation, and in establishing further assumptions on a measure~$\QP$ sufficient for~\ref{ass:CWP} to hold, we list some of their properties here.

\begin{defs}[Poisson measures]\label{d:Poisson}
The \emph{Poisson} (\emph{random}) \emph{measure~$\PP_\mssm$} with intensity~$\mssm$ is the unique probability measure~$\QP$ on~$\ttonde{\dUpsilon, \A_\mrmv(\msE)}$ satisfying~\ref{ass:GNZ} with~$\rho\equiv \car$, viz.\ (cf.~\cite[Satz~3.1]{Mec67})
\begin{align}\label{eq:Mecke}
\iint_{\dUpsilon\times X} u(\gamma, x) \diff\gamma(x) \diff\PP_\mssm(\gamma)= \iint_{\dUpsilon\times X} u(\gamma+\delta_x,x) \diff\mssm(x) \diff\PP_\mssm(\gamma)
\end{align}
for every bounded $\A_\mrmv(\msE)\hotimes \A$-measurable~$u\colon \dUpsilon\times X\rar \R$.
\end{defs}

Alternatively,~$\PP_\mssm$ is characterized by its Laplace transform, cf.~\cite[Thm.~3.9]{LasPen18},
\begin{align}\label{eq:LaplacePoisson}
\int_{\dUpsilon} e^{f^\trid\gamma} \diff\PP_\mssm(\gamma)=\exp\tonde{\int_X (e^{f}-1) \diff\mssm}\comma \qquad f\in \Sb(X)^+\fstop
\end{align}

For~$E\in\msE$ set~$\mssm_E\eqdef \mssm\mrestr{E}$ and let~$\PP_{\mssm_E}$ be the Poisson measure with intensity~$\mssm_E$ on~$\Upsilon(E)$.
We recall some well-known properties of Poisson measures; see e.g.~\cite{AlbKonRoe98}.
Firstly:
\begin{itemize}
\item if~$\mssm X=\infty$, then~$\PP_\mssm$ is concentrated in~$\Upsilon^\sym{\infty}(\msE)$, viz.\ $\PP_\mssm \Upsilon^\sym{\infty}(\msE)=1$;
\item if~$\mssm X<\infty$, then~$\PP_\mssm$ is concentrated in~$\Upsilon^\sym{<\infty}(\msE)$, viz.\ $\PP_\mssm \Upsilon^\sym{<\infty}(\msE)=1$.
\end{itemize}

Secondly,~$\PP_\mssm$ is the unique probability measure on~$\dUpsilon$ satisfying
\begin{align}\label{eq:PoissonRestriction}
\pr^{E}_\pfwd \PP_\mssm=\PP_{\mssm_E}\comma \qquad E\in\msE \fstop
\end{align}
The projected measure~$\PP_{\mssm_E}$ ---~or~$\PP_\mssm$ itself, if~$\mssm X<\infty$~--- may be further characterized in the following way.
Let~$\mfS_n$ be the $n^\text{th}$ symmetric group, and denote by
\begin{equation}\label{eq:ProjSymmetricG}
\pr^\sym{n}\colon E^\tym{n}\rar E^\sym{n}\eqdef E^\tym{n}/\mfS_n
\end{equation}
the quotient projection, and by~$\mssm_E^\sym{n}$ the quotient measure~$\mssm_E^\sym{n}\eqdef \pr^\sym{n}_\pfwd \mssm^\hotym{n}$.
The space~$E^\sym{n}$ is naturally isomorphic to~$\dUpsilon^\sym{n}(E)$.
Under this isomorphism, the measure~$\PP_{\mssm_E}$, also called the \emph{Poisson--Lebesgue measure of~$\PP_\mssm$}, satisfies, e.g.~\cite[Eqn.~(2.5),~(2.6)]{AlbKonRoe98},
\begin{align}\label{eq:PoissonLebesgue}
\PP_{\mssm_E}=e^{-\mssm_E}\sum_{n=0}^\infty \frac{\mssm_E^\sym{n}}{n!} \fstop
\end{align}

\begin{rem}\label{r:PoissonClosability}
The measure~$\PP_\mssm$ satisfies Assumption~\ref{ass:Closability}.
\begin{proof}
As a consequence of~\eqref{eq:Mecke}, $\PP_\mssm$ has intensity measure~$\mssm$.
It follows that Assumption~\ref{ass:Closability}\iref{i:ass:Closability:1} is satisfied with~$\int \gamma E\diff\PP_\mssm(\gamma)=\mssm E$ for every~$E\in\msE$.
Furthermore, as soon as~$(\mcX,\rep\cdc)$ is an~\LDS, then Assumption~\ref{ass:Closability}\iref{i:ass:Closability:5} is satisfied with~$\mssm_\gamma=\mssm$ for every $\gamma\in\dUpsilon$, and Assumption~\ref{ass:Closability}\iref{i:ass:Closability:6} is satisfied with~$w\equiv \car$.
\end{proof}
\end{rem}

Let us further collect here some additional properties of Poisson measures which we shall need in the following.
For~$i\in [n]$ let~$a_i$ be non-negative integers,~$E_i\in \msE$ be pairwise disjoint.
Since the sets~$E_i$ are disjoint, one has, e.g.~\cite[Eqn.~(2.7)]{AlbKonRoe98},
\begin{align}\label{eq:AKR2.7}
\PP_{\mssm_F} \set{\gamma\in \Upsilon(F) : \gamma E_i=a_i \text{~~for } i\in [n] }=\prod_{i=1}^n \frac{(\mssm E_i)^{a_i} e^{-\mssm E_i}}{a_i!}\comma
\end{align}
and
\begin{align}\label{eq:PropertiesPP:4}
\PP_{\mssm_F}\ttonde{\Xi_{\geq n_1}(E_1) \cap \Xi_{\geq n_2}(E_2)}=\PP_{\mssm_F}\ttonde{\Xi_{\geq n_1}(E_1)} \cdot \PP_{\mssm_F}\ttonde{\Xi_{\geq n_2}(E_2)} \fstop
\end{align}

\paragraph{Concentration sets}
In order to study further properties of measures on configuration spaces, we introduce here a family of \emph{concentration sets}.
Namely, for~$E\in\msE$ we define the $n$-concentration sets of~$E$ as
\begin{equation}\label{eq:ConcentrationSet}
\Xi_{=n}(E)\eqdef\ \set{\gamma\in\dUpsilon: \gamma E= n} \comma
\end{equation}
and similarly for~`$\geq$' or~`$\leq$' in place of~`$=$'.
Analogously to the notation established for configuration spaces, we write~$\Xi^\sym{\infty}_{=n}(E)=\Xi_{=n}(E)\cap \dUpsilon^\sym{\infty}$.

\section{Analytic structure: topological local diffusion spaces}\label{s:Analytic}
A local structure is sufficient to the construction of the configuration space.
However, if the base space is endowed with additional data ---~for instance, if it is a topological space~---, it is intuitive that any additional property coming from that data may be lifted to the configuration space only if the local structure is ``compatible'' with the data.

\subsection{Base spaces}
In this section we study local diffusion spaces endowed with a topology compatible with both the localizing ring and the square field operator.
Besides collecting the main properties of such \emph{topological local diffusion spaces}, we shall additionally recall the concept of \emph{lifting} in measure theory, which will be of importance in the definition of forms on configuration spaces, as in~\S\ref{ss:TopConfig}, in a way that is compatible with the diffusion structure on most of the non-smooth base spaces discussed in~\S\ref{ss:ExamplesBase}.

\subsubsection{Topological local structures}\label{ss:TLS}
We recall that a \emph{Hausdorff} topological space~$(X,\T)$ is
\begin{itemize}
\item \emph{Polish} if there exists a distance~$\mssd$ on~$X$ inducing the topology~$\T$ and additionally so that~$(X,\mssd)$ is a separable and complete metric space;
\item \emph{Luzin} if it is a continuous bijective image of a Polish space;
\item \emph{Suslin} if it is a continuous image of a Polish space.
\end{itemize}

In the case when a local structure~$\mcX$ is additionally a topological space, the compatibility between the local structure and the topology is captured by the notion of \emph{topological local structure}.
\begin{defs}[Topological local structures]\label{d:TLS}
A \emph{topological local structure} is $\mcX\eqdef (X,\T,\A,\mssm,\msE)$ so that
\begin{enumerate}[$(a)$]
\item\label{i:d:TLS:1} $(X,\T)$ is a completely regular Luzin topological space, with Borel $\sigma$-algebra~$\Bo{\T}$;
\item $(X,\A,\mssm,\msE)$ is a local structure, with~$\Bo{\T}\subset \A\subset \Bo{\T}^\mssm$;
\item\label{i:d:TLS:3} $\mssm$ is a Radon measure on~$(X,\T,\A)$ with full $\T$-support;
\item\label{i:d:TLS:4} for every~$x\in X$ there exists a $\T$-neighborhood~$U_x$ of~$x$ so that~$U_x\in\msE$.
\end{enumerate}

When~$\mcX$ is a topological local structure, write:
\begin{itemize}
\item $\Cb(X)$ for the space of $\T$-continuous bounded functions on~$X$;

\item $\Cz(\msE)$ for the space of $\T$-continuous bounded $\msE$-\emph{eventually vanishing} functions on~$X$, viz.
\begin{align}\label{eq:CzExhaustion}
\Cz\ttonde{\msE}\eqdef \set{f\in\Cb(X) : f\equiv 0 \text{~~on } E^\complement \text{~~for some } E\in\msE}  \fstop
\end{align}
\end{itemize}
\end{defs}

We collect here some comments about topological local structures.
An explanation of the necessity and sufficiency of the Luzin property to our purposes is postponed to Remark~\ref{r:QRegBase}.

\begin{rem}[Miscellaneous topological properties]\label{r:TLS} Let~$\mcX=(X,\T,\A,\mssm,\msE)$ be a topological local structure.
\begin{enumerate*}[$(a)$]
\item\label{i:r:TLS:1} It follows from Definitions~\ref{d:MS}\ref{i:d:MS:4} and~\ref{d:TLS}\iref{i:d:TLS:1} that~$(X,\T)$ is in fact second countable, thus~$(X,\T)$ is separable metrizable~\cite[Thm.~4.2.9]{Eng89}. In particular: $(X,\T)$ is normal and paracompact; every $\sigma$-finite measure~$\mu$ on~$(X,\T,\A)$ has support~$\supp[\mu]$~\cite[II.7.2.9]{Bog07}.

\item\label{i:r:TLS:1.5} Since~$\mssm$ is assumed to be atomless by Definition~\ref{d:MS}\iref{i:d:MS:3} and has full support by Definition~\ref{d:TLS}\iref{i:d:TLS:3}, it follows that~$(X,\T)$ is perfect, i.e.\ no point in~$X$ is isolated.

\item\label{i:r:TLS:2}If~$(X,\T)$ is locally compact, then it is Polish~\cite[Thm.~5.3]{Kec95}.

\item\label{i:r:TLS:3}$\Cb(\T)\cap L^p(\mssm)$ is dense in~$L^p(\mssm)$ for every~$1\leq p <\infty$~\cite[Prop.~3.3.49]{HeiKosShaTys15}.

\item\label{i:r:TLS:4} We allow for some freedom in the choice of a $\sigma$-algebra~$\A$ on~$X$, for technical reasons. In particular, it is convenient to allow for either:
\begin{enumerate*}[$({d}_1)$]
\item\label{i:r:TLS:5.1}$\A=\Bo{\T}^*$ the universally measurable $\sigma$-algebra induced by~$\Bo{\T}$;
or
\item\label{i:r:TLS:5.2}$\A=\Bo{\T}^{\mssm_0}$ for any~$\mssm_0\ll \mssm$.
\end{enumerate*}
\iref{i:r:TLS:5.1} is motivated by results in the theory of metric measure spaces, e.g.\ the $\Bo{\T}^*$-measurability of slopes in~\cite[Lem.~2.6]{AmbGigSav14} or the fine analysis of Rademacher and Sobolev-to-Lipschitz properties in~\cite{LzDSSuz20};
\iref{i:r:TLS:5.2} is motivated by the need for intensity measures of point processes; see Definition~\ref{ass:Closability} below.

\item\label{i:r:TLS:6} Some natural choices for a localizing ring are: the algebra~$\A_\mssm$; the family~$\Ed$ of $\mssd$-bounded sets in~$\A_\mssm$ for some metric~$\mssd$ on~$X$; when~$(X,\T)$ is locally compact, the family~$\rKo{\T,\A}$ of all $\A$-measurable relatively $\T$-compact subset of~$X$, noting that all such sets have finite $\mssm$-measure since~$\mssm$ is Radon.
When~$\msE=\rKo{\T,\A}$, resp.\ when~$\msE=\Ed$, then~$\Cz(\msE)=\Cc(\T)$, the space of continuous compactly supported functions on~$X$, resp.~$\Cz(\msE)=\Cbs(\mssd)$, the space of continuous bounded functions on~$X$ with $\mssd$-bounded support.

\item\label{i:r:TLS:7} Combining Definitions~\ref{d:LS}\iref{i:d:LS:1} and~\ref{d:TLS}\iref{i:d:TLS:4}, it is readily seen that~$\msE$ contains a local basis of open neighborhoods of~$x$ for every~$x\in X$. In particular,~$\Cz(\msE)$ generates the topology~$\T$.
Furthermore, by paracompactness (see~\iref{i:r:TLS:1} above), there exists a countable locally finite open cover~$\msU$ of~$X$ with~$\msU\subset \msE$, and a partition of unity of~$(X,\T)$ subordinate to~$\msU$. In particular,~$\Cz(\msE)$ separates points in~$X$.
Finally, let~$K\in\Ko{\T}$. Since~$K$ may be covered by finitely many~$U_i\in \msU\subset \msE$, and since~$\msE$ is an algebra, then~$K\in \msE$. Thus,~$\Ko{\T}\subset \msE$.
\end{enumerate*}
\end{rem}

\begin{notat}\label{notat:ContRep}
Any $\mssm$-class~$f$ defined on a topological local structure has at most \emph{one} continuous $\mssm$-representative, by the properties of~$\mssm$.
Everywhere in the following, if~$f$ has a continuous $\mssm$-representative, say~$\rep f$, we shall always assume to be concerned with~$\rep f$; in this case ---~and only in this case~--- we drop the notation for representatives, thus writing~$f$ for both the $\mssm$-class \emph{and} for the continuous $\mssm$-representative.
In particular, for every~$f\in\Cz(\msE)$, we shall always make sense of~\eqref{eq:Trid} by means of the continuous $\mssm$-representative of~$f$.
\end{notat}

\subsubsection{Liftings}\label{ss:Liftings} Let~$\mcX$ be a local structure endowed with a square field operator~$(\cdc,\Dz)$. In order to construct a Dirichlet form on~$\dUpsilon$ by means of Theorem~\ref{t:Closability}, we need to define functionals of the form~$\cdc(f)^\trid$. This may be achieved na\"ively by the choice of representatives for the $\mssm$-class~$\cdc(f)$.
However, this choice of representatives ought to be consistent with the standard algebraic operations on functions, and, more importantly, with the chain rule and the Leibniz rule for~$\nabla$ encoded in~\eqref{eq:i:d:SF:3}. This will be made possible by the notion of lifting.

Let~$\mcX$ be a topological local structure.
For~$\rep\ell\colon \mcL^\infty(\mssm)\rar \mcL^\infty(\mssm)$, define the following properties
\begin{enumerate*}[$(a)$]
\item\label{i:d:Liftings:1} $\tclass{\rep\ell(f)}=\class{f}$; 
\item\label{i:d:Liftings:2} if $\class{f}=\class{g}$, then~$\rep\ell(f)=\rep\ell(g)$;
\item\label{i:d:Liftings:3} $\rep\ell(\car)=\car$;
\item\label{i:d:Liftings:4} if~$f\geq 0$, then~$\rep\ell(f)\geq 0$;
\item\label{i:d:Liftings:5} $\rep\ell(a f+b g)=a\rep\ell(f)+b\rep\ell(g)$ for~$a,b\in \R$;
\item\label{i:d:Liftings:6} $\rep\ell(fg)=\rep\ell(f)\rep\ell(g)$;
\item\label{i:d:Liftings:7} $\rep\ell(\phi)=\phi$ for~$\phi\in\Cb(X)$.
\end{enumerate*}

\begin{defs}[Liftings]\label{d:Liftings}
A \emph{linear lifting} is a map~$\rep\ell\colon \mcL^\infty(\mssm)\rar \mcL^\infty(\mssm)$
satisfying~\iref{i:d:Liftings:1}--\iref{i:d:Liftings:5}. Any such~$\rep\ell$ is a (\emph{multiplicative}) \emph{lifting} if it satisfies~\iref{i:d:Liftings:1}--\iref{i:d:Liftings:6}. Finally, it is a \emph{strong lifting} if it satisfies~\iref{i:d:Liftings:1}--\iref{i:d:Liftings:7}.
A \emph{Borel} (linear) \emph{lifting} is a (linear) lifting with~$\A=\Bo{\T}$.
As a consequence of~\iref{i:d:Liftings:1}, a linear lifting~$\rep\ell\colon \mcL^\infty(\mssm)\rar \mcL^\infty(\mssm)$ descends to a linear order-preserving inverse~$\ell\colon L^\infty(\mssm)\rar \mcL^\infty(\mssm)$ of the quotient map~$\class[\mssm]{\emparg}\colon\mcL^\infty(\mssm)\rar L^\infty(\mssm)$.
As customary, by a (\emph{linear}/\emph{multiplicative}/\emph{strong}/\emph{Borel}) \emph{lifting} we shall mean without distinction either~$\rep\ell$ or~$\ell$ as above.
\end{defs}

The existence of a lifting is non-trivial, and essentially related to the ``smallness'' of the $\sigma$-algebra~$\A$. This is encoded in our very Definition~\ref{d:LS}\iref{i:d:MS:4} of a measure space, thus granting the existence of a lifting.

\begin{thm}[e.g.,~{\cite[Thm.~4.12]{StrMacMus02}}]
Every topological local structure admits a strong (Borel) lifting.
\end{thm}

Let~$\mcX$ be a topological local structure, and~$(\Gamma,\Dz)$ be a square field operator on~$\mcX$ with~$\rep\Dz\subset \Cb(X)$. Here and in the following we write~$\rep\Dz\subset \Cb(\T)$ to indicate that every $\mssm$-class in~$\Dz$ has a continuous representative. For any strong lifting~$\ell\colon L^\infty(\mssm)\rar \mcL^\infty(\mssm)$, we set
\begin{align}\label{eq:LiftCdC}
\rep\cdc_\ell(f)\eqdef \ell(\cdc(f)) \comma \qquad f\in\Dz\comma
\end{align}
and denote again by~$\rep\cdc_\ell\colon \ell(\Dz)^{\otym 2}\rar \mcL^\infty(\mssm)$ the bilinear form induced by~$\rep\cdc_\ell$ by polarization. By the strong lifting property for~$\ell$ it is readily verified that~$\ell(\Dz)=\rep\Dz\subset\Cb(\T)$ and we have the following.

\begin{prop}\label{p:StrongLiftingCdC} $(\rep\cdc_\ell,\rep\Dz)$ is a pointwise defined square field operator satisfying
\begin{align*}
\tclass[\mssm]{\rep\cdc_\ell(\emparg)}=\cdc(\emparg) \qquad \text{on} \qquad \rep\Dz\fstop
\end{align*}
\end{prop}

We may now introduce the main object of our analysis. We assume the reader to be familiar with the notions of \emph{nests}, \emph{polar sets}, and \emph{quasi-regularity}, all from the theory of Dirichlet forms.
\begin{defs}[Topological local diffusion spaces]\label{d:TLDS}
A \emph{topological local diffusion space} (in short: \TLDS) is a pair $(\mcX,\cdc)$ so that
\begin{enumerate}[$(a)$]
\item\label{i:d:TLDS:1} $\mcX$ is a topological local structure;
\item\label{i:d:TLDS:2} $\Dz\subset \Cz(\msE)$ is a subalgebra of~$\Cz(\msE)$ generating the topology~$\T$ of~$\mcX$;
\item\label{i:d:TLDS:3} $\cdc\colon \Dz^\otym{2}\rar \class[\mssm]{\Sb(\msE)}$ is a square field operator;
\item\label{i:d:TLDS:4} the bilinear form~\eqref{eq:i:d:DS:3} is closable and densely defined in~$L^2(\mssm)$;
\item\label{i:d:TLDS:5} its closure~$\ttonde{\EE{X}{\mssm},\dom{\EE{X}{\mssm}}}$ is a quasi-regular Dirichlet form in the sense of e.g.~\cite{CheMaRoe94}.
\end{enumerate}
\end{defs}

In comparison with the Definitrion~\ref{d:DS} of~\LDS, we stress that the square field operator~$\cdc$ on a \TLDS takes values in a space of $\mssm$-classes, \emph{not} representatives.

\begin{rem}[On quasi-regularity and regularity]\label{r:QRegBase}
The $\T$-quasi-regularity of~$\ttonde{\EE{X}{\mssm},\dom{\EE{X}{\mssm}}}$ can be checked on a wide class of examples, including all our examples of base spaces, as discussed in~\S\ref{ss:ExamplesBase}.

When~$\ttonde{\EE{X}{\mssm},\dom{\EE{X}{\mssm}}}$ is quasi-regular, there exists an $\EE{X}{\mssm}$-polar subset~$N\subset X$ so that~$X\setminus N$ is a Luzin space, thus when dealing with quasi-regular Dirichlet forms one can assume without loss of generality that the undelrying topological space~$(X,\T)$ be Luzin (cf.\ e.g.~\cite[Rmk.~IV.3.2(iii), p.~101]{MaRoe92}), which motivates Definition~\ref{d:TLS}\iref{i:d:TLDS:1}.

Let~$(\mcX,\cdc)$ be a \TLDS, and assume further that its topology~$\T$ is locally compact, and that~$\msE=\rKo{\T,\A}$ as in Remark~\ref{r:TLS}\iref{i:r:TLS:6}.
In this case,~$\Dz$ is a core for the form~$\ttonde{\EE{X}{\mssm},\dom{\EE{X}{\mssm}}}$ by Definition~\ref{d:DS}\iref{i:d:DS:3}, consisting of $\T$-continuous compactly supported functions, since we have that~$\Cz(\msE)=\Cz(\T)$ by Remark~\ref{r:TLS}\iref{i:r:TLS:7}.
Thus finally, the form~$\ttonde{\EE{X}{\mssm},\dom{\EE{X}{\mssm}}}$ is regular, since additionally~$\mssm$ is Radon by~\ref{d:TLS}\iref{i:d:TLS:3}.
\end{rem}

As a first application of the theory of liftings, we show how to extend Theorem~\ref{t:Closability} to the case of square field operators on~$X$ defined on~$\mssm$-a.e..

\begin{prop}[Closability I]\label{p:MRLifting}
Let~$(\mcX,\cdc)$ be a \TLDS, and~$\ell$ be a strong lifting.
Further let~$\QP$ be a probability measure on~$\ttonde{\dUpsilon,\A_\mrmv(\msE)}$ satisfying Assumption~\ref{ass:Closability} with~$(\rep\cdc_\ell,\rep\Dz)$ in place of~$(\rep\cdc,\rep\Dz)$.
Then, the form
\begin{align*}
\ttonde{\EE{\dUpsilon}{\QP},\dom{\EE{\dUpsilon}{\QP}}}=\ttonde{\EE{\dUpsilon}{\QP}_\ell,\dom{\EE{\dUpsilon}{\QP}_\ell}}
\end{align*}
is well-defined, densely defined and closable on~$L^2(\QP)$, and it is independent of~$\ell$.
\begin{proof}
It follows from Proposition~\ref{p:StrongLiftingCdC}, Definition~\ref{d:Liftings}\iref{i:d:Liftings:2} and Remark~\ref{r:DensityQP}\iref{i:r:DensityQP:2} that we may associate to~$\rep\cdc_\ell$ and~$\QP$ the lifted Dirichlet form~$\ttonde{\EE{\dUpsilon}{\QP}_\ell,\dom{\EE{\dUpsilon}{\QP}_\ell}}$, possibly depending on~$\ell$, constructed in Theorem~\ref{t:Closability}.
By Proposition~\ref{p:ExtDom}, the form is uniquely associated to the square field operator~$(\cdc,\Dz)$, thus it does not depend on~$\ell$.
\end{proof}
\end{prop}

\noindent In other words, the form~$\ttonde{\EE{\dUpsilon}{\QP},\dom{\EE{\dUpsilon}{\QP}}}$ is independent of the way we chose to lift the square field operator~$(\cdc,\Dz)$ to a pointwise defined square field operator~$(\rep\cdc,\rep\Dz)$.

\begin{rem}
\begin{enumerate*}[$(a)$]
\item Proposition~\ref{p:MRLifting} establishes in particular the construction of the Dirichlet form $\ttonde{\EE{\dUpsilon}{\PP_\mssm},\dom{\EE{\dUpsilon}{\PP_\mssm}}}$ for the Poisson measure~$\PP_\mssm$.
This form will be instrumental to the construction of the forms for general~$\QP$ as discussed later in this section.
\item We stress that the choice of a topology on~$\mcX$ is only instrumental to the existence of a strong lifiting~$\ell$.
If the existence of a linear lifting respecting the diffusion structure (i.e.\ the chain rule for~$\cdc$) were given \emph{a priori}, then the same construction would apply.
In this sense, Proposition~\ref{p:MRLifting} is \emph{independent} of the topology on~$\mcX$.
\end{enumerate*}
\end{rem}

The full extent of the applications of lifting theory to the closability of Dirichlet forms on~$L^2(\QP)$ will be discussed in~\S\ref{ss:FurtherClosability} below, for a larger class of measures~$\QP$.

\begin{rem}[Comparison with~{\cite[\S4]{MaRoe00}}, part I]\label{r:ComparisonMR00-1}
The choice of a topology on~$X$ will be essential in determining whether the Dirichlet form~$\ttonde{\EE{\dUpsilon}{\QP},\dom{\EE{\dUpsilon}{\QP}}}$ is (properly) associated to a Markov process on~$\dUpsilon$. %, \purple{which will be the subject of upcoming work by the authors}.
To this end, we note that our definition of \TLDS matches the same level of generality as the assumptions in~\cite[\S4]{MaRoe00}, and in particular that
\begin{enumerate*}[$(a)$]
\item\label{i:r:ComparisonMR:1} the continuity of functions in~$\Dz$ is assumed in~\cite[\S4, p.~298]{MaRoe00} shortly before Equation~(4.1) there;
\item together with the inclusion~$\Dz\subset \Sb(\msE)$ in~\cite[\S1.1, p.~276]{MaRoe00},~\iref{i:r:ComparisonMR:1} above exactly amounts to the assumption that~$\Dz\subset \Cz(\msE)$.
\end{enumerate*}
\end{rem}

\begin{notat}\label{notat:LiftingOmitted}
Recalling Notation~\ref{n:Form}, we write~$\repSF{X}{\mssm}_\ell\eqdef \ell\circ \SF{X}{\mssm}$, extending~\eqref{eq:LiftCdC}.
Whenever unnecessary, the specification of the lifting~$\ell$ is omitted, thus we only write~$\repSF{X}{\mssm}$.
\end{notat}

\subsection{Configuration spaces}\label{ss:TopConfig}
Let us study here more in detail the topological properties of configuration spaces over topological local structures.
We start with the definition of a natural topology on~$\dUpsilon$.

\begin{defs}[Vague topology on~$\dUpsilon$]
Let~$\mcX$ be a topological local structure.
We endow~$\dUpsilon$ with the $\msE$-\emph{vague topology}~$\T_\mrmv(\msE)$, generated by functions of the form
\begin{align*}
f^\trid\colon \gamma\longmapsto \gamma f \comma\qquad f\in\Cz(\msE)\fstop
\end{align*}
\end{defs}

We note that, if~$\mcX$ is a topological local structure, then~$\Bo{\T_\mrmv(\msE)}\subset\A_\mrmv(\msE)$, with equality if~$\A=\Bo{\T}$.

\begin{prop}\label{p:TopologyUpsilon}
Let~$\mcX$ be a topological local structure. Then,
\begin{enumerate}[$(i)$]
\item\label{i:p:TopologyUpsilon:1} $\ttonde{\Upsilon(\msE),\T_\mrmv(\msE)}$ and $\ttonde{\dUpsilon,\T_\mrmv(\msE)}$ are separable metrizable;

\item\label{i:p:TopologyUpsilon:2} if $(X,\T)$ is Polish, then so are $\ttonde{\Upsilon(\msE),\T_\mrmv(\msE)}$ and $\ttonde{\dUpsilon,\T_\mrmv(\msE)}$;

\item\label{i:p:TopologyUpsilon:3} if $(X,\T)$ is locally compact and~$\msE=\rKo{\T,\A}$, then~$\ttonde{\Upsilon(\msE),\T_\mrmv(\msE)}$ and~$\ttonde{\dUpsilon,\T_\mrmv(\msE)}$ are Polish;

\item\label{i:p:TopologyUpsilon:4} $\ttonde{\dUpsilon, \T_\mrmv(\msE), \A_\mrmv(\msE), \QP,  \A_\mrmv(\msE)}$ is a topological local structure for every probability measure~$\QP$ with full support.
\end{enumerate}

\begin{proof}
\iref{i:p:TopologyUpsilon:1} follows from~\cite[Thm.~3.6, Prop.~3.9]{MaRoe00}. \iref{i:p:TopologyUpsilon:2} for~$\dUpsilon$ is again~\cite[Thm.~3.6, Prop.~3.9]{MaRoe00}. By~\cite[Prop.~3.10]{MaRoe00},~$\ttonde{\Upsilon(\msE),\T_\mrmv(\msE)}$ is a $G_\delta$ set in~$\ttonde{\dUpsilon,\T_\mrmv(\msE)}$. Since the latter is Polish by~\cite[Thm.~3.6, Prop.~3.9]{MaRoe00}, then so is~$\ttonde{\Upsilon(\msE),\T_\mrmv(\msE)}$.
\iref{i:p:TopologyUpsilon:3} The coincidence of spaces and topologies is noted in~\cite[p.~295]{MaRoe00}. If~$(X,\T)$ is locally compact, then it is Polish, see Rmk.~\ref{r:TLS}, hence~\iref{i:p:TopologyUpsilon:2} applies.
\iref{i:p:TopologyUpsilon:4} The space~$\ttonde{\dUpsilon, \T_\mrmv(\msE)}$ is separable metrizable by~\iref{i:p:TopologyUpsilon:1}. Every finite topological measure on a separable metrizable space is Radon; see~\cite[451M]{Fre00}.
\end{proof}
\end{prop}

\begin{prop}\label{p:ConfigLuzin}
Let~$\mcX$ be a local structure. Then,~$\ttonde{\dUpsilon,\T_\mrmv(\msE)}$ is a Luzin space.
\begin{proof}
Firstly, note that reference measures are inessential to the discussion, hence they will be dropped from the notation throughout the rest of the proof.
By definition of the Luzin property for~$(X,\T)$, there exist a Polish space~$(P,\T^P)$ and a $\T^P/\T$-continuous bijection~$\varphi \colon P\to X$ so that~$X=f(P)$.
Let~$\msE$ be the localizing ring of~$\mcX$ and note that~$\varphi^*\msE\eqdef \set{\varphi^{-1}(E): E\in\msE}$ is a localizing ring of~$P$. 
It readily checked that~$\mcP\eqdef (P,\T^P,\Bo{\T^P},\mssp,\varphi^*\msE)$ is a topological local structure for every atomless (Borel) probability measure~$\mssp$.
Since~$P$ is a Polish space,~$\ttonde{\dUpsilon(\varphi^*\msE), \T^P_\mrmv(\varphi^*\msE)}$ is Polish as well by Proposition~\ref{p:TopologyUpsilon}\ref{i:p:TopologyUpsilon:2}.
Now, let~$\eta\in\dUpsilon(\varphi^*\msE)$.
Since~$\varphi$ is in particular $\Bo{\T^P}/\Bo{\T}$-measurable, $\varphi_\pfwd\eta$~is a Borel measure on~$X$, $\overline\N_0$-valued, and additionally satisfying~$\varphi_\pfwd \eta\set{x}\in \N_0$ for every~$x\in X$, since~$\varphi$ is injective.
Furthermore,~$\varphi_\pfwd\eta E=\eta \varphi^{-1}(E)\in \N_0$ for every~$E\in\msE$ by definition of~$\varphi^*\msE$ and~$\eta$.
Thus,~$\varphi_\pfwd$ maps~$\dUpsilon(\varphi^*\msE)$ into~$\dUpsilon$.
Since~$\varphi\colon P\to X$ is bijective and by definition of the localizing ring~$\varphi^*\msE$ on~$P$, the map~$\varphi_\pfwd\colon \dUpsilon(\varphi^*\msE)\to\dUpsilon$ is bijective.

Now, let~$f\in\Cz(\T,\msE)$ and note that~$f^\trid(\varphi_\pfwd\gamma)=(f\circ \varphi)^\trid \ttonde{(\varphi_\pfwd)^{-1}(\gamma)}$.
Since~$\varphi$ is $\T^P/\T$-continuous, we have that~$f\circ \varphi\in\Cz(\T^P,\varphi^*\msE)$.
Since the $\varphi^*\msE$-vague topology~$\T^P_\mrmv(\varphi^*\msE)$ on~$\dUpsilon(\varphi^*\msE)$ is generated by all maps of the form~$g^\trid$,~$g\in\Cz(\T^P,\varphi^*\msE)$, we conclude that the pullback maps~$(\varphi_\pfwd)^*f^\trid$ are $\T^P_\mrmv(\varphi^*\msE)$-continuous.
Thus, finally, we may conclude that $\varphi_\pfwd\colon \dUpsilon(\varphi^*\msE)\to\dUpsilon$ is $\T^P_\mrmv(\varphi^*\msE)/\T_\mrmv(\msE)$-continuity.
Thus,~$\dUpsilon$ is a continuous bijective image, via~$\varphi_\pfwd$, of the Polish space~$\ttonde{\dUpsilon(\varphi^*\msE),\T_\mrmv(\varphi^*\msE)}$, and therefore it is a Luzin space.
\end{proof}
\end{prop}

\begin{rem}\label{r:CylContinuity} The following statements are straightforward:
\begin{enumerate}[$(a)$]
\item\label{i:r:CylContinuity:1} $\Cyl{\Cz(\msE)}\subset \Cb\ttonde{\dUpsilon,\T_\mrmv(\msE)}$;
\item\label{i:r:CylContinuity:2} if~$f\in \Sb(\msE)$ is $\T$-u.s.c.\ (resp.\ l.s.c.), then~$f^\trid$ is $\T_\mrmv(\msE)$-u.s.c.\ (resp.\ l.s.c.);
\item\label{i:r:CylContinuity:3} if~$F\in\msE$ is $\T$-closed, then the concentration set~$\Xi_{\leq n}(F)$ in~\eqref{eq:ConcentrationSet} is $\T_\mrmv(\msE)$-open;
\item\label{i:r:CylContinuity:4} if~$U\in\msE$ is $\T$-open, then the concentration set~$\Xi_{\leq n}(U)$ in~\eqref{eq:ConcentrationSet} is $\T_\mrmv(\msE)$-closed.
\end{enumerate}
\end{rem}

\subsubsection{Labeling maps and cylinder sets}\label{ss:Labelings}
Let~$\mcX$ be a local structure. It is clear that for each~$\gamma\in \dUpsilon$, say with~$\gamma X=N$, we may choose an \emph{ordered} sequence of points~$\seq{x_i}_i\in X^\tym{N}$, so that~$\gamma=\sum_{i=1}^N \delta_{x_i}$. We show here that, at least in the case when~$\mcX$ is a topological local structure, we can make this choice in a measurable way.

We shall start by collecting the necessary notation.  Let
\begin{align}\label{eq:d:Labelings:1}
\mbfX\eqdef \set{\emp} \sqcup \bigsqcup_{N\in \overline\N_1} X^\tym{N}
\end{align}
be endowed with its natural (direct sum of products) $\sigma$-algebra~$\boldSigma$. For~$M\in\overline\N_1$, we define the truncation~$\tr^M\colon \mbfX\rar \mbfX$ by
\begin{align}\label{eq:TruncationDef}
\tr^M\colon \mbfx\longmapsto \mbfx^\asym{M}\eqdef& \begin{cases} \seq{x_1,\dotsc, x_M} &\text{if } \mbfx=\seq{x_p}_{p\leq N} \text{ and } M\leq N \text{ and } M<\infty\comma
\\
\mbfx &\text{if } \mbfx=\seq{x_p}_{p\leq N} \text{ and } M> N \text{ or } N=M=\infty
\end{cases} \fstop
\end{align}
The map~$\tr^M$ is clearly both $\boldSigma/\boldSigma$- and Borel measurable for every~$M$.

\medskip

We denote by~$\mbfX_\locfin(\msE)$ the space of $\msE$-locally finite, finite or infinite sequences in~$X$, viz.
\begin{align}\label{eq:Xlocfin}
\mbfX_\locfin(\msE)\eqdef& \set{\mbfx\eqdef \seq{x_i}_{i\leq N}\in \mbfX: \#\set{i:x_i\in E}<\infty \text{~for all } E\in\msE}\comma
\end{align}
endowed with the subspace $\sigma$-algebra~$\boldSigma_\locfin(\msE)$. Analogously to the case of~$\dUpsilon$, let us also set
\begin{align*}
\mbfX_\locfin^\asym{\infty}(\msE)\eqdef&X^\tym{\infty}\cap\mbfX_\locfin(\msE)\fstop
\end{align*}
We note that, in the definition~\eqref{eq:Xlocfin} of~$\mbfX_\locfin(\msE)$ above, we might equivalently require, for every~$h\in \N_1$, that $\#\set{i:x_i\in E_h}<\infty$, where~$\seq{E_h}_h$ is any localizing sequence generating~$\msE$ as in Definition~\ref{d:LS}\iref{i:d:LS:2}. Since $\mbfx\mapsto \#\set{i:x_i\in A}$ is $\boldSigma$-measurable if $A\in \A$, it follows that~$\mbfX_\locfin(\msE)$, hence~$\mbfX^\asym{\infty}_\locfin(\msE)$, is an element of~$\boldSigma$.
Since, by Remark~\ref{r:TLS}\iref{i:r:TLS:1} and~\iref{i:r:TLS:7},~$\msE$ is countably generated by open sets,~$\mbfX^{\asym{\infty}}_\locfin(\msE)$, hence~$\mbfX_\locfin(\msE)$, is also~$\Bo{\T^\tym{\infty}}$-measurable.
Finally, let~$\Lb\colon \mbfX_\locfin(\msE)\longrar \dUpsilon$ be defined by
\begin{equation}\label{eq:d:Labelings}
\begin{aligned}
\Lb\colon \mbfx\eqdef\seq{x_p}_{p\leq N} &\longmapsto \gamma \eqdef \sum_{p=1}^N\delta_{x_p} \comma\qquad \emp \longmapsto \zero\fstop
\end{aligned}
\end{equation}

\begin{rem}[Lack of continuity of~$\Lb$]\label{r:DiscontinuityL}
The map~$\Lb$ is not continuous.
Indeed, let~$\sigma_n$ be the transposition~$(1\, n)$ in~$\mfS_\infty$, fix~$\mbfx_1\eqdef \seq{x_i}_{i=1}^\infty\in \mbfX^\asym{\infty}_\locfin(\msE)$ with~$x_i\neq x_j$ for~$i\neq j$, and set~$\mbfx_n\eqdef(\mbfx_1)_{\sigma_n}$.
Then,~$\T^\tym{\infty}$-$\nlim \mbfx_n=\mbfx_\infty\eqdef\seq{x_2,x_3,\dotsc}$, yet~$\Lb(\mbfx_n)=\Lb(\mbfx_1)$ for every~$n$, hence~$\T_\mrmv(\msE)$-$\nlim \Lb(\mbfx_n)\neq \Lb(\mbfx_\infty)$.
\end{rem}

\begin{lem}[Measurability of~$\Lb$]\label{l:MeasurabilityL}
The map~$\Lb\colon \mbfX_\locfin(\msE)\longrar \dUpsilon$ is $\boldSigma_\locfin(\msE)/\A_\mrmv(\msE)$-measurable.
In particular, it is $\Bo{\T^\tym{\infty}}/\Bo{\T_\mrmv(\msE)}$-measurable and $\Bo{\T^\tym{\infty}}^*/\Bo{\T_\mrmv(\msE)}^*$-measurable.
\begin{proof}
We adapt the arguments in~\cite[p.~12]{KonLytRoe02}.
Note that~$\A_\mrmv(\msE)$ is generated by sets of the form~$\Xi_{=n}(E)$ as in~\eqref{eq:ConcentrationSet} with~$E\in\msE$ and~$n\in \N_0$.

Thus, it suffices to show that~$\Lb^{-1}\ttonde{\Xi_{=n}(E)}\in\boldSigma$ for every~$\Xi_{=n}(E)$ as above. 
In fact, it suffices to show that~$\Lb\colon \mbfX_\locfin^\asym{\infty}(\msE)\longrar \dUpsilon^\sym{\infty}$ is $\boldSigma_\locfin(\msE)/\A_\mrmv(\msE)$-measurable, so we can further restrict our attention to sets of the form~$\Xi^\sym{\infty}_{=n}(E)$.
It holds that
\begin{align*}
\Lb^{-1}\ttonde{\Xi_{=n}^\sym{\infty}(E)}=\bigcup_{\sigma\in \mfS_0(\N_1)} \set{\mbfx\in\mbfX_\locfin^\sym{\infty}(\msE):\; \begin{matrix} x_{\sigma(i)}\in E \text{ for } i\leq n\comma
\\
x_{\sigma(i)}\in E^\complement \text{ for } i>n\phantom{\comma}
\end{matrix}
}\comma
\end{align*}
where~$\mfS_0(\N_1)$ denotes the group of all bijections of~$\N_1$ with cofinitely many fixed points.
Since points in~$X$ are measurable, $\mbfx\eqdef\seq{x_p}_{p\leq N}\mapsto \mbfx_\sigma\eqdef \seq{x_{\sigma(p)}}_{p\leq N}$ is $\boldSigma/\boldSigma$-measurable.
Thus, since $\mfS_0(\N_1)$ is countable, then~$\Lb^{-1}\ttonde{\Xi_{=n}^\sym{\infty}(E)}\in\boldSigma$, which concludes the proof.
\end{proof}
\end{lem}

Let~$\QP$ be any probability measure on~$\ttonde{\dUpsilon,\A_\mrmv(\msE)}$, and recall that~$\A_\mrmv(\msE)^\QP$  denotes the completion of~$\A_\mrmv(\msE)$ w.r.t.~$\QP$.

\begin{prop}\label{p:Selection}
For every~$\QP$ as above, the correspondence $\Lb^{-1}\colon \dUpsilon\rightrightarrows \mbfX_\locfin(\msE)\subset \mbfX$ has a $\A_\mrmv(\msE)^\QP/\Bo{\mbfX_\locfin(\msE)}$-measur\-able selection.
\begin{proof}
Note that~$(X,\T)$ is a Suslin space (since Luzin). Hence,~$\mbfX$ is Suslin as well, and thus also~$\mbfX_\locfin(\msE)$ is Suslin, being a Borel subset of a Suslin space (e.g.,~\cite[Cor.~6.6.7]{Bog07}).
By Proposition~\ref{p:ConfigLuzin},~$\ttonde{\dUpsilon,\T_\mrmv(\msE)}$ is a Luzin space, hence Suslin.
As a consequence, both~$\mbfX_\locfin(\msE)$ and~$\dUpsilon$ are Suslin spaces.

Thus,~$\Lb$ is a surjective map between Suslin spaces, Borel-measurable by Lemma~\ref{l:MeasurabilityL}.
The conclusion now follows from Yankov's Measurable Selection Theorem~\cite[Thm.~6.9.1]{Bog07}.
\end{proof}
\end{prop}

The previous proposition will be sufficient for our purposes.
However, let us note that it holds true even if we replace the $\sigma$-algebra~$\A_\mrmv(\msE)^\QP$ with either the (smaller) $\sigma$-algebra~$\A_\mrmv(\msE)^*$ of universally measurable subsets of~$\ttonde{\dUpsilon,\T_\mrmv(\msE)}$ or the (smaller still) $\sigma$-algebra generated by all Suslin sets in~$\dUpsilon$.

\paragraph{Labeling maps}
Proposition~\ref{p:Selection} justifies the following definition.
\begin{defs}[Labeling maps]
A \emph{labeling map}~$\lb\colon \dUpsilon\rar \mbfX_\locfin(\msE)$ is any $\A_\mrmv(\msE)^\QP/\Bo{\mbfX_\locfin(\msE)}$-measurable right inverse of~$\Lb$.
\end{defs}

\begin{rem}[Properties of labeling maps]
Again as a consequence of Yankov's Measurable Selection Theorem (in the form~\cite[Thm.~6.9.1]{Bog07}), we have that, for every labeling map~$\lb$, the set~$\lb\ttonde{\dUpsilon}$ belongs to the $\sigma$-algebra generated by the Suslin sets in~$\mbfX_\locfin(\msE)$.
In particular, it is universally measurable.
Let~$\QP$ be a probability measure on~$\ttonde{\dUpsilon,\Bo{\T_\mrmv(\msE)}}$, and denote by~$\QP^*$ the \emph{unique} extension of~$\QP$ to the $\sigma$-algebra~$\Bo{\T_\mrmv(\msE)}^*$ of universally measurable sets.
Everywhere in the following ---~with slight abuse of notation~--- we shall denote the push-forward via~$\lb$ of~$\QP^*$ simply by by~$\lb_\pfwd\QP$ (as opposed to:~$\lb_\pfwd\QP^*$).
\end{rem}

It is worth to state here a negative result on the continuity of labeling maps, paired with the lack of continuity for~$\Lb$ noted in Remark~\ref{r:DiscontinuityL}.
Whereas this result will not be used in the rest of the paper, it is important to stress that the topological properties of~$\dUpsilon$ and those of~$\mbfX_\locfin(\msE)$ cannot be related via the study of~$\Lb$ or its inverses.

\begin{prop}[Lack of continuity of~$\lb$]\label{p:Discontinuity-l}
Let~$\mcX$ be a topological local structure, and assume that \emph{either}
\begin{enumerate}[$(a)$]
\item\label{i:p:Discontinuity-l:1} it contains a non-trivial simple loop with image in~$\msE$;
\item\label{i:p:Discontinuity-l:2} it contains a simple infinite curve with image not in~$\msE$.
\end{enumerate}
Then no labeling map~$\lb\colon \dUpsilon\to \mbfX_\locfin(\msE)$ is $\T_\mrmv(\msE)/\T^\tym{\infty}$-continuous.

\begin{proof}
\iref{i:p:Discontinuity-l:1} Let~$\seq{a_t}_{t\in [0,2]} \subset X$ be a non-trivial simple continuous loop.
Consider the curve of configurations~$\gamma_t=\delta_{x^1_t}+\delta_{x^2_t}$ with~$x^1_t\eqdef a_t$ and~$x^2_t\eqdef a_{t+1}$, and note that~$\gamma_0=\gamma_{1}$ since~$a_0=a_1$.
Now, argue by contradiction that there exists a continuous labeling map~$\lb$. Without loss of generality, we may assume that $\lb(\gamma_0)_1=a_0 \in X$. Here, $\lb(\gamma_0)_1 := \tr^1\lb(\gamma_0)$, the first coordinate of~$\lb(\gamma_0) \in \mbfX_\locfin(\msE)$.

By definition of~$\T^\tym{\infty}$, the map~$t\mapsto \lb(\gamma_t)_1\in X$ is a continuous curve in~$X$.
Set
\begin{equation*}
\delta\eqdef \sup\set{t\in [0,1]: \lb(\gamma_s)_1=a_s, s\leq t}\fstop
\end{equation*}
By continuity of~$\lb(\emparg)_1$, we have $\lim_{t\to 0^+} \lb(\gamma_t)_1=\lb(\gamma_0)_1=a_0$, and thus~$\delta>0$.
We divide the following argument into two cases: $\delta=1$ or~$\delta \in (0,1)$. 

Assume~$\delta=1$. By the definition of~$\delta$,  and continuity of~$\lb(\emparg)_1$ and~$t\mapsto a_{t}$, we see that 
\begin{equation*}
a_{1}=\lim_{t \to 1^-} a_{t} = \lim_{t \to 1^-}\lb(\gamma_t)_1 = \lb(\gamma_{1})_1 \fstop
\end{equation*}
Using $\gamma_0=\gamma_{1}$ and $\lb(\gamma_0)_1=a_0$,  we get $a_0=\lb(\gamma_0)_1=\lb(\gamma_{1})_1=a_{1}$, a contradiction since $a_0 \neq a_{1}$ by the assumption that the loop is simple.

Assume now~$\delta \in (0,1)$. The definition of~$\delta$ yields that, for~$n>0$ with $\delta < \delta+1/n < 1$, there exists $s_n>0$ so that $\delta < s_n < \delta+1/n$ and $\lb(\gamma_{s_n})_1 = a_{s_n+1}$. By the continuity of the maps~$\lb(\emparg)_1$ and~$t\mapsto a_{t}$,  and noting $s_n \to \delta$, we obtain
\begin{equation*} 
a_{\delta+1}=\lim_{n \to \infty}a_{s_n + 1}=\lim_{n \to \infty}\lb(\gamma_{s_n})_1 = \lb(\gamma_{\delta})_1= \lim_{t \to \delta^-}\lb(\gamma_{t})_1 = \lim_{t \to \delta^-}a_{t} = a_{\delta}\comma
\end{equation*}
which is a contradiction since~$a_\delta \neq a_{\delta+1}$ by the assumption that the loop is simple.

\iref{i:p:Discontinuity-l:2} Let~$\seq{a_t}_{t\in\R}$ be a simple infinite curve with image not in~$\msE$, and note that~$\seq{a_{t+i}}_{t\in [0,1]}\in\msE$ for every~$i\in \Z$ by Remark~\ref{r:TLS}\iref{i:r:TLS:7}, compactness of~$[i,i+1]$, and continuity of~$t\mapsto a_{i+t}$.
Consider the curve of configurations~$\gamma_t\eqdef \sum_{i\in\Z}\delta_{a_{t+i}}$, and note that~$\gamma_t=\gamma_{t+1}$.
The rest of the proof follows exactly as in~\iref{i:p:Discontinuity-l:1}.
\end{proof}
\end{prop}

Labeling maps will play a key role throughout the present work, and their importance in our analysis cannot be overstated.
Indeed, they will allow us to discuss properties of the configuration space~$\dUpsilon^\sym{\infty}$ in relationship with corresponding properties on the infinite-product space~$\mbfX^\asym{\infty}_\locfin(\msE)$.
The strength of labeling maps as a technical tool lies in their multitude, exemplified in Figure~\ref{fig:LabelingMaps} below.
Since labeling maps are right inverses to the quotient map of the action of the (projective) infinite symmetric group~$\mfS_\infty$ on~$X^\tym{\infty}$, they are uncountably many.
On the one hand, this implies that any assertion which holds for every labeling map characterizes~$\dUpsilon^\sym{\infty}$ as well as~$X^\tym{\infty}$.
On the other hand, it will be possible to choose labeling maps with special properties, in particular in~\S\ref{s:Geometry}, where we will construct labeling maps that are additionally radial isometries around points in~$\dUpsilon^\sym{\infty}$.

\begin{figure}[htb!]
\begin{tikzcd}
&& X^\tym{\infty}
\\
&& X^\asym{\infty}_\locfin \arrow[draw=none]{u}[sloped,auto=false]{\subset} \arrow[dd, "\Lb", two heads]
\\
\mfl_1'\ttonde{\dUpsilon^\sym{\infty}} \arrow[urr, bend left=30, hook]\arrow[r, bend left=10, "\sigma"] &\lb_1\ttonde{\dUpsilon^\sym{\infty}} \arrow[ur, hook]\arrow[l, bend left=10, "\sigma^{-1}"]
&&
\lb_2\ttonde{\dUpsilon^\sym{\infty}} \arrow[ul, hook'] \arrow[r, bend left=10, "\sigma"] & \mfl_2'\ttonde{\dUpsilon^\sym{\infty}} \arrow[ull, bend right=30, hook']\arrow[l, bend left=10, "\sigma^{-1}"] 
\\
&& \dUpsilon^\sym{\infty} \arrow[ul, "\lb_1"{name=L3}] \arrow[ur, "\lb_2 "'{name=L1} ] 
%right arrows
\arrow[urr, bend right=50, "\lb_2'=\sigma\circ \lb_2 "{name=L2}] \arrow[from=L1, to=L2, bend left=50] \arrow[from=L2, to=L1, bend left=50]
%left arrows
\arrow[ull, bend left=50, "\lb_1'=\sigma\circ \lb_1 " '{name=L4}] \arrow[from=L3, to=L4, bend left=50] \arrow[from=L4, to=L3, bend left=50]
\end{tikzcd}
\caption{Different labeling maps~$\lb_1$,~$\lb_2$,~$\lb_1'$,~$\lb_2'$, each inverting~$\Lb$ on the right.
For each labeling map~$\lb$, and each permutation~$\sigma\in\mfS_\infty$, a different labeling map is defined by setting~$\lb'\eqdef \sigma\circ \lb$.
}
\label{fig:LabelingMaps}
\end{figure}
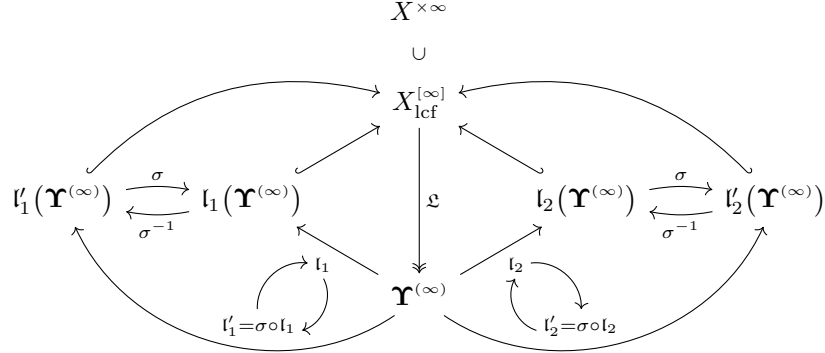

\paragraph{Cylinder sets}
Let us introduce here a class of \emph{cylinder sets} which shall be of use throughout the next sections. 
Heuristically, one can picture the configuration space~$\dUpsilon^\sym{\infty}$ as the quotient of~$\mbfX_\locfin^\asym{\infty}(\msE)$ by the action of the infinite symmetric group~$\mfS(\N_0)$ consisting of all bijections of~$\N_0$.
Cylinder sets are minimal pre-images of subsets of~$\dUpsilon$ of the form~$\Xi_{\geq n}(E)$ via the quotient projection induced by the said action. 

\begin{defs}[Cylinder sets]
Let~$\mcX$ be a topological local structure. A set~$\mbfA\subset X^\tym{\infty}$ is a \emph{cylinder set} if there exist~$n\in\N$ and sets
\begin{align*}
A_1,\dotsc, A_n\in \Bo{\T}\cap\msE\qquad \text{with} \qquad \mssm A_i>0 \comma \quad i\in [n]\comma
\end{align*}
so that
\begin{align*}
\mbfA=\tr_n^{-1}(A_1\times\cdots\times A_n)=A_1\times \cdots \times A_n\times X \times X\times\cdots
\end{align*}
We denote by~$\CylSet{\msE}$ the family of all cylinder sets.
Finally, for every cylinder set~$\mbfA$, we let
\begin{align*}
\tilde\mbfA\eqdef \mbfA\cap\mbfX^\asym{\infty}_\locfin(\msE) \fstop
\end{align*}
\end{defs}

Recall the definition~\eqref{eq:ConcentrationSet} of \emph{concentration sets}. The following fact is straightforward.

\begin{prop}[Concentration $\iff$ Cylinder]\label{p:FundamentalSets}
Let~$\mbfA=\tr_n^{-1}(A_1\times \cdots \times A_n)$ be a cylinder set, and set
\begin{equation*}
j_i\eqdef \min\set{j\leq n: A_i = A_j }\comma\qquad m\eqdef\max_{i\leq n} j_i \comma \qquad k_i\eqdef \#\set{j\leq n: A_i=A_j}\fstop
\end{equation*}
Then,
\begin{align}\label{eq:p:FundamentalSets:0}
\Lb(\tilde\mbfA)=\bigcap_{i=1}^m \Xi_{\geq k_i}^\sym{\infty} (A_{j_i}) \fstop
\end{align}

Viceversa, let~$\mbfk\eqdef\seq{k_i}_i$, and
\begin{subequations}
\begin{gather}
\label{eq:FundamentalSet}
\Xi^\sym{\infty}_{\geq \mbfk}(A_1,\dotsc, A_m)\eqdef \set{\gamma\in\dUpsilon^\sym{\infty}:\gamma E_j\geq k_j\comma j\leq m}\comma \qquad n\eqdef \sum_{j=1}^m j_i\comma
\\
\label{eq:FundamentalSet2}
\mbfA\eqdef \tr_n^{-1}\ttonde{\underbrace{A_1\times \cdots \times A_1}_{k_1} \times \cdots \times \underbrace{A_m\times \cdots \times A_m}_{k_m}} \fstop
\end{gather}
\end{subequations}
Then,
\begin{equation*}
\Lb^{-1}\ttonde{\Xi^\sym{\infty}_{\geq \mbfk}(A_1,\dotsc, A_m)}= \bigcup_{\sigma\in\mfS_0(\N_1)} \tilde\mbfA_\sigma \subset \mbfX^\asym{\infty}_\locfin(\msE) \comma
\end{equation*}
where~$\mfS_0(\N_1)$ denotes the group of bijections of~$\N_1$ with cofinitely many fixed points, and we set $\mbfA_\sigma\eqdef\set{\mbfx_\sigma:\mbfx\in \mbfA}$.
\end{prop}

Finally, let us show some further properties of Poisson measures.

\begin{lem}\label{l:SupportPoisson}
Let~$\mcX$ be a topological local structure. Then,~$\PP_\mssm$ has full $\T_\mrmv(\msE)$-support.
\begin{proof}
For every~$n\in\N_0$,~$\mbfk\in \N_0^n$, and~$\seq{A_i}_{i\leq n}\subset \msE\cap\T$ with~$\mssm A_i\neq \emp$, it is not difficult to show that the set~$\Xi_{>\mbfk}(A_1,\dotsc, A_n)$ as in~\eqref{eq:FundamentalSet} is $\T_\mrmv(\msE)$-open, and that the family of all such sets~$\Xi$ is a basis for the topology~$\T_\mrmv(\msE)$.
Thus, it suffices to show that~$\PP_\mssm \Xi>0$ for every~$\Xi$ as above (and non-empty).
This follows by the direct computation~\eqref{eq:AKR2.7} below and the fact that, since~$\mssm$ has full $\T$-support by Definition~\ref{d:TLS}\iref{i:d:TLS:3},~$\mssm A_i>0$.
\end{proof}
\end{lem}

\begin{rem}
In fact, it is not difficult to show that the assumption on~$\mssm$ having full $\T$-support is both necessary and sufficient for~$\PP_\mssm$ to have full $\T_\mrmv(\msE)$-support.
\end{rem}

\subsection{Well-posedness}\label{ss:AnalyticForms}
As already discussed in~\S\ref{s:Intro}, throughout this work we shall make extensive use of the following assumption about the finite-dimensional marginalizations of~$\QP$.

\begin{ass}[Absolute continuity of marginalizations]\label{d:ass:AC}
Let~$(\mcX,\cdc)$ be a \TLDS, and~$\QP$ be a probability measure on~$\ttonde{\dUpsilon,\A_\mrmv(\msE)}$.
Further set~$\QP^\sym{\geq n}\eqdef \QP\mrestr{\dUpsilon^\sym{\geq n}(\msE)}$.
We say that~$\QP$ is \emph{marginally absolutely continuous} if for \emph{every} labeling map~$\lb$, it holds that
\begin{equation}\label{ass:AC}\tag*{$(\mssA\mssC)_{\ref{d:ass:AC}}$}
\QP^\asym{n}\eqdef (\tr^n\circ \lb)_\pfwd \QP^\sym{\geq n}  \ll \mssm^\otym{n} \comma \qquad n\in\N \fstop
\end{equation}
\end{ass}

\begin{rem}\label{r:AC}
\begin{enumerate*}[$(a)$]
\item As discussed in~\S\ref{sss:ExamplesAC}, all quasi-Gibbs measures (Dfn.~\ref{d:QuasiGibbs}) satisfy this assumption.
\item\label{i:r:AC:0} If~$\QP \dUpsilon^\sym{\infty}(\msE)=1$ we have, in particular, that~$\QP^\sym{\geq n}\dUpsilon(\msE)=1$ for each~$n$.
Thus~\ref{ass:AC} reads
\end{enumerate*}
\begin{equation*}
\QP^\asym{n}\eqdef (\tr^n\circ \lb)_\pfwd \QP  \ll \mssm^\otym{n} \comma \qquad n\in\N \fstop
\end{equation*}
\begin{enumerate*}[$(a)$]\setcounter{enumi}{2}
\item\label{i:r:AC:2} If~$\QP$ satisfies Assumption~\ref{ass:AC} and~$\QP'$ is another probability measure on~$\ttonde{\dUpsilon,\A_\mrmv(\msE)}$ with $\QP'\ll\QP$, then~$\QP'$ satisfies Assumption~\ref{ass:AC} as well.

\item\label{i:r:AC:1.5} If~$\mssm X<\infty$, then the Poisson measure~$\PP_\mssm$ with (diffuse) intensity measure~$\mssm$ satisfies Assumption~\ref{ass:AC} in light of~\eqref{eq:PoissonLebesgue}.

\item\label{i:r:AC:1} If~$\mssm X=\infty$, then it is well-known that the Poisson measure~$\PP_\mssm$ with (diffuse) intensity measure~$\mssm$ satisfies~$\PP_\mssm \Upsilon^\sym{\infty}(\msE)=1$.
In this case,~$\PP_\mssm$ satisfies Assumption~\ref{ass:AC}, as we discuss now.
\end{enumerate*}
\end{rem}

In order to show that the Poisson measure~$\PP_\mssm$ satisfies Assumption~\ref{ass:AC} in the case~$\mssm X=\infty$, it suffices to show that
\begin{equation*}
\PP^\asym{n}_\mssm\mrestr{E^\tym{n}}\eqdef \ttonde{(\tr^n\circ\lb)_\pfwd \PP_\mssm}\mrestr{E^\tym{n}} \ll \mssm^\otym{n}\mrestr{E^\tym{n}}\comma \qquad n\in \N_1\comma
\end{equation*}
which will be a consequence of Remark~\ref{r:AC}\iref{i:r:AC:1.5}, together with the next Proposition.

\begin{prop}\label{p:PropertiesPP}
Let~$\mcX$ be a topological local structure with~$\mssm X=\infty$, and~$\lb\colon \dUpsilon(\msE)\rar \mbfX_\locfin(\msE)$ be a labeling map. For every~$n\geq 1$ and~$E\in\msE$ set
\begin{align*}
\PP^\sym{\geq n}_\mssm\eqdef& \PP_\mssm\mrestr{\Upsilon^\sym{\geq n}(\msE)} \qquad \text{and} \qquad \PP^\sym{\geq n}_{\mssm_E}\eqdef \PP_{\mssm_E}\mrestr{\Upsilon^\sym{\geq n}(E)}\comma
\\
\PP^\asym{n}_\mssm\eqdef& (\tr^n\circ \lb)_\pfwd \PP_\mssm^\sym{\geq n} \qquad \text{and} \qquad \PP^\asym{n}_{\mssm_E}\eqdef (\tr^n\circ \lb)_\pfwd \PP_{\mssm_E}^\sym{\geq n}\fstop
\end{align*}

Then, $\PP^\asym{n}_\mssm\mrestr{E^\tym{n}}=\PP^\asym{n}_{\mssm_E}$ for every~$n\geq 1$;

\begin{proof}
Throughout the proof let~$\set{E_i}_{i\in [n]}\subset \msE$, and set~$E\eqdef \cup_{i\in [n]} E_i\in\msE$ and
\begin{align*}
E^{\asym{n}}\eqdef E_1\times \cdots \times E_n\subset E^\tym{n} \fstop
\end{align*}
Note that sets of this form~$E^{\asym{n}}$ generate~$(\A^{\hotimes n})_{E^\tym{n}}$.

By Remark~\ref{r:TLS}\iref{i:r:TLS:7},~$\msE^\otym{n}$ contains a basis for~$\T^\tym{n}$.
Thus, since both~$\ttonde{\PP^\asym{n}_\mssm\mrestr{E^\tym{n}}}$ and~$\PP^\asym{n}_{\mssm_E}$ are finite measures, by e.g.~\cite[Lem.~II.7.1.2]{Bog07} it suffices to show that
\begin{align}\label{eq:p:PropertiesPP:0}
\ttonde{\PP^\asym{n}_\mssm\mrestr{E^\tym{n}}}E^{\asym{n}}=\PP^\asym{n}_{\mssm_E} E^{\asym{n}} \fstop
\end{align}
Indeed,
\begin{align*}
\PP_{\mssm_E}^\asym{n} E^{\asym{n}}=&\ \pr^{E}_\pfwd\PP_\mssm^\sym{\geq n} \tonde{\ttonde{\tr^n \circ \lb}^{-1}(E^{\asym{n}})}
\\
=&\ \pr^{E}_\pfwd\PP_\mssm^\sym{\geq n}\set{\gamma\in \Upsilon^\sym{\geq n}(E) : \lb(\gamma)_j\in E_j \text{~~for } j\in[n] }
\\
=&\ \PP_\mssm^\sym{\geq n}\set{\gamma\in \Upsilon^\sym{\geq n}(X) : \lb(\gamma)_j\in E_j \text{~~for } j\in[n] }=  \PP_\mssm^\sym{\geq n} (\tr^n\circ \lb)^{-1}(E^{\asym{n}})
\\
=&\ \PP_\mssm^\asym{n} E^{\asym{n}} \fstop \qedhere
\end{align*}
\end{proof}
\end{prop}

\subsubsection{Finite products}
Let~$n\geq 2$.
If~$(X,\A,\mssm)$ is a measure space in the sense of Dfn.~\ref{d:MS}, we denote by~$\A^{\hotym n}$ the product $\sigma$-algebra on~$X^{\tym{n}}$, by~$\mssm^{\otym{n}}$ the product measure on~$(X^{\otym n}, \A^{\hotym n})$.
If~$\msE$ is a localizing ring (Dfn.~\ref{d:LS}) on~$(X,\A,\mssm)$, we denote by~$\msE^{\otym n}$ the localing ring generated by the algebra of pluri-rectangles generated by the family of rectangles~$\msE^\tym{n}$.
Finally, if~$(\rep\cdc, \rep\Dz)$ is a pointwise defined square-field operator (Dfn.~\ref{d:SF}), we denote by~$\rep\Dz^{\otym n}$ the $n$-fold product algebra generated by~$\rep\Dz$, endowed with the natural product operator~$\rep\cdc^\tym{n}\colon (\rep \Dz^{\otym n})^{\tym 2}\rar \mcL^\infty(\mssm^{\otym{n}})$, defined as follows.
Let~$\rep f^\asym{n}\colon X^\tym{n}\rar \R$ be $\A^{\hotimes{n}}$-measurable.
For~$\mbfx^\asym{n}\in X^{\tym{n}}$ and~$p\in[n]$, define the $p$-section~$\rep f^\asym{n}_{\mbfx, {p}}\colon X\rar \R$ by
\begin{align}\label{eq:FuncSection}
\rep f^\asym{n}_{\mbfx, {p}}\colon y \longmapsto \rep f^\asym{n}(x_1,\dotsc, x_{p-1},y,x_{p+1},\dotsc, x_n)\comma \qquad \mbfx^\asym{n}\eqdef\seq{x_1,\dotsc, x_n}\in X^{\tym{n}}\fstop
\end{align}
Define the \emph{product square field operator}~$(\rep\cdc^{\tym{n}},\rep\Dz^{\otym{n}})$
\begin{align}\label{eq:CdCTensor}
\rep\cdc^{\tym{n}}(\rep f^\asym{n})(\mbfx^\asym{n})\eqdef \sum_{p=1}^n \rep\cdc(\rep f^\asym{n}_{\mbfx, {p}})(x_p) \comma \qquad \rep f^\asym{n}\in \rep\Dz^{\otym{n}}\comma \qquad \mbfx^\asym{n}\in X^{\tym{n}}\fstop
\end{align}

Some essentials about product spaces are collected in the following proposition. A proof is quite standard, and therefore it is omitted.

\begin{prop}[Product spaces]\label{p:Products} For every~$n\geq 2$, 
\begin{itemize}
\item the triple~$\ttonde{X^{\tym n}, \A^{\hotym n}, \mssm^{\otym n}}$ is a measure space in the sense of Dfn.~\ref{d:MS}; 
\item the quadruple~$\mcX^{\otym n}\eqdef (X^{\tym{n}}, \A^{\hotym{n}},\mssm^{\otym{n}}, \msE^{\otym n})$ is a local structure;
\item the pair~$(\rep\cdc^\tym{n}, \rep\Dz^{\otym n})$ is a pointwise defined square field operator;
\item the pair~$(\mcX^{\otym n}, \rep\cdc^\tym{n})$ is an~\LDS.
\end{itemize}

If~$(\mcX,\cdc)$ is additionally a \TLDS, then so is~$(\mcX^{\otym n}, \cdc^{\tym n})$.
\end{prop}

\begin{notat}\label{n:ProductCdC}
In the following, we write~$\cdc^{\otym n}\eqdef \cdc^{X^{\tym n}, \mssm^{\otym n}}$ to denote the closure of~$(\cdc^{\tym n},\Dz^{\otym n})$, as well as the extension of the former to its extended domain~$\domext{\cdc^{\otym n}}$.
\end{notat}

\subsubsection{Pre-domains on infinite products}\label{sss:PreDomains}
Let~$\QP$ be a probability measure on~$\ttonde{\dUpsilon,\A_\mrmv(\msE)}$ satisfying Assumption~\ref{ass:Mmu}.

Proposition~\ref{p:ExtDom} establishes that, in defining (extended) cylinder functions, we do not need any specification of $\mssm$-representatives for the inner functions~$\mbff$, nor of $\QP$-representatives for the resulting function~$u=F\circ \mbff^\trid$. However, in order to transfer objects from the configuration space to~$\mbfX$, we shall need to make sense of the pullback map
\begin{align*}
\Lb^*\colon u\longmapsto u\circ \Lb \fstop
\end{align*}
Since no $\sigma$-ideal in~$\boldSigma$ of negligible sets is assigned on~$\mbfX$, we need to interpret all functions~$\Lb^*u$ as pointwise defined everywhere on~$\mbfX$. The same holds for functions of the form~$\Lb^* \ttonde{\SF{\dUpsilon}{\QP}(u)}$, which motivates a more thorough study of the choice of representatives.

We shall start by defining a suitable core of differentiable functions.

\begin{notat}\label{n:LabelingUniversalAlgebra}
For a labeling map~$\lb$ set~$\QP^\asym{\infty}\eqdef \lb_\pfwd \QP$, and denote by~$\Bo{\T^\tym{\infty}}^{\lb_\pfwd\QP}$ the completion of~$\Bo{\T^\tym{\infty}}$ w.r.t.~$\QP^\asym{\infty}$.
Further define the \emph{labeling-universal} $\sigma$-algebra on~$\mbfX^\asym{\infty}_\locfin(\msE)$ by
\begin{align*}
\boldSigma^*(\msE)\eqdef \bigcap_{\lb \text{ labeling map}} \Bo{\T^\tym{\infty}}^{\lb_\pfwd\QP} \fstop
\end{align*}
\end{notat}
Note that~$\Lb\colon \mbfX^\asym{\infty}_\locfin(\msE)\to \dUpsilon$ is $\boldSigma^*(\msE)/\A_\mrmv(\msE)^\QP$-measurable.

\begin{notat}
For a bounded $\boldSigma^*(\msE)$-measurable function~$\rep U\colon X^\tym{\infty}\rar \R$ let
\begin{align}\label{eq:d:Di:0}
\rep U_{\mbfx, {p}}\colon z\longmapsto \rep U(x_1,\dotsc, x_{p-1},z, x_{p+1},\dotsc)\comma \qquad \mbfx\in X^\tym{\infty}\comma \quad p\in \N_1\comma
\end{align}
and set, for every~$N\in\overline\N_1$,
\begin{align}\label{eq:d:Di:1}
\begin{aligned}
\rep\cdc^p(\rep U)(\mbfx)\eqdef&\ \repSF{X}{\mssm}(\rep U_{\mbfx, p})(x_p) \comma
\\
\rep\cdc^\asym{N}(\rep U)(\mbfx)\eqdef&\ \sum_{p=1}^N \rep\cdc^p(\rep U)(\mbfx) \comma
\end{aligned}
\qquad \mbfx\in X^\tym{\infty}\comma
\end{align}
whenever this makes sense.
We denote by~$\rep\cdc^\asym{N}(\emparg,\emparg)$ the bilinear form induced by~$\rep\cdc^\asym{N}(\emparg)$ by polarization.
\end{notat}

For a labeling map~$\lb$ recall that~$\QP^\asym{\infty}\eqdef \lb_\pfwd \QP$ and set~$H_\infty\eqdef L^2(\QP^\asym{\infty})$. We further define the Sobolev \emph{semi}-norm
\begin{align*}
\ttnorm{\rep U}_{\llb}\eqdef \norm{\ttabs{\rep U}+ \rep\cdc^\asym{\infty}(\rep U)^{1/2}}_{H_\infty} \fstop
\end{align*}

\begin{defs}[pre-Sobolev class]\label{d:preSobolev}
We say that $\rep U\colon \mbfX^{\asym{\infty}}_\locfin(\msE) \rar \R$ is \emph{pre-Sobolev} if
\begin{enumerate}[$(a)$]
\item\label{i:d:preSobolev:1} $\rep U$ is bounded $\boldSigma^*(\msE)$-measur\-able;
\item\label{i:d:preSobolev:3} there exists a constant~$M>0$ so that~$\ttnorm{\rep U}_{\llb}\leq M$ for \emph{every} labeling map~$\lb$.
\end{enumerate}
We denote by~$\preW(\msE)$ the space of all pre-Sobolev functions on~$\mbfX^{\asym{\infty}}_\locfin(\msE)$, and by~$\spclass[\llb]{\preW(\msE)}$ the corresponding space of~$\QP^\asym{\infty}$-classes for some fixed labeling map~$\lb$.
\end{defs}

The requirement of $\boldSigma^*(\msE)$ measurability in the definition of pre-Sobolev functions arises from the necessity to guarantee that any pre-Sobolev function~$\rep U\colon \mbfX^\asym{\infty}_\locfin(\msE)\to\R$ be measurable w.r.t.\ the completion of~$\QP^\asym{\infty}\eqdef \lb_\pfwd\QP$ \emph{for every labeling map~$\lb$}.

\medskip

If not otherwise stated, everywhere in the following $\lb$ is a \emph{fixed} labeling map.
For every~$N\in \overline\N_1$ set
\begin{align}\label{eq:FormInfty}
\EE{\asym{N}}{\llb_\pfwd\QP}(\rep U,\rep V)\eqdef \int_{X^\tym{\infty}} \rep\cdc^\asym{N}(\rep U,\rep V) \diff\QP^\asym{\infty} \comma\qquad \rep U,\rep V\in \preW(\msE) \fstop
\end{align}
We aim to show that the quadratic form above descends to a well-defined quadratic form on the space of classes~$\spclass[\llb]{\preW(\msE)}$.
To this end, it suffices to show~\eqref{eq:Ss} with~$\rep\cdc^\asym{N}$ in place of~$\rep\cdc$, which is the content of the next Lemma~\ref{l:WellPosedness}.

\paragraph{Disintegration of measures}
We recall here some essentials about disintegrations of measures, which we will use in the proof of Lemma~\ref{l:WellPosedness} below.
\begin{defs}[Disintegration (cf.~{\cite[452E]{Fre00}})]\label{d:Disintegrations}
Let~$(X,\A, \mssm)$ and~$(Y,\Tau,\mssn)$ be measure space. A \emph{disintegration} of~$\mssm$ over~$\mssn$ is a family~$\seq{\mssm_y}_{y\in Y}$ of non-zero sub-probability measures on~$(X,\A)$ such that~$\int \mssm_y A \diff\mssn(y)$ is defined in~$[0,\infty]$ and equal to~$\mssm A$ for every~$A\in\A$. If~$f\colon X\rar Y$ is measurable and satisfying~$\mssn=f_\pfwd \mssm$, a disintegration~$\seq{\mssm_y}_{y\in Y}$ of~$\mssm$ over~$\mssn$ is \emph{consistent} with~$f$ if
\begin{align*}
\mssm \ttonde{A\cap f^{-1}(B)}= \int_B \mssm_y A \diff \mssn(y) \comma \qquad A\in\A\comma B\in\Tau \semicolon
\end{align*}
\emph{strongly consistent with~$f$} if additionally, for $\mssn$-a.e.~$y\in Y$, the set~$f^{-1}(y)$ is $\mssm_y$-conegligible.

A disintegration~$\seq{\mssm_y}_{y\in Y}$ of~$\mssm$ over~$\mssn$ is a \emph{system of regular conditional probabilities} if~$\mssm$,~$\mssn$ and all the~$\mssm_y$'s are probability measures.
\end{defs}

\begin{defs}[Countable separation, cf.~{\cite[343D]{Fre00}}]\label{d:CountableSep}
A measurable space~$(Y,\Tau)$ is \emph{countably separated} if there exists a countable family of subsets~$\mcA\subset\Tau$ separating points in~$Y$.
\end{defs}
In our context, the existence of disintegrations may be stated as follows.

\begin{thm}[{Cf.~\cite[452O, 452G(c)]{Fre00}}]\label{t:Disintegration}
Let~$(X,\T,\Bo{\T}^\mssm,\hat\mssm)$ be a Radon probability space\footnote{Note that the definition~\cite[411H(b)]{Fre00} of a Radon measure space \emph{requires} it to be complete in the sense of measures.}, and $(Y,\Tau)$ be countably separated. Further let~$f\colon X\rar Y$ be $\Bo{\T}^\mssm/\Tau$-measurable and set~$\mssn\eqdef f_\pfwd \mssm$. Then, there exists a disintegration~$\seq{\mssm_y}_{y\in Y}$ of~$\hat\mssm$ over~$\mssn$ strongly consistent with~$f$.
\end{thm}

\subsubsection{Well-posedness on infinite products}
We the following well-posedness result.

\begin{lem}[Well-posedness on products]\label{l:WellPosedness}
Let~$(\mcX,\cdc)$ be a \TLDS, and~$\QP$ be a probability measure on~$\ttonde{\dUpsilon,\A_\mrmv(\msE)}$ satisfying Assumption~\ref{ass:AC}. 
Further let~$\rep U,\rep V\in \preW(\msE)$ with~$\rep U\equiv 0$ $\QP^\asym{\infty}$-a.e.. Then,
\begin{align}\label{eq:l:WellPosedness:00}
\rep\cdc^\asym{N}(\rep U,\rep V)\equiv 0 \as{\QP^\asym{\infty}} \fstop
\end{align}
\end{lem}

Before dwelling into the details, the idea of the proof is as follows.
It suffices to show that~\eqref{eq:l:WellPosedness:00} holds when~$\rep\cdc^\asym{N}$ is replaced by any of its (at most countable) summands~$\rep\cdc^p$ as in~\eqref{eq:d:Di:1}.
This is achieved via a disintegration argument, projecting~$X^\tym{\infty}$ to one of its coordinates (without loss of generality, say the first one), and considering the section~$\rep\cdc^1(\rep U,\rep V)_{\mbfx,1}\colon X\to\R$.
By assumption on~$\rep U$ and locality of~$\SF{X}{\mssm}$, this section vanishes $\mssm$-a.e..
As a consequence of Assumption~\ref{ass:AC}, it vanishes~$\tr^1_\pfwd\QP^\asym{\infty}$-a.e.\ as well, and the assertion is concluded by reintegrating over all sections in a consistent way.
\begin{proof}[Proof of Lemma~\ref{l:WellPosedness}]
Set~$\QP^\sym{N}\eqdef \QP\mrestr{\dUpsilon^\sym{N}}$.
Since~$\QP=\sum_{N\in\N_0\cup\set{+\infty}} \QP^\sym{N}$, it suffices to show the statement with~$\QP^\sym{N}$ in place of~$\QP$ for each~$N\in N_0\cup\set{+\infty}$.
For simplicity, we show the statement in the case when~$\QP\dUpsilon^\sym{\infty}=1$, in which case we may restrict each labeling map to~$\lb\colon \dUpsilon^\sym{\infty}\to X^\tym{\infty}$.
A proof for the analogous statement with~$\QP^\sym{n}$ in place of~$\QP^\sym{\infty}$ is similar, and therefore it is omitted.

\paragraph{Restriction to $1$-sections} Suppose we have already shown that for every~$p\in \N_1$ there exists a $\widehat{\QP^\asym{\infty}}$-measurable set~$\mbfZ_p$ of full $\widehat{\QP^\asym{\infty}}$-measure so that~$\rep\cdc^p(\rep U, \rep V)\equiv \zero$ on~$\mbfZ_p$.
Then, the conclusion follows since the set~$\bigcap_{p\geq 1} \mbfZ_p$ has full $\QP^\asym{\infty}$-measure as well.
Thus, it suffices to show the statement with~$\rep\cdc^p(\rep U, \rep V)$ in place of~$\rep\cdc^\asym{N}(\rep U,\rep V)$ and we may and will in fact choose $p=1$, the proof for $p>1$ being analogous, although notationally more involved.

\paragraph{Sets of interest}
Now, denote by~$\Pi\colon \mbfX\rar \mbfX$ the projection~$\Pi\colon \mbfx\eqdef\seq{x_p}_{p=1}^\infty\mapsto \mbfx'\eqdef \seq{x_p}_{p=2}^\infty$, and note that~$\Pi$ globally fixes~$\mbfX_\locfin(\msE)$.
In the following, we also write~$\mbfx=x_1\oplus\mbfx'$.
Further set
\begin{align*}
\mbfA\eqdef& \tset{\mbfx\in \mbfX^\asym{\infty}_\locfin(\msE): \rep U(\mbfx)=0}\comma & A_{\mbfx'}\eqdef& \tset{y\in X: \rep U_{\mbfx,{1}}(y)=0} \comma
\\
\mbfZ\eqdef& \tset{\mbfx\in  \mbfX^\asym{\infty}_\locfin(\msE): \rep\cdc^1(\rep U, \rep V)(\mbfx)=0} \comma & Z_{\mbfx'}\eqdef& \tset{y\in X : \ttonde{\rep\cdc^1(\rep U, \rep V)}_{\mbfx,1}(y)=0} \fstop
\end{align*}
We write~$A_{\mbfx'}$ (as opposed to:~$A_\mbfx$) since, for fixed~$\mbfx=\seq{x_p}_{p=1}^\infty$, the set~$A_{\mbfx'}$ does not depend on~$x_1$. The same holds for~$Z_{\mbfx'}$.
Note that the set~$\mbfA$ above is $\boldSigma^*(\msE)$-measurable, by $\boldSigma^*(\msE)$-measurability of the corresponding defining function.
%Analogously,~$A_{\mbfx'}$ is $\Bo{\T}^*$-measurable for every~$\mbfx'\in\mbfX$.
Since~$\mbfA\in\boldSigma^*(\msE)$, there exists a $\boldSigma^*(\msE)$-measurable $\QP^\asym{\infty}$-negligible set~$\mbfN$ and a $\Bo{\T^\tym{\infty}}$-measurable set~$\mbfA_0\subset \mbfA$ so that~$\mbfA=\mbfA_0\cup \mbfN$.

Further set~$\mbfA'\eqdef \Pi(\mbfA_0)$. Since~$\Pi$ is $\T^\tym{\infty}/\T^\tym{\infty}$-continuous and~$\mbfA_0\in\Bo{\T^\tym{\infty}}$, then~$\mbfA'$ is a Suslin set, thus $\Pi_\pfwd\QP^\asym{\infty}$-measurable.
Analogously,~$\Pi^{-1}(\mbfA')$ is Suslin, thus $\QP^\asym{\infty}$-measurable.

\paragraph{Proof for the first coordinate} Assume now~$\widehat{\QP^\asym{\infty}}\mbfA=1$. We want to show that~$\widehat{\QP^\asym{\infty}}\mbfZ=1$.
Since~$\widehat{\QP^\asym{\infty}}\mbfA=1$, we have that~$\QP^\asym{\infty}\mbfA_0=1$ as well.
Thus, up to replacing~$\mbfA$ with~$\mbfA_0$, we may and shall assume with no loss of generality that~$\mbfA$ is an element of~$\Bo{\T^\tym{\infty}}\subsetneq \boldSigma^*(\msE)$.

By Remark~\ref{r:TLS}\iref{i:r:TLS:1},~$(X,\Bo{\T})$ is countably separated by a countable basis of~$\T$, and so the same holds for~$(X^\tym{\infty},\Bo{\T^\tym{\infty}})$ and~$(\mbfX,\Bo{\T^\tym{\infty}})$.
By Theorem~\ref{t:Disintegration} there exists a disintegration~$\seq{\nu_{\mbfx'}}_{\mbfx'\in\mbfX}$ of the completed measure~$\widehat{\QP^\asym{\infty}}$ over~$\Pi_\pfwd \QP^\asym{\infty}$, strongly consistent with~$\Pi$.
In particular,
\begin{align}\label{eq:Disintegration}
\QP^\asym{\infty}(\mbfB\cap \mbfA)=\widehat{\QP^\asym{\infty}}\ttonde{\mbfB\cap \Pi^{-1}(\mbfA')}=\int_{\mbfA'} \nu_{\mbfx'} \mbfB \, \diff \Pi_\pfwd \QP^\asym{\infty}(\mbfx') \comma \qquad \mbfB\in\Bo{\T^\tym{\infty}}\fstop
\end{align}

Since~$\QP^\asym{\infty}\mbfA=1$, it follows from~\eqref{eq:Disintegration} with~$\mbfB=\mbfA$ and the strong consistency of the disintegration with~$\Pi$ that
\begin{align}\label{eq:l:WellPosedness:1}
\nu_{\mbfx'}\ttonde{A_{\mbfx'}\times \set{\mbfx'}} = \nu_{\mbfx'} \ttonde{\mbfA\cap \Pi^{-1}(\mbfx')}=1 \quad \forallae{\Pi_\pfwd\QP^\asym{\infty}} \mbfx'\in\mbfA' \comma
\end{align}
where the first equality holds by identity of the sets.

By the local property of~$\SF{X}{\mssm}$, we have~$\repSF{X}{\mssm}(f, h)=\repSF{X}{\mssm}(g,h)$ $\mssm$-a.e.\ on the set~$\ttset{\rep f= \rep g}$ for every~$f$,~$g$, and~$h\in\domext{\EE{X}{\mssm}}$.
Since~$\rep U$, $\rep V\in\preW(\msE)$, we may choose~$\rep f=\rep U_{\mbfx,{1}}$,~$\rep g\equiv \zero$ and~$\rep h=\rep V_{\mbfx,{1}}$, to conclude
\begin{align*}
\ttonde{\rep\cdc^1(\rep U, \rep V)}_{\mbfx,1}=\repSF{X}{\mssm}(\zero, \rep V_{\mbfx,{1}})\equiv 0 \quad \as{\mssm} \text{~on~} A_{\mbfx'} \qquad \mbfx'\in\mbfA'
\end{align*}
by definition of~$A_{\mbfx'}$. Since~$\QP^\asym{1}\ll \mssm$ by Assumption~\ref{ass:AC}, we have as well that
\begin{align}\label{eq:l:WellPosedness:2}
\ttonde{\rep\cdc^1(\rep U, \rep V)}_{\mbfx,1}=\repSF{X}{\mssm}(\zero, \rep V_{\mbfx,{1}})\equiv 0 \quad \as{\QP^\asym{1}} \text{~on~} A_{\mbfx'} \qquad \mbfx'\in\mbfA' \fstop
\end{align}
As a consequence we have that~$\widehat{\QP^\asym{1}} \ttonde{Z_{\mbfx'}\triangle A_{\mbfx'}}=0$ by definition of~$Z_{\mbfx'}$.
Thus, letting~$Z_{\mbfx'}^*\eqdef A_{\mbfx'}\cap Z_{\mbfx'}$, it holds that~$\widehat{\QP^\asym{1}} Z_{\mbfx'}^*= \widehat{\QP^\asym{1}} A_{\mbfx'}$ for all~$\mbfx'\in\mbfA'$.
Therefore, by~\eqref{eq:l:WellPosedness:1} and~\eqref{eq:Disintegration},
\begin{align}
\nonumber
\widehat{\QP^\asym{1}} Z_{\mbfx'}^*= \widehat{\QP^\asym{1}} A_{\mbfx'} =& \int_{\mbfA'} \widehat{\tr^1_\pfwd \nu_{\mbfx'}} \, A_{\mbfx'}\diff \widehat{\Pi_\pfwd \QP^\asym{\infty}} (\mbfx') =  \int_{\mbfA'} \nu_{\mbfx'} \, \ttonde{A_{\mbfx'}\times\set{\mbfx'}}\diff \widehat{\Pi_\pfwd \QP^\asym{\infty}} (\mbfx')
\\
\label{eq:l:WellPosedness:3}
=&\ \widehat{\Pi_\pfwd \QP^\asym{\infty}} \mbfA'=\QP^\asym{\infty} \mbfA=1
\end{align}
for all~$\mbfx'\in\mbfA'$. Again by~\eqref{eq:Disintegration} and strong consistency of the disintegration
\begin{equation}\label{eq:l:WellPosedness:4}
\begin{aligned}
\widehat{\QP^\asym{1}} Z_{\mbfx'}^*=&\ \widehat{\QP^\asym{\infty}} (\tr^1)^{-1}(Z_{\mbfx'}^*)=\int_{\mbfA'} \tr^1_\pfwd \nu_{\mbfx'} \, Z_{\mbfx'}^* \diff \widehat{\Pi_\pfwd \QP^\asym{\infty}}(\mbfx')
\\
=&\ \int_{\mbfA'} \nu_{\mbfx'}\ttonde{Z_{\mbfx'}^*\times \set{\mbfx'}} \diff \widehat{\Pi_\pfwd \QP^\asym{\infty}}(\mbfx') \fstop
\end{aligned}
\end{equation}

Combining~\eqref{eq:l:WellPosedness:3} and~\eqref{eq:l:WellPosedness:4} yields
\begin{align}\label{eq:l:WellPosedness:5}
1=\nu_{\mbfx'} \ttonde{Z_{\mbfx'}^* \times \set{\mbfx'}}\leq \nu_{\mbfx'}\ttonde{Z_{\mbfx'}\times\set{\mbfx'}} \leq 1\quad \forallae{\Pi_\pfwd\QP^\asym{\infty}}\,\mbfx'\in\mbfA' \fstop
\end{align}

Finally, by~\eqref{eq:Disintegration}, strong consistency of the disintegration, and~\eqref{eq:l:WellPosedness:5}
\begin{align*}
\widehat{\QP^\asym{\infty}} \mbfZ=&\int \nu_{\mbfx'} \mbfZ \diff \widehat{\Pi_\pfwd \QP^\asym{\infty}}(\mbfx')\geq \int_{\mbfA'} \nu_{\mbfx'} \ttonde{\mbfZ\cap \Pi^{-1}(\mbfx')} \diff \widehat{\Pi_\pfwd \QP^\asym{\infty}}(\mbfx')
\\
=& \int_{\mbfA'} \nu_{\mbfx'} \ttonde{Z_{\mbfx'}\times\set{\mbfx'}} \diff \widehat{\Pi_\pfwd \QP^\asym{\infty}}(\mbfx')
\\
=&\ \widehat{\Pi_\pfwd \QP^\asym{\infty}} \mbfA'= \QP^\asym{\infty} \mbfA=1 \fstop \qedhere
\end{align*}
\end{proof}

\subsubsection{Well-posedness on configuration spaces} 
We are now ready to show that, under Assumption~\ref{ass:AC}, the pre-Dirichlet form~\ref{eq:Temptation} is well-defined, that is, that~\eqref{eq:WPCdC} holds.

\begin{prop}[Well-posedness on~$\dUpsilon$]\label{p:NewWP}
Let~$(\mcX,\cdc)$ be a \TLDS, and~$\QP$ be a probability measure on~$\ttonde{\dUpsilon,\A_\mrmv(\msE)}$ satisfying Assumption~\ref{ass:AC}.
Then,~\eqref{eq:WPCdC} holds for~$\QP$.

\begin{proof}
Let~$\rep v\in\Cyl{\rep\Dz}$ be arbitrary, and note that~$\Lb^*\rep v\in \preW(\msE)$ by Lemma~\ref{l:InfDiff}.
By~\eqref{eq:l:CdCCylinderTruncation:1c}, for every strong lifting~$\ell\colon L^\infty(\mssm)\rar\mcL^\infty(\mssm)$ and every labeling map~$\lb$,
\begin{align}\label{eq:p:NewWP:1}
\Lb^*\tonde{\class[\QP]{\cdc^\dUpsilon(0,\rep v)}}=\class[\QP^\asym{\infty}]{\rep\cdc^\asym{\infty}_\ell(0,\Lb^*\rep v)}\as{\QP^\asym{\infty}}\fstop
\end{align}
By Lemma~\ref{l:WellPosedness}, the right-hand side of~\eqref{eq:p:NewWP:1} vanishes $\QP^\asym{\infty}$-a.e., hence
\begin{align*}
\Lb^*\tonde{\class[\QP]{\cdc^\dUpsilon(0,\rep v)}}=0 \as{\QP^\asym{\infty}}\comma
\end{align*}
whence~$\rep\cdc^\dUpsilon(0,\rep v)=0$ $\QP$-a.e..
\end{proof}
\end{prop}

\subsection{Closability and conditional measures}\label{ss:FurtherClosability}
In this section we prove the closability of the form~\eqref{eq:Temptation} under a new set of assumptions, different from Assumption~\ref{ass:Closability}.
We collect examples of measures satisfying this new assumptions in~\S\ref{sss:ExamplesAC} below.

\subsubsection{Projected conditional measures}
Let~$(\mcX,\cdc)$ be a \TLDS, and recall that~$\cdc$ is assumed to be defined on~$\Dz\subset \Cz(\msE)$.
In particular, we shall always assume that~$f\in\Dz$ is identified with its continuous representative in~$\rep\Dz=\ell(\Dz)$, where $\ell\colon L^\infty(\mssm)\to\mcL^\infty(\mssm)$ is any strong lifting.

\begin{defs}[Projected conditional measures]\label{d:ConditionalQP}
For a probability measure~$\QP$ on~$\ttonde{\dUpsilon,\A_{\mrmv}(\msE)}$, for each fixed~$\eta\in\dUpsilon$, and for each fixed~$E\in\msE$, we let~$\QP^{\eta_{E^\complement}}$ be the regular conditional probability strongly consistent with~$\pr^{E^\complement}$ (Dfn.~\ref{d:Disintegrations}), satisfying
\begin{equation}\label{eq:ConditionalQP}
\QP\tonde{\Lambda\cap \pr_{E^\complement}^{-1}(\Xi)}=\int_\Xi \QP^{\eta_{E^\complement}} \Lambda \, \diff\QP(\eta) \comma
\end{equation}
where, with slight abuse of notation, we regard~$\dUpsilon(E)$ as a subset of~$\dUpsilon$, and thus~$\pr^{E^\complement}$ as a map~$\pr^{E^\complement}\colon\dUpsilon\to\dUpsilon$.
In probabilistic notation,
\begin{equation*}
\QP^{\eta_{E^\complement}}=\QP\ttonde{\emparg \big |\, \pr^{E^\complement}(\emparg)=\eta_{E^\complement}} \fstop
\end{equation*}
The \emph{projected conditional probabilities of~$\QP$} are the system of measures
\begin{equation}\label{eq:ProjectedConditionalQP}
\QP^\eta_E\eqdef \pr^E_\pfwd \QP^{\eta_{E^\complement}}\comma \qquad \eta\in\dUpsilon\comma \qquad E\in\msE \fstop
\end{equation}
\end{defs}

\begin{defs}[Conditional absolute continuity]\label{d:ConditionalAC}
A probability measure~$\QP$ on $\ttonde{\dUpsilon,\A_{\mrmv}(\msE)}$ is \emph{conditionally absolutely continuous} (\emph{w.r.t.~$\PP_\mssm$}) if there exists a localizing sequence~$\seq{E_h}_h$ such that the projected conditional probabilities satisfy
\begin{equation}\tag*{$(\mathsf{CAC})_{\ref{d:ConditionalAC}}$}
\label{ass:CAC}
\forallae{\QP} \eta\in\dUpsilon \qquad \QP^\eta_{E_h} \ll \PP_{\mssm_{E_h}} \comma \qquad h\in\N \fstop
\end{equation}
It is \emph{conditionally equivalent} (\emph{to~$\PP_\mssm$}) if its projected conditional probabilities satisfy
\begin{equation}\tag*{$(\mathsf{CE})_{\ref{d:ConditionalAC}}$}
\label{ass:CE}
\forallae{\QP} \eta\in\dUpsilon \qquad \QP^\eta_{E_h} \sim \PP_{\mssm_{E_h}} \comma \qquad h\in\N \fstop
\end{equation}
\end{defs}

\begin{rem}\label{r:ConditionalAC}
\begin{enumerate*}[$(a)$]
\item\label{i:r:ConditionalAC:1} We note that, for~$\QP$-a.e.~$\eta$, the probability measure~$\QP^\eta_E$ satisfies Assumption~\ref{d:ass:AC} by combining Remark~\ref{r:AC}\iref{i:r:AC:1.5} and~\iref{i:r:AC:2}.
\item\label{i:r:ConditionalAC:2} Assumption~\ref{ass:CAC} is a rephrasing of the more `symmetric' condition~$\pr^{E_h}_\pfwd \QP^{\eta_{E_h^\complement}}\ll \pr^{E_h}_\pfwd (\PP_\mssm)^{\eta_{E_h^\complement}}$.
Indeed, we have that~$\pr^{E_h}_\pfwd (\PP_\mssm)^{\eta_{E_h^\complement}}=\pr^{E_h}_\pfwd \PP_\mssm$, e.g.~\eqref{eq:PropertiesPP:4}, and~$\pr^{E_h}_\pfwd \PP_\mssm=\PP_{\mssm_{E_h}}$ by~\eqref{eq:PoissonRestriction}.
The same is true for Assumption~\ref{ass:CE}.
\end{enumerate*}
\end{rem}

In fact, the same holds for the original measure~$\QP$, as shown by the next result.

\begin{prop}[{\protect{\ref{ass:CAC}} $\Longrightarrow$ \protect{\ref{ass:AC}}}]\label{p:CACtoAC}
Let~$\mcX$ be a topological local structure, and~$\QP$ be a probability measure on $\ttonde{\dUpsilon,\A_\mrmv(\msE)}$ satisfying Assumption~\ref{ass:CAC}.
Then,~$\QP$ satisfies as well Assumption~\ref{ass:AC}.
\begin{proof}
Let~$\seq{E_h}_h\subset \msE$ be a localizing sequence witnessing~\ref{ass:CAC}.
Further let~$A^\asym{n}\in\A^\otym{n}$ be $\mssm^\otym{n}$-negligible, and set
\begin{align}\label{eq:p:CACtoAC:1}
\mbfA\eqdef \tr_n^{-1}(A^\asym{n})\comma \qquad \mbfA_h\eqdef \tr_n^{-1}(A^\asym{n}\cap E_h^\tym{n})\fstop
\end{align}
For each~$k\in\N_0$ set
\begin{equation*}
\tr^k(A^\asym{n})\times X^\tym{k-n}=\begin{cases} \tr^k A^\asym{n} & \text{if~} k\leq n\\ A^\asym{n}\times X^\tym{k-n} & \text{if } k>n\end{cases} \fstop
\end{equation*}
Finally, recall the notation in~\eqref{eq:ProjSymmetricG} and, for each~$k\in\N_0$ and each~$B^\asym{k}\subset X^\tym{k}$, denote its symmetrization by
\begin{equation*}
(B^\asym{k})^\sym{k}\eqdef \bigcup_{\sigma\in\mfS_k} (B^\asym{k})_\sigma=(\pr^\sym{k})^{-1}\ttonde{\pr^\sym{k}(B^\asym{k})}\fstop
\end{equation*}

Letting~$\mbfE_h\eqdef \sqcup_{k=0}^\infty E_h^\tym{k}\sqcup E_h^\tym{\infty}$, it follows from~\eqref{eq:PoissonLebesgue} that for each~$h\in\N_1$,
\begin{equation}\label{eq:p:QuasiGibbsAC:1}
(\pr_{E_h})_{\pfwd} \PP_\mssm \,\,\Lb(\widetilde{\mbfA_h\cap \mbfE_h}) = e^{-\mssm E_h} \sum_{k=0}^\infty\frac{1}{k!} \mssm_{E_h}^\sym{k}\pr^\sym{k} \tonde{\ttonde{\tr^k(A^\asym{n})\times X^\tym{k-n}}^\sym{k}}=0\fstop
\end{equation}

Note now that~$\Lb(\widetilde{\mbfA_h})\subset \pr_{E_h}^{-1}\ttonde{\Lb(\widetilde{\mbfA_h\cap\mbfE_h})}$, hence
\begin{equation}\label{eq:p:QuasiGibbsAC:2}
\QP^{\eta_{E_h^\complement}} \Lb(\widetilde{\mbfA_h})\leq (\pr_{E_h})_\pfwd\QP^{\eta_{E_h^\complement}} \Lb(\widetilde{\mbfA_h\cap\mbfE_h}) \fstop
\end{equation}
By definition of~$\QP^\eta_{E_h^\complement}$, and combining~\eqref{eq:ConditionalQP}--\eqref{eq:p:QuasiGibbsAC:2},
\begin{align*}
\QP\, \Lb(\widetilde{\mbfA_h})=&\ \int_\dUpsilon \QP^{\eta_{E_h^\complement}} \Lb(\widetilde{\mbfA_h}) \diff\QP(\eta)= \int_\dUpsilon \QP^\eta_{E_h} \Lb(\widetilde{\mbfA_h\cap\mbfE_h}) \diff\QP(\eta) = 0\fstop
\end{align*}

By monotonicity of~$\mbfA_h\uparrow \mbfA$, and since~$\QP$ is a probability measure,
\begin{equation}
\QP \, \Lb(\widetilde{\mbfA}) = \lim_{h\rar\infty} \QP\, \Lb(\widetilde{\mbfA_h})=0\fstop
\end{equation}

Thus, finally, for any labeling map~$\lb$,
\begin{equation*}
(\tr^n\circ\lb)_\pfwd\QP \, A^\asym{n}=\QP\ttonde{(\tr^n\circ \lb)^{-1} (A^\asym{n})}\leq \QP\, \Lb(\widetilde\mbfA)=0 \comma
\end{equation*}
where the inequality follows by~\eqref{eq:p:CACtoAC:1} and since~$\lb^{-1}$ coincides with~$\Lb$ on the image~$\lb\ttonde{\dUpsilon}$.
\end{proof}
\end{prop}

\subsubsection{Closability}
Now, let~$(\mcX,\cdc)$ be a \TLDS.
For simplicity, we shall consider a strong lifting~$\ell\colon L^\infty(\mssm)\rar \mcL^\infty(\mssm)$ fixed throughout this section, and we shall write~$\rep\cdc$ in place of $\rep\cdc_\ell$.

For every $\A_\mrmv(\msE)$-measurable~$\rep u\colon \dUpsilon\to [-\infty,\infty]$ define now
\begin{equation}\label{eq:ConditionalFunction}
\rep u_{E,\eta}(\gamma)\eqdef \rep u(\gamma+\eta_{E^\complement})\comma \qquad \gamma\in \dUpsilon(E) \fstop
\end{equation}
For a probability measure~$\QP$ on~$\ttonde{\dUpsilon,\A_\mrmv(\msE)}$, and for the corresponding system of projected conditional probabilities~\eqref{eq:ProjectedConditionalQP}, we have the following standard result.
\begin{prop}\label{p:ConditionalIntegration}
Let~$(\mcX,\cdc)$ be a \TLDS,~$\QP$ be a probability measure on~$\ttonde{\dUpsilon,\A_\mrmv(\msE)}$, and~$\rep u\in \mcL^1(\QP)$. Then,
\begin{align*}
\int_{\dUpsilon}\rep u \diff\QP = \int_{\dUpsilon} \quadre{\int_{\dUpsilon(E)} \rep u_{E,\eta} \diff \QP^\eta_E }\diff\QP(\eta) \fstop
\end{align*}
\begin{proof}
By definition of conditional probability
\begin{align*}
\int_{\dUpsilon} \rep u\diff\QP = \int_{\dUpsilon} \quadre{\int_{\dUpsilon} \rep u \diff\QP^{\eta_{E^\complement}} }\diff\QP(\eta) \fstop
\end{align*}
By regularity of the conditional system~$\seq{\QP^{\eta_{E^\complement}}}_{\eta\in\dUpsilon}$, the measure~$\QP^{\eta_{E^\complement}}$ is concentrated on the set
\begin{equation}\label{eq:RoeSch99Set}
\Lambda_{\eta, E^\complement} \eqdef \set{\gamma \in \dUpsilon : \gamma_{E^\complement}=\eta_{E^\complement} } 
\end{equation}
and we have that
\begin{equation}\label{eq:p:ConditionalIntegration:1}
\rep u\equiv \rep u_{E,\eta}\circ \pr_E \quad \text{everywhere on } \Lambda_{\eta, E^\complement}\fstop
\end{equation}
As a consequence,
\begin{align*}
\int_{\dUpsilon} \rep u\diff\QP = \int_{\dUpsilon} \quadre{\int_{\dUpsilon} \rep u_{E,\eta} \circ \pr_E \diff\QP^{\eta_{E^\complement}} }\diff\QP(\eta) \comma
\end{align*}
and the conclusion follows by definition of the projected conditional system.
\end{proof}
\end{prop}

In light of~\eqref{eq:MaRoeckner}, we define a restricted square field operator on~$\dUpsilon$ associated to~$E$ as
\begin{equation}\label{eq:RestrictedCdCUpsilon}
\rep\cdc^\dUpsilon_E(\rep u)(\gamma) \eqdef \sum_{x\in \gamma_E} \gamma_x^{-1}\cdot \rep\cdc\tonde{\rep u\ttonde{\car_{X\setminus\set{x}}\cdot\gamma + \gamma_x\delta_x}-\rep u\ttonde{\car_{X\setminus\set{x}}\cdot\gamma}} \comma \quad \rep u\in\Cyl{\Dz} \comma
\end{equation}
and the associated quadratic functionals
\begin{subequations}\label{eq:VariousForms}
\begin{align}\label{eq:VariousFormsA}
\EE{\dUpsilon}{\QP}_E(\rep u)\eqdef& \int_{\dUpsilon} \rep\cdc^\dUpsilon_E(\rep u)\diff\QP\comma && E\in\msE\comma\quad \rep u\in\Cyl{\Dz}
\\
\label{eq:VariousFormsB}
\EE{\dUpsilon(E)}{\QP^\eta_E}(\rep u) \eqdef& \int_{\dUpsilon(E)} \rep\cdc^{\dUpsilon(E)}(\rep u)\diff\QP^\eta_E\comma &&\begin{gathered} E\in\msE\comma\quad\eta\in\dUpsilon\comma\\ \rep u\in\Cyl{\Dz}\fstop \end{gathered}
\end{align}
\end{subequations}

Before proving the next results, let us comment on the meaning of the forms above.
For simplicity, let us assume we have already shown that both the forms in~\eqref{eq:VariousForms} are quasi-regular and strongly local, so that their properties can be recast in terms of the properly associated Markov diffusions.
Let~$\mbfM_E$, resp.~$\mbfM^{\eta,E}$, be the diffusion associated to~\eqref{eq:VariousFormsA}, resp.~\eqref{eq:VariousFormsB}.
The corresponding sample paths~$\gamma^E_t$ and~$\gamma^{\eta,E}_t$ coincide almost surely when restricted to~$E$.
Restricting on~$E^\complement$ we have instead that~$t\mapsto\gamma^E_t\restr_{E^\complement}$ is almost surely a constant configuration on~$E^\complement$ randomly distributed according to~$\pr^{E^\complement}_\pfwd\QP$, whereas~$t\mapsto\gamma^{\eta,E}_t\restr_{E^\complement}=\eta_{E^\complement}$ almost surely.

\begin{prop}\label{p:MarginalWP}
Let~$(\mcX,\cdc)$ be a \TLDS, and~$\QP$ be a probability measure on $\ttonde{\dUpsilon,\A_\mrmv(\msE)}$ satisfying Assumption~\ref{ass:CAC} for some localizing sequence~$\seq{E_h}_h$.
Then,
\begin{align}\label{eq:p:MarginalWP:0}
\EE{\dUpsilon}{\QP}_{E_h}(\rep u)=\int_\dUpsilon  \EE{\dUpsilon(E_h)}{\QP^\eta_{E_h}}(\rep u_{E_h,\eta}) \diff\QP(\eta) \comma\qquad h\in\N\comma\qquad \rep u\in\Cyl{\Dz}\fstop
\end{align}
Furthermore,~$\EE{\dUpsilon}{\QP}_{E_h}$ is well-defined on~$\CylQP{\QP}{\Dz}$, in the sense that:
\begin{enumerate}[$(i)$]
\item\label{i:p:MarginalWP:1} it does not depend on the choice of the $\QP$-representative~$\rep u$ of~$u\in \CylQP{\QP}{\Dz}$;
\item\label{i:p:MarginalWP:2} it does not depend on the choice of the strong lifting~$\ell$.
\end{enumerate}

\begin{proof}
Let~$h$ and~$\eta$ be fixed and set~$E\eqdef E_h$ for simplicity of notation.
By Assumption~\ref{ass:CAC} and Remark~\ref{r:ConditionalAC}, we may apply Proposition~\ref{p:NewWP} to~$\QP^\eta_E$ on~$\dUpsilon(E)$ to obtain that the form~$\EE{\dUpsilon(E)}{\QP^\eta_E}$ is well-posed, in the sense that it satisfies~\iref{i:p:MarginalWP:1}.
Furthermore, by Lemma~\ref{l:MaRoeckner} applied to~$\dUpsilon(E)$,
\begin{align*}
&\tonde{\rep\cdc^{\dUpsilon(E)}(\rep u_{E,\eta})\circ \pr_E}(\gamma)=
\\
=&\ \sum_{x\in\gamma_E} \ttonde{(\gamma_E)_x}^{-1}\cdot \rep\cdc\tonde{\rep u_{E,\eta}\ttonde{\car_{E\setminus \set{x}} \cdot\gamma_E+(\gamma_E)_x\delta_\bullet}-\rep u_{E,\eta}\ttonde{\car_{E\setminus\set{x}}\cdot\gamma_E}}
\\
=&\ \sum_{x\in\gamma_E} \ttonde{(\gamma_E)_x}^{-1}\cdot \rep\cdc\tonde{\rep u_{E,\eta}\ttonde{\car_{X\setminus\set{x}}\cdot\gamma_E+(\gamma_E)_x\delta_\bullet}-\rep u_{E,\eta}\ttonde{\car_{X\setminus\set{x}}\cdot\gamma_E}}
\\
=&\ \sum_{x\in\gamma_E} \ttonde{(\gamma_E+\eta_{E^\complement})_x}^{-1}\cdot \rep\cdc\tonde{\rep u\ttonde{\car_{X\setminus\set{x}}\cdot\gamma_E+(\gamma_E)_x\delta_\bullet+\eta_{E^\complement}}-\rep u\ttonde{\car_{X\setminus\set{x}}\cdot\gamma_E+\eta_{E^\complement}}}
\\
=&\ \sum_{x\in\gamma_E} \ttonde{(\gamma_E+\eta_{E^\complement})_x}^{-1}\cdot \rep\cdc\tonde{\rep u\ttonde{\car_{X\setminus\set{x}}\cdot(\gamma_E+\eta_{E^\complement})+(\gamma_E+\eta_{E^\complement})_x\delta_\bullet}-\rep u\ttonde{\car_{X\setminus\set{x}}\cdot(\gamma_E+\eta_{E^\complement})}}
\\
=&\ \rep\cdc^\dUpsilon_E(\rep u) (\gamma_E+\eta_{E^\complement})\comma
\end{align*}
where the last equality holds by definition~\eqref{eq:RestrictedCdCUpsilon} of~$\rep\cdc^\dUpsilon_E$. This shows that
\begin{equation}\label{eq:CdCRestrConditionalFormCylinderF}
\rep\cdc^{\dUpsilon(E)}(\rep u_{E,\eta})\circ \pr_E\equiv \rep\cdc^\dUpsilon_E(\rep u) \qquad \text{on} \quad \Lambda_{\eta, E^\complement} \comma
\end{equation}
where~$\Lambda_{\eta,E^\complement}$ is defined as in~\eqref{eq:RoeSch99Set}.
Since~$\QP^{\eta_{E^\complement}}$ is concentrated on~$\Lambda_{\eta, E^\complement}$, we thus have
\begin{equation*}
\int_{\dUpsilon} \rep\cdc^\dUpsilon_E(\rep u) \diff \QP^{\eta_{E^\complement}}
=\int_{\dUpsilon} \rep\cdc^{\dUpsilon(E)}(\rep u_{E,\eta})\circ \pr_E \diff\QP^{\eta_{E^\complement}}= \int_{\dUpsilon(E)} \rep\cdc^{\dUpsilon(E)}(\rep u_{E,\eta}) \diff\QP^\eta_E
\end{equation*}
for every~$\rep u\in\Cyl{\Dz}$, which concludes the proof of~\eqref{eq:p:MarginalWP:0} by integration w.r.t.~$\QP$ and Proposition~\ref{p:ConditionalIntegration}.

Assertion~\iref{i:p:MarginalWP:1} immediately follows from the well-posedness of the right-hand side in~\eqref{eq:p:MarginalWP:0}.
Assertion~\iref{i:p:MarginalWP:2} follows similarly, as soon as we show that, for~$\QP$-a.e.~$\eta$, the form~$\EE{\dUpsilon(E)}{\QP^\eta_E}$ is independent of~$\ell$.

To this end, note that, for different strong liftings~$\ell_1, \ell_2$,
\begin{align*}
\int_{\dUpsilon(E)} &\abs{ \rep\cdc^{\dUpsilon(E)}_{\ell_1}(\rep u_{E,\eta})- \rep\cdc^{\dUpsilon(E)}_{\ell_2}(\rep u_{E,\eta}) }\diff\QP^\eta_E
\\
=&\int_{\dUpsilon} \tonde{\abs{\rep\cdc^{\dUpsilon(E)}_{\ell_1}(\rep u_{E,\eta})-\rep\cdc^{\dUpsilon(E)}_{\ell_2}(\rep u_{E,\eta})} \cdot \frac{\diff \QP^\eta_E}{\diff (\pr_E)_\pfwd \PP_\mssm}} \circ \pr_E \diff\PP_\mssm
\\
=&\int_\dUpsilon \int_X \abs{\rep\cdc^{\dUpsilon(E)}_{\ell_1}(\rep u_{E,\eta})-\rep\cdc^{\dUpsilon(E)}_{\ell_2}(\rep u_{E,\eta})}(\gamma_E+\car_E\delta_x) \, \cdot 
\\
&\qquad \cdot \frac{\diff \QP^\eta_E}{\diff (\pr_E)_\pfwd \PP_\mssm}(\gamma_E+\car_E\delta_x) \diff\mssm(x)\, \diff\PP_\mssm(\gamma) \comma
\end{align*}
where the second equality follows from~\eqref{eq:Mecke}.
By definition of lifting, and computing~$\rep\cdc^{\dUpsilon(E)}_{\ell_i}(\rep u_{E,\eta})$ by~\eqref{eq:d:LiftCdCRep} for $i=1,2$, we have that
\begin{align*}
x\longmapsto \abs{\rep\cdc^{\dUpsilon(E)}_{\ell_1}(\rep u_{E,\eta})-\rep\cdc^{\dUpsilon(E)}_{\ell_2}(\rep u_{E,\eta})}(\gamma_E+\car_E\delta_x) \equiv 0 \as{\mssm}\comma
\end{align*}
and the conclusion follows.
\end{proof}
\end{prop}

In light of Proposition~\ref{p:MarginalWP}, the following assumption is well-posed.

\begin{ass}[Conditional closability]\label{ass:ConditionalClosability}
Let~$(\mcX,\cdc)$ be a \TLDS, and~$\QP$ be a probability measure on~$\ttonde{\dUpsilon,\A_{\mrmv}(\msE)}$ satisfying Assumptions~\ref{ass:Mmu} and~\ref{ass:CAC} for some localizing sequence~$\seq{E_h}_h$.
We say that~$\QP$ satisfies the \emph{conditional closability} assumption~\ref{ass:ConditionalClos} if the forms
\begin{equation}\tag*{$(\mathsf{CC})_{\ref{ass:ConditionalClosability}}$}\label{ass:ConditionalClos}
\EE{\dUpsilon(E_h)}{\QP^\eta_{E_h}}(u,v)=\int_{\dUpsilon(E_h)} \cdc^{\dUpsilon(E_h)}(u,v) \diff\QP^\eta_{E_h}\comma\qquad
\begin{aligned}
u,v\in&\ \CylQP{\QP^\eta_{E_h}}{\Dz}\comma
\\ 
h\in\N&\comma \quad \eta\in\dUpsilon\comma
\end{aligned}
\end{equation}
are closable on~$L^2\ttonde{\dUpsilon(E_h),\QP^\eta_{E_h}}$ for $\QP$-a.e.~$\eta\in\dUpsilon$, and for every~$h\in\N$.
\end{ass}

For~$E\eqdef E_h$ in a suitable localizing sequence, combining Lemma~\ref{l:MmuL1} with the disintegration result in Proposition~\ref{p:ConditionalIntegration} shows that~$\EE{\dUpsilon(E)}{\QP^\eta_E}$ is finite on~$\CylQP{\QP^\eta_E}{\Dz}$ for every~$E$, for $\QP$-a.e.~$\eta\in\dUpsilon$.
Since each of the forms~\ref{ass:ConditionalClos} is densely defined by Remark~\ref{r:DensityQP}\iref{i:r:DensityQP:1}, its closure
\begin{equation*}
\ttonde{\EE{\dUpsilon(E)}{\QP^\eta_E},\dom{\EE{\dUpsilon(E)}{\QP^\eta_E}}}
\end{equation*}
is a Dirichlet form.

\begin{ese}
When~$X=\R^n$ is a standard Euclidean space, then all canonical Gibbs measures  and the laws of some determinantal/permanental point processes (e.g., the Ginibre,~$\mathrm{sine}_\beta$, $\mathrm{Airy}_\beta$, $\mathrm{Bessel}_{\alpha,\beta}$) satisfy~\ref{ass:ConditionalClos}.
See~\S\ref{sss:ExamplesAC} for references and further examples.
\end{ese}

\begin{thm}[Closability II]\label{t:ClosabilitySecond}
Let~$(\mcX,\cdc)$ be a \TLDS,~$\QP$ be a probability measure on $\ttonde{\dUpsilon,\A_\mrmv(\msE)}$ satisfying Assumptions~\ref{ass:CAC} and~\ref{ass:ConditionalClos}.
Then, the form~$\ttonde{\EE{\dUpsilon}{\QP},\CylQP{\QP}{\Dz}}$ as defined in~\eqref{eq:Temptation} is well-defined, densely defined, and closable, and its closure~$\ttonde{\EE{\dUpsilon}{\QP},\dom{\EE{\dUpsilon}{\QP}}}$ is a Dirichlet form on~$L^2(\QP)$.

\begin{proof}
Let~$\seq{E_h}_h$ be a localizing sequence witnessing~\ref{ass:CAC} and~\ref{ass:ConditionalClos}.
For simplicity of notation let~$h$ be fixed and set~$E\eqdef E_h$.
The form~$\ttonde{\EE{\dUpsilon}{\QP}_E,\CylQP{\QP}{\Dz}}$ in~\eqref{eq:p:MarginalWP:0} is finite by Lemma~\ref{l:MmuL1}, well-defined on~$\QP$-classes by Proposition~\ref{p:MarginalWP}, and closable by replacing~$\mbbD$ in~\cite[Prop.~V.3.1.1]{BouHir91} with the core~$\CylQP{\QP}{\Dz}$.
We may apply~\cite[Prop.~V.3.1.1]{BouHir91} by the assumption on the closability of the forms~\ref{ass:ConditionalClos}, and Proposition~\ref{p:ConditionalIntegration}.
Its closure~$\ttonde{\EE{\dUpsilon}{\QP}_E,\dom{\EE{\dUpsilon}{\QP}_E}}$ is a Dirichlet form for each~$h\in\N$.

It follows from the representation of~$\rep\cdc^\dUpsilon_E$ in~\eqref{eq:RestrictedCdCUpsilon} that~$h\mapsto \rep\cdc^\dUpsilon_{E_h}(\rep u)$ is monotone increasing in~$h$, w.r.t.\ the inclusion of sets, for every~$\rep u\in\Cyl{\Dz}$.
As a consequence, by the equality in~\eqref{eq:p:MarginalWP:0} and by the representation in~\eqref{eq:VariousFormsA} the association
\begin{equation}\label{eq:MonotonicityE}
h\longmapsto \ttonde{\EE{\dUpsilon}{\QP}_{E_h},\CylQP{\QP}{\Dz}} \comma\quad h\in\N\comma \qquad \text{is monotone}
\end{equation}
in the sense of Dirichlet forms.

Now, fix~$\rep u\in\Cyl{\Dz}$ and~$\eta\in\dUpsilon$, and observe that the  conditional expectation~(see e.g.,~\cite[233E(b)]{Fre00})~$u_h\colon \gamma\mapsto \mbfE_\QP\quadre{\rep u(\gamma)\middle| \gamma_{E_h^\complement}=\eta_{E_h^\complement}}$ satisfies~$u_h(\gamma_E+\eta_{E^\complement})= \class[\QP^\eta_{E_h}]{\rep u}$.
By definition of conditional expectation,~$h\mapsto u_h$ satisfies the tower property and is therefore an $L^2(\QP)$-martingale (see e.g.~\cite[275A]{Fre00}), adapted to the filtration of $\sigma$-algebras on~$\dUpsilon$ generated by~$\CylQP{\QP^\eta_{E_h}}{\QP}$.
Since~$\seq{E_h}_h$ is a monotone exhaustion of~$X$, the martingale~$\tseq{\rep u_h}_h$ converges to~$\class[\QP]{\rep u}$ in~$L^2(\QP)$.
As a consequence,
\begin{equation*}
\bigcap_h \CylQP{\QP^\eta_{E_h}}{\Dz}=\CylQP{\QP}{\Dz}\comma
\end{equation*}
and the form
\begin{gather*}
\Dz^{\dUpsilon,\QP}\eqdef \set{u\in L^2(\QP): u\in \bigcap_h \CylQP{\QP^\eta_{E_h}}{\Dz} \comma \sup_h \EE{\dUpsilon}{\QP}_{E_h}(u)<\infty}\comma
\\
\widetilde{\EE{\dUpsilon}{\QP}}(u)\eqdef \sup_h \EE{\dUpsilon}{\QP}_{E_h}(u)\comma \qquad u\in \dom{\widetilde{\EE{\dUpsilon}{\QP}}}\fstop
\end{gather*}
is a well-defined pre-Dirichlet form on~$L^2(\QP)$ with core~$\CylQP{\QP}{\Dz}$, closable by~\cite[Prop.~I.3.7.(ii)]{MaRoe92}. Again since~$\seq{E_h}_h$ exhausts~$X$, it is readily verified that
\begin{equation*}
\widetilde{\EE{\dUpsilon}{\QP}}(u)=\int_\dUpsilon \cdc^\dUpsilon(u) \diff\QP\comma \qquad u\in \CylQP{\QP}{\Dz}
\end{equation*}
for~$\cdc^\dUpsilon$ given by the~$\QP$-class of~\eqref{eq:d:LiftCdCRep}.
Thus, finally,
\begin{equation*}
\EE{\dUpsilon}{\QP}(u)=\widetilde{\EE{\dUpsilon}{\QP}}(u) \comma \qquad u\in \CylQP{\QP}{\Dz}
\end{equation*}
and therefore~$\ttonde{\EE{\dUpsilon}{\QP},\CylQP{\QP}{\Dz}}$ is closable, and its closure coincides with $\ttonde{\widetilde{\EE{\dUpsilon}{\QP}},\dom{\widetilde{\EE{\dUpsilon}{\QP}}}}$.
In particular, the form constructed above does not depend on the localizing sequence~$\seq{E_h}_h$.
\end{proof}
\end{thm}

\begin{rem}
\begin{enumerate*}[$(a)$]
\item We stress that Theorem~\ref{t:ClosabilitySecond} can be proved in many different ways:
\begin{enumerate*}[$({a}_1)$] 
\item as an application of the theory of superpositions of Dirichlet forms in~\cite{BouHir91};

\item as an application of the theory of direct integrals of Dirichlet forms developed by the first named author in~\cite{LzDS20};

\item by the very same arguments as in the proof of~\cite[Thm.~4]{Osa96}, noting that~\cite[Prop.~4.1]{Osa96} there is replaced by our Assumption~\ref{ass:ConditionalClosability}.
\end{enumerate*}

\item It is possible to show that the projected conditional systems~$\set{\QP^\eta_E}_{\eta\in\dUpsilon,E\in\msE}$ are consistent similarly to \emph{specifications} in the sense of Preston~\cite[\S6]{Pre76}.
\item If the measure~$\QP$ has full $\T_\mrmv(\msE)$-support, the well-posedness of the form~$\ttonde{\EE{\dUpsilon}{\QP},\CylQP{\QP}{\Dz}}$ is immediate, since every function in~$\CylQP{\QP}{\Dz}$ has a unique continuous representative~$\rep u\in \Cyl{\Dz}$.
\end{enumerate*}
\end{rem}

\subsection{Identification of forms on product spaces}\label{ss:IdentificationFormsProduct}
Let~$(\mcX,\cdc)$ be a \TLDS.
In this section we show how to identify objects on~$\dUpsilon$ ---~e.g.\ the form~$\EE{\dUpsilon}{\QP}$ or, in the case~$\QP=\PP_\mssm$, the corresponding heat kernel~$\hh{\dUpsilon}{\PP_\mssm}_\bullet$~--- with a suitable counterpart thereof on the infinite product~$\mbfX$.

The identification of forms will be instrumental to the proof of the Rademacher property for~$\EE{\dUpsilon}{\QP}$ established in~\S\ref{sss:Rademacher}, and constitutes one of the main results of this paper.
It relies on the notion of abstract completion of pre-Dirichlet spaces and on the finite-dimensional approximations thorougly discussed in~\S\ref{ss:IdentificationFormsProduct} below.

\medskip

In the following, the well-posedness of the pre-Dirichlet form~$\ttonde{\EE{\dUpsilon}{\QP},\CylQP{\QP}{\Dz}}$ will be often established by means of Proposition~\ref{p:NewWP} under Assumption~\ref{ass:AC}.
The form is densely defined by Remark~\ref{r:DensityQP}\iref{i:r:DensityQP:1}.
In light of the results in the previous sections, it will be convenient from now on to simply \emph{assume} that the pre-Dirichlet form~$\ttonde{\EE{\dUpsilon}{\QP},\CylQP{\QP}{\Dz}}$ be \emph{closable}.
In particular, the next assumption includes all measures satisfying Assumption~\ref{ass:Closability} (by Theorem~\ref{t:Closability}), as well as measures satisfying the assumptions in Theorem~\ref{t:ClosabilitySecond}.

\begin{ass}[Closability]\label{d:ass:CWP}
Let~$\QP$ be a probability measure on~$\ttonde{\dUpsilon,\A_\mrmv(\msE)}$. It satisfies Assumption~\ref{ass:CWP} if
\begin{equation}\tag*{$(\mssC)_{\ref{d:ass:CWP}}$}\label{ass:CWP}
\begin{gathered}
\textbf{$\QP$ satisfies Assumption~\ref{ass:Mmu} and}
\\
\textbf{the pre-Dirichlet form~\eqref{eq:Temptation} is closable on~$L^2(\QP)$.}
\end{gathered}
\end{equation}
\end{ass}

Most of the results of this section will make use only of Assumptions~\ref{ass:AC} and~\ref{ass:CWP}.
While~\ref{ass:AC} will be essential to our analysis,~\ref{ass:CWP} ought to be understood as a placeholder for any assumption granting the closability of~$\ttonde{\EE{\dUpsilon}{\QP},\dom{\EE{\dUpsilon}{\QP}}}$ in a specific case of interest.
As a first result granted by these assumptions, let us reprove Lemma~\ref{l:MaRoe4.3} and Proposition~\ref{p:ExtDom} in this greater generality.

\begin{lem}\label{l:MaRoe4.3II}
Let~$(\mcX,\cdc)$ be a \TLDS, and~$\QP$ be a probability measure on~$\ttonde{\dUpsilon,\A_\mrmv(\msE)}$ satisfying Assumption~\ref{ass:AC} and~\ref{ass:CWP}.
Then, all the assertions in Lemma~\ref{l:MaRoe4.3} hold true.
\begin{proof}
\cite[Lem.~4.2]{MaRoe00} is replaced by our Lemma~\ref{l:AbstractCompletion}.
The validity of $(S^\Gamma.\mu)$ in~\cite[p.~282]{MaRoe00} is precisely our Proposition~\ref{p:NewWP}.
The proof thus follows exactly as in~\cite[Lem.~4.3]{MaRoe00}.
\end{proof}
\end{lem}

\begin{prop}
Let~$(\mcX,\cdc)$ be a \TLDS, and~$\QP$ be a probability measure on~$\ttonde{\dUpsilon,\A_\mrmv(\msE)}$ satisfying Assumption~\ref{ass:AC} and~\ref{ass:CWP}.
Then, all the assertions in Proposition~\ref{p:ExtDom} hold true.
\begin{proof}
Set (cf.~\cite[p.~301]{MaRoe00})
\begin{align*}
\tnorm{\rep f}_{\mssm_\QP,1}\eqdef \quadre{\int \rep\cdc(\rep f)\diff\mssm_\QP}^{1/2}+\int\tabs{\rep f} \diff\mssm_\QP\comma \qquad \rep f\in \Dz\comma
\end{align*}
which is finite by Assumption~\ref{ass:Mmu} (part of~\ref{ass:CWP}). This defines a corresponding norm~$\norm{f}_{\mssm_\QP,1}$ on the space of $\mssm_\QP$-classes~$\class[\mssm_\QP]{\Dz}$, and the inclusion~$\cok_{1,\mssm_\QP}\colon \class[\mssm_\QP]{\Dz}\to L^1(\mssm_\QP)$ uniquely extends to a continuous linear map~$\cokb_{1,\mssm_\QP}\colon \co_{1,\mssm_\QP}(\Dz)\to L^1(\mssm_\QP)$, injective by Lemma~\ref{l:AbstractCompletion} (which substitutes~\cite[Lem.~4.4]{MaRoe00}).
The proof now follows as in~\cite[Prop.~4.6]{MaRoe00} having care to substitute~\cite[Lem.~4.3]{MaRoe00} with the corresponding statement previously established in Lemma~\ref{l:MaRoe4.3II}.
\end{proof}
\end{prop}

We aim to lift a form~$\EE{\dUpsilon}{\QP}$ satisfying Assumption~\ref{ass:CWP} on the configuration space~$\dUpsilon$ to a quadratic form~$\EE{\asym{\infty}}{\llb_\pfwd\QP}$ on the infinite product space~$X^{\tym{\infty}}$.
The latter form will depend on the choice of a labeling map~$\lb$ around which this lifting procedure is pivoting.
Our goal is to transfer properties of~$\EE{\asym{\infty}}{\llb_\pfwd\QP}$ to~$\EE{\dUpsilon}{\QP}$.
These properties will be studied by approximating~$\EE{\asym{\infty}}{\llb_\pfwd\QP}$ with forms on the finite-product spaces~$X^\tym{n}$.

\smallskip

In light of Lemma~\ref{l:WellPosedness}, we shall denote again by~$\EE{\asym{N}}{\llb_\pfwd\QP}$ the quadratic form defined on~$\spclass[\llb]{\preW(\msE)}$.
Without an integration-by-parts formula, the closability of~$\ttonde{\EE{\asym{N}}{\llb_\pfwd\QP}, \spclass[\llb]{\preW(\msE)}}$ on its natural Hilbert space~$H_\infty\eqdef L^2(\QP^\asym{\infty})$ appears to be out of reach.
We shall therefore work with the abstract Hilbert completions
\begin{align}\label{eq:AbstractCompletionN}
K_N\eqdef \coK{\spclass[\llb]{\preW(\msE)}}{N}\comma \qquad N\in \overline\N_1\comma
\end{align}
of~$\spclass[\llb]{\preW(\msE)}$ w.r.t.~$(\EE{\asym{N}}{\llb_\pfwd\QP})^{1/2}_1$ given as in Definition~\ref{d:AbstractCompletion} with~$\QP^\asym{\infty}$~in place of~$\mssm$, $\spclass[\llb]{\preW(\msE)}$ in place of~$\Dz$, and~$\EE{\asym{N}}{\llb_\pfwd\QP}$ in place of~$\EE{X}{\mssm}$.
Again as in Definition~\ref{d:AbstractCompletion}, let us set~$\coi_N\eqdef \coi_{2,N}$ and~$\cokb_N\eqdef \cokb_{2,N}$.

\smallskip

Now, let us assume that~$\ttonde{\EE{\dUpsilon}{\QP},\dom{\EE{\dUpsilon}{\QP}}}$ satisfies Assumption~\ref{ass:CWP}, and admits square field operator $\ttonde{\SF{\dUpsilon}{\QP},\dom{\SF{\dUpsilon}{\QP}}}$ as in~\eqref{eq:d:LiftCdC}. 
%\purple{Recall Proposition~\ref{p:ExtDom}\iref{i:p:ExtDom:1}.}
The next lemma constitutes a key step in the sought identification.

\begin{lem}[Square field approximation]\label{l:CdCCylinderTruncation}
Let~$(\mcX,\cdc)$ be a \TLDS, and~$\QP$ be a probability measure on~$\ttonde{\dUpsilon,\A_\mrmv(\msE)}$ satisfying Assumptions~\ref{ass:AC} and~\ref{ass:CWP}.
For~$\rep u= F\circ \rep\mbff^\trid \in \Cyl{\rep\Dz}$ set
\begin{align}\label{eq:CylFReduction}
\rep u^\sym{n}\colon \mbfy \longmapsto (\Lb^*\rep u)(\mbfy) \comma \qquad \mbfy\eqdef\seq{y_i}_{i\leq n} \in X^\tym{n}\fstop
\end{align}
Then, $u^{\sym{n}}\eqdef \class[\mssm^\otym{n}]{\rep u^\sym{n}}\in \domext{\cdc^\otym{n}}$,
\begin{align}
\label{eq:l:CdCCylinderTruncation:1}
\SF{\dUpsilon}{\QP}(u)=\nlim \cdc^{\otym n}\ttonde{u^\sym{n} } \circ \tr^n\circ \lb \as{\QP}\comma
%\\
%\label{eq:l:CdCCylinderTruncation:1b}
%\cdc^\asym{\infty}(u\circ\Lb)=&\nlim \cdc^\asym{n}\ttonde{u^\sym{n} \circ \tr^n} &\as{\PP_\mssm^\asym{\infty}}\comma&
\end{align}
and, for every strong lifting~$\ell\colon L^\infty(\mssm)\rar \mcL^\infty(\mssm)$,
\begin{align}\label{eq:l:CdCCylinderTruncation:1c}
\Lb^*\ttonde{\SF{\dUpsilon}{\QP}(u)}=\class[\QP^\asym{\infty}]{\rep\cdc_\ell^\asym{\infty}(\Lb^*\rep u)} \as{\QP^\asym{\infty}}\fstop
\end{align}

\begin{proof}
We divide the proof into several steps.

\paragraph{Sets} Firstly, we shall construct a set~$\Omega_u$, of full $\QP$-measure depending on~$u$, on which we show~\eqref{eq:l:CdCCylinderTruncation:1}.
By Proposition~\ref{p:Products}, the space~$\mcX^{\otym n}$ is a topological local structure for every~$n\geq 1$, thus it admits a strong lifting, say~$\ell^\asym{n}$. We write~$\ell$ in place of~$\ell^\asym{1}$.
For~$\rep u^\sym{n}$ as above and~$\mbfy\eqdef\seq{y_i}_{i\leq n}\in X^\tym{n}$, set
\begin{align*}
\rep u^\sym{n}_{\mbfy, {p}}\colon z&\longmapsto \rep u^\sym{n}(y_1,\dotsc, y_{p-1},z,y_{p+1},\dotsc, y_n) \comma
\\
B^\asym{n}_u\eqdef&\set{\mbfy \in X^\tym{n}: 
\sum_{p=1}^n \repSF{X}{\mssm}_\ell \tonde{u^\sym{n}_{\mbfy, {p}}}(x_p)=\ell^\asym{n}\tonde{\cdc^\otym{n}\ttonde{u^\sym{n}}}(\mbfy)}\fstop
\end{align*}
Note that~$B^\asym{n}_u$ has full $\mssm^{\otimes n}$-measure for every~$n\geq 1$. In fact,~$B^\asym{1}_u=X$.

By assumption on~$f_i$, there exists~$E\in\msE$ so that~$\bigcup_i\supp f_i\subset E$. By definition of~$\Upsilon(\msE)$,
\begin{equation}\label{eq:l:CdCCylinderTruncation:2}
\begin{aligned}
\forall \gamma\in \Upsilon(\msE) \quad \exists n_\gamma\geq&\ 1: \\
\rep u(\gamma)=&\ \ttonde{\rep u^\sym{n}\circ \tr^n \circ \lb}(\gamma)\comma \qquad n\geq n_\gamma \fstop
\end{aligned}
\end{equation}

Let further~$\Omega_{0,u}\subset \Upsilon(\msE)$ be a set of full~$\QP$-measure such that~\eqref{eq:d:LiftCdC} holds with~$\rep\cdc_\ell$ in place of~$\rep\cdc$, and~$\Omega_{n,u}\eqdef\ttonde{\tr^n\circ \lb}^{-1} (B^\asym{n}_u)$.
Since~$B^\asym{n}_u$ is $\mssm^\otym{n}$-conegligible,~$\QP \Omega_{n,u}=1$ as a consequence of Assumption~\ref{ass:AC}.
Thus,~$\Omega_u\eqdef \bigcap_{n\geq 0} \Omega_{n,u}$ has full~$\QP$-measure as well.

\paragraph{Statements on configuration spaces} Let us now turn to the choice of a representative for $\SF{\dUpsilon}{\QP}(u)$ satisfying~\eqref{eq:l:CdCCylinderTruncation:1} on~$\Omega_u$. Indeed, set
\begin{align*}
\rep g_u\colon \gamma\longmapsto \sum_{i,j}^{k,k} (\partial_i F \cdot \partial_j F)(\rep \mbff^\trid \gamma) \cdot \ell\ttonde{\cdc(f_i , f_j)}^\trid \gamma
\end{align*}
and note that it is a $\QP$-representative of~$\SF{\dUpsilon}{\QP}(u)$ by~\eqref{eq:d:LiftCdC} and~\eqref{eq:LiftCdC}.
For all~$\gamma\in\Omega_u$, for all~$n\geq n_\gamma$ as in~\eqref{eq:l:CdCCylinderTruncation:2},
\begin{align}
\nonumber
\rep g_u(\gamma)%=&\SF{\Upsilon}{\PP_\mssm}\ttonde{u^\sym{n}\circ \tr^n \circ \lb_\mbfx}(\gamma)
%\\
=&\sum_{i,j}^{k,k} (\partial_i F \cdot \partial_j F)(\rep \mbff^\trid \gamma) \cdot \ell\ttonde{\cdc(f_i , f_j)}^\trid \gamma
\\
\label{eq:l:CdCCylinderTruncation:2.1}
=&\sum_{i,j}^{k,k} (\partial_i F \cdot \partial_j F)\ttonde{(\rep \mbff^\trid\circ \Lb)(\mbfx^\asym{n})} \sum_{p=1}^n\ell\ttonde{\cdc(f_i,f_j)}(x_p)
\\
\nonumber
=&\sum_{p=1}^n \repSF{X}{\mssm}_\ell \tonde{u^\sym{n}_{\mbfx, {p}} }(x_p)
\\
\label{eq:l:CdCCylinderTruncation:2.11}
=&\ \ell^\asym{n}\tonde{\cdc^\otym{n}\ttonde{u^\sym{n}}}(\mbfx^\asym{n})\comma
\end{align}
where~$\mbfx\eqdef\lb(\gamma)$. That is,
\begin{equation}\label{eq:l:CdCCylinderTruncation:3}
\begin{aligned}
\forall \gamma\in \Omega_u\quad \exists n_\gamma\geq&\ 1: 
\\
\rep g_u(\gamma)=&\ \tonde{ \ell^\asym{n}\tonde{\cdc^\otym{n}\ttonde{u^\sym{n}}}\circ \tr^n\circ \lb} (\gamma) \comma \qquad n\geq n_\gamma \fstop
\end{aligned}
\end{equation}
Formula~\eqref{eq:l:CdCCylinderTruncation:1} follows by passing to $\QP$-classes, since, for fixed~$\gamma$, the left-hand side of~\eqref{eq:l:CdCCylinderTruncation:3} is independent of the chosen family of liftings~$\rep \ell^\asym{n}$.

\paragraph{Statements on product spaces}
Now we show~\eqref{eq:l:CdCCylinderTruncation:1c}.
By pre-composing \eqref{eq:l:CdCCylinderTruncation:1} with~$\Lb$ we have
\begin{align}\label{eq:l:CdCCylinderTruncation:4}
\ttonde{\SF{\dUpsilon}{\QP}(u)\circ \Lb}(\mbfx)=\ttonde{\nlim \cdc^\otym{n}\ttonde{u^\sym{n}}\circ \tr^n}(\mbfx)\comma \qquad \gamma\in \Omega_u\comma \mbfx\eqdef\lb(\gamma) \fstop
\end{align}

Further note that, since
\begin{align*}
\ttonde{\rep\mbff^\trid\circ\Lb}(\mbfx^\asym{n})=\ttonde{\rep\mbff^\trid\circ\Lb}(\mbfx)\comma \qquad \gamma\in \Omega_u\comma \mbfx\eqdef\lb(\gamma)\comma n\geq n_\gamma\comma
\end{align*}
then
\begin{align}\label{eq:l:CdCCylinderTruncation:2.0}
\rep\cdc_\ell^\asym{n}\ttonde{\rep u^\sym{n}\circ\tr^n}\ttonde{\lb(\gamma)}=\eqref{eq:l:CdCCylinderTruncation:2.1}=\rep\cdc_\ell^\asym{n}(\rep u\circ \Lb)\ttonde{\lb(\gamma)} \comma \qquad \gamma\in \Omega_u\comma n\geq n_\gamma\fstop
\end{align}

Finally, combining the first equality in~\eqref{eq:l:CdCCylinderTruncation:2.0} with~\eqref{eq:l:CdCCylinderTruncation:2.11},
\begin{align*}
\forall \gamma\in\Omega_u \quad \exists& n_\gamma \geq 1 :
\\
\rep\cdc_\ell^\asym{n}&\ttonde{\rep u^\sym{n}\circ\tr^n}\ttonde{\lb(\gamma)} = \ell^\asym{n}\ttonde{\cdc^\otym{n}(u^\sym{n})}\ttonde{(\tr^n\circ \lb)(\gamma)}\comma \qquad n\geq n_\gamma \comma
\end{align*}
whence
\begin{align*}
\forall \mbfx\in\lb\ttonde{\Omega_u} \quad \exists& n_\mbfx \geq 1 :
\\
\rep\cdc_\ell^\asym{n}&\ttonde{\rep u^\sym{n}\circ\tr^n}(\mbfx)= \ell^\asym{n}\ttonde{\cdc^\otym{n}(u^\sym{n})}(\mbfx^\asym{n})\comma \qquad n\geq n_\mbfx \fstop
\end{align*}
By comparison of the latter equation with~\eqref{eq:l:CdCCylinderTruncation:4} and the second equality in~\eqref{eq:l:CdCCylinderTruncation:2.0}, letting~$n\rar\infty$,
\begin{align}\label{eq:l:CdCCylinderTruncation:5}
\rep g_u\circ \Lb=&\nlim \rep\cdc_\ell^\asym{n}\ttonde{\rep u^{\sym{n}}\circ\tr^n}=\nlim \rep\cdc_\ell^\asym{n}(\rep u\circ \Lb) = \rep\cdc_\ell^\asym{\infty}(\rep u\circ \Lb) \comma \qquad \mbfx\in \lb(\Omega_u)\fstop
\end{align}
and~\eqref{eq:l:CdCCylinderTruncation:1c} follows by passing to~$\QP^\asym{\infty}$-classes.
\end{proof}
\end{lem}

\begin{lem}\label{l:InfDiff} Let~$(\mcX,\cdc)$ be a \TLDS, and~$\QP$ be a probability measure on~$\ttonde{\dUpsilon,\A_\mrmv(\msE)}$ satisfying Assumptions~\ref{ass:AC} and~\ref{ass:CWP}.
Then,~$\Lb^*\Cyl{\rep\Dz}\subset \preW(\msE)$. In particular, the form $\ttonde{\EE{\asym{N}}{\llb_\pfwd\QP}, \spclass[\llb]{\preW(\msE)}}$ is densely defined on~$H_\infty$.

\begin{proof}
Definition~\ref{d:preSobolev}\iref{i:d:preSobolev:1} holds by definition of cylinder function, while~\iref{i:d:preSobolev:3} follows by integrating~\eqref{eq:l:CdCCylinderTruncation:1c} w.r.t.~$\QP^\asym{\infty}$, since the integral of the left-hand side does not in fact depend on~$\lb$.
Thus, the inclusion~$\spclass[\llb]{\preW(\msE)}\subset H_\infty$ is dense, and the form~$\ttonde{\EE{\asym{N}}{\llb_\pfwd\QP}, \spclass[\llb]{\preW(\msE)}}$ is densely defined on~$H_\infty$ by Remark~\ref{r:DensityQP}\iref{i:r:DensityQP:1}.
\end{proof}
\end{lem}

Let~$H\eqdef L^2(\QP)$. We are now in the position to identify the Hilbert space 
\begin{align*}
K\eqdef\ttonde{\dom{\EE{\dUpsilon}{\QP}}, (\EE{\dUpsilon}{\QP})^{1/2}_1}
\end{align*}
with a Hilbert subspace of the abstract completion~$K_\infty$.
In order to do so, we shall show that
\begin{itemize}
\item $\Lb^*\colon H\rar H_\infty\eqdef L^2(\QP^\asym{\infty})$ is a unitary isomorphism (by definition);
\item the space~$K_n$ `converges' (see below) to~$K$ (see Lem.~\ref{l:KSPoisson});
\item the `convergence' of~$K_n$ to~$K$ is invariant under unitary isomorphism of the limit space (see Lem.~\ref{l:KSUnitary}).
\end{itemize}
The notion of convergence of Hilbert spaces we use here is the one introduced by K.~Kuwae and~T.~Shioya in~\cite{KuwShi03}. We recall the main definitions and related results below.
Since \emph{Kuwae--Shioya's convergence} induces a Hausdorff topology on the space of classes of Hilbert spaces modulo unitary isomorphisms, the limit space of the sequence~$\seq{K_n}_n$ above is uniquely identified,~\cite[Cor.~2.2]{KuwShi03}.
We may thus conclude that
\begin{itemize}
\item $\Lb^*\colon K\rar K_\infty$ is isometric onto its image in~$K_\infty$ (Prop.~\ref{p:LipIsometry}).
\end{itemize}

The identification of~$K$ with the subspace~$\Lb^* K$ of~$K_\infty$ will be of particular importance. Indeed, if we show that the completion embedding~$\cokb_\infty\colon K_\infty \rar H_\infty\eqdef L^2(\QP^\asym{\infty})$ is injective on a further subspace~$V\subset \Lb^*K$, then we may transfer properties of~$V$ to properties of~$\lb^* V\subset K$. In particular, in~\S\ref{sss:Rademacher} below, choosing~$\spclass[\llb]{\preW(\msE)}$ for some suitable labeling maps~$\lb$, we will show how to transfer the Lipschitz property of functions.

\paragraph{Kuwae--Shioya's convergence}
We recall the main definitions in~\cite[\S2]{KuwShi03}. 
In the following:~$(\mcA,\T,\preceq)$ is a directed Hausdorff topological space; every Hilbert space is real and \emph{separable} and, for a Hilbert space~$H$, we denote by~$\Lin(H)$ the space of bounded linear operators~$B\colon H\rar H$.

\begin{defs}[Convergence of spaces]\label{d:KSSpaces} A net~$\seq{H_\alpha}_{\alpha\in\mcA}$ of Hilbert spaces \emph{converges} to a Hilbert space~$H$ if there exist a dense linear subspace~$\KSD\subset H$ and densely defined linear maps~$\Phi_\alpha\colon \KSD\rar H_\alpha$ such that
\begin{align*}
\limsup_\alpha \norm{\Phi_\alpha u}_{H_\alpha}=\norm{u}_H\comma\qquad u\in \KSD\fstop
\end{align*}
\end{defs}

\begin{rem} The above definition is \emph{not} independent on the system of maps~$\seq{\Phi_\alpha}_{\alpha\in\mcA}$ and on the linear subspace~$\KSD$. If it is necessary to specify the system of maps~$\seq{\Phi_\alpha}$ and the subspace~$\KSD$, we write that~$\seq{H_\alpha}_{\alpha\in\mcA}$ converges to~$H$ \emph{w.r.t.~the pair~$(\Phi_\alpha, \KSD)$}.
\end{rem}

\begin{defs}[Convergences of vectors%, of operators
]\label{d:KSOperators}
Assume that a net~$\seq{H_\alpha}_{\alpha\in \mcA}$ of Hilbert spaces converges to a Hilbert space~$H$.
We say that
\begin{itemize}
\item a net~$\seq{u_\alpha}_{\alpha\in\mcA}$, with $u_\alpha\in H_\alpha$, (\emph{strongly}) \emph{converges} to~$u\in H$ if there exists a net~$\seq{\tilde u_\beta}_{\beta\in\mcB}\subset \KSD$ such that
\begin{align}\label{eq:KSStrongConvergence}
\lim_\beta \norm{\tilde u_\beta-u}_H=0 \qquad \text{and} \qquad \lim_\beta\limsup_\alpha \norm{\Phi_\alpha \tilde u_\beta-u_\alpha}_{H_\alpha}=0 \semicolon
\end{align}

\item a net~$\seq{u_\alpha}_{\alpha\in\mcA}$, with $u_\alpha\in H_\alpha$, \emph{weakly converges} to~$u\in H$ if 
\begin{align}\label{eq:KSWeakConvergence}
\lim_\alpha \scalar{u_\alpha}{v_\alpha}_{H_\alpha}=\scalar{u}{v}_H
\end{align}
for every net~$\seq{v_\alpha}_{\alpha\in\mcA}$, with~$v_\alpha\in H_\alpha$, strongly converging to~$v\in H$.
\end{itemize}
\end{defs}

\begin{lem}[Unitary invariance]\label{l:KSUnitary}
Let~$\seq{H_\alpha}_{\alpha\in\mcA}$ be a net of Hilbert spaces converging to a Hilbert space~$H$ w.r.t.~the pair~$(\Phi_\alpha, \KSD)$. Further let~$H'$ be a Hilbert space and~$U\colon H\rar H'$ be a unitary operator. Then,
\begin{enumerate*}[$(i)$]
\item $\seq{H_\alpha}_{\alpha}$ converges to~$H'$ w.r.t.~the pair~$\ttonde{\Phi_\alpha\circ U^{-1}, U(\KSD)}$;
\item if~$\seq{u_\alpha}_{\alpha\in\mcA}$, with~$u_\alpha\in H_\alpha$ strongly, resp.~weakly, converges to~$u\in H$ w.r.t.~the pair~$(\Phi_\alpha, \KSD)$, then~$\seq{u_\alpha}_{\alpha\in\mcA}$ strongly, resp.~weakly, converges to~$U(u)\in H'$ w.r.t.~the pair~$\ttonde{\Phi_\alpha\circ U^{-1}, U(\KSD)}$.
\end{enumerate*}
\begin{proof}
Straightforward.
\end{proof}
\end{lem}

We are now ready to show the main result of this section.

\begin{lem}[Domains' convergence]\label{l:KSPoisson}
Let~$(\mcX,\cdc)$ be a \TLDS, and~$\QP$ be a probability measure on~$\ttonde{\dUpsilon,\A_\mrmv(\msE)}$ satisfying Assumptions~\ref{ass:AC} and~\ref{ass:CWP}.
Further let~$\KSD\eqdef \CylQP{\QP}{\Dz}$ and $\Psi_n\colon \CylQP{\QP}{\Dz}\rar K_n$
be defined by
\begin{align*}
\Psi_n \colon u\mapsto (\Lb^*u)\circ \tr^n\comma
\end{align*}
where~$\tr^n$ is defined as in~\eqref{eq:TruncationDef}.
Then,
%\begin{enumerate*}[$(i)$]
%\item\label{i:l:KSPoisson:1} $H_n \xrightarrow{n\rar\infty} H$
%and
%\item\label{i:l:KSPoisson:2}
$K_n \xrightarrow{n\rar\infty} K$
%\end{enumerate*}
in the sense of Definition~\ref{d:KSSpaces}.

\begin{proof} %In both~\iref{i:l:KSPoisson:1} and~\iref{i:l:KSPoisson:2}, 
Firstly, recall that~$\Lb^* \Cyl{\Dz}\subset \preW(\msE)$ by Lemma~\ref{l:InfDiff}. Furthermore, it is readily checked that~$\Psi_n\ttonde{\CylQP{\QP}{\Dz}}\subset \spclass[\llb]{\preW(\msE)} \subset K_n$ for every~$n \geq 1$.
Secondly, note that the inclusion $\KSD \subset K$ is dense, by definition of~$K$.
Since the functions~$\mbff$ in the definition of~$u=F\circ \mbff^\trid\in \Cyl{\Dz}$ are $\msE$-eventually vanishing, and by definition of~$\Lb$, we have that
\begin{align}\label{eq:l:KSPoisson:0}
\nlim u\circ \Lb\circ\tr^n=u\circ \Lb \qquad \text{pointwise on~$\mbfX^\asym{\infty}_\locfin(\msE)$}\fstop
\end{align}
Then
\begin{align}\label{eq:l:KSPoisson:1}
\nlim \norm{\Psi_n(u)}_{H_\infty}^2=\nlim \int_\Upsilon \tabs{u\circ \Lb\circ \tr^n \circ \lb }^2 \diff\QP = \norm{u}_H^2
\end{align}
by Dominated Convergence.

Now, let~$\rep\Psi_n \colon \rep u\mapsto (\Lb^*\rep u)\circ \tr^n$. In order to compute
\begin{align}\label{eq:l:CdCCylinderTruncation:4.25}
\nlim \int_{X^\tym{\infty}} \rep\cdc_\ell^\asym{n}\ttonde{\rep\Psi_n(\rep u)} \diff\QP^\asym{\infty} \eqdef \nlim \int_{X^\tym{\infty}} \rep\cdc_\ell^\asym{n}(\rep u\circ \Lb \circ \tr^n) \diff\QP^\asym{\infty}\comma
\end{align}
note that, by~\eqref{eq:l:CdCCylinderTruncation:5},
\begin{equation}\label{eq:l:CdCCylinderTruncation:4.5}
\nlim \rep\cdc_\ell^\asym{n}(\rep u\circ \Lb \circ \tr^n)= \rep\cdc_\ell^\asym{\infty}(\rep u\circ \Lb) \diff\QP^\asym{\infty} \as{\QP^\asym{\infty}} \fstop
\end{equation}
Furthermore, for every~$\gamma\in\dUpsilon$ and every~$\mbfx\in\lb(\gamma)$,
\begin{align*}
\rep\cdc_\ell^\asym{n}(\rep u\circ \Lb \circ \tr^n)(\mbfx) \eqdef& \sum_{p=1}^n \ttonde{\rep\cdc^p_\ell (\rep u\circ\Lb\circ\tr^n)_{\mbfx, {p}}}(x_p) \leq \norm{\nabla F}^2_\infty \sum_{i,j}^{k,k}\sum_{p=1}^n \ell\ttonde{\cdc(f_i,f_j)}(x_p)
\\
\leq&\ \norm{\nabla F}^2_\infty \sum_{i,j}^{k,k} \int_X \ell\tabs{\cdc(f_i,f_j)} \diff\Lb(\mbfx)
\\
\leq&\ C_u \cdot \gamma\ttonde{ \cup_{i\leq k} \supp[f_i]}\comma
\end{align*}
for some constant~$C_u>0$, where we used that~$\rep\cdc_\ell$ is~$\mcL^\infty(\mssm)$-valued.
By definition of~$\Cyl{\Dz}$ we have that~$E_u\eqdef \cup_{i\leq k} \supp[f_i]\in\msE$. 
We have thus shown that there exists a constant~$C_u>0$ and a set~$E_u\in\msE$, so that
\begin{align*}
\sup_n \rep\cdc_\ell^\asym{n}(\rep u\circ \Lb \circ \tr^n)(\mbfx) \leq C_u \cdot \gamma E \comma \qquad \mbfx\in \lb(\gamma) \fstop
\end{align*}
Setting~$G_u\colon\mbfx\mapsto C_u \cdot \Lb(\mbfx) E$ we have therefore that~$\norm{G_u}_{L^1(\QP^\asym{\infty})}=C_u \cdot \mssm_\QP E_u<\infty$ by Assumption~\ref{ass:Mmu} (part of~\ref{ass:CWP}).

By~\eqref{eq:l:CdCCylinderTruncation:4.25},~\eqref{eq:l:CdCCylinderTruncation:4.5} and Dominated Convergence with Dominating function~$G_u$, and~\eqref{eq:l:CdCCylinderTruncation:1c}, we thus have
\begin{align*}
\nlim \int_{X^\tym{\infty}} \rep\cdc_\ell^\asym{n}\ttonde{\rep\Psi_n(\rep u)} \diff\QP^\asym{\infty} =& \nlim \int_{X^\tym{\infty}}\rep\cdc_\ell^\asym{n}(\rep u\circ \Lb \circ \tr^n) \diff\QP^\asym{\infty}
\\
=& \int_{X^\tym{\infty}}\rep\cdc_\ell^\asym{\infty}(\Lb^* \rep u) \diff\QP^\asym{\infty}
\\
=& \int_{X^\tym{\infty}} \SF{\dUpsilon}{\QP}(u)\circ \Lb \,\diff\QP^\asym{\infty} \comma
\end{align*}
and the conclusion follows.
\end{proof}
\end{lem}

\begin{prop}\label{p:LipIsometry}
Let~$(\mcX,\cdc)$ be a \TLDS, and~$\QP$ be a probability measure on $\ttonde{\dUpsilon,\A_\mrmv(\msE)}$ satisfying Assumptions~\ref{ass:AC} and~\ref{ass:CWP}.
The linear operator
\begin{equation*}
\Lb^*\colon (K,\norm{\emparg}_K)\longrar \ttonde{\Lb^* K, \norm{\emparg}_{K_\infty}}
\end{equation*}
is unitary.
\begin{proof}
Firstly, recall that~$\Lb^* \CylQP{\QP}{\Dz}\subset K_\infty$ by Lemma~\ref{l:InfDiff}.
It suffices to show that
\begin{align*}
\Lb^*\colon \ttonde{\CylQP{\QP}{\Dz}, \norm{\emparg}_K}\longrar \ttonde{\Lb^*\CylQP{\QP}{\Dz},\norm{\emparg}_{K_\infty}}
\end{align*}
is isometric. Since~$\Lb^*\colon L^2(\QP)\rar L^2(\QP^\asym{\infty})$ is clearly unitary with inverse~$\lb^*$, it suffices to show the intertwining property
\begin{align}\label{eq:l:LipIsometry:1}
\EE{\dUpsilon}{\QP}(u)=\EE{\asym{\infty}}{\llb}(\Lb^* u) \comma\qquad u\in \CylQP{\QP}{\Dz}\comma
\end{align}
which in turn follows from~\eqref{eq:l:CdCCylinderTruncation:1c}.
\end{proof}
\end{prop}

\begin{cor}\label{c:Isometry}
Let~$(\mcX,\cdc)$ be a \TLDS, and~$\QP$ be a probability measure on~$\ttonde{\dUpsilon,\A_\mrmv(\msE)}$ satisfying Assumptions~\ref{ass:AC} and~\ref{ass:CWP}.
Equation~\eqref{eq:l:CdCCylinderTruncation:1c}, resp.~\eqref{eq:l:LipIsometry:1}, holds for all~$u$ in~$\dom{\SF{\dUpsilon}{\QP}}$, resp.~$\dom{\EE{\dUpsilon}{\QP}}$.
\end{cor}

\begin{cor}[Finite-dimensional approximations]\label{c:KSUpsilon}
Let~$(\mcX,\cdc)$ be a \TLDS, and~$\QP$ be a probability measure on~$\ttonde{\dUpsilon,\A_\mrmv(\msE)}$ satisfying Assumptions~\ref{ass:AC} and~\ref{ass:CWP}.
Further let~$\KSD$ and~$\Psi_n$ be defined as in Lemma~\ref{l:KSPoisson}. Then,
\begin{equation*}
\text{$K_n$ converges to~$\Lb^*K$ in the sense of Definition~\ref{d:KSSpaces}}
\end{equation*}
w.r.t.~the pair~$\ttonde{\Psi_n\circ (\Lb^*)^{-1}, \Lb^*(\KSD)}$.
\begin{proof}
It suffices to combine Lemma~\ref{l:KSPoisson} with Proposition~\ref{p:LipIsometry} and Lemma~\ref{l:KSUnitary}.
\end{proof}
\end{cor}

\begin{figure}[htb!]
\centerfloat
{
\begin{tikzcd}[column sep=tiny]
& K_\infty^{\lb_1} \arrow[ddl, hook', "?"' near start, dotted] & L^2(\QP) \arrow[ddrr, leftrightarrow, "\lb_2^*" near start] & K_\infty^{\lb_2} \arrow[ddr, hook, "?" near start, dotted]
\\
&& K \arrow[draw=none]{u}[sloped, auto=false]{\subset} \arrow[dddl, dashed, bend right=15, leftarrow] \arrow[dddr, dashed, bend left=15, leftarrow]
\\
L^2\ttonde{\QP^{\asym{\infty}, \lb_1}} \arrow[uurr, leftrightarrow, "\lb_1^*" near end] \arrow[draw=none]{r}[sloped,auto=false]{\supset} & \ttonde{\preW, (\EE{\asym{\infty}}{\lb_1})_1^{1/2}} \arrow[uu, hook, crossing over, "\co"] & & \ttonde{\preW, (\EE{\asym{\infty}}{\lb_2})_1^{1/2}} \arrow[uu, hook', crossing over, "\co"'] & L^2\ttonde{\QP^{\asym{\infty}, \lb_2}} \arrow[draw=none]{l}[sloped,auto=false]{\subset}
\\
&& \ttonde{\Cyl{\Dz},(\EE{\dUpsilon}{\QP})_1^{1/2}} \arrow[uu, hook, "\co"] \arrow[ul, "\Lb^*"', hook, crossing over] \arrow[ur, "\Lb^*", hook', crossing over] \arrow[dl, "\Psi_n"'] \arrow[dr, "\Psi_n"] \arrow[ddl, "\Psi_{n-1}"] \arrow[ddr, "\Psi_{n-1}"']
\\
& K_n^{\lb_1} && K_n^{\lb_2}
\\
& K_{n-1}^{\lb_1} \arrow[u, no head, dashed] && K_{n-1}^{\lb_2} \arrow[u, no head, dashed] 
\end{tikzcd}
}
\caption{
A description of the constructions in this section for two different labeling maps~$\lb_1$,~$\lb_2$.
For simplicity of notation, in the diagram we wright~$\QP^{\asym{\infty},\lb_i}$ in place of~$(\lb_i)_\pfwd \QP$, and~$\EE{\asym{\infty}}{\lb_i}$ in place of~$\EE{\asym{\infty}}{(\lb_i)_\pfwd\QP}$ for $i=1,2$.
The dashed arrows represent Kuwae--Shioya convergence.
The dotted arrows represent the lift to the completion of the identical embedding~$\ttonde{\preW, (\EE{\asym{\infty}}{\lb_i})_1^{1/2}} \subset L^2\ttonde{\QP^{\asym{\infty}, \lb_i}}$.
Question marks refer to the potential lack of injectivity of the corresponding arrows.
Arrows labeled by~`$\co$' denote completion embeddings.
}
\label{fig:KS}
\end{figure}

\section{Geometric structure: extended metric local diffusion spaces}\label{s:Geometry}
In this section we add further structure to the base spaces, by considering a distance function~$\mssd$, suitably compatible with the topology~$\T$ and with the localizing ring~$\msE$ of a topological local structure~$\mcX$.
For the sake of this introductory discussion, let us assume for simplicity that~$\T$ is the topology induced by a separable and complete distance~$\mssd\colon X^\tym{2}\rar[0,\infty)$, and that~$\msE=\Ed$ is the localizing ring of all bounded Borel sets with finite $\mssm$-measure.

For~$i=1,2$ let~$\pr^i\colon X^{\times 2}\rar X$ denote the projection to the~$i^\text{th}$ coordinate.
For~$\gamma,\eta\in \dUpsilon$, let~$\Cpl(\gamma,\eta)\subset \Meas(X^{\tym{2}},\A^{\hotimes 2})$ be the set of all couplings of~$\gamma$ and~$\eta$, viz.
\begin{align*}
\Cpl(\gamma,\eta)\eqdef \set{\cpl\in \Meas(X^{\tym{2}},\A^{\hotimes 2}) \colon \pr^1_\pfwd \cpl =\gamma \comma \pr^2_\pfwd \cpl=\eta} \fstop
\end{align*}
Similarly to the case of Wasserstein spaces, any distance function~$\mssd$ on~$X$ induces a distance on configuration spaces, in the sense of the following definition.

\begin{defs}[$L^2$-transportation distance]
The \emph{$L^2$-transportation} (\emph{extended}) \emph{distance} on~$\dUpsilon$ is defined as
\begin{align}\label{eq:d:W2Upsilon}
\mssd_\dUpsilon(\gamma,\eta)\eqdef \inf_{\cpl\in\Cpl(\gamma,\eta)} \tonde{\int_{X^{\times 2}} \mssd^2(x,y) \diff\cpl(x,y)}^{1/2}\comma \qquad \inf{\emp}=+\infty \fstop
\end{align}
\end{defs}

As shown by several works about configuration spaces over Riemannian manifolds, e.g.~\cite{RoeSch99,ErbHue15}, the $L^2$-transportation distance on~$\dUpsilon$ induced by the Riemannian distance on the underlying manifold is a natural object.
In particular, the same metric properties of the underlying manifold, such as completeness, hold for the corresponding multiple configuration space as well.

In fact~$\mssd_\dUpsilon$ is an extended distance, attaining the value~$+\infty$ on a set of positive~$\QP^\otym{2}$-measure in~$\dUpsilon^\tym{2}$, making~$\ttonde{\dUpsilon,\mssd_\dUpsilon, \QP}$ into an extended metric measure space.

\subsection{Base spaces}
We collect here some definitions detailing the interplay between the metric and diffusion-space structures on base spaces.
We refer the reader to~\cite{LzDSSuz20} for a more thorough account of this interplay, and for the proof of some of the results listed below.

For the sake of simplicity, we state most of our definitions for \emph{extended} metric diffusion spaces (see Definition~\ref{d:EMLDS} below) in such a way to include both the base spaces, and the configuration spaces endowed with the (extended) $L^2$-transportation distance.
In most of the assertions, we shall however restrict our attention to the case when the base spaces are endowed with a distance, rather than with an extended distance.

\subsubsection{Extended metric local structures}
Let~$X$ be any non-empty set.

\paragraph{Distances}
A function~$\mssd\colon X^\tym{2}\rar [0,\infty]$ is an \emph{extended pseudo-distance} if it is symmetric and satisfying the triangle inequality. Any such~$\mssd$ is called a \emph{pseudo-distance} if it is finite, i.e.~$\mssd\colon X^{\tym{2}}\rar [0,\infty)$; an \emph{extended distance} if it does not vanish outside the diagonal in~$X^{\tym{2}}$, i.e.~$\mssd(x,y)=0$ iff~$x=y$; a \emph{distance} if it is both finite and not vanishing outside the diagonal.

Let~$x_0\in X$ and~$r\in [0,\infty]$. We write~$\mssd_{x_0}\eqdef \mssd(\emparg, x_0)\colon X\rar [0,\infty]$ and
\begin{align*}
B^\mssd_r(x_0)\eqdef \set{x\in X: \mssd(x,x_0)<r}\fstop
\end{align*}
Note that, if~$\mssd$ in an extended pseudo-metric, then, in general,~$B^\mssd_\infty(x_0)\subsetneq X$. 
We say that an extended metric space is, \emph{complete}, resp.\ \emph{length}, \emph{geodesic}, if each~$B^\mssd_\infty(x)$ is complete, resp.\ length, geodesic, for each~$x\in X$.

We denote by~$\Bdd{\mssd}$ the family of all $\mssd$-bounded sets.
Finally, for every~$A\subset X$, set
\begin{align}\label{eq:Point-to-SetDistance}
\mssd(\emparg, A)\eqdef \inf_{x\in A} \mssd_x \colon X\longrar [0,\infty] \fstop
\end{align}

\paragraph{Curves}
Let~$I\eqdef [0,1]$ be the standard unit interval, and let~$p\in [1,\infty]$.
A curve~$x_\emparg\colon I\to (X,\mssd)$ is $p$-\emph{absolutely continuous} if there exists~$g\in L^p(I)$ such that
\begin{align}\label{eq:MetricDerivative}
\mssd(x_s,x_t)\leq \int_s^t g(r)\diff r\comma\qquad s,t\in I \comma s\leq t\fstop
\end{align}
We denote by~$\AC^p(I;X)$ the family of all $p$-absolutely continuous curves.
If~$\seq{x_t}_{t\in I}$ is $p$-absolutely continuous, we denote by~$\abs{\dot x}\colon I\to [0,\infty]$ its \emph{metric derivative}
\begin{equation*}
\abs{\dot x}_t\eqdef \lim_{h \to 0}\frac{\mssd(x_{t+h}, x_t)}{\abs{h}}\comma \qquad t\in (0,1)
\end{equation*}
well-defined for a.e.~$t\in I$ and the a.e.-minimal~$g$ in~\eqref{eq:MetricDerivative}, e.g.~\cite[Thm.~1.1.2]{AmbGigSav08}.

\paragraph{Lipschitz functions} A function~$f\colon X\rar \R$ is $\mssd$-Lipschitz if there exists a constant~$L>0$ so that
\begin{align}\label{eq:Lipschitz}
\abs{f(x)-f(y)}\leq L\, \mssd(x,y) \comma \qquad x,y\in X \fstop
\end{align}
The infimal constant~$L$ such that~\eqref{eq:Lipschitz} holds is the (global) \emph{Lipschitz constant of~$f$}, denoted by~$\Li[d]{f}$. We write~$\Lip(\mssd)$, resp.~$\bLip(\mssd)$ for the family of all finite, resp.\ bounded, $\mssd$-Lipschitz functions on~$(X,\mssd)$.

In the following, we shall be concerned with triples~$(X,\T,\mssd)$, where~$(X,\T)$ is a (completely regular Hausdorff) topological space, and~$\mssd$ is a (possibly extended) metric on~$X$ not necessarily compatible with~$\T$.
We caution that, on any such space, the standard intuition of $\mssd$-Lipschitz function as an amenable class of functions fails.
In particular, let~$f\colon (X,\T,\mssd)\rar \Rex$ be $\mssd$-Lipschitz with~$\Dom(f)\neq \emp$.
In general, $f$ is \emph{neither} finite, nor $\T$-continuous, nor $\Bo{\T}$-measurable.
For a given $\sigma$-algebra~$\A$, resp.\ a topology~$\T$, on~$X$, this motivates to set
\begin{align*}
\bLip(\mssd,\A)\eqdef \bLip(\mssd)\cap \Sb(X)\qquad \textrm{and}\qquad \bLip(\mssd,\T)\eqdef \bLip(\mssd)\cap \Cb(\T)\fstop
\end{align*}
Let the analogous definitions for unbounded Lipschitz functions be given. Finally, set
\begin{align*}
\Lipu(\mssd)\eqdef \set{f\in \Lip(\mssd): \Li[\mssd]{f}\leq 1} \comma
\end{align*}
and let the analogous definitions for continuous, measurable, or bounded Lipschitz functions be given.

The next lemma is an easy adaptation of McShane~\cite{McS34} to extended metric spaces.

\begin{lem}[Constrained McShane extensions,~{\cite[Lem.~2.1]{LzDSSuz20}}]\label{l:McShane}
Let~$(X,\mssd)$ be an extended metric space.
Fix~$A\subset X$,~$A\neq \emp$, and let~$\rep f\colon A\rar \R$ be a function in~$\bLip(A,\mssd)$. Further set
\begin{equation}\label{eq:McShane}
\begin{aligned}
\overline{f}\colon x&\longmapsto \sup_A \rep f\wedge \inf_{a\in A} \ttonde{\rep f(a)+\Li[\mssd]{\rep f}\,\mssd(x,a)} \comma
\\
\underline{f}\colon x&\longmapsto \inf_A \rep f \vee \sup_{a\in A} \ttonde{\rep f(a)-\Li[\mssd]{\rep f}\,\mssd(x,a)} \fstop
\end{aligned}
\end{equation}

Then,
\begin{enumerate}[$(i)$]
\item\label{i:l:McShane:1} $\underline{f}=\rep f=\overline{f}$ on~$A$ and~$\inf_A \rep f\leq \underline{f}\leq \overline{f}\leq \sup_A \rep f$ on~$X$;

\item\label{i:l:McShane:2} $\underline{f}$, $\overline{f}\in \bLip[\mssd]$ with~$\Li[\mssd]{\underline{f}}=\Li[\mssd]{\overline{f}}=\Li[\mssd]{\rep f}$;

\item\label{i:l:McShane:3} $\underline{f}$, resp.~$\overline{f}$, is the minimal, resp.\ maximal, function satisfying~\iref{i:l:McShane:1}-\iref{i:l:McShane:2}, that is, for every~$\rep g\in \bLip(\mssd)$ with~$\inf_A \rep f\leq \rep g\leq \sup_A \rep f$, $\rep g\restr_A=\rep f$ on~$A$ and~$\Li[\mssd]{\rep g}= \Li[\mssd]{\rep f}$, it holds that~$\underline{f}\leq \rep g \leq \overline{f}$.
\end{enumerate}
\end{lem}

\paragraph{Slopes}
Let~$(X,\mssd)$ be an extended metric space.
The \emph{slope} of $u\colon X\rar \overline{\R}$ is defined as
\begin{align}\label{eq:SlopeDef}
\slo[\mssd]{f}(x)\eqdef \limsup_{y\to x} \frac{\abs{f(x)-f(y)}}{\mssd(x,y)}
\end{align}
where, conventionally,~$\slo[\mssd]{f}(x)=0$ whenever~$x$ is isolated.
The specification of the distance is omitted whenever apparent from context.

In the following, we shall make use of the following facts on slopes, occasionally without explicit mention.

\begin{rem}\label{r:Slopes}
\begin{enumerate*}[$(a)$]
\item It is shown in~\cite[Lem.~2.6]{AmbGigSav14} that~$\slo{f}$ is $\Bo{\T}^*$-measurable whenever~$f$ is $\Bo{\T}$-measur\-able.

\item\label{i:r:Slopes:1} As a consequence of~\cite[Eqn.~(2.9)]{AmbGigSav14}, the slope is $\mssd_\T$-local in the following sense: whenever~$f$ is constant on a $\mssd$-open set~$A$, then~$\slo{f}\equiv 0$ on~$A$.
Since $\mssd$ is $\T$-l.s.c.\ by assumption, $\T_\mssd$ is finer than~$\T$, hence the slope is $\T$-local as well.
\end{enumerate*}
\end{rem}

\paragraph{Extended metric local structures}
Let~$(X,\T)$ be a Hausdorff topological space. A family~$\UP$ of pseudo-distances~$\mssd_i\colon X^\tym{2}\to [0,\infty]$ is a \emph{uniformity} (\emph{of pseudo-distances}) if:
\begin{enumerate*}[$(a)$]
\item it is directed, i.e., $\mssd_1\vee \mssd_2\in\UP$ for every~$\mssd_1,\mssd_2\in \UP$;
and
\item it is order-closed, i.e., $\mssd_2\in \UP$ and~$\mssd_1\leq \mssd_2$ implies~$\mssd_1\in \UP$ for every pseudo-distance~$\mssd_1$ on~$X$.
\end{enumerate*}
A uniformity is: \emph{bounded} if every~$\mssd\in \UP$ is bounded; \emph{Hausdorff} if it separates points.

The next definition is taken from~\cite[Dfn.~2.2]{LzDSSuz20}.
It is a reformulation of~\cite[Dfn.~4.1]{AmbErbSav16}; see~\cite[Rmk.~2.3]{LzDSSuz20}.

\begin{defs}[Extended metric-topological spaces]\label{d:AES}
Let~$(X,\T)$ be a topological space, $\mssd\colon X^{\tym{2}}\rar [0,\infty]$ be an extended distance on~$X$.
The triple~$(X,\T,\mssd)$ is an \emph{extended metric-topological space} if there exists a (Hausdorff) uniformity~$\UP$ of $\T$-continuous pseudo-distances~$\mssd'\colon X^\tym{2}\rar [0,\infty)$, so that~$\UP$ generates~$\T$ and
\begin{align}\label{eq:d:AES}
\mssd=\sup\set{\mssd':\mssd'\in\UP}\fstop
\end{align}
\end{defs}

\begin{defs}[Extended metric local structure]\label{d:EMLS}
Let~$\mcX$ be a topological local structure as in Definition~\ref{d:TLS}, and~$\mssd\colon X^\tym{2}\rar[0,\infty]$ be an extended distance. We say that~$(\mcX,\mssd)$ is an \emph{extended metric local structure}, if
\begin{enumerate}[$(a)$]
\item\label{i:d:EMLS:0} $\mcX$ is a topological local structure in the sense of Definition~\ref{d:TLS};

\item\label{i:d:EMLS:1} the space~$(X,\T,\mssd)$ is a \emph{complete} extended metric-topological space;% in the sense of Definition~\ref{d:AES};

\item\label{i:d:EMLS:2} $\msE=\msE_\mssd\eqdef \Bdd{\mssd}\cap \A$ is the localizing ring of $\A$-measurable $\mssd$-bounded sets (in particular~$\mssm$ is finite on~$\Ed$).
\end{enumerate}

If~$\mssd$ is finite, then it metrizes~$\T$,~\iref{i:d:EMLS:1} reduces to the requirement that~$(X,\mssd)$ be a complete metric space, and we say that~$(\mcX,\mssd)$ is a \emph{metric local structure}.
\end{defs}

\begin{rem}\label{r:EMetTop}
\begin{enumerate*}[$(a)$]
\item\label{i:r:EMetTop:1} As noted in~\cite[Eqn.~(4.4)]{AmbErbSav16}, if~$(X,\T,\mssd)$ is an extended metric-topological space, then $\mssd$-convergence implies $\T$-convergence.

\item\label{i:r:EMetTop:2}
If~$(\mcX,\mssd)$ is a \emph{metric} local structure, then~$\Cz(\Ed)$ is a convergence-determining class for the narrow convergence of probability measures on~$(X,\A)$.
\end{enumerate*}
Furthermore, the space
\begin{equation*}
\Cz(\Ed)\restr_E\eqdef \set{f\restr_E:f\in\Cz(\Ed)} \subset \Cb(E,\T_E)
\end{equation*}
is a convergence-determining class for the narrow convergence of probability measures on~$(E,\A_E)$, for every~$E\in\Ed$.
Both assertions hold as a consequence of~\cite[Cor.~7]{BloKou10}.
\end{rem}

The Definition~\ref{d:TLDS} of \TLDS can be now combined with that of a metric local structure above.

\begin{defs}[Metric local diffusion spaces]\label{d:EMLDS}
An (\emph{extended}) \emph{metric local diffusion space} (in \linebreak short: \parEMLDS) is a triple $(\mcX,\cdc,\mssd)$ so that
\begin{enumerate}[$(a)$]
\item $(\mcX,\mssd)$ is an (extended) metric local structure;
\item $(\mcX,\cdc)$ is a \TLDS.
\end{enumerate}
\end{defs}

\subsubsection{Intrinsic distances and maximal functions}
In this section, we recall some basic properties of intrinsic distances and maximal functions of Dirichlet spaces.
We assume the reader to be familiar with `quasi-notions' and broad local spaces in the sense of Dirichlet forms theory, see e.g.~\cite{Kuw98, LzDSSuz20}.

Let~$(\mcX,\cdc)$ be a \TLDS, and~$\ttonde{\EE{X}{\mssm},\dom{\EE{X}{\mssm}}}$ be the (quasi-regular strongly local) Dirichlet form with square field operator~$\SF{X}{\mssm}$.
We denote by~$\domloc{\EE{X}{\mssm}}$ its broad local domain, and consider the non-relabeled extension of~$\cdc$ to~$\domloc{\EE{X}{\mssm}}$; see e.g.~\cite[\S2.4]{LzDSSuz20}.
The \emph{broad local space of functions with $\mssm$-uniformly bounded $\EE{X}{\mssm}$-energy} is the space
\begin{align*}
\DzLoc{\mssm}\eqdef \set{f\in \domloc{\EE{X}{\mssm}}: \cdc(f)\leq 1 \as{\mssm}}\fstop
\end{align*}
Let us additionally set
\begin{align*}
\DzLocB{\mssm,\T}\eqdef \DzLoc{\mssm}\cap \Cb(\T)\qquad \text{and} \qquad \DzB{\mssm,\T}\eqdef \DzLocB{\mssm,\T}\cap \dom{\EE{X}{\mssm}} \fstop
\end{align*}

\begin{defs}[Intrinsic distance]
Let~$(\mcX,\cdc)$ be a \TLDS. The \emph{intrinsic distance associated to~$\EE{X}{\mssm}$} is the extended pseudo-distance
\begin{align}\label{eq:d:IntrinsicDist}
\mssd_\mssm(x,y)\eqdef \sup_{f\in \DzLocB{\mssm,\T}} \abs{f(x)-f(y)} \fstop
\end{align}
\end{defs}
In the case when~$\mcX$ is a locally compact Polish Radon measure space and~$\ttonde{\EE{X}{\mssm},\dom{\EE{X}{\mssm}}}$ is a regular strongly local Dirichlet form,~\eqref{eq:d:IntrinsicDist} coincides with the standard definition of intrinsic distance, see~\cite[Prop.~2.31]{LzDSSuz20}. For more information on intrinsic distances, we refer the reader to~\cite[\S{2.6}]{LzDSSuz20}.

\paragraph{Maximal functions}
Together with intrinsic distances, a second class of functions playing a role in the analysis of Dirichlet spaces is that of \emph{maximal functions}.

\begin{defs}[Maximal functions]\label{d:MaximalFunction}
For each~$A\in \A$ there exists an $\mssm$-a.e.\ unique function $\hr{\mssm, A}\colon X\to[0,\infty]$ so that, for each~$r>0$,
\begin{align*}
\hr{\mssm, A}\wedge r = \mssm\text{-}\esssup \set{f: f\in \DzLocB{\mssm}\comma f\equiv 0 \as{\mssm} \text{ on~$A$}\comma f\leq r \as{\mssm}} \fstop
\end{align*}
\end{defs}

Maximal functions were introduced by M.~Hino and J.~Ram\'{i}rez in~\cite{HinRam03}, in order to study the short-time asymptotics of the heat kernel associated to a general Dirichlet spaces.
Their main result was further generalized by T.~Ariyoshi and M.~Hino in~\cite{AriHin05}. An adaptation of this latter result to our setting reads as follows.

\begin{thm}[Ariyoshi--Hino~{\cite[Thm.~5.2(i)]{AriHin05}}]
Let~$(\mcX,\cdc)$ be a \TLDS, and~$\nu\sim \mssm$ be any probability measure on~$(X,\A)$.
Further let~$A\in \A$ be so that~$\mssm A\in(0,\infty)$, and set~$u_t\eqdef -2t \log \TT{X}{\mssm}_t\car_A$.
Then, $\nu$-$\lim_{t\downarrow 0} u_t\cdot \car_{\set{u_t <\infty}} = \hr{\mssm, A}^2$.%;
\end{thm}

\subsubsection{Rademacher and Sobolev-to-Lipschitz properties}\label{sss:RadStoL}
In this section, we focus on the interplay between the diffusion-space structure and the metric structure of an \parEMLDS.
We refer the reader to~\cite{LzDSSuz20} for a detailed discussion on the subject.

\begin{defs}[Rademacher and Sobolev-to-Lipschitz properties]
We say that an \EMLDS \linebreak $(\mcX,\cdc,\mssd)$ has:
\begin{itemize}
\item[($\Rad{\mssd}{\mssm}$)] the \emph{Rademacher property} if, whenever~$\rep f\in \Lipu(\mssd,\A)$, then~$f\in \DzLoc{\mssm}$;
%\begin{align}\label{eq:Rad}
%\tag{$\Rad{\mssd}{\mu}$} \rep f\in \Lipu(\mssd,\A) \qquad \implies \qquad f\in \DzLoc{\mssm}\semicolon
%\end{align}
%
\item[($\dRad{\mssd}{\mssm}$)] the \emph{distance-Rademacher property} if~$\mssd\leq \mssd_\mssm$;
%\begin{align}\label{eq:dRad}
%\tag{$\dRad{\mssd}{\mu}$} \mssd\leq \mssd_\mssm\fstop
%\end{align}
\item[($\ScL{\mssm}{\T}{\mssd}$)] the \emph{Sobolev--to--continuous-Lipschitz property} if each~$f\in\DzLoc{\mssm}$ has an $\mssm$-representative $\rep f\in\Lip^1(\mssd,\T)$;

\item[($\SL{\mssm}{\mssd}$)] the \emph{Sobolev--to--Lipschitz property} if each~$f\in\DzLoc{\mssm}$ has an $\mssm$-representat\-ive $\rep f\in\Lip^1(\mssd,\A)$;

\item[($\dcSL{\mssd}{\mssm}{\mssd}$)] the \emph{$\mssd$-continuous-Sobolev--to--Lipschitz property} if each~$f\in \DzLoc{\mssm}$ having a $\mssd$-continuous $\A$-measurable representative~$\rep f$ also has a representative $\reptwo f\in\Lip^1(\mssd,\A^\mssm)$ (possibly,~$\reptwo f\neq \rep f$);

\item[($\cSL{\T}{\mssm}{\mssd}$)] the \emph{continuous-Sobolev--to--Lipschitz property} if each~$f\in \DzLoc{\mssm,\T}$ satisfies $f\in\Lip^1(\mssd,\T)$;

\item[($\dSL{\mssd}{\mssm}$)] the \emph{distance Sobolev-to-Lipschitz property} if~$\mssd\geq \mssd_\mssm$.
\end{itemize}
\end{defs}

We refer the reader to~\cite[Rmk.s~3.2, 4.3]{LzDSSuz20} for comments on the terminology and to~\cite[Lem.~3.6, Prop.~4.2]{LzDSSuz20} for the interplay of all such properties. In the setting of~\EMLDS's, they reduce to the following scheme:

\begin{equation}\label{eq:EquivalenceRadStoL}
\begin{tikzcd}
& (\Rad{\mssd}{\mssm}) \arrow[r, Rightarrow] & (\dRad{\mssd}{\mssm}) & & \text{\cite[Lem.~3.6]{LzDSSuz20}}
\\
(\ScL{\mssm}{\T}{\mssd}) \arrow[r, Rightarrow] & (\SL{\mssm}{\mssd}) \arrow[r, Rightarrow] & (\cSL{\mssm}{\T}{\mssd}) \arrow[r, Leftrightarrow] & (\dSL{\mssd}{\mssm}) &  \text{\cite[Prop.~4.2]{LzDSSuz20}}
\end{tikzcd}
\end{equation}

The delicate interplay between the Rademacher and Sobolev-to-Lipschitz properties and maximal functions was investigated in the setting of quasi-regular Dirichlet spaces by the authors, in~\cite{LzDSSuz20}.

\begin{rem}[About~$(\dcSL{\mssd}{\mssm}{\mssd})$]\label{r:dcSL}
Let us note that, whereas both~$(\dcSL{\mssd}{\mssm}{\mssd})$ and $(\cSL{\T}{\mssm}{\mssd})$ are implied by~$(\SL{\mssm}{\mssd})$ and coincide on metric spaces, they do \emph{not} ---~at least in principle~--- imply each other on extended metric spaces.
In particular, while the $\mssd$-Lipschitz representative in~$(\cSL{\T}{\mssm}{\mssd})$ is taken to coincide with the given $\T$-continuous one, it is important in the definition of~$(\dcSL{\mssd}{\mssm}{\mssd})$ to allow for the $\mssd$-Lipschitz representative to be different from the $\mssd$-continuous one, and for the former to be only~$\A^\mssm$-measurable, rather than $\A$-measurable.
This will be clarified in the case of configuration spaces in Remark~\ref{r:RoeSch99SL} below.
\end{rem}

\begin{lem}\label{l:RadCompleteness}
Let~$(\mcX,\cdc,\mssd)$ be an \MLDS satisfying~$(\dRad{\mssd}{\mssm})$. Then,~$(X,\T,\mssd_\mssm)$ is a complete extended metric-topological (Radon measure) space in the sense of Definition~\ref{d:AES}.
\begin{proof}
See~\cite[Prop.~3.7]{LzDSSuz20}.
\end{proof}
\end{lem}

\paragraph{Locality}
Let~$(\mcX,\cdc,\mssd)$,~$(\mcX',\cdc',\mssd)$ be \EMLDS's, with same underlying topological measurable space~$(X,\T,\A)$, equivalent measures~$\mssm\sim\mssm'$, and same extended distance~$\mssd$.
Let~$(\EE{X}{\mssm},\dom{\EE{X}{\mssm}})$, resp.~$(\EE{X}{\mssm'},\dom{\EE{X}{\mssm'}})$, be the corresponding quasi-regular strongly local Dirichlet forms, defined on~$L^2(\mssm)$, resp.~$L^2(\mssm')$, and admitting carr\'e du champ operator~$\SF{X}{\mssm}$, resp.~$\SF{X}{\mssm'}$.
We write
\begin{equation*}
\ttonde{\SF{X}{\mssm}, \msA}\leq \ttonde{\SF{X}{\mssm'}, \msA'}
\end{equation*}
to indicate that~$\msA\supset \msA'$ and~$\SF{X}{\mssm'}\geq \SF{X}{\mssm}$ on~$ \msA'$, and analogously for the opposite inequality.
Let~$(\mssP)$ denote either~$(\Rad{\mssd}{\mssm})$, $(\ScL{\mssm}{\T}{\mssd})$, $(\SL{\mssm}{\mssd})$, or~$(\cSL{\T}{\mssm}{\mssd})$.
Note that~$\Lipu(\mssd,\A)$ and~$\Cb(\T)$ do not depend on either~$(\mcE,\dom{\mcE})$ or~$\mssm$.
Furthemore, since~$\mssm\sim\mssm'$, we have that~$L^\infty(\mssm)= L^\infty(\mssm')$ as Banach spaces.
As a consequence of the facts above,~$(\mssP)$ is a \emph{local} property in the sense of the following Proposition, a proof of which is straightforward.

\begin{prop}[Weighted spaces]\label{p:Locality}
Retain the notation above. Then,
\begin{enumerate}[$(i)$]
\item\label{i:p:Locality:1} if $\ttonde{\SF{X}{\mssm}, \DzLocB{\mssm}}\leq \ttonde{\SF{X}{\mssm'},\DzLocBprime{\mssm'}}$ and~$(\ScL{\mssm}{\T}{\mssd})$, resp.~$(\SL{\mssm}{\mssd})$,~$(\cSL{\T}{\mssm}{\mssd})$ holds,
then $(\ScL{\mssm'}{\T}{\mssd})$, resp.~$(\SL{\mssm'}{\mssd})$, or~$(\cSL{\T}{\mssm'}{\mssd})$ holds as well;

\item\label{i:p:Locality:2} if $\ttonde{\SF{X}{\mssm}, \DzLocB{\mssm}}\geq \ttonde{\SF{X}{\mssm'},\DzLocBprime{\mssm'}}$ and~$(\Rad{\mssd}{\mssm})$ holds, then~$(\Rad{\mssd}{\mssm'})$ holds as well.
\end{enumerate}
\end{prop}

Let us now show under some additional assumptions for~$\mssm'$ and~$\mssm$, that it suffices to verify the assumptions in Proposition~\ref{p:Locality} only on a common core.
We spell this out only for the assumption in Proposition~\ref{p:Locality}\ref{i:p:Locality:1}, which is the only one we shall need in the following.
The easy adaptation for the case of Proposition~\ref{p:Locality}\ref{i:p:Locality:2} is left to the reader.
We say that two sets~$A_1$, $A_2\subset X$ are \emph{well-separated} if~$\mssd(A_1,A_2)>0$.
An increasing sequence of sets~$\seq{A_k}_k$ is well-separated if~$A_k$ and~$A_{k+1}^\complement$ are well-separated for every~$k$.
 
\begin{prop}[Comparison of square fields]\label{p:LocalityProbab}
Let~$(\mcX,\cdc,\mssd)$,~$(\mcX',\cdc',\mssd)$ be \MLDS's with same underlying topological measurable space~$(X,\T,\A)$ and same distance~$\mssd$, and assume that:
\begin{enumerate}[$(a)$]
\item\label{i:p:LocalityProbab:1} $\mssm$, $\mssm'$ are finite measures;
\item\label{i:p:LocalityProbab:2} for every pair of well-separated $\T$-open sets~$U_1, U_2$, there exists~$\varrho\eqdef\varrho_{U_1,U_2}\in \dom{\EE{X}{\mssm}}$ with~$\varrho\equiv 1$ $\mssm$-a.e.\ on~$U_1$, $\varrho\equiv 0$ $\mssm$-a.e.\ on~$U_2$, and~$\SF{X}{\mssm}(\varrho)\leq C$ for some finite constant~$C\eqdef C_{U_1,U_2}>0$;
\item\label{i:p:LocalityProbab:3} $\mssm'=g\mssm$ for some~$g\in L^1(\mssm)$, and there exists an increasing well-separated $\EE{X}{\mssm}$-nest~$\seq{F_k}_k$, and a sequence of constants~$\seq{a_k}_k$ with~$a_k>0$, such that
\begin{align*}
0<a_k\leq g \leq a_k^{-1}<+\infty \quad \as{\mssm} \text{ on} \quad \inter_\T(F_k) \semicolon
\end{align*}
\item\label{i:p:LocalityProbab:4} there exists~$\Dz$ a core for both~$\ttonde{\EE{X}{\mssm},\dom{\EE{X}{\mssm}}}$ on~$L^2(\mssm)$ and~$\ttonde{\EE{X}{\mssm'},\dom{\EE{X}{\mssm'}}}$ on~$L^2(\mssm')$, additionally so that~$\SF{X}{\mssm}\leq \SF{X}{\mssm'}$ on~$\Dz$.
\end{enumerate}
Then,~$\ttonde{\SF{X}{\mssm}, \DzLocB{\mssm}}\leq \ttonde{\SF{X}{\mssm'},\DzLocBprime{\mssm'}}$.

\begin{proof}
Since~$\mssm,\mssm'$ are finite measures, it follows from~\cite[Prop.~2.26]{LzDSSuz20} that
\begin{equation}\label{eq:p:LocalityProbab:1}
\DzLocB{\mssm}=\DzB{\mssm}\eqdef \DzLocB{\mssm}\cap \dom{\EE{X}{\mssm}}\comma
\end{equation}
and analogously for~$\mssm'$.

Let~$\varrho_k$ be the cut-off function of the well-separated pair of open sets~$\inter_\T(F_k)$,~$F_{k+1}^\complement$ given by~\ref{i:p:LocalityProbab:2}, with $\SF{X}{\mssm}(\varrho_k)\leq C_k$.
Without loss of generality, by the Markov property of~$\ttonde{\EE{X}{\mssm},\dom{\EE{X}{\mssm}}}$, we may and shall assume that~$0\leq \varrho_k\leq 1$.

Now, let~$f\in \DzLocBprime{\mssm'}= \DzBprime{\mssm'}$, and~$\seq{f_n}_n\subset \Dz$ be $\EE{X}{\mssm'}_1$-converging to~$f$.
By the Leibniz rule for~$\SF{X}{\mssm}$, we have that~$\seq{f_n\varrho_k}_n\subset \dom{\EE{X}{\mssm}}_b$ satisfies
\begin{align*}
\SF{X}{\mssm}(f_n\varrho_k-f_m\varrho_k) =&\ \abs{f_n-f_m}^2 \SF{X}{\mssm}(\varrho_k) + \abs{\varrho_k}^2 \SF{X}{\mssm}(f_n-f_m) 
\\
\leq&\ C_k\abs{f_n-f_m}^2+\SF{X}{\mssm'}(f_n-f_m)\fstop
\end{align*}
Since~$\varrho_k$ is supported on~$F_{k+1}$, and since~$\inter_\T(F_{k+1})$ is $\T$-open, hence in particular $\EE{X}{\mssm}$-quasi-open, by locality of~$\SF{X}{\mssm}$ we conclude that
\begin{align}
\nonumber
\SF{X}{\mssm}(f_n\varrho_k-f_m\varrho_k)=&\ \car_{\inter_\T(F_{k+1})}\SF{X}{\mssm}(f_n\varrho_k-f_m\varrho_k) 
\\
\label{eq:p:LocalityProbab:2}
\leq&\ \car_{\inter_\T(F_{k+1})}\ttonde{C_k\abs{f_n-f_m}^2+ \SF{X}{\mssm}(f_n)} \fstop
\end{align}
Now, since~$0<a_k\leq g\leq a_k^{-1}<+\infty$ on~$F_{k+1}$, the $L^2(\mssm\mrestr{F_{k+1}})$-convergence is equivalent to the $L^2(\mssm'\mrestr{F_{k+1}})$-convergence.
Thus, there exists
\begin{equation}\label{eq:p:LocalityProbab:3}
L^2(\mssm)\text{-}\nlim f_n\varrho_k = L^2(\mssm')\text{-}\nlim f_n\varrho_k = f\varrho_k \comma
\end{equation}
which, together with~\eqref{eq:p:LocalityProbab:2}, implies the existence of
\begin{align*}
\EE{X}{\mssm}_1\text{-}\nlim f_n\varrho_k=f\varrho_k\in\dom{\EE{X}{\mssm}}_b \fstop
\end{align*}

By the same argument as above, with~$f_n$ in place of~$f_n-f_m$,
\begin{align*}
\car_{\inter_\T(F_{k+1})}\SF{X}{\mssm}(f_n\varrho_k)=\car_{\inter_\T(F_{k+1})} \SF{X}{\mssm}(f_n)\leq \car_{\inter_\T(F_{k+1})}  \SF{X}{\mssm'}(f_n)\comma
\end{align*}
and, taking the limit as~$n\to\infty$ (possibly, up to choosing a suitable non-relabeled subsequence),
\begin{align*}%\label{eq:p:LocalityProbab:3}
\car_{\inter_\T(F_{k+1})} \SF{X}{\mssm}(f\varrho_k)\leq \car_{\inter_\T(F_{k+1})} \nlim \SF{X}{\mssm'}(f_n)=\car_{\inter_\T(F_{k+1})}  \SF{X}{\mssm'}(f) \leq 1 \as{\mssm}\fstop
\end{align*}
Since~$\seq{F_k}_k$ is a nest,~$\seq{\inter_\T(F_k)}_k$ is a sequence of ($\EE{X}{\mssm}$-quasi-)open sets increasing to~$X$ $\EE{X}{\mssm}$-quasi-everywhere, and~$\seq{f\varrho_k}_k\subset \dom{\EE{X}{\mssm}}$ is a sequence of functions satisfying~$f\varrho_k\equiv f$ on~$\inter_\T(F_k)$.
Thus,~$f\in\dotloc{\dom{\EE{X}{\mssm}}}$ by definition of broad local space.
Again by locality of both~$\SF{X}{\mssm}$ and~$\SF{X}{\mssm'}$,
\begin{align*}
\car_{\inter_\T(F_k)} \SF{X}{\mssm}(f) =\car_{\inter_\T(F_k)} \SF{X}{\mssm}(f\varrho_k)\leq \car_{\inter_\T(F_k)}  \SF{X}{\mssm'}(f) \as{\mssm}\comma
\end{align*}
hence, letting~$k\to\infty$ and since~$\klim \mssm F_k^\complement=0$,
\begin{align*}
\SF{X}{\mssm}(f) \leq \SF{X}{\mssm'}(f) \leq 1 \as{\mssm}\comma
\end{align*}
which shows that~$f\in\DzLocB{\mssm}$. %The conclusion~\eqref{eq:p:LocalityProbab:1}.
\end{proof}
\end{prop}

\begin{rem}\label{r:CutOff}
\begin{enumerate*}[$(a)$]
\item In the language of~\cite{LzDSSuz20}, one can rephrase assumption~\ref{i:p:LocalityProbab:3} in Proposition~\ref{p:LocalityProbab} by saying that $g, g^{-1}\in \dotloc{\ttonde{L^\infty(\mssm)}}$.
\item\label{i:r:CutOff:2} The existence of the cut-off functions in~Proposition~\ref{p:LocalityProbab}\ref{i:p:LocalityProbab:2} is a quite mild assumption. Indeed, it is readily verified whenever~$(X,\mssd,\mssm)$ satisfies~$(\Rad{\mssd}{\mssm})$, for the choice~$\varrho_{U_1,U_2}\eqdef \mssd(U_1,U_2)^{-1}\mssd(\emparg,U_2)\wedge 1$ with constant~$C\eqdef \mssd(U_1,U_2)^{-1}$.
\end{enumerate*}
\end{rem}

\subsubsection{Cheeger energies}\label{sss:CheegerE}
The content of this section is mostly taken from the detailed discussion of extended metric measure spaces put forward by L.~Ambrosio, N.~Gigli, and G.~Savar\'e in~\cite{AmbGigSav14}, Ambrosio, M.~Erbar, and Savar\'e in~\cite{AmbErbSav16}, and the one ---~more general still~--- in Savar\'e's monograph~\cite{Sav19}.
For the sake of simplicity, we adapt the definitions in~\cite{AmbErbSav16,Sav19} to include information on localizing rings.
This additional structure, however, will not be used throughout this section.

By a \emph{Borel-probability extended-metric local structure}~$(\mcX,\mssd)$ we mean an extended metric local structure additionally so that~$\A=\Bo{\T}$ and~$\mssm$ is a (Radon) probability measure on~$\A$.
We note that every such space is an \emph{extended metric measure space} in the sense of~\cite[Dfn.~4.7]{AmbErbSav16}.
We shall therefore make use of results in~\cite{AmbErbSav16} without further reference to the assumptions there.

\paragraph{Minimal relaxed slopes}
Let~$\mcX$ be a Borel-probability extended-metric local structure.
A $\Bo{\T}$-measurable function~$G\colon X\to [0,\infty]$ is a ($\mssm$-)\emph{relaxed slope} of~$f\in L^2(\mssm)$ if there exist~$\seq{f_n}_n\subset \Lip(\mssd,\Bo{\T})$ so that
\begin{itemize}
\item $L^2(\mssm)$-$\nlim f_n=f$ and~$\slo{f_n}$ $L^2(\mssm)$-weakly converges to~$\tilde G\in L^2(\mssm)$;
\item $\tilde G\leq G$ $\mssm$-a.e..
\end{itemize}

We say that~$G$ is the ($\mssm$-)\emph{minimal} ($\mssm$-)\emph{relaxed slope} of~$f\in L^2(\mssm)$, denoted by~$\slo[*]{f}$, if its $L^2(\mssm)$-norm is minimal among relaxed slopes.
The notion is well-posed, and $\slo[*]{f}$ is in fact $\mssm$-a.e.\ minimal as well, see~\cite[\S{4.1}]{AmbGigSav14}.

\paragraph{Minimal weak upper gradients}
Let~$\mcX$ be a Borel-probability extended-metric local structure.
A Borel probability measure~$\boldpi$ on~$\AC^2(I;X)$ is a \emph{test plan of bounded compression} if there exists a constant~$C=C_\boldpi>0$ such that
\begin{align*}
(\ev_t)_\pfwd \boldpi \leq C\mssm \comma \qquad \ev_t\colon \AC^2\to X\comma \ev_t\colon x_\emparg\mapsto x_t\fstop
\end{align*}
A Borel subset~$A\subset \AC^2(I;X)$ is \emph{negligible} if~$\boldpi(A)=0$ for every test plan of bounded compression.
A property of $\AC^2(I;X)$-curves is said to hold for a.e.-curve if it holds for every curve in a co-negligible set.

A $\Bo{\T}$-measurable function~$G\colon X\to [0,\infty]$ ($\mssm$-)\emph{weak upper gradient} of~$\rep f\in \mcL^0(\mssm)$ if
\begin{align*}
\abs{\rep f(x_1)-\rep f(x_0)}\leq \int_0^1 G(x_r)\, \abs{\dot x}_r\diff r<\infty \quad \text{for a.e.\ curve}\fstop
\end{align*}

We say that~$G$ is the ($\mssm$-)\emph{minimal} ($\mssm$-)\emph{weak upper gradient} of~$f\in L^2(\mssm)$, denoted by~$\slo[w]{f}$, if it is $\mssm$-.a.e.-minimal among the weak upper gradients of~$\rep f$ for every representative~$\rep f$ of~$f$.
See e.g.~\cite[Dfn.~2.12]{AmbGigSav14} for the well-posedness of this notion, independently of the representatives of~$f$.

\paragraph{Asymptotic Lipschitz constants and asymptotic slopes}
Let~$\mcX$ be a Borel-probability extended-metric local structure.
For~$f\in \bLip(\mssd,\T)$ set
\begin{align*}
\Lia[\mssd]{f}(x,r)\eqdef \sup_{\substack{y,z\in B^\mssd_r(x)\\ \mssd(y,z)>0}} \frac{\abs{f(z)-f(y)}}{\mssd(z,y)}\comma \qquad x\in X\comma r>0 \fstop
\end{align*}
The \emph{asymptotic Lipschitz constant}~$\Lia[\mssd]{f}\colon X\rar [0,\infty]$ is defined as
\begin{align*}
\Lia[\mssd]{f}(x)\eqdef \lim_{r\downarrow 0} \Lia[\mssd]{f}(x,r) \comma
\end{align*}
with the usual convention that~$\Lia[\mssd]{f}(x)=0$ whenever~$x$ is a $\mssd$-isolated point in~$X$.
We drop the subscript~$\mssd$ whenever apparent from context.
Note that~$\Lia{f}$ is \mbox{$\mssd$-u.s.c.}\ by construction, thus if~$(\mcX,\mssd)$ is a metric (as opposed to: extended metric) local structure, then~$\Lia{f}$ is $\T$-u.s.c.\ as well, and therefore it is $\Bo{\T}$-measurable.

\paragraph{Cheeger energies}
Each metric measure space~$(X,\mssd,\mssm)$ is naturally endowed with a convex local energy functional, called the \emph{Cheeger energy} of the space; e.g.~\cite{AmbGigSav14}.
In the case when the Cheeger energy is quadratic, it is a local Dirichlet form, and the l.s.c.\ $L^2(\mssm)$-relaxation of the natural energy on Lipschitz functions.
Several ---~\emph{a priori} inequivalent~--- definitions of Cheeger energy are possible. We collect here three of them, referring to~\cite{Sav19} for additional ones.

\begin{defs}[{\cite[Thm.~4.5]{AmbGigSav14}}]
Let~$(\mcX,\T)$ be a Borel-probability extended-metric local structure.
The \emph{Cheeger energy} of~$f\in L^2(\mssm)$ is defined as
\begin{gather*}
\Ch[*,\mssd,\mssm](f)\eqdef \int \slo[*]{f}^2 \diff\mssm\comma \qquad
\dom{\Ch[*,\mssd,\mssm]}\eqdef \set{f\in L^2(\mssm) : \Ch[*,\mssd,\mssm](f)<\infty} \fstop
\end{gather*}
\end{defs}

\begin{defs}[{\cite[Dfn.~6.1]{AmbErbSav16}}]\label{d:Cheeger}
Let~$(\mcX,\T)$ be a Borel-probability extended-metric local structure.
The \emph{Cheeger energy} of~$f\in L^2(\mssm)$ is defined as
\begin{gather*}
\Ch[a,\mssd,\mssm](f)\eqdef \inf\nliminf \int g_n^2 \diff\mssm\comma \quad
\dom{\Ch[a,\mssd,\mssm]}\eqdef \set{f\in L^2(\mssm) : \Ch[a,\mssd,\mssm](f)<\infty} \comma
\end{gather*}
where the infimum is taken over all sequences~$\seq{f_n}_n \subset \bLip(\mssd,\T)$, $L^2(\mssm)$-strongly converging to~$f$, and all sequences~$\seq{g_n}_n$ of $\Bo{\T}^\mssm$-measurable functions satisfying~$g_n\geq \Lia{f_n}$ $\mssm$-a.e..
\end{defs}

\begin{defs}[{\cite[Rmk.~5.12]{AmbGigSav14}}]\label{d:CheegerW}
Let~$(\mcX,\T)$ be a Borel-probability extended-metric local structure.
The \emph{Cheeger energy} of~$f\in L^2(\mssm)$ is defined as
\begin{gather*}
\Ch[w,\mssd,\mssm](f)\eqdef \int \slo[w]{f}^2 \diff\mssm\comma \qquad
\dom{\Ch[w,\mssd,\mssm]}\eqdef \set{f\in L^2(\mssm) : \Ch[w,\mssd,\mssm](f)<\infty} \fstop
\end{gather*}

\end{defs}

In all cases, we shall omit the specification of either~$\mssd$,~$\mssm$ or both, whenever apparent from context.
We refer the reader to~\cite[\S4]{AmbGigSav14} for a thorough treatment of~$\Ch[*]$, and to~\cite[\S6]{AmbErbSav16} for a thorough treatment of~$\Ch[a]$ in the setting of probability extended metric measure space.

In the following, in order to refer to results in the literature, we shall need to make use of both definitions. To this end, in the next quite technical result, we show that they coincide in the present setting.

\begin{prop}\label{p:ConsistencyCheeger}
Let~$(\mcX,\T)$ be a Borel-probability extended-metric local structure. Then,
\begin{equation*}
\Ch[*,\mssd,\mssm]=\Ch[a,\mssd,\mssm]=\Ch[w,\mssd,\mssm]\comma
\end{equation*}
and each of these functionals is densely defined on~$L^2(\mssm)$.

\begin{proof}
The domain of~$\Ch[*,\mssd,\mssm]$ contains~$\bLip(\mssd,\Bo{\T})$, and the latter is dense in~$L^2(\mssm)$ by e.g.~\cite[Prop.~4.1]{AmbGigSav14}.
Let us show the identification.
Firstly, note that our Definition~\ref{d:Cheeger} (i.e.~\cite[Dfn.~6.1]{AmbErbSav16}) differs from~\cite[Dfn.~5.1]{Sav19} ---~as do our definition of the asymptotic Lipschitz constant (again after~\cite{AmbErbSav16}) and the one in~\cite{Sav19}~---. Thus, we need to show that our definition of~$\Ch[a]$ coincides with the one of~$\CE[2]$ in~\cite[Dfn.~5.1]{Sav19}.
Once this identity is established, the assertion will be a consequence of the identification of both~$\Ch[*]$ and~$\CE[2]$ with~$\Ch[w]$.%, the Cheeger energy~\cite[Dfn.~10.24]{Sav19} defined via weak upper gradients.

The identification of~$\CE[2]$ with~$\Ch[w]$ is shown in~\cite[Thm.~11.7]{Sav19} for the choice $\msA\eqdef \bLip(X,\mssd,\T)$.
The identification of~$\Ch[*]$ with~$\Ch[w]$ is shown in~\cite[Thm.~6.2]{AmbGigSav14}. The assumption in Equation~(4.2) there is trivially satisfied since~$\mssm X=1$.
Thus, the proof is concluded if we show that
\begin{equation}\label{eq:p:Cheeger:1}
\Ch[w]\leq \Ch[a]\leq \CE[2]\fstop
\end{equation}

Since~$\T_\mssd$ is finer than~$\T$, for each~$x\in X$ and for each neighborhood~$U$ of~$x$, there exists~$r>0$ so that~$B^\mssd_r(x)\subset U$. Thus, for any $\Bo{\T}$-measurable~$f\colon X\rar \overline{\R}$,
\begin{align*}
\sup_{x,y\in B^\mssd_r(x)} \frac{f(z)-f(y)}{\mssd(x,y)} \leq \sup_{x,y\in U} \frac{f(z)-f(y)}{\mssd(x,y)} \fstop
\end{align*}
As a consequence, $\Lia[\mssd]{f}$~is dominated by the asymptotic Lipschitz constant of~$f$ as defined in~\cite[Eqn.~(2.48)]{Sav19}, and the second inequality in~\eqref{eq:p:Cheeger:1} follows.

The first inequality in~\eqref{eq:p:Cheeger:1} is a consequence of~\cite[Prop.~6.3$(b)$ and $(g)$]{AmbErbSav16}. Importantly, we note that~\cite{AmbErbSav16} denotes by~$\slo[w]{f}$ the minimal relaxed slope~$\slo[*]{f}$.
\end{proof}
\end{prop}

\noindent In light of the Proposition, we shall henceforth simply write~$\Ch$ in place of either~$\Ch[*]$, $\Ch[w]$, or~$\Ch[a]$.

%\smallskip

\subsubsection{Tensorization and its consequences}\label{sss:TensorizationConseq}
In this section we present some results on product spaces, and in particular the tensorization of the Rademacher property for~\MLDS's.
As explained in~\S\ref{sss:Rademacher} below, the tensorization of the Rademacher property for~$(\mcX,\cdc,\mssd)$ is one cornerstone of our proof of the Rademacher property for~$\ttonde{\dUpsilon,\SF{\dUpsilon}{\QP},\mssd_\dUpsilon}$.
However, the tensorization of the Rademacher property is also of independent interest.
In the literature, it seems to have been addressed only in the case when~$\EE{X}{\mssm}=\Ch[\mssd,\mssm]$ is the Cheeger energy, in connection with the tensorization of the Cheeger energy under local doubling and Poincar\'e or under the $\RCD$ condition (see~\S~\ref{sss:RCD}), cf.~\cite{AmbGigSav14b},~\cite{AmbPinSpe15}.
Here, we discuss the case of general quasi-regular strongly local~$\EE{X}{\mssm}$, without any geometric (e.g.\ curvature) assumption.

\medskip

Let~$(\mcX,\mssd)$ be a (extended) metric local structure and~$n\in\N$, and denote by~$\mssd_\tym{n}$ the $n$-fold (extended) $\ell^2$-product distance on~$X^\tym{n}$.
The next result, consequence of the inequality $\norm{\emparg}_{\ell^\infty}\leq \norm{\emparg}_{\ell^2}$ on~$\R^n$, complements Proposition~\ref{p:Products}.

\begin{prop}\label{p:Products2}
Let~$(\mcX,\mssd)$ be a (extended) metric local structure, resp.\ $(\mcX,\cdc,\mssd)$ be an \parEMLDS. Then, for every~$n\geq 2$, the pair~$(\mcX^\otym{n},\mssd_\tym{n})$ is a (extended) metric local structure, resp.\ the triple $(\mcX^\otym{n},\cdc^\otym{n},\mssd_\tym{n})$ is an \parEMLDS.
\end{prop}

In the same spirit of~\S\ref{ss:AnalyticForms}, in the next sections we will compare the metric properties of~$\ttonde{\dUpsilon,\mssd_\dUpsilon}$ with those of the infinite-product space~$(X^\tym{\infty},\mssd_\tym{\infty})$.
In order to make sense of the approximation of the latter extended metric space in terms of its finite-dimensional approximations, we shall need the following assumption.

\begin{ass}[Tensorization]\label{d:ass:Tensorization}
We say that an \EMLDS $(\mcX,\cdc,\mssd)$ satisfies the \emph{tensorization} assumption if
\begin{align}\tag*{$(\otimes)_{\ref{d:ass:Tensorization}}$}\label{ass:T}
\cdc^\otym{n}(f^\asym{n})=\slo[\mssd_\tym{n}]{f^\asym{n}}^2 \quad \as{\mssm} \comma \qquad f^\asym{n}\in \bLip(\mssd,\A)\fstop
\end{align}
\end{ass}

Let us put Assumption~\ref{ass:T} into context: while it is presented here, for it ought to be regarded as an assumption on base spaces, rather than on configuration spaces, the assumption will be used only in~\S\ref{sss:AnalyticCheeger} to show the identification of the Cheeger energy~$\Ch[\mssd_\dUpsilon,\QP]$ with the form~$\EE{\dUpsilon}{\QP}$.
As further clarified in the preliminary discussion to Theorem~\ref{t:IdentificationCheeger} below, this identification is ---~from the point of view of Dirichlet-forms theory~--- a very strong statement.
In the same spirit of all our results, we will proceed to a proof lifting properties of the base spaces to the corresponding configuration spaces.
With the above goal in mind, it is therefore natural to assume a similar coincidence of the diffusion structure of an \MLDS with the correspoding metric measure structure, which takes the form~\ref{ass:T}.
We collect some of its consequences in the next remark.

\begin{rem}\label{r:Ass:T}
Assumption~\ref{ass:T} entails both the coincidence of the square field~$\SF{X}{\mssm}$ with the $\mssd$-slope, and the tensorization of the slope on product spaces of every order.
In particular,
\begin{enumerate*}[$(a)$]
\item since the $\mssd$-slope of a $\mssd$-Lipschitz function is always dominated by its Lipschitz constant, the assumption implies~$(\Rad{\mssd_{\tym{n}}}{\mssm^{\otym n}})$ for every~$n\in \N_1$;
\item since~$\cdc^\otym{n}$ is a quadratic functional, the assumption implies that the metric measure space~$(X^\tym{n},\mssd_{\tym n},\mssm^{\otym n})$ is \emph{infinitesimally Hilbertian}, i.e.\ that the corresponding Cheeger energy~$\Ch[*,\mssd_\tym{n}, \mssm^\otym{n}]$ is a quadratic form (see Definition~\ref{d:IH}) for every~$n\in \N_1$;
\item under Assumption~\ref{ass:T}, we have the coincidence of~$\ttonde{\EE{X}{\mssm},\dom{\EE{X}{\mssm}}}$ with the Cheeger energy~$\Ch[*,\mssd,\mssm]$.
Indeed, since Assumption~\ref{ass:T} clearly implies~$(\Rad{\mssd}{\mssm})$, we have~$\EE{X}{\mssm}\leq \Ch[*,\mssd,\mssm]$ by Proposition~\ref{p:IneqCdC}.
The opposite inequality holds by $L^2(\mssm)$-lower-semicontinuity of the square field~$\SF{X}{\mssm}$ and definition of~$\slo[*,\mssd,\mssm]{\emparg}$.
\end{enumerate*}
\end{rem}

Whereas at some instance in the following sections we will assume~\ref{ass:T}, it is worth to point out that on well-behaved spaces the assumption is a consequence of the same assertion \emph{only for~$n=1$}.
This can be checked by separately studing the two opposite inequalities, corresponding to local versions of the Rademacher and Sobolev-to-Lipschitz properties.
Let us now provide a proof of the tensorization of the Rademacher property, relying on some technical arguments in~\cite{AmbGigSav14b}.

\begin{prop}\label{p:IneqCdC}
Let~$(\mcX,\cdc,\mssd)$ be an \MLDS satisfying~$(\Rad{\mssd}{\mssm})$. Then,
\begin{equation*}
\SF{X}{\mssm}(f)\leq \slo[w,\mssd_{\mssm}]{f}^2 \leq \slo[w,\mssd]{f}^2 \quad \as{\mssm}\comma \qquad f\in \bLip(\mssd)\fstop
\end{equation*}
In particular,~$\EE{X}{\mssm}\leq \Ch[\mssd_{\mssm}]\leq \Ch[\mssd]$.
\begin{proof}
By~\eqref{eq:EquivalenceRadStoL} the space~$(\mcX,\cdc,\mssd)$ satisfies~$(\dRad{\mssd}{\mssm})$, which implies the second inequality by definition of the objects involved.
Thus, it suffices to show the first equality.
Firstly, let us note that~$(X,\mssd_{\mssm})$ is a complete extended metric-topological space in the sense of Definition~\ref{d:AES} by Lemma~\ref{l:RadCompleteness}.
Suppose momentarily that~$\mssm$ be a probability measure.
Then, we can apply~\cite[Thm.~12.5]{AmbErbSav16} to the complete extended metric-topological Radon probability space~$(X,\T,\mssd,\mssm)$ and conclude the assertion.

\paragraph{Heuristics}
Now, let us show how to extend the statement to the case of any $\sigma$-finite measure~$\mssm$.
We need to treat simulatneously the square field operator~$\SF{X}{\mssm}$ and the minimal weak upper gradient~$\slo[w,\mssd_\mssm,\mssm]{\emparg}$.
To this end, we find a probability density~$\theta\in L^1(\mssm)$, set~$\tilde\mssm\eqdef \theta\mssm$, and show that~$\SF{X}{\mssm}=\SF{X}{\tilde\mssm}$ and~$\slo[w,\mssd_\mssm,\mssm]{\emparg}=\slo[w,\mssd_\mssm,\tilde\mssm]{\emparg}$.
The conclusion follows from these equalities together with the inequality established in the probability case applied to~$\tilde\mssm$.

For the square field, we make use of the result on the Girsanov transform of~$\EE{X}{\mssm}$ by a factor~$\sqrt{\theta}\in\dom{\EE{X}{\mssm}}$, thoroughly discussed in the generality of quasi-regular Dirichlet spaces by C.-Z.~Chen and W.~Sun in~\cite{CheSun06}.
For the minimal weak upper gradient, we make use of the locality result for~$\slo[w,\mssd_\mssm,\mssm]{\emparg}$ under transformation of the reference measure by a factor~$\theta$ locally bounded away from~$0$ and infinity on neighborhoods of compact sets, established by L.~Ambrosio, N.~Gigli, and G.~Savar\'e in~\cite{AmbGigSav14}.

\paragraph{Construction of a density} We start by showing that there exists a $\Bo{\T}$-measurable~$\theta\colon X\to \R$ satisfying:
\begin{equation*}
\theta\in L^1(\mssm)\cap L^\infty(\mssm)\comma\qquad \norm{\theta}_{L^1(\mssm)}=1\comma\qquad \theta>0 \as{\mssm}\comma\qquad \sqrt\theta\in \dom{\EE{X}{\mssm}}\fstop
\end{equation*}
Indeed, let~$x_0\in X$ be fixed,~$\seq{r_i}_i$,~$r_{i+1}-r_i>1$ be a diverging sequence of radii, and note that~$B_i\eqdef B^\mssd_{r_i}(x_0)\in\Ed$ is an exhaustion of~$X$.
Further define
\begin{align*}
g_i\colon x\mapsto 0\vee \ttonde{r_i-\mssd(x_0,x)}\wedge 1\comma \qquad f_i\eqdef \frac{g_i}{\sqrt{\mssm B_{i+1}} }\comma\qquad \varphi_n\eqdef \sum_i^n 2^{-i-1} f_i\fstop
\end{align*}
Since~$\mssm B_i<\infty$ for every~$i$, we have that~$0<f_i\in L^2(\mssm)$ for every~$i$.
Up to choosing~$r_1\gg 1$, we may assume with no loss of generality that~$\mssm B_{i+1} \geq \norm{g_i}_{L^2(\mssm)}^2\geq 1$ for every~$i$, in such a way that~$\norm{f_i}_{L^2(\mssm)}\leq 1$ for every~$i$.
Thus, the sequence~$\seq{\varphi_n}_n$ satisfies
\begin{align*}
\norm{\varphi_n}_{L^2(\mssm)}\leq 1\comma \qquad  L^2(\mssm)\text{-}\nlim \varphi_n=\varphi\eqdef \sum_i^\infty 2^{-i-1} f_i\fstop
\end{align*}
Furthermore,~$f_i\in \bLip(\mssd)$ and~$\Li[\mssd]{f_i}\leq \Li[\mssd]{g_i}\leq 1$ for every~$i$ by definition, hence~$\varphi_n\in \bLip(\mssd)$ and~$\Li[\mssd]{\varphi_n}\leq 1$ for every~$n$ by triangle inequality for the Lipschitz semi-norm~$\Li[\mssd]{\emparg}$.
As a consequence,~$\seq{\varphi_n}_n$ is uniformly bounded in~$\dom{\EE{X}{\mssm}}$ by~$(\Rad{\mssd}{\mssm})$ and thus~$\varphi\in \dom{\EE{X}{\mssm}}$ by e.g.~\cite[Lem.~2.36, Rmk.~2.37]{LzDSSuz20}.
Since
\begin{equation}\label{eq:l:IneqCdC:0.25}
\varphi\restr_{B_i}\geq 2^{-i-1}f_i\restr_{B_i}=2^{-i-1}/\norm{g_i}_{L^2(\mssm)}>0
\end{equation}
and since~$\seq{B_i}_i$ is a covering of~$X$, then~$\varphi>0$ everywhere on~$X$.
Letting~$\theta\eqdef \varphi^2/\norm{\varphi}_{L^2(\mssm)}^2$ shows the assertion.

\paragraph{Dirichlet forms}
Now, let us set~$\tilde\mssm\eqdef \theta\mssm\sim \mssm$.
Since~$\sqrt\theta\in\dom{\EE{X}{\mssm}}$, by~\cite[Thm.~2.2]{CheSun06} the form
\begin{align*}
\EE{X}{\tilde\mssm}(f)\eqdef \int_X \SF{X}{\mssm}(f)\diff\tilde\mssm\comma \qquad f\in\dom{\EE{X}{\mssm}}\comma %: \int_X \ttonde{\SF{X}{\mssm}(f)+f^2} \diff \tilde\mssm<\infty
\end{align*}
is closable, and its closure~$\ttonde{\EE{X}{\tilde\mssm},\dom{\EE{X}{\tilde\mssm}}}$ is a Dirichlet form on~$L^2(\tilde\mssm)$ with square field operator~$\SF{X}{\tilde\mssm}$ satisfying
\begin{align}\label{eq:l:IneqCdC:0.5}
\SF{X}{\tilde\mssm}(f)=\SF{X}{\mssm}(f) \as{\tilde\mssm}\comma \qquad f\in \dom{\EE{X}{\mssm}}\cap L^\infty(\tilde\mssm)\fstop
\end{align}
Further observe that~$\dom{\EE{X}{\tilde\mssm}}\cap L^\infty(\tilde\mssm)\supset \dom{\EE{X}{\mssm}}\cap L^\infty(\mssm)\supset \bLip(\mssd)$ by~$(\Rad{\mssd}{\mssm})$.

\paragraph{Metric objects} 
On the complete extended metric-topological Radon measure space~$(X,\mssd_\mssm,\mssm)$ we have
\begin{equation}\label{eq:l:IneqCdC:1}
\slo[*,\mssd_\mssm]{f}=\slo[w,\mssd_\mssm]{f}\comma \qquad f\in\bLip(\mssd_\mssm,\T)\fstop
\end{equation}
This readily follows from the locality of both~$\slo[*]{\emparg}$ and~$\slo[w]{\emparg}$ in the sense of e.g.~\cite[Prop.~4.8(b)]{AmbGigSav14} and~\cite[Eqn.~(2.18)]{AmbGigSav14b}.

Secondly, let us verify that the space~$(X,\mssd_\mssm,\mssm)$ satisfies the assumptions of~\cite[Lem.~4.11]{AmbGigSav14}.
It suffices to show that~$\theta$ as above satisfies~\cite[Eqn.~(4.21)]{AmbGigSav14}, that is:
for every $\T$-compact~$K\subset X$, there exist~$r_K>0$ and constants~$c_K$, $C_K$ so that
\begin{align}\label{eq:AGS4.11}
0<c_K\leq \theta \leq C_K <\infty \quad \as{\mssm} \text{ on~} B^{\mssd_\mssm}_{r_K}(K)\eqdef \set{x: \mssd_\mssm(x, K)\leq r_K} \fstop
\end{align}
Since~$\theta\in L^\infty(\mssm)$, it suffices to show the existence of~$c_K$ and~$r_K$.
In fact, we show that there exists~$c_K$ satisfying~\eqref{eq:AGS4.11} with~$r_K=1$.
Since~$\seq{B_i}_i$ is an exhaustion of~$X$, and since~$\mssd\leq \mssd_\mssm$ by~$(\Rad{\mssd}{\mssm})$ and~\eqref{eq:EquivalenceRadStoL}, for every $\T$-compact~$K$ there exists~$i_K$ so that~$B_1^{\mssd_\mssm}(K)\subset B_1^\mssd(K)\subset B_{i_K}$.
By~\eqref{eq:l:IneqCdC:0.25}, it suffices to set~$c_K\eqdef 2^{-2i_K-2} \norm{\varphi}_{L^2(\mssm)}^{-2}\norm{g_{i_K}}_{L^2(\mssm)}^{-1}$.

Now, applying~\cite[Lem.~4.11]{AmbGigSav14} to the probability measure~$\tilde\mssm$, we have that
\begin{align}\label{eq:l:IneqCdC:2}
\slo[*,\mssd_\mssm,\mssm]{f}=\slo[*,\mssd_\mssm,\tilde\mssm]{f}\comma \qquad f\in\bLip(\mssd)\subset \bLip(\mssd_\mssm, \T)\fstop
\end{align}

\paragraph{Conclusion}
Respectively by~\eqref{eq:l:IneqCdC:0.5},~\cite[Lem.~12.3]{AmbErbSav16},~\eqref{eq:l:IneqCdC:1},~\eqref{eq:l:IneqCdC:2}, and again~\eqref{eq:l:IneqCdC:1},
\begin{align*}
\SF{X}{\mssm}(f)=\SF{X}{\tilde\mssm}\leq \slo[w,\mssd_\mssm,\tilde\mssm]{f} = \slo[*,\mssd_\mssm,\tilde\mssm]{f} = \slo[*,\mssd_\mssm,\mssm]{f} = \slo[w,\mssd_\mssm,\mssm]{f}\comma \qquad f\in\bLip(\mssd)\comma
\end{align*}
which concludes the proof.
\end{proof}
\end{prop}

\begin{thm}[Tensorization of $(\mathsf{Rad})$]\label{t:Tensor} Let~$(X,\mssd,\cdc,\mssm)$ be an \MLDS satisfying~$(\Rad{\mssd}{\mssm})$.
Then $\ttonde{X^\tym{n},\mssd_\tym{n}, \cdc^{\otym{n}},\mssm^{\hotym n}}$ satisfies~$(\Rad{\mssd_\tym{n}}{\mssm^\otym{n}})$.
\begin{proof} It suffices to show the statements for~$n=2$.
The case of arbitrary~$n$ follows by induction.

By Proposition~\ref{p:Products2}, the product space is an \MLDS.
By Propositions~\ref{p:Products},~\ref{p:Products2}, and recalling~\eqref{eq:FuncSection},~\eqref{eq:CdCTensor}, and Notation~\ref{n:ProductCdC},
\begin{align}\label{eq:t:Tensor:1}
\cdc^\otym{2}(f^\asym{2})(x_1,x_2)=\SF{X}{\mssm}(f^\asym{2}_{x_2,1})(x_1)+\SF{X}{\mssm}(f^\asym{2}_{x_1,2})(x_2)\comma \qquad \forallae{\mssm^\otym{2}}~ (x_1,x_2)\in X^\tym{2} \fstop
\end{align}

Further denote by
\begin{align*}
\wslo[\mssd]{f^\asym{2}}^{\otym{2}}(x_1,x_2)\eqdef \sqrt{\wslo[\mssd]{f^\asym{2}_{x_2,1}}^2(x_1)+\wslo[\mssd]{f^\asym{2}_{x_1,2}}^2(x_2)} \comma \qquad f^\asym{2}\in \Lip(\mssd_\tym{2})\comma
\end{align*}
the \emph{Cartesian gradient} in~\cite[p.~1477]{AmbGigSav14b}, by
\begin{align*}
\slo[c,\mssd]{f^\asym{2}}^{\otym{2}}(x_1,x_2)\eqdef \sqrt{\slo[\mssd]{f^\asym{2}_{x_2,1}}^2(x_1)+\slo[\mssd]{f^\asym{2}_{x_1,2}}^2(x_2)} \comma \qquad f^\asym{2}\in \Lip(\mssd_\tym{2})\comma
\end{align*}
the \emph{Cartesian slope} in~\cite[Eqn.~(3.3)]{AmbPinSpe15}, and by~$\slo[*,c,\mssd]{f^\asym{2}}^{\otym{2}}$ the minimal relaxed gradient associated to the Cheeger energy
\begin{align*}
\Ch[*,c,\mssm^\otym{2}](f^\asym{2})\eqdef \inf\set{\liminf_n \int_{X^\tym{2}} \ttonde{\slo[c,\mssd]{f^\asym{2}_n}^{\otym{2}}}^2 \diff\mssm^\otym{2}}=\int_{X^\tym{2}} \ttonde{\slo[*,c,\mssd]{f^\asym{2}}^{\otym{2}}}^2 \diff\mssm^\otym{2} \comma
\end{align*}
where the infimum is taken over all sequences~$\seq{f^\asym{2}_n}_n\subset \Lip(\mssd_\tym{2})$ with~$L^2(\mssm^\otym{2})\text{-}\nlim f_n^\asym{2}=f^\asym{2}$.
Recall that, by e.g.~\cite[Thm.~2.2(ii)]{AmbPinSpe15}, for every~$f^\asym{2}\in\Lip(\mssd_\tym{2})$ there exists a sequence of functions~$\seq{f^\asym{2}_n}_n\subset \Lip(\mssd_\tym{2})$ such that
\begin{align}\label{eq:t:Tensor:2}
L^2(\mssm^\otym{2})\text{-}\nlim f^\asym{2}_n=f^\asym{2}\qquad \text{and} \qquad \slo[*,c,\mssd]{f^\asym{2}}^{\otym{2}}=L^2(\mssm^\otym{2})\text{-}\nlim \slo[*,c,\mssd]{f^\asym{2}_n}^{\otym{2}}\fstop
\end{align}

Now, by applying Proposition~\ref{p:IneqCdC} to the right-hand side of~\eqref{eq:t:Tensor:1}, we have that
\begin{align}\label{eq:t:Tensor:3}
\cdc^\otym{2}(f^\asym{2})\leq \ttonde{\wslo[\mssd]{f^\asym{2}}^{\otym{2}}}^2\comma \qquad f^\asym{2}\in \Lip(\mssd_\tym{2})\fstop
\end{align}
Letting~$f^\asym{2}$ and~$\seq{f^\asym{2}_n}_n$ be as in~\eqref{eq:t:Tensor:2}, it follows by $L^2(\mssm^\otym{2})$-lower semi-continuity of~$\cdc^\otym{2}$, that
\begin{align}\label{eq:t:Tensor:4}
\cdc^\otym{2}(f^\asym{2})\leq&\ \nliminf \cdc^\otym{2}(f^\asym{2}_n)\leq \nliminf \tonde{\wslo[\mssd]{{f^\asym{2}_n}_{\emparg,1}}^2+\wslo[\mssd]{{f^\asym{2}_n}_{\emparg,2}}^2}
\\
\leq&\ \nliminf \tonde{\slo[\mssd]{{f^\asym{2}_n}_{\emparg,1}}^2+\slo[\mssd]{{f^\asym{2}_n}_{\emparg,2}}^2}= \slo[*,c,\mssd]{f^\asym{2}}^{\otym{2}} \comma
\end{align}
where the second inequality is~\eqref{eq:t:Tensor:3}, the third one holds by definition of minimal weak upper gradient, and the final equality holds by~\eqref{eq:t:Tensor:2}.

By~\cite[Thm.~3.2]{AmbPinSpe15}, we have that~$\slo[*,c,\mssd]{f^\asym{2}}^{\otym{2}}=\slo[*,\mssd_\tym{2}]{f^\asym{2}}$.
We may thus continue from~\eqref{eq:t:Tensor:4} to obtain that
\begin{align*}
\cdc^\otym{2}(f^\asym{2})\leq& \slo[*,\mssd_\tym{2}]{f^\asym{2}} \leq \slo[\mssd_\tym{2}]{f^\asym{2}} \leq \Li[\mssd_\tym{2}]{f^\asym{2}} \comma \qquad f^\asym{2}\in \bLip(\mssd_{\tym{2}})\comma
\end{align*}
which concludes the proof.
\end{proof}
\end{thm}

\subsection{Configuration spaces}\label{ss:MetricConfig}
Let~$(\mcX,\cdc,\mssd)$ be an \MLDS.
We start by proving some metric properties of~$\mssd_\dUpsilon$, including: a generalization to our setting of~\cite[Lem.s~4.1 and 4.2]{RoeSch99}, the Borel measurability of McShane extensions, and a study of the local Lipschitz property for cylinder functions.
These properties eventually allow us to prove both the Rademacher (\S\ref{sss:Rademacher}) and the Sobolev-to-Lipschitz (\S\ref{ss:VariousStoL}) properties for~$\ttonde{\EE{\dUpsilon}{\QP},\dom{\EE{\dUpsilon}{\QP}}}$ w.r.t.~$\mssd_\dUpsilon$, from which we will conclude the identification of the analytic and the geometric structure.

\begin{center}
\textbf{Everywhere in this section,~$(\mcX,\mssd)$ is a metric (\emph{not} extended) local structure.}
\end{center}

Throughout this section, we shall make use of the following fact.

\begin{rem}\label{r:QPRadon}
Any Borel probability measure~$\QP$ on $\ttonde{\dUpsilon, \T_\mrmv(\Ed)}$ is Radon.
\begin{proof}
Every topological probability measure on a Polish space is Radon.
Since an \MLDS is in particular Polish,~$\ttonde{\dUpsilon,\T_\mrmv(\Ed)}$ is Polish by Proposition~\ref{p:TopologyUpsilon}\iref{i:p:TopologyUpsilon:2}. It follows that~$\QP$ is Radon.
\end{proof}
\end{rem}

\subsubsection{The \texorpdfstring{$L^2$}{L2}-transportation distance}\label{sss:DistanceUpsilon}
We collect here several properties of~$\mssd_\dUpsilon$, adapting the proofs of~\cite[Lem.s~4.1,~4.2]{RoeSch99}.
As usual, we denote by~$\Opt(\eta,\gamma)$ the set of~$\cpl$'s in~$\Cpl(\eta,\gamma)$ attaining the infimum in~\eqref{eq:d:W2Upsilon}.
We write~$B^\dUpsilon_r(\gamma)$ in place of~$B^{\mssd_\dUpsilon}_r(\gamma)$ for every~$\gamma\in \dUpsilon$ and~$r\in [0,\infty]$.
Finally, recall that~$\pr^i\colon X^{\tym{2}}\rar X$ denotes the projection on the $i^\text{th}$ coordinate, $i=1,2$. 

\begin{prop}[Properties of~$\mssd_\dUpsilon$, cf.~{\cite[Lem.~4.1]{RoeSch99}}]\label{p:W2Upsilon}
The following assertions hold.
\begin{enumerate}[$(i)$]
\item\label{i:p:W2Upsilon:1} Let~$\eta\in\dUpsilon$ and~$N\eqdef \eta X$. Then~$B^\dUpsilon_\infty(\eta)\subset \dUpsilon^{\sym{N}}(\Ed)$;

\item\label{i:p:W2Upsilon:2} the map~$I\colon \dUpsilon(X^{\tym{2}},\Ed^{\otym{2}}) \longrar [0,\infty]$ defined as
\begin{align*}
I\colon \alpha &\longmapsto \int_{X^\tym{2}} \mssd^2(x,y)\diff \alpha(x,y)
\end{align*}
is $\T_\mrmv(\Ed^{\otym{2}})$-lower semi-continuous;

\item\label{i:p:W2Upsilon:3} for~$i=1,2$ the projection map $\pr^i_\pfwd\colon \dUpsilon(X^{\tym{2}},\Ed^{\otym{2}})\rar \dUpsilon$ is proper, i.e.\ the pre-image via~$\pr^i_\pfwd$ of a relatively $\T_\mrmv(\Ed)$-compact set is relatively $\T_\mrmv(\Ed^{\otym{2}})$-compact.

\item\label{i:p:W2Upsilon:4} for every~$r\geq 0$ and~$i=1,2$, the projection map~$\pr^i_\pfwd\colon \dUpsilon(X^{\tym{2}},\Ed^{\otym{2}})\rar \dUpsilon$ restricted to the closed set
\begin{align}\label{p:W2Upsilon:0}
G_r\eqdef \set{I\leq r^2}
\end{align}
is continuous;

\item\label{i:p:W2Upsilon:5} for every~$\eta,\gamma\in\dUpsilon$ with~$\mssd_\dUpsilon(\eta,\gamma)<\infty$, there exists~$\cpl\in \Opt(\eta,\gamma)$, additionally a \emph{matching}, i.e.~$\cpl\in \dUpsilon(X^{\tym{2}},\Ed^{\otym{2}})$. In particular,~$\Opt(\eta,\gamma)\neq \emp$, and, if~$\eta,\gamma\in\Upsilon$, then $\cpl\in \Upsilon(X^{\tym{2}},\Ed^{\otym{2}})$;

\item\label{i:p:W2Upsilon:6} the map~$G_r\ni\alpha \longmapsto (\pr^1_\pfwd \alpha,\pr^2_\pfwd \alpha)\in \dUpsilon^{\tym{2}}$ is closed, i.e.\ the set
\begin{align*}
\tset{(\pr^1_\pfwd \alpha,\pr^2_\pfwd \alpha) : \alpha\in F\cap G_r}
\end{align*}
is $\T_\mrmv(\Ed)^{\tym{2}}$-closed for every $\T_\mrmv(\Ed^{\otym{2}})$-closed $F\subset \dUpsilon(X^{\tym{2}},\Ed^{\otym{2}})$. In particular, $\pr^i_\pfwd \colon G_r\rar \dUpsilon$ is closed;

\item\label{i:p:W2Upsilon:7} $\mssd_\dUpsilon$ is $\T_\mrmv(\Ed)^{\tym{2}}$-l.s.c.\ on~$\dUpsilon^{\tym{2}}$; (in particular:~$\Bo{\T_\mrmv(\Ed)}^{\otym{2}}$-measur\-able);

\item\label{i:p:W2Upsilon:8} Let~$\Lambda \subset \dUpsilon$ be $\T_\mrmv(\Ed)$-closed.
Then,~$\mssd_\dUpsilon(\emparg,\Lambda)$ is  $\T_\mrmv(\Ed)$-l.s.c.\ (in particular: $\Bo{\T_\mrmv(\Ed)}$-measur\-able) on~$\dUpsilon$.
If~$\Lambda$ is $\T_\mrmv(\Ed)$-compact, then
\begin{align*}
(\Lambda)_r\eqdef \set{\mssd_\dUpsilon(\emparg, \Lambda)\leq r}
\end{align*}
is compact, for all~$r\geq 0$. In particular, $\mssd_\dUpsilon$-closed balls are $\T_\mrmv(\Ed)$-compact;

\item\label{i:p:W2Upsilon:9} $\mssd_\dUpsilon(\emparg, \Lambda)\wedge r$ belongs to $\Lip(\mssd_\dUpsilon)$ with~$\Li[\mssd_\dUpsilon]{\mssd_\dUpsilon(\emparg, \Lambda)\wedge r}\leq 1$ for every~$\Lambda \subset \dUpsilon$ and~$r\geq 0$;
\end{enumerate}
\begin{proof}
\iref{i:p:W2Upsilon:1} If~$\eta\in\dUpsilon^\sym{N_1}$ and~$\gamma\in\dUpsilon^\sym{N_2}$ for some~$N_1\neq N_2$, then~$\Cpl(\eta,\gamma)=\emp$ and~$\mssd_\dUpsilon(\eta,\gamma)=+\infty$.
\iref{i:p:W2Upsilon:2} By Proposition~\ref{p:Products} and Remark~\ref{r:TLS}\iref{i:r:TLS:7}, there exists~$\msU^{\asym{2}}\subset \Ed^{\otym{2}}$ a countable locally finite open cover of~$(X^{\tym{2}},\T^{\tym{2}})$, and~$\seq{f_k}_k$ a partition of unity of~$(X^{\tym{2}},\T^{\tym{2}})$ subordinate to~$\msU^{\asym{2}}$. In particular,~$\seq{f_k}_k\subset \Cz(\Ed^{\otym{2}})$.
Note that~$\mssd_n\eqdef \sum_{k=1}^n f_k\mssd$ belongs to~$\Cz(\Ed^{\otym{2}})$ for every~$n$, since~$\msU^{\asym{2}}$ is locally finite and~$\Ed^{\otym{2}}$ is an algebra.
By monotone convergence of~$\mssd_n$ to~$\mssd$,
\begin{align*}
\int_{X^{\tym{2}}} \mssd^2(x,y) \diff \alpha(x,y)\ =&\ \nlim \int_{X^{\tym{2}}} \mssd_n^2(x,y) \diff \alpha(x,y)
\\
=&\ \nlim \mliminf \int_{X^{\tym{2}}} \mssd_n^2(x,y) \diff \alpha_m(x,y)
\\
\leq&\ \mliminf \int_{X^{\tym{2}}} \mssd^2(x,y) \diff \alpha_m(x,y) \fstop
\end{align*}

\iref{i:p:W2Upsilon:3}
Denote by~$\Meas(\Ed)$ the space of $\Ed$-locally finite nonnegative measures on $(X,\A)$. By~\cite[Thm.~4.2]{Kal17}, a set~$\Lambda \subset \Meas(\Ed)$ is $\T_\mrmv(\Ed)$-relatively compact if and only if
\begin{align*}
\sup_{\lambda \in \Lambda} \lambda E <\infty \qquad \text{and} \qquad \inf_{K\in \Ko{\T}}\sup_{\lambda \in \Lambda} \lambda (E\setminus K)=0 \comma \qquad E\in\Ed \fstop 
\end{align*}
The same characterization holds for subsets~$\Lambda \subset \dUpsilon$ and~$\Delta \subset \dUpsilon(X^{\tym{2}},\Ed^{\otym{2}})$.
Thus,
\begin{align*}
\sup_{\alpha\in \Delta} \alpha E^{\asym{2}}\leq& \sup_{\alpha\in \Delta} \pr^i_\pfwd \alpha \; \pr^i(E^{\asym{2}}) \leq \sup_{\gamma\in \Lambda} \gamma\, \pr^i(E^{\asym{2}}) <\infty\comma \qquad E^\asym{2}\in\Ed^\otym{2}\fstop
\end{align*}
Now, setting~$E_i\eqdef \pr^i(E^\asym{2})\in\Ed$, and since~$\Ko{\T^\tym{2}}\supset \Ko{\T}^{\tym{2}}$,
\begin{align*}
\inf_{K^{\asym{2}}\in \Ko{\T^{\tym{2}}}} &\sup_{\alpha\in \Delta} \alpha \ttonde{E^{\asym{2}}\setminus K^{\asym{2}}}\leq
\\
\leq&\inf_{K_1, K_2\in \Ko{\T}} \sup_{\alpha\in \Delta} \alpha \ttonde{(E_1\times E_2)\setminus (K_1\times K_2)}
\\
\leq&\inf_{K_1, K_2\in \Ko{\T}} \tonde{ \sup_{\alpha\in \Delta} \alpha \ttonde{E_1\times (E_2\setminus K_2)} + \sup_{\alpha\in \Delta} \alpha \ttonde{(E_1\setminus K_1)\times E_2} }
\\
\leq&\inf_{K_1, K_2\in \Ko{\T}} \Bigg( \sup_{\alpha\in \Delta} \pr^2_\pfwd\alpha \; \pr^2\ttonde{E_1\times (E_2\setminus K_2)} 
%\\
%&\phantom{\inf_{K_1, K_2\in \Ko{\T}}}
+ \sup_{\alpha\in \Delta} \pr^1_\pfwd \alpha \; \pr^1\ttonde{(E_1\setminus K_1)\times E_2} \Bigg)
\\
\leq&\inf_{K_1, K_2\in \Ko{\T}} \tonde{ \sup_{\gamma\in \Lambda} \gamma  (E_2\setminus K_2) + \sup_{\gamma\in \Lambda} \gamma (E_1\setminus K_1) }=0 \fstop
%\\
\end{align*}

\iref{i:p:W2Upsilon:4} By Remark~\ref{r:TLS}\iref{i:r:TLS:7}, for every~$r>0$ there exists~$\msU\eqdef\seq{U_k}_k \subset\Ed$ a countable locally finite open cover of~$(X,\T)$ so that every~$E\in\msU$ has diameter~$\diam_\mssd E<r$. Let~$\seq{f_k}_k\subset \Cz(\Ed)$ with~$\supp f_k\subset U_k$ be a partition of unity subordinate to~$\msU$.
For~$f\in \Cz(\Ed)$ set
\begin{align*}
(\supp f)_r\eqdef \set{\mssd(\emparg,\supp f)\leq r}\in\Ed \qquad \text{and} \qquad g\eqdef \sum_{k: U_k \cap (\supp f)_r \neq \emp} f_k \fstop
\end{align*}
Then,~$g\in\Cz(\Ed)$ and~$g\equiv \uno$ on~$(\supp f)_r$. The proof now follows as in~\cite[Lem.~4.1(iii)]{RoeSch99}.

\iref{i:p:W2Upsilon:5} is a consequence of~\iref{i:p:W2Upsilon:2},~\iref{i:p:W2Upsilon:3} and the closedness of~$G_r$ in~\iref{i:p:W2Upsilon:4}.

\iref{i:p:W2Upsilon:6}--\iref{i:p:W2Upsilon:9} are concluded as in the proof of~\cite[Lem.~4.1(v)--(vii)]{RoeSch99}.
\end{proof}
\end{prop}

For further purposes, let us expand the scope of the measurability statement in Proposition~\ref{p:W2Upsilon}\ref{i:p:W2Upsilon:8}.

\begin{cor}[Universal measurability of point-to-set distance]\label{c:MeasurabilitydU}
The map~$\mssd_\dUpsilon(\emparg,\Xi)$ is $\A_\mrmv(\Ed)^*$-measurable for every~$\Xi\in\Bo{\T_\mrmv(\Ed)}$.
\begin{proof}
It suffices to show that the sub-level sets~$(\Xi)_r\eqdef \set{\gamma\in\dUpsilon: \mssd_\dUpsilon(\gamma,\Xi)\leq r}$ satisfy~$(\Xi)_r\in\A_\mrmv(\Ed)^*$.
Let~$G_r\subset \dUpsilon(X^\tym{2},\Ed^\otym{2})$ be as in~\eqref{p:W2Upsilon:0}, and note that
\begin{align*}
(\Xi)_r=\set{\pr^1_\pfwd\alpha: \alpha\in G_r\comma \pr^2_\pfwd\alpha\in\Xi}= \pr^1_\pfwd\ttonde{G_r\cap (\pr^2_\pfwd)^{-1}(\Xi)} \fstop
\end{align*}
Since~$\Xi\in\Bo{\T_\mrmv(\Ed)}$ and since~$G_r$ is $\T_\mrmv(\Ed^\otym{2})$-closed by Proposition~\ref{p:W2Upsilon}\ref{i:p:W2Upsilon:4}, hence $\Bo{\T_\mrmv(\Ed^\otym{2})}$-measur\-able, the set~$G_r\cap (\pr^2_\pfwd)^{-1}(\Xi)$ is $\Bo{\T_\mrmv(\Ed^\otym{2})}$-measurable, and thus a Borel subset of the Polish space $\dUpsilon(X^\tym{2},\Ed^\otym{2})$.
Since~$\pr^1_\pfwd\colon \dUpsilon(X^\tym{2},\Ed^\otym{2})\to \dUpsilon$ is $\T_\mrmv(\Ed^\otym{2})/\T_\mrmv(\Ed)$-continuous by Proposition~\ref{p:W2Upsilon}\ref{i:p:W2Upsilon:4}, the set~$(\Xi)_r$ is a continuous image of a Borel subset in a Polish space, thus a Suslin set, and therefore universally measurable (e.g.~\cite[Thm.~21.10]{Kec95}).
\end{proof}
\end{cor}

\begin{prop}[Further properties of~$\mssd_\dUpsilon$, cf.~{\cite[Lem.~4.2]{RoeSch99}}]\label{p:W2Upsilon2}
Let~$U\in\Ed$ be \emph{open} and so that $U^\complement\eqdef X\setminus U$ has non-empty interior.
For~$\gamma\in\dUpsilon$ set
\begin{align}\label{eq:W2Upsilon:0}
\Lambda_{\gamma,U}\eqdef \set{\eta\in\dUpsilon:\eta_U=\gamma_U} \qquad \text{and} \qquad \rho_{\gamma,U}\eqdef \mssd_\dUpsilon(\emparg,\Lambda_{\gamma,U}) \fstop
\end{align}
Then,

\begin{enumerate}[$(i)$]%\setcounter{enumi}{9}
\item\label{i:p:W2Upsilon:10} the set~$\Lambda_{\gamma,U}$ is closed;
\item\label{i:p:W2Upsilon:11} $\rho_{\gamma,U}(\eta)<\infty$ if and only if~$\eta X\geq \gamma U$;
\item\label{i:p:W2Upsilon:12} $\rho_{\gamma,U}\colon\dUpsilon\rar [0,\infty]$ is $\T_\mrmv(\Ed)$-continuous;
\item\label{i:p:W2Upsilon:13} $\rho_{\gamma, U}(\eta) \uparrow \mssd_\dUpsilon(\gamma,\eta)$ as~$U\uparrow X$ for all~$\gamma,\eta\in\dUpsilon$;
\item\label{i:p:W2Upsilon:14} $\gamma\mapsto \rho_{\gamma,U}(\eta)$ is lower semi-continuous for every fixed~$\eta\in\dUpsilon$.
\end{enumerate}

\begin{proof} 
\iref{i:p:W2Upsilon:10} Since~$U$ is open, $\gamma_U=\eta_U$ if and only if~$f^\trid\gamma=f^\trid\eta$ for every~$f\in\Cz(\Ed)$ with~$\supp f\subset U$ by Remark~\ref{r:TLS}\iref{i:r:TLS:7}. Thus, $\Lambda_{\gamma,U}$ is closed, since it coincides with the intersection of the closed sets
\begin{align*}
\Lambda_{\gamma,U}=\bigcap_{\substack{f\in\Cz(\Ed)\\ \supp f\subset U}} (f^\trid)^{-1} (f^\trid\gamma) \fstop
\end{align*}

\iref{i:p:W2Upsilon:11} Firstly, we show that for every~$m\in \N_0$ there exists~$E_m\in\Ed$, with~$U\subset E_m$, so that~$E_m\setminus U$ contains at least $m$ points.
Recall that~$(X,\T)$ is perfect by Remark~\ref{r:TLS}\ref{i:r:TLS:1.5}.
Since an open subset of a perfect space does not have isolated points, the $\T$-open set~$\interior(U^\complement)$ is at least countable and the claim follows by Definition~\iref{d:LS}\iref{i:d:LS:2}. 

Now, let~$\gamma\in\dUpsilon$ be fixed. If~$\eta X <\gamma U$, then~$\mssd_\dUpsilon(\eta,\gamma)=\infty$ by Proposition~\ref{p:W2Upsilon}\iref{i:p:W2Upsilon:1}, and the `only if' part easily follows.
For the converse implication, assume~$\eta X\geq \gamma U$ and let $\gamma = \sum_{i=1}^{\gamma X} \delta_{y_i}$, and $\eta = \sum_{i=1}^{\eta X} \delta_{x_i}$.
Without loss of generality, up to relabeling the atoms of~$\gamma$ and~$\eta$, we may assume that~$\seq{y_i}_{i\leq \gamma U}\subset U$ and~$\seq{x_i}_{i\leq \eta U}\subset U$.
Set $\gamma'=\sum_{i=1}^{\gamma U}\delta_{y_i}+ \sum_{i=1}^{\eta U} \delta_{z_i} + \eta - \sum_{i=1}^{\gamma U + \eta U} \delta_{x_i}$, where $\seq{z_i}_{i \leq \eta U} \subset U^c$ and $z_i \neq z_j$ whenever $i \neq j$ (existence by the previous claim).
Since~$\gamma'_U = \gamma_U$ 
\begin{align*}
\rho_{\gamma, U}(\eta) = \rho_{\gamma', U}(\eta)  \leq \mssd_\dUpsilon(\gamma',\eta)\leq  \sqrt{\sum_{i=1}^{\gamma U} \mssd^2(y_i, x_i) + \sum_{i=1}^{\eta U} \mssd^2(z_i, x_{\gamma U+i}) }<\infty\comma%\\
\end{align*}
which concludes the proof.

\iref{i:p:W2Upsilon:12}--\iref{i:p:W2Upsilon:14} are concluded as in the proof of~\cite[Lem.~4.2(ii)--(iv)]{RoeSch99} having care to replace the set~$B_r$ there by~$U$ as in~\iref{i:p:W2Upsilon:11}, resp.\ a continuous compactly supported function~$f$ with~$\supp f\subset B_r$ with a function~$f\in\Cz(\Ed)$ with~$\supp f\subset U$.
\end{proof}
\end{prop}

As a consequence of the previous proposition, we also obtain that configuration spaces are extended metric-topological spaces.

\begin{prop}\label{p:ConfigEMTS}
Let~$(\mcX,\mssd)$ be a metric local structure.
Then,
\begin{align*}
\ttonde{\dUpsilon,\T_\mrmv(\Ed),\mssd_\dUpsilon}
\end{align*}
is an extended metric-topological space in the sense of Definition~\ref{d:AES}.
\begin{proof}
Let~$x_0\in X$ be fixed, and, for each~$r>0$, set~$B_r\eqdef B^\mssd_r(x_0)\in\Ed$.
Note that~$B_r\in\T$ as well, since~$\mssd$ metrizes~$\T$.
For each~$\gamma\in\dUpsilon$ further let~$\rho_{\gamma,r}\eqdef \rho_{\gamma, B_r}$ be defined as in~\eqref{eq:W2Upsilon:0}.
Now, the function~$\mssd_{\gamma,r}\colon \dUpsilon^\tym{2}\rar[0,\infty]$ defined as
\begin{align*}
\mssd_{\gamma,r}\colon (\eta,\xi)\longmapsto \abs{\rho_{\gamma,r}(\eta)-\rho_{\gamma,r}(\xi)}
\end{align*}
is $\T_\mrmv(\Ed)$-continuous by Pro\-po\-sition~\ref{p:W2Upsilon2}\iref{i:p:W2Upsilon:12}, finite everywhere on~$\dUpsilon^\sym{\infty}(\Ed)^\tym{2}$ (Prop.~\ref{p:W2Upsilon2}\iref{i:p:W2Upsilon:11}), $1$-Lipschitz w.r.t.\ $\mssd_\dUpsilon$ by triangle inequality for~$\mssd_\dUpsilon$, and a pseudo-distance for each~$\gamma$ and~$r$, by triangle inequality for the Euclidean distance on~$\R$.
By the $1$-Lipschitz property,~$\mssd_{\gamma,r}\leq \mssd_\dUpsilon$ everywhere on~$\dUpsilon^\tym{2}$ for every~$\gamma\in\dUpsilon$.
Together with Proposition~\ref{p:W2Upsilon2}\iref{i:p:W2Upsilon:13}, the latter inequality implies that
\begin{align*}
\mssd_\dUpsilon(\eta,\xi)\geq \sup_{r>0}\sup_{\gamma\in\dUpsilon} \mssd_{\gamma,r}(\eta,\xi)\geq \sup_{r>0} \mssd_{\eta,r}(\eta,\xi)=\mssd_\dUpsilon(\eta,\xi)\comma \qquad \eta,\xi\in\dUpsilon \fstop
\end{align*}
In particular, since~$\mssd_\dUpsilon$ separates points, the uniformity~$\UP$ generated by the family of $\T_\mrmv(\Ed)$-continuous pseudo-distances~$\set{\mssd_{\gamma,r}}_{\gamma\in\dUpsilon, r>0}$ separates points as well, and it is therefore Hausdorff and generating~$\T_\mrmv(\Ed)$.
\end{proof}
\end{prop}

\paragraph{Metric properties} In the following, let~$\mssd_\Upsilon$ be the restriction of~$\mssd_\dUpsilon$ to~$\Upsilon^{\tym{2}}$.
We collect some metric properties of the spaces~$(\Upsilon,\mssd_\Upsilon)$.

\begin{prop}[Metric properties of~$\dUpsilon$]\label{p:EH}
The following assertions hold:
\begin{enumerate}[$(i)$]
\item\label{i:p:EH:1} $\ttonde{\dUpsilon,\mssd_\dUpsilon}$ is (isometric to) the completion of~$\ttonde{\Upsilon,\mssd_\Upsilon}$ and~$\ttonde{\dUpsilon^\sym{N},\mssd_\dUpsilon}$ is (isometric to) the completion of~$\ttonde{\Upsilon^\sym{N}(\Ed),\mssd_\Upsilon}$ for each~$N\in\overline\N_0$;

\item\label{i:p:EH:2} \cite[Cor.~2.7]{ErbHue15} if~$(X,\mssd)$ is a geodesic metric space, then so is~$\ttonde{\dUpsilon,\mssd_\dUpsilon}$;
\end{enumerate}
\begin{proof}
Since~$\mssd$ completely metrizes~$\T$, and since~$\msE=\Ed$, a proof of~\iref{i:p:EH:1} is straightforward.
Assertion~\iref{i:p:EH:2} is proven in~\cite{ErbHue15} for~$(X,\mssd)$ a Riemannian manifold. Its proof carry over \emph{verbatim} to our setting. 
\end{proof}
\end{prop}

\paragraph{$L^2$-transportation distance and topology in the finite-volume case}
Let us collect here some result about the interplay between the $\Ed$-vague topology~$\T_\mrmv(\msE)$ and the $L^2$-transportation distance~$\mssd_\dUpsilon$ in the case when~$X=E\in\Ed$ is a $\mssd$-bounded set.
In this case, the extended metric space~$\ttonde{\dUpsilon(E),\mssd_\dUpsilon}$ has only countably many $\mssd_\dUpsilon$-accessible components, which makes the analysis of the metric-topological properties much easier than in the general case.

In particular the following properties hold true.

\begin{prop}[Topology over finite configurations]\label{p:DistTopE}
Let~$(\mcX,\mssd)$ be a metric local structure. Then, for every $\T$-closed ($\mssd$-bounded)~$E\in\Ed$,
\begin{enumerate*}[$(i)$]
\item\label{i:p:DistTopE:1} a sequence~$\seq{\gamma_n}_n$ of configurations in~$\dUpsilon(E)$ $\Ed$-vaguely converges to~$\gamma\in\dUpsilon(E)$ if and only if~$\nlim \mssd_\dUpsilon(\gamma_n,\gamma)=0$;
\item\label{i:p:DistTopE:2} a function~$u\colon \dUpsilon(E)\to\R$ is $\T_\mrmv(\Ed)$-continuous if and only if it is $\mssd_\dUpsilon$-continuous.
\end{enumerate*}

\begin{proof}
The backwards implication in~\iref{i:p:DistTopE:1} is a consequence of Proposition~\ref{p:ConfigEMTS} and Remark~\ref{r:EMetTop}\iref{i:r:EMetTop:1}.
In order to show the converse implication, let~$\seq{\gamma_n}_n\subset \dUpsilon(E)$ be $\T_\mrmv(\msE)$-vaguely convergent to~$\gamma\in\dUpsilon(E)$.
Then~$\gamma_n E=\gamma E\defeq m$ for all sufficiently large~$n$, and
\begin{equation*}
\nlim \int_E f\diff\frac{\gamma_n}{m}= \int_E f \diff\frac{\gamma}{m} \comma \qquad f\in\Cz(\Ed)\restr_E\fstop
\end{equation*}
Since~$\Cz(\Ed)\restr_E$ is a convergence-determing class for the narrow topology on probability measures on~$E$ by Remark~\ref{r:EMetTop}\iref{i:r:EMetTop:2}, we have that~$\seq{\gamma_n/m}_n$ narrowly converges to~$\gamma/m$.
Since~$E$ is $\T$-closed,~$(E,\mssd)$ is a bounded complete and separable metric space, hence the narrow convergence of probability measures on~$E$ coincides with the convergence in the $L^2$-Wasserstein distance~$W_2$ on~$\msP(E)$, see e.g.~\cite[Cor.~6.13]{Vil09}.
As a consequence,
\begin{equation*}
\nlim \mssd_\dUpsilon(\gamma_n,\gamma)= \nlim m \, W_2(\gamma_n/m, \gamma/m) = 0\comma
\end{equation*}
and the conclusion follows.
\iref{i:p:DistTopE:2} is an immediate consequence of~\iref{i:p:DistTopE:1}.
\end{proof}
\end{prop}

\subsubsection{\texorpdfstring{$L^2$}{L2}-transportation distance and labeling maps}\label{sss:dUpsilonLb}
Let~$\mcX$ be a metric local structure,~$\mbfX$ be defined by~\eqref{eq:d:Labelings:1}, and set
\begin{align*}
\bmssd(\mbfx,\mbfy)\eqdef& 
\begin{cases} 0 & \text{if } \mbfx=\mbfy=\emp
\\
\mssd_\tym{N}(\mbfx, \mbfy) &\text{if } \mbfx,\mbfy\in X^\tym{N}\comma \quad N\in \overline\N_1
\\
+\infty & \text{otherwise}
\end{cases}
\fstop
\end{align*}

\begin{prop}[Radially isometric labeling maps]\label{p:PoissonLabeling}
Fix~$\eta\in \dUpsilon$,~$\mbfx\in\Lb^{-1}(\eta)$, and set $N\eqdef \eta X$. Then,
\begin{enumerate}[$(i)$]
\item\label{i:p:PoissonLabeling:1} for every~$\gamma\in B^\dUpsilon_\infty(\eta)$ there exists a labeling~$\mbfy=\mbfy(\mbfx)\in\Lb^{-1}(\gamma)$ such that
\begin{align}\label{eq:p:W2UpsilonLabeling:1}
\mssd_\dUpsilon(\eta,\gamma)=\bmssd(\mbfx,\mbfy)\semicolon
\end{align}

\item\label{i:p:PoissonLabeling:2} there exists a labeling map~$\lb_\mbfx\colon \dUpsilon\rar \mbfX_\locfin(\Ed)$ satisfying
\begin{equation}\label{eq:p:W2Upsilon:1}
\begin{aligned}
\mssd_\dUpsilon(\emparg,\eta)=&\mssd_\tym{N}\ttonde{\lb_\mbfx(\emparg),\mbfx} \quad \text{on} \quad \dUpsilon^\sym{N}\comma
\\
\mssd_\dUpsilon\ttonde{\emparg, \Lb(\mbfx^\asym{n})}=& \mssd_\tym{n}\ttonde{\lb_\mbfx(\emparg),\mbfx^\asym{n}} \quad \text{on} \quad \dUpsilon^\sym{N}\comma\qquad n< N \fstop
\end{aligned}
\end{equation}

\item\label{i:p:PoissonLabeling:3} $\mssd_\dUpsilon(\emparg,\eta)=\mssd_\tym{N}(\emparg,\mbfx)\circ \lb_\mbfx$ on~$\dUpsilon$;

\item\label{i:p:PoissonLabeling:4} if~$N<\infty$, then $\mssd_\dUpsilon(\emparg,\eta)\geq \mssd_\tym{n}\ttonde{\emparg,\mbfx^\asym{n}}\circ \tr^n \circ \lb_\mbfx$ on~$B_\infty^\dUpsilon(\eta)$ for every~$n\leq N$;

\item\label{i:p:PoissonLabeling:5} if~$N=\infty$, then $\mssd_\dUpsilon(\emparg,\eta)=\nlim \mssd_\tym{n}\ttonde{\emparg, \mbfx^\asym{n}}\circ \tr^n\circ \lb_\mbfx$ on~$B_\infty^\dUpsilon(\eta)$;

\item\label{i:p:PoissonLabeling:6} for every~$n\leq N$
\begin{align}\label{eq:SymmetricDistance}
\mssd_\dUpsilon\ttonde{\Lb(\mbfx^\asym{n}),\Lb(\mbfy^\asym{n})}= \inf_{\sigma\in\mfS_n}\mssd_\tym{n}\ttonde{\mbfx^\asym{n},{\mbfy^\asym{n}}_\sigma} \comma\qquad \mbfy^\asym{n}\in \cl \ttonde{(\tr^n\circ \lb_\mbfx)(\dUpsilon)}\subset X^\tym{n} \fstop
\end{align}
\end{enumerate}

\begin{proof}

\iref{i:p:PoissonLabeling:1}~is a consequence of Proposition~\ref{p:W2Upsilon}\iref{i:p:W2Upsilon:7}, cf.~\cite[Eqn.~(4.2)]{ErbHue15}.

\iref{i:p:PoissonLabeling:2}
Set~$\mbfX_\locfin(\Ed)_\mbfx\eqdef\set{\mbfy\in\mbfX_\locfin(\Ed):\mssd_\dUpsilon\ttonde{\Lb(\mbfy),\eta}=\bmssd(\mbfy,\mbfx)}$, and note that the restriction of~$\Lb$ to~$\mbfX_\locfin(\Ed)_\mbfx$ is surjective by~\iref{i:p:PoissonLabeling:1}.
The rest of the proof thus follows as in Proposition~\ref{p:Selection} if we show that~$\mbfX_\locfin(\Ed)_\mbfx$ is a Suslin space.
Recall that~$\mbfX_\locfin(\Ed)$ is a Suslin space as in the proof of Proposition~\ref{p:Selection}.
We show that~$\mbfX_\locfin(\Ed)_\mbfx$ is a Borel subset of~$\mbfX_\locfin(\Ed)$, hence a Suslin space by e.g.~\cite[Cor.~6.6.7]{Bog07}.
In order to show that~$\mbfX_\locfin(\Ed)_\mbfx$ is Borel, it suffices to use the $\Bo{\T_\mrmv(\Ed)}$-measurability of~$\mssd_\dUpsilon(\emparg,\eta)$ (Prop.~\ref{p:W2Upsilon}\iref{i:p:W2Upsilon:8}), the $\Bo{\mbfX_\locfin(\Ed)}/\Bo{\T_\mrmv(\Ed)}$-measurability of~$\Lb$ (Lem.~\ref{l:MeasurabilityL}), and the~$\Bo{\mbfX_\locfin(\Ed)}$-measurability of~$\bmssd(\emparg,\mbfx)$.

\iref{i:p:PoissonLabeling:3}--\iref{i:p:PoissonLabeling:5} are a straightforward consequence of Proposition~\ref{p:W2Upsilon}. 

The inequality~``$\geq$'' in~\iref{i:p:PoissonLabeling:6} is trivial. Firstly, let us prove the reverse inequality when~$\mbfy= \lb_\mbfx(\gamma)$ for some~$\gamma\in\dUpsilon$ and~$\mbfy^\asym{n}=\tr^n(\mbfy)$. Let~$\cpl$ be an optimal matching for the pair~$(\eta,\gamma)$ as in Proposition~\ref{p:W2Upsilon}\iref{i:p:W2Upsilon:5}. By an adaptation of Step~2 in the proof of~\cite[Thm.~5.10(ii)]{Vil09}, the matching~$\cpl$ is concentrated on the $\mssd^2$-cyclically monotone set
\begin{align*}
C\eqdef \bigcup_{i\geq 1}\tset{(x_i, \lb_\mbfx(\gamma)_i)}\subset X^{\times 2}\fstop
\end{align*} 
That is, for every~$n\leq \eta X$ and every~$\set{(x_{i_1},y_{i_1}),\dotsc, (x_{i_n},y_{i_n})}\subset C$ with~$i_j\leq N$ for~$j\leq n$, it holds that
\begin{align*}
\sum_j^n \mssd^2(x_{i_j},y_{i_j})\leq \sum_j^n \mssd^2(x_{i_j},y_{i_{j+1}})\comma \qquad y_{i_{n+1}}\eqdef y_{i_1} \fstop
\end{align*}
Equivalently, for every $n\leq N$ and every~$\sigma\in\mfS_n$, it holds that
\begin{align*}
\mssd_\tym{n}^2\ttonde{\mbfx^\asym{n},\mbfy^\asym{n}} \leq \mssd_\tym{n}^2\ttonde{\mbfx^\asym{n},{\mbfy^\asym{n}}_\sigma}\comma \qquad \set{(x_1,y_1),\dotsc, (x_n,y_n)}\subset C \comma
\end{align*}
which implies the inequality~``$\leq$'' in~\eqref{eq:SymmetricDistance}.
Finally, the same inequality holds for~$\mbfy^\asym{n}$ in the $\T^{\tym{n}}$-closure of~$(\tr^n\circ \lb_\mbfx)\ttonde{\dUpsilon}$ by continuity of~$\mssd_\tym{n}^2$.
\end{proof}
\end{prop}

As a consequence of Proposition~\ref{p:PoissonLabeling}\iref{i:p:PoissonLabeling:1}, we obtain the following.

\begin{prop}\label{p:LbLipschitz}
The map~$\Lb\colon \mbfX^\asym{\infty}_\locfin(\Ed)\longrar \dUpsilon$ is $\mssd_\tym{\infty}/\mssd_\dUpsilon$-short, i.e.
\begin{equation*}
\mssd_\dUpsilon\ttonde{\Lb(\mbfx),\Lb(\mbfy)}\leq \mssd_\tym{\infty}(\mbfx,\mbfy) \comma\qquad \mbfx,\mbfy\in\mbfX^\asym{\infty}_\locfin(\Ed) \fstop
\end{equation*}
\end{prop}

\subsubsection{Metric properties of functions}
It is not straightforward to exhibit non-trivial $\mssd_\dUpsilon$-Lipschitz functions on the configuration space $\ttonde{\dUpsilon,\mssd_\dUpsilon}$.
Two opposite examples of this fact are that
\begin{enumerate*}[$(a)$]
\item\label{i:r:Lipschitz:1.1} there exist uncountably many~$\gamma\in\dUpsilon$ so that $\mssd_\dUpsilon(\emparg, \gamma)\wedge 1\equiv \car$ $\PP_\mssm$-a.e.;
\item\label{i:r:Lipschitz:1.2} there exist cylinder functions, or exponential cylinder functions, that are \emph{not} $\mssd_\dUpsilon$-Lipschitz.
\end{enumerate*}
\iref{i:r:Lipschitz:1.1}~holds since we have~$\PP_\mssm B^\dUpsilon_\infty(\gamma)=0$ for uncountably many~$\gamma\in\dUpsilon$. For~\iref{i:r:Lipschitz:1.2} we provide a simple Example~\ref{e:NonLipCyl} of such a non-Lipschitz cylinder function on~$\dUpsilon(\R)$.

\begin{ese}[A non-Lipschitz cylinder function]\label{e:NonLipCyl}
Consider the standard real line~$\R$, endowed with the one-dimension\-al Lebesgue measure~$\Leb^1$ and with the localizing ring generated by compact sets. On the configuration space~$\dUpsilon(\R)$, let~$\PP_1\eqdef\PP_{\Leb^1}$.
For~$\eps>0$ set~$A_{\eps}\eqdef [-2,-2+\eps]\cup [2-\eps,2]$ and note that
\begin{align}
\Xi_{n,\eps}\eqdef \set{\gamma \in \dUpsilon(\R): \gamma A_\eps= n\comma \gamma\ttonde{[-3,3] \setminus A_\eps}=0}
\end{align}
has positive $\PP_1$-measure for every~$n\in \N$ and every~$\eps>0$, e.g.~by~\eqref{eq:AKR2.7}.

Now, let~$f\in\Cc^\infty(\R)$ be satisfying~$f\restr_{[1,2]}\equiv 1-\abs{\emparg-1}$,~$f\restr_{[-3,3]^\complement}\equiv 0$, and~$f(x)=f(-x)$. 
Set~$u\eqdef \arctan \circ f^\trid \in\Cyl{\Cc^\infty(\R)}$.
Then,~$f^\trid\gamma=f\restr_{A_\eps}^\trid\gamma \leq \eps n$ for every~$\gamma\in \Xi_{n,\eps}$, and
\begin{align*}
\SF{\dUpsilon}{\PP_1}(u)(\gamma)=\frac{(\abs{\nabla f}^2)^\trid \gamma}{\ttonde{1+(f^\trid\gamma)^2}^2}\geq \frac{\car_{A_\eps}^\trid\gamma}{\ttonde{1+(f\restr_{A_\eps}^\trid\gamma)^2}^2}\geq \frac{n}{(1+\eps^2 n^2)^2} \comma \qquad \gamma\in \Xi_{n,\eps}\fstop
\end{align*}

In particular, choosing~$\eps_n\eqdef 1/n$, we have~$\SF{\dUpsilon}{\PP_1}(u)(\gamma) \geq n/4$ for every~$\gamma\in\Xi_{n,\eps_n}$, thus~$\SF{\dUpsilon}{\PP_1}(u)$ is \emph{not} in~$L^\infty(\PP_\mssm)$.
Argue by contradiction that~$u$ were $\mssd_\dUpsilon$-Lipschitz.
By the Rademacher Theorem for configuration spaces over smooth spaces shown in~\cite[Thm.~1.3(i)]{RoeSch99}, we have that~$\PP_1$-$\esssup \SF{\dUpsilon}{\PP_1}(u)\leq \Li[\mssd_\dUpsilon]{u}$, a contradiction.
\end{ese}

In order to construct a $\PP_\mssm$-non-trivial example of a measurable Lipschitz function, let~$U\subset X$ be open, and so that~$\mssm U>0$ and~$\mssm U^\complement >0$.
Then the set~$\Lambda_{\gamma, U}$ defined in~\eqref{eq:W2Upsilon:0} satisfies~$\PP_\mssm \Lambda_{\gamma, U}>0$ and~$\PP_\mssm \Lambda_{\gamma, U}^\complement>0$.
As a consequence, the function~$\rho_{\gamma, U}\eqdef \mssd_\dUpsilon(\emparg, \Lambda_{\gamma, U})$ is an example of a non-trivial $\mssd_\dUpsilon$-Lipschitz function.

Since~$\mssd_\dUpsilon$ is not $\T_\mrmv(\Ed)$-continuous, it is not true in general that $\mssd_\dUpsilon$-Lipschitz functions on~$\dUpsilon$ are Borel measurable.
In fact, one can show that there exist bounded $\mssd_\dUpsilon$-Lipschitz functions that are not $\QP$-measurable for every Borel probability measure~$\QP$ on~$\dUpsilon$ (cf.~\cite[Ex.~3.4]{LzDSSuz20}).
A good class of examples of Borel-measurable $\mssd_\dUpsilon$-Lipschitz functions is provided by the next result.

\begin{prop}[$\Bo{\T_\mrmv(\Ed)}$-measurability of McShane extensions]\label{p:BorelMcShane}
Let~$(\mcX,\mssd)$ be a metric local \linebreak structure.
Further let~$u \in \Cb(\T_\mrmv(\Ed))$, and~$\Lambda \subset \dUpsilon$ be $\T_\mrmv(\Ed)$-closed.
If the restriction~$u\restr_\Lambda$ of~$u$ to~$\Lambda$ is $\mssd_\dUpsilon$-Lipschitz, then the constrained McShane extension~$\overline{u}\colon \dUpsilon\rar \R$ of~$u\restr_\Lambda$ is $\Bo{\T_\mrmv(\Ed)}$-measurable.
\begin{proof}
It suffices to show that~$\overline{u}$ is $\T_\mrmv(\Ed)$-l.s.c., i.e.\ that~$\set{\overline{u}\leq r}$ is $\T_\mrmv(\Ed)$-closed for each real~$r$.
For simplicity, set~$L\eqdef\Li[\mssd_\dUpsilon]{u}$, and~$M\eqdef \sup u$.

\paragraph{Step~1}
Firstly, let us show that the infimum in the definition~\eqref{eq:McShane} of~$\overline{u}$ is achieved.
Indeed, for fixed~$\gamma\in\dUpsilon$, let~$\seq{\zeta_n}_n$ be an infimizing sequence in~\eqref{eq:McShane}, so that
\begin{align}\label{eq:BorelMcShane:1}
M \wedge \ttonde{u(\zeta_n)+L\, \mssd_\dUpsilon(\gamma,\zeta_n)}\leq \overline{u}(\gamma)+1/n \fstop
\end{align}
If~$\overline{u}(\gamma)=M$, we may choose~$\zeta_n=\zeta^\gamma$ for some fixed~$\Lambda \cap B^{\dUpsilon}_\infty(\gamma)^\complement$.
If otherwise~$\overline{u}(\gamma)<M$, then~$L\mssd_\dUpsilon(\gamma,\zeta_n)\leq 2M+1$, hence~$\seq{\zeta_n}_n\subset \Lambda \cap B^\dUpsilon_s(\gamma)$ for any~$s\geq (2M+1)/L$.
Since~$B^\dUpsilon_s(\gamma)$ is $\T_\mrmv(\Ed)$-compact by Proposition~\ref{p:W2Upsilon}\iref{i:p:W2Upsilon:8}, there exists a $\T_\mrmv(\Ed)$-cluster point~$\zeta^\gamma$ of~$\seq{\zeta_n}_n$, and~$\zeta^\gamma\in \Lambda$ since the latter set is $\T_\mrmv(\Ed)$-closed.
Letting~$n$ to infinity in~\eqref{eq:BorelMcShane:1}, the $\T_\mrmv(\Ed)$-continuity of~$u$ and the lower $\T_\mrmv(\Ed)$-semi-continuity of~$\mssd_\dUpsilon$ imply that~$\overline{u}(\gamma)=u(\zeta^\gamma)+ L\mssd_\dUpsilon(\gamma,\zeta^\gamma)$.
Summarizing, for each~$\gamma\in\dUpsilon$ there exists~$\zeta^\gamma\in \Lambda$ so that
\begin{align*}
\overline{u}(\gamma)=M \wedge \ttonde{u(\zeta^\gamma)+L\, \mssd_\dUpsilon(\gamma,\zeta^\gamma)} \comma
\end{align*}
whence
\begin{align*}
\set{\overline{u}\leq r}=\set{\gamma\in\dUpsilon: M\wedge\ttonde{u(\zeta^\gamma)+L\,\mssd_\dUpsilon(\gamma,\zeta^\gamma)}\leq r} \fstop
\end{align*}

\paragraph{Step~2}
Retaining the notation established in Proposition~\ref{p:W2Upsilon}, let
\begin{align*}
H_r=\set{\alpha\in\dUpsilon^\tym{2}: u(\pr^2_\pfwd\alpha)+L\cdot I(\alpha)^{1/2}\leq r}\fstop
\end{align*}
Further let~$G_r$ be as in~\eqref{p:W2Upsilon:0}, and note that~$H_r\subset G_s$ for~$s\geq (r+\sup\abs{u})^2/\Li[\mssd_\dUpsilon]{u}$.
As a consequence,~$H_r$ is a $\T_\mrmv(\Ed)$-closed set by Proposition~\ref{p:W2Upsilon}\iref{i:p:W2Upsilon:2} and~\iref{i:p:W2Upsilon:4}.

Since~$\pr^2_\pfwd$ is closed on~$G_s$ by Proposition~\ref{p:W2Upsilon}\iref{i:p:W2Upsilon:6}, and since~$H_r$ is $\T_\mrmv(\Ed^\otym{2})$-closed, the set~$\pr^2_\pfwd(H_r)$ is $\T_\mrmv(\Ed)$-closed.
Since~$\Lambda$ is $\T_\mrmv(\Ed)$-closed, the set~$\Lambda \cap\pr^2_\pfwd(H_r)$ is $\T_\mrmv(\Ed)$-closed as well.
Since~$H_r$ is $\T_\mrmv(\Ed^\otym{2})$-closed, and since~$\pr^2_\pfwd$ is continuous on~$G_s$ by Proposition~\ref{p:W2Upsilon}\iref{i:p:W2Upsilon:4}, the set
\begin{align*}
H_r \cap (\pr^2_\pfwd)^{-1}(\Lambda)= H_r\cap (\pr^2_\pfwd)^{-1}\ttonde{\Lambda \cap \pr^2_\pfwd (H_r)}
\end{align*}
is $\T_\mrmv(\Ed^\otym{2})$-closed.
Since~$\pr^1_\pfwd$ is closed on~$G_s$ by Proposition~\ref{p:W2Upsilon}\iref{i:p:W2Upsilon:6}, we conclude that
\begin{align*}
K_r\eqdef \pr^1_\pfwd \ttonde{H_r \cap (\pr^2_\pfwd)^{-1}(\Lambda)}=\set{\pr^1_\pfwd\alpha: \alpha\in H_r\comma \pr^2_\pfwd \alpha\in \Lambda}
\end{align*}
is $\T_\mrmv(\Ed)$-closed.

\paragraph{Step 3} In order to conclude the proof, it suffices to show that~$\set{\overline{u}\leq r}=K_r$ for each~$r\in\R$.
For each~$\gamma\in\set{\overline{u}\leq r}$, there exists by Proposition~\ref{p:W2Upsilon}\iref{i:p:W2Upsilon:5} an optimal matching~$\alpha\in\dUpsilon(\Ed^\otym{2})$ between~$\gamma=\pr^1_\pfwd\alpha$ and~$\zeta^\gamma=\pr^2_\pfwd\alpha$, thus satisfying
\begin{align*}
\overline{u}(\gamma)= M\wedge \tonde{u(\pr^2_\pfwd\alpha)+L\tonde{\int_{X^\tym{2}}\mssd^2(x,y)\diff\alpha(x,y)}^{1/2}}\comma
\end{align*}
which shows that~$\gamma\in K_r$.

Viceversa, let~$\gamma\in K_r$.
By definition of~$K_r$, there exists~$\alpha\in H_r\cap (\pr^2_\pfwd)^{-1}(\Lambda)$ so that~$\pr^1_\pfwd \alpha=\gamma$.
Therefore, we have that
\begin{align*}
u(\pr^2_\pfwd \alpha)+L\cdot I(\alpha)^{1/2}\leq&\ r
\\
M\wedge \ttonde{u(\pr^2_\pfwd \alpha)+L\cdot I(\alpha)^{1/2}}\leq&\ M\wedge r \leq r \fstop
\end{align*}
By definition of~$\overline{u}$, and since~$\pr^2_\pfwd \alpha\in \Lambda$,
\begin{align*}
\overline{u}(\gamma)=M\wedge \inf_{\zeta\in \Lambda} \ttonde{u(\zeta) + L\, \mssd_\dUpsilon(\gamma,\zeta)}\leq M\wedge \ttonde{u(\pr^2_\pfwd \alpha)+L\cdot I(\alpha)^{1/2}} \leq r\comma
\end{align*}
which shows that~$\gamma\in \set{\overline{u}\leq r}$, and thus concludes the proof.
\end{proof}
\end{prop}

Whereas cylinder functions are typically not $\mssd_\dUpsilon$-Lipschitz on~$\dUpsilon$, as shown by Example~\ref{e:NonLipCyl}, they are locally Lipschitz in the following weak sense.

\begin{thm}[Local Lipschitz property of cylinder functions]\label{t:CylinderLocLip}
Let~$(\mcX,\cdc,\mssd)$ be an \MLDS with $\Dz\subset \bLip(\mssd)$, and fix~$u=F\circ \mbff^\trid\in\Cyl{\Dz}$ with~$\mbff=\seq{f_i}_{i\leq k}$. Further let~$E\in\Ed$ be so that~$\cup_{i\leq k}\supp f_i\subset E$.
Then,~$u$ is $\mssd_\dUpsilon$-Lipschitz on
\begin{align}\label{eq:t:CylinderLocLip:0}
\Xi^\sym{\infty}_{\leq n}(E)\eqdef \set{\gamma\in\dUpsilon^\sym{\infty}:\gamma E\leq n}
\end{align}
for each~$n\in\N$.
\begin{proof}
Fix~$\gamma,\eta\in \Xi^\sym{\infty}_{\leq n}(E)$.
By Proposition~\ref{p:PoissonLabeling}\iref{i:p:PoissonLabeling:1} there exist~$\mbfx\in \Lb^{-1}(\gamma)$ and $\mbfy\in \Lb^{-1}(\eta)$ so that $\mssd_\dUpsilon(\gamma,\eta)=\mssd_\tym{\infty}(\mbfx,\mbfy)$.
Up to permutation of coordinates for both~$\mbfx$ and~$\mbfy$, we may additionally choose them in such a way that~$x_p\in E$ for all~$p\leq n$, and~$y_p\notin E$ for all~$p>2n$.
Since~$\gamma\in \Xi^\sym{\infty}_{\leq n}(E)$, we deduce that~$x_p\notin E$ for all~$p>n$.
Now, letting~$u^\sym{2n}\eqdef \Lb^*u\restr_{X^\tym{2n}}$, it follows from the definition of~$u$ that
\begin{align*}
\abs{u(\gamma)-u(\eta)}=&\ \abs{u^\sym{2n}(\mbfx^\asym{2n})-u^\sym{2n}(\mbfy^\asym{2n})}
\\
\leq &\ \Li[\R^k]{F}\cdot \quadre{\sum_{i=1}^k \abs{f_i^\trid \Lb(\mbfx^\asym{2n})-f_i^\trid \Lb(\mbfy^\asym{2n})}^2}^{1/2}
\\
\leq&\ \Li[\R^k]{F}\cdot \sqrt{k} \max_{i\leq k} \quadre{\abs{\sum_{p=1}^{2n} f_i(x_p)- f_i(y_p)}^2}^{1/2}
\\
\leq&\ \Li[\R^k]{F}\cdot \sqrt{2n} \sqrt{k} \max_{i\leq k} \Li[\mssd]{f_i}\, \mssd_\tym{2n}(\mbfx^\asym{2n},\mbfy^\asym{2n}) %\tonde{\sum_{p=1}^{2n}\mssd(x_p,y_p)^2}^{1/2}
\\
=&\ \sqrt{n}\, C_u \,\mssd_\tym{2n}(\mbfx^\asym{2n},\mbfy^\asym{2n}) \leq \sqrt{n}\, C_u\, \mssd_\tym{\infty}(\mbfx,\mbfy) = \sqrt{n}\, C_u\, \mssd_\dUpsilon(\gamma,\eta)
\end{align*}
for some constant~$C_u\geq 0$ only depending on~$u$.
\end{proof}
\end{thm}

\section{Interplay between the analytic and the geometric structure}\label{s:Interplay}
In this section, we undertake a careful analysis of the metric and metric-measure properties of the extended metric-measure space~$\ttonde{\dUpsilon, \mssd_\dUpsilon,\QP}$, and of their interplay with the Dirichlet-space structure provided by either Proposition~\ref{p:MRLifting} or Theorem~\ref{t:ClosabilitySecond}, with intrinsic distance~$\mssd_\QP$.
Our goal is to show that these two structures coincide. 
In particular, for the class of measures~$\QP$ and of spaces~$(\mcX,\cdc,\mssd)$ discussed in~\S\ref{s:Intro}, we shall identify the energy functionals
\begin{align}\label{eq:IdentificationCheeger}
\ttonde{\Ch[\mssd_\dUpsilon,\QP],\dom{\Ch[\mssd_\dUpsilon,\QP]}}=\ttonde{\EE{\dUpsilon}{\QP},\dom{\EE{\dUpsilon}{\QP}}} \comma
\end{align}
as well as the extended distances
\begin{align}\label{eq:IdentificationDistance}
\mssd_\dUpsilon=\mssd_\QP\fstop
\end{align}
As a byproduct of the former identification, we may conclude that:
%\begin{itemize}[--]
%\item 
$\Ch[\mssd_\dUpsilon,\QP]$ is a quadratic functional, and thus defines a Dirichlet form by polarization;
%\item 
$\mssd_\QP$ is an extended \emph{distance} (as opposed to: pseudo-distance) on~$\dUpsilon$.
%\end{itemize}

\smallskip

As it turns out, this analysis is particularly delicate, mainly due to the following facts:
\begin{itemize}
\item $\ttonde{\EE{\dUpsilon}{\QP},\dom{\EE{\dUpsilon}{\QP}}}$ is a \emph{quasi}-regular (\emph{non}-regular) Dirichlet space, since~$\ttonde{\dUpsilon,\T_\mrmv(\Ed)}$ is \emph{not} locally compact;
\item $\mssd_\dUpsilon$ (or, equivalently,~$\mssd_\QP$) is an \emph{extended} distance, not generating the $\Ed$-vague topology.
\end{itemize}

In order to address these difficulties, we shall make extensive use of results in \S\S\ref{s:Analytic}--\ref{s:Geometry}, and of the study of intrinsic distances of quasi-regular strongly local Dirichlet spaces undertaken by the authors in~\cite{LzDSSuz20} (partly recalled in~\S\ref{sss:RadStoL}).

\subsection{The Rademacher property}\label{sss:Rademacher}
We present here one of the main results in this paper, the Rademacher property for the form~$\EE{\dUpsilon}{\QP}$ on~$\ttonde{\dUpsilon,\mssd_\dUpsilon}$ under Assumption~\ref{ass:AC}.
Before dwelling into the details, let us note as a guideline the following three heuristic facts:
\begin{enumerate}[$(a)$]
\item\label{i:HeuristicRad1} The configuration space~$\ttonde{\dUpsilon^\sym{\infty},\mssd_\dUpsilon}$ is a $1$-Lipschitz image, via~$\Lb$ (Proposition~\ref{p:LbLipschitz}), of~$\mbfX^\asym{\infty}_\locfin(\Ed)$, and, for every~$\gamma\in \dUpsilon$, there exists a radial isometry around~$\gamma$ inverting the map~$\Lb$ (Prop.~\ref{p:PoissonLabeling}\iref{i:p:PoissonLabeling:2}).
It is therefore not surprising that a Lipschitz function~$u$ on~$\ttonde{\dUpsilon^\sym{\infty},\mssd_\dUpsilon}$ may be studied via its counterpart~$\Lb^*u$ on~$\ttonde{\mbfX^\asym{\infty}_\locfin(\Ed),\mssd_{\tym{\infty}}}$.

\item\label{i:HeuristicRad2} If~$(\mcX,\cdc,\mssd)$ satisfies~$(\Rad{\mssd}{\mssm})$, then the product space~$(\mcX^\otym{n},\cdc^\otym{n},\mssd_\tym{n})$ satisfies~$(\Rad{\mssd_\tym{n}}{\mssm^\otym{n}})$ (Thm.~\ref{t:Tensor}).

\item\label{i:HeuristicRad3} The condition defining~$\mbfX^\asym{\infty}_\locfin(\Ed)\subset X^{\tym{\infty}}$ is a `tail event', in the sense that~$\tr^n\colon \mbfX^\asym{\infty}_\locfin(\Ed)\rar X^{\tym{n}}$ is surjective for every~$n\in \N_1$.
As a consequence, statements obtained by finite-dimensional approximation of~$X^{\tym{\infty}}$ by~$X^{\tym{n}}$ are likely to hold for~$\mbfX^\asym{\infty}_\locfin(\Ed)$ as well.
\end{enumerate}

\paragraph{Sketch of proof} Our proof relies in an essential way on many of the results previously established, in particular:
\begin{itemize}
\item on the well-posedness results for forms on infinite-product spaces, established in~\S\ref{ss:AnalyticForms};
\item on the identification of said forms and on their approximation in the sense of Kuwae--Shioya convergence, established in~\S\ref{ss:IdentificationFormsProduct};
\item on the tensorization results for minimal weak upper gradients, minimal relaxed slopes, and for the Rademacher property, established in~\S\ref{sss:TensorizationConseq};
\item on the metric properties of the map~$\Lb$ and of its radially isometric right inverses, established in~\S\ref{sss:dUpsilonLb}.
\end{itemize}

Now, let~$u$ be a $\mssd_\dUpsilon$-Lipschitz function.
We want to show that~$u\in\dom{\EE{\dUpsilon}{\QP}}$.
Firstly, we will show by~\iref{i:HeuristicRad1} above that~$\Lb^*u$ is a pre-Sobolev function in~$\preW(\Ed)$ (Dfn.~\ref{d:preSobolev}), and that~$\norm{\Lb^*u}_{\llb}\leq \norm{\Lb^*u}_{\bLip}$.
Secondly, since~$\Lb^*u\in\preW(\Ed)$, we may regard~$\Lb^*u$ as an element~$\class[n]{\Lb^*u}$ of each of the abstract completions~$K_n$ defined in~\eqref{eq:AbstractCompletionN}.

Furthermore, we can identify~$(\Lb^*u)\circ \tr^n$ with a function in~$L^\infty(\mssm^{\otym{n}})$ by virtue of Assumption~\ref{ass:AC}.
Since the spaces~$K_n$ are `spaces of functions differentiable along $n$ directions', combining the last two assertions with~\iref{i:HeuristicRad2} above shows that the sequence~$\seq{\class[n]{\Lb^*u}}_n$, where again~$\Lb^*u$ is regarded as an element of~$K_n$, is uniformly bounded by~$\norm{\Lb^* u}_{\bLip}$, and thus weakly convergent in the Kuwae--Shioya sense.

In the spirit of~\iref{i:HeuristicRad3}, it is now not surprising that the limit of~$\seq{\class[n]{\Lb^*u}}_n$ is the element~$\class[\infty]{\Lb^*u}$ of~$K_\infty$.
Thus, we have shown that~$\Lb^*u$ is an element of~$K_\infty$ in the range of~$\Lb^*$.
Since~$\Lb^*\colon \dom{\EE{\dUpsilon}{\QP}}\to K_\infty$ is an isometry (onto its range in~$K_\infty$, by Cor.~\ref{c:Isometry}), we may finally conclude that~$u\in\dom{\EE{\dUpsilon}{\QP}}$.

\medskip

Let us now turn to a rigorous proof.
We will dispense from the distinction of classes and representatives, as it would be too burdensome to distinguish both representatives from measure classes, and classes in the different Hilbert completions~$K_N$ of~$\preW(\Ed)$.
Similarly, we shall not mention the underlying choice of strong lifting.
Furthermore, we will implicitly assume that~$\QP$ is concentrated on~$\dUpsilon^\sym{\infty}$, which allows to restrict our statements and arguments to~$\mbfX^\asym{\infty}_\locfin(\Ed)$.
The easy adaptation of both statements and arguments to~$\mbfX_\locfin(\Ed)$ is left to the reader.

For~$N\in \overline\N_1$, let~$\Li[N]{\emparg}\eqdef \Li[\mssd^{\tym{N}}]{\emparg}$.
If not otherwise specified,~$\lb$ is a fixed labeling map.
We still need one preparatory lemma.
Recall (cf.\ Notation~\ref{n:LabelingUniversalAlgebra}) that the \emph{labeling-universal} $\sigma$-algebra on~$\mbfX^\asym{\infty}_\locfin(\Ed)$ is defined by
\begin{align*}
\boldSigma^*(\msE)\eqdef \bigcap_{\lb \text{ labeling map}} \Bo{\T^\tym{\infty}}^{\lb_\pfwd\QP} \fstop
\end{align*}
Note that~$\Lb\colon \mbfX^\asym{\infty}_\locfin(\Ed)\to \dUpsilon$ is $\boldSigma^*(\Ed)/\A_\mrmv(\Ed)^\QP$-measurable, hence that~$\Lb^*\rep u\colon \mbfX^\asym{\infty}_\locfin(\Ed)\to[-\infty,\infty]$ is $\boldSigma^*(\Ed)$-measurable whenever $\rep u\colon \dUpsilon\to[-\infty,\infty]$ is $\A_\mrmv(\Ed)^\QP$-measurable.

\begin{lem}\label{l:CdCBLip}
Let~$(\mcX,\cdc,\mssd)$ be an \MLDS satisfying~$(\Rad{\mssd}{\mssm})$, and~$\QP$ be a probability measure on~$\dUpsilon$ satisfying Assumption~\ref{ass:AC}. Then,
\begin{align*}
\bLip\ttonde{\mbfX^\asym{\infty}_\locfin(\Ed),\boldSigma^*(\Ed),\mssd_\tym{\infty}}\subset \preW(\Ed)\comma \qquad \norm{U}_{\llb}\leq \norm{U}_{\bLip(X^\tym{\infty},\mssd_\tym{\infty})}\comma
\end{align*}
for every labelling map~$\lb$.
\begin{proof}
Let~$\mbfx=\seq{x_i}_i$ and let~$U_{\mbfx,n}\colon X^\tym{n}\longrar \R$ be defined by
\begin{align*}
U_{\mbfx,n}\colon \seq{y_1,\dotsc, y_n}&\longmapsto U(y_1,\dotsc, y_n,x_{n+1},x_{n+2},\dotsc)\fstop
\end{align*}
Then,
\begin{align}\label{eq:l:CdCBLip:1}
\cdc^\asym{n}(U)(\mbfx)=\cdc^\otym{n}(U_{\mbfx,n})(\mbfx^\asym{n})\comma \qquad \mbfx\in X^\tym{\infty}\comma
\end{align}
and~$\Li[n]{U_{\mbfx,n}}\leq \Li[\infty]{U}$ for every~$\mbfx\in X^\tym{\infty}$ and~$n\in \N_1$.
Thus, by Theorem~\ref{t:Tensor},
\begin{align}\label{eq:l:CdCBLip:2}
\cdc^\otym{n}(U_{\mbfx,n})(\emparg)\leq \Li[n]{U_{\mbfx,n}}^2\leq \Li[\infty]{U}^2 \as{\mssm^{\otimes n}} \comma
\end{align}
whence~$\QP^\asym{n}$-a.e., by Assumption~\ref{ass:AC}. As a consequence, combining~\eqref{eq:l:CdCBLip:2} and~\eqref{eq:l:CdCBLip:1},
\begin{align}\label{eq:l:CdCBLip:3}
\cdc^\asym{n}(U)(\emparg)=\cdc^\otym{n}(U_{\emparg,n})\circ \tr^n (\emparg)\leq \Li[\infty]{U}^2 \as{\QP^\asym{\infty}}\comma
\end{align}
and the conclusion follows by~\eqref{eq:d:Di:1} letting~$n\rar\infty$.
\end{proof}
\end{lem}

\begin{thm}[Rademacher]\label{t:Rademacher}
Let~$(\mcX,\cdc,\mssd)$ be an \MLDS satisfying~$(\Rad{\mssd}{\mssm})$ and assume that $\Dz\subset \bLip(\T,\mssd)$. Further let~$\QP$ be a probability measure on~$\dUpsilon$ satisfying Assumptions~\ref{ass:AC} and~\ref{ass:CWP}.

Then,~$(\dUpsilon, \SF{\dUpsilon}{\QP}, \mssd_\dUpsilon)$ satisfies~$(\Rad{\mssd_\dUpsilon}{\QP})$, that is, every~$\rep u\in \bLip\ttonde{\dUpsilon,\A_\mrmv(\Ed)^\QP,\mssd_\dUpsilon}$ satisfies~$u\in\dom{\EE{\dUpsilon}{\QP}}$ and~$\QP$-$\esssup \SF{\dUpsilon}{\QP}(u)\leq \Li[\dUpsilon]{\rep u}$.
\begin{proof}
Fix~$\rep u\in \bLip\ttonde{\dUpsilon,\A_\mrmv(\Ed)^\QP,\mssd_\dUpsilon}$ and note that
\begin{align*}
\Lb^*u\in \spclass[\llb]{\bLip\ttonde{\mbfX^\asym{\infty}_\locfin(\Ed),\boldSigma^*(\Ed),\mssd_\tym{\infty}}}\subset \spclass[\llb]{\preW(\Ed)}\subset K_\infty
\end{align*}
by Lemma~\ref{l:CdCBLip}, and that~$\Li[\infty]{\Lb^*\rep u}\leq\Li[\dUpsilon]{\rep u}$ by Proposition~\ref{p:LbLipschitz}. Thus, by~\eqref{eq:l:CdCBLip:3},
\begin{align*}
\cdc^\asym{n}(\Lb^*u)\leq \cdc^\asym{\infty}(\Lb^*u)\leq \Li[\infty]{\Lb^*\rep u}^2\leq \Li[\dUpsilon]{\rep u}^2 \as{\QP^\asym{\infty}}\fstop
\end{align*}

As a consequence,
\begin{align*}
\sup_n\norm{\Lb^*u}_{K_n}\eqdef \sup_n \sqrt{\int \cdc^\asym{n}(\Lb^*u)\diff\QP^\asym{\infty}+\norm{\Lb^*u}^2_{L^2(\QP^\asym{\infty})}} \leq \Li[\dUpsilon]{\rep u}+\norm{u}_{L^\infty(\QP)} \leq \norm{\rep u}_{\bLip(\mssd_\dUpsilon)}\fstop
\end{align*}
Thus,~$\seq{\Lb^*u}_n$, with~$\Lb^*u\in K_n$, has a weakly converging subsequence in the sense of Definition~\ref{d:KSOperators} w.r.t.~the pair~$\ttonde{\Psi_n\circ (\Lb^*)^{-1},\Lb^*\Cyl{\Dz}}$ by~\cite[Lem.~2.2]{KuwShi03}. Let~$U$ be its limit.
Since~$K_n$ converges to~$\Lb^*K$ by Corollary~\ref{c:KSUpsilon}, there exists~$v\in K$ such that~$U= \Lb^*v$.
Since~$\Lb^*u$ is an element of~$\spclass[\llb]{\preW(\Ed)}$, the (extensions of the) canonical inclusions~$\cokb_N\colon K_N\to L^2(\QP^\asym{\infty})$ act identically on~$\Lb^*u$, that is,~$\cokb_n(\Lb^*u)=\cokb_\infty (\Lb^*u)=\Lb^*u$ is a constant sequence in~$L^2(\QP^\asym{\infty})$.
Therefore,~$\Lb^*u=\Lb^*v$ \emph{as elements of}~$L^2(\QP^\asym{\infty})$, and thus~$u=v$ in~$L^2(\QP)$ by Proposition~\ref{p:LipIsometry}.
By closability of~$\ttonde{\EE{\dUpsilon}{\QP},\dom{\EE{\dUpsilon}{\QP}}}$, and since~$v\in K\eqdef \dom{\EE{\dUpsilon}{\QP}}$ we have~$u\in \dom{\EE{\dUpsilon}{\QP}}$.
We may now apply the equality in~\eqref{eq:l:CdCCylinderTruncation:1c} established in Corollary~\ref{c:Isometry} to obtain
\begin{align*}
\SF{\dUpsilon}{\QP}(u)\circ \Lb=\cdc^\asym{\infty}(\Lb^*u)\leq \Li[\dUpsilon]{u}^2 \as{\QP^\asym{\infty}}\comma
\end{align*}
which concludes the proof.
\end{proof}
\end{thm}

\begin{rem}[Comparison with~{\cite{RoeSch99}}, part I]\label{r:RoeSch99SL0}
In~\cite[Thm.~1.3]{RoeSch99}, the authors construct four closed energy forms on~$\dUpsilon$ in the case when~$(X,\mssd)$ is a smooth Riemannian manifold~$(M,\mssd_g)$.
The forms all coincide on cylinder functions~$\CylQP{\QP}{\mcC^\infty_c(M)}$, but differ for their domains, satisfying, cf.~\cite[Prop.~1.4]{RoeSch99},
\begin{equation}\label{eq:RoeSch}
\msF_0\subset \msF^{(c)}\subset \msF\subset W^{1,2}\comma
\end{equation}
with~$\msF_0$ the $\EE{\dUpsilon}{\QP}$-closure of~$\CylQP{\QP}{\mcC^\infty_c(M)}$, i.e.\ the minimal domain containing cylinder functions.

The authors show the Rademacher property for~$\ttonde{\EE{\dUpsilon}{\QP},\msF}$ assuming that~$\QP$ is quasi-invariant w.r.t.\ the lift by push-forward of the natural action of the group of diffeomorphisms~$\Diff_c(M)$, in a rather strong sense, cf.~\cite[Ass.~1.1(d)]{RoeSch99}.
Whereas not comparable to the assumptions in~\cite{RoeSch99}, our assumptions are ---~in essence~--- much weaker than~\cite{RoeSch99}.
Furthermore, we stress that the Rademacher property shown here holds for the minimal domain~$\msF_0$ rather than for~$\msF$, which is a stronger statement than the one in~\cite{RoeSch99}.
The coincidence of the domains in~\eqref{eq:RoeSch} in the setting of~\cite{RoeSch99}, holds under essential self-adjointness of the generator~$\LL{\dUpsilon}{\QP}$ on~$\CylQP{\QP}{\Cc^\infty(M)}$ see~\cite[Proof of Prop.~1.4(iii), p.~350]{RoeSch99}.
While in our setting it is not possible to define the spaces~$\msF^{(c)}$ and~$\msF$ \emph{a priori}, the identification~$\msF_0=W^{1,2}$ holds at least for Poisson measures by the essential self-adjointness results in~\cite{LzDSSuz22a}.
\end{rem}

\subsection{Identification of forms}\label{sss:AnalyticCheeger}
Our aim in this section is to establish sufficient conditions for the coincidence of the forms~$\EE{\dUpsilon}{\QP}$ and~$\Ch[\mssd_\dUpsilon, \QP]$ on~$L^2(\QP)$.

\begin{lem}\label{l:EleCh}
Let~$(\mcX,\cdc, \mssd)$ be an \MLDS, and~$\QP$ be a probability measure on~$\ttonde{\dUpsilon,\A_\mrmv(\Ed)}$ satisfying Assumption~\ref{ass:CWP}.
If~$\ttonde{\dUpsilon,\SF{\dUpsilon}{\QP},\mssd_\dUpsilon}$ satisfies~$(\Rad{\mssd_\dUpsilon}{\QP})$, then
\begin{align*}
\EE{\dUpsilon}{\QP}\leq \Ch[\mssd_\dUpsilon,\QP] \fstop
\end{align*}
\begin{proof}
By Proposition~\ref{p:ConfigEMTS},~$\ttonde{\dUpsilon,\T_\mrmv(\Ed),\mssd_\dUpsilon}$ is an extended metric-topo\-logic\-al space, hence we may apply~\cite[Lem.~3.6]{LzDSSuz20} to obtain that~$\ttonde{\dUpsilon,\SF{\dUpsilon}{\QP},\mssd_\dUpsilon}$ satisfies~$(\dRad{\mssd_\dUpsilon}{\QP})$, viz.~$\mssd_\dUpsilon\leq \mssd_{\QP}$.
As a consequence, by Lemma~\ref{l:RadCompleteness}, $\ttonde{\dUpsilon,\T_\mrmv(\Ed),\mssd_\QP}$ is a complete extended metric-topological measure space.

Note that~$\ttonde{\EE{\dUpsilon}{\QP},\dom{\EE{\dUpsilon}{\QP}}}$ is local by construction of~$\SF{\dUpsilon}{\QP}$, conservative since~$\QP$ is a probability measure, and therefore strongly local. 
Further recall that~$\QP$ is a Radon measure by Remark~\ref{r:QPRadon}.
Combining this two facts together, the quadruple~$\ttonde{\dUpsilon, \Bo{\T_\mrmv(\Ed)}, \EE{\dUpsilon}{\QP},\QP}$ is an energy measure space in the sense of~\cite[Dfn.~10.1]{AmbErbSav16}.
Therefore, all results in~\cite[\S12]{AmbErbSav16} apply.
By~\cite[Thm.~12.5]{AmbErbSav16} and~$\dRad{\mssd_\dUpsilon}{\QP}$,
\begin{equation*}
\EE{\dUpsilon}{\QP}\leq \Ch[\mssd_{\QP},\QP]\leq \Ch[\mssd_\dUpsilon,\QP] \fstop \qedhere
\end{equation*}
\end{proof}
\end{lem}

In order to show the opposite inequality, we shall need pointwise estimates for the slope of cylinder functions.
As usual, let~$u^\sym{n}\eqdef \Lb^*u\restr_{X^\tym{n}}$, and recall the definition~\eqref{eq:SlopeDef} of slope of a function.

\begin{lem}\label{l:CylinderLocLip}
Let~$(\mcX,\cdc)$ be an \MLDS with~$\Dz\subset \bLip(\mssd)$, and fix~$u\in\Cyl{\Dz}$.
Then,
\begin{align}\label{eq:l:CylinderLocLip:0}
\slo[\mssd_\dUpsilon]{u} \circ \Lb \leq \nlimsup \slo[\mssd_\tym{n}]{u^\sym{n}} \circ \tr^{n} \fstop
\end{align}
\begin{proof}
Since~$\dUpsilon^\sym{N}$ is a $\mssd_\dUpsilon$-accessibility component of~$\dUpsilon$ for every~$N\in\overline{\N}_0$, a proof of~\eqref{eq:l:CylinderLocLip:0} on~$\dUpsilon^\sym{N}$ is independent of the same argument for~$M\neq N$.
For simplicity, we show that~\eqref{eq:l:CylinderLocLip:0} holds on~$\dUpsilon^\sym{N}$ when~$N=\infty$.
A proof of the case when~$N<\infty$ is similar and simpler, and therefore it is omitted.

Let $E\eqdef \cup_{i\leq k}\supp f_i\in\Ed$, and, for each~$n\in\N$, let~$\Xi_{\leq n}(E)$ be defined as in~\eqref{eq:t:CylinderLocLip:0}.
Since $E$ is $\T$-closed, $\Xi$ is $\T_\mrmv(\Ed)$-open by Remark~\ref{r:CylContinuity}\iref{i:r:CylContinuity:3}, and therefore~$\mssd_\dUpsilon$-open as well, since $\T_{\mssd_\dUpsilon}$ is finer than~$\T_\mrmv(\Ed)$.
Further fix~$\eta\in\dUpsilon$, and set~$n\eqdef \eta E$ and~$\Xi\eqdef \Xi_{\leq n}(E)$.
Similarly to the proof of Theorem~\ref{t:CylinderLocLip}, for each fixed $\gamma\in\dUpsilon$ we have that
\begin{align}\label{eq:l:CylinderLocLip:1}
\slo[\mssd_\dUpsilon]{u}(\eta)=&\ \limsup_{\gamma \to \eta} \frac{\abs{u(\gamma)-u(\eta)}}{\mssd_\dUpsilon(\gamma,\eta)}= \limsup_{\substack{\gamma \to \eta \\ \gamma \in B^\dUpsilon_\infty(\eta)\cap \Xi}} \frac{\abs{u(\gamma)-u(\eta)}}{\mssd_\dUpsilon(\gamma,\eta)} \fstop
\end{align}

Assume now that~$\gamma\in B^\dUpsilon_\infty(\eta)\cap\Xi$ be $\mssd_\dUpsilon$-convergent to~$\eta$.
Fix~$\mbfy\in\Lb^{-1}(\eta)$ with~$y_p\notin E$ for~$p> n$, let~$\alpha\in \dUpsilon(\Ed^\otym{2})$ be any optimal matching of~$(\gamma,\eta)$, and define a labeling
\begin{align*}
\gamma\longmapsto \mbfx\eqdef& \seq{x_1,\dotsc, x_{\eta\set{y_1}}, x_{\eta\set{y_1}+1},\dotsc, x_{\eta\set{y_1}+\eta\set{y_2}},\dotsc}
\\
&x_{\sum_{j=1}^{p-1} \eta\set{y_j}+i}\in \supp\, \alpha\set{(\emparg, y_p)}\comma \qquad i\in \set{1,\dotsc, \eta\set{y_p}}\comma \quad p\in \N \fstop
\end{align*}
Observe that $\bmssd(\mbfx,\mbfy)=\mssd_\dUpsilon(\gamma,\eta)$ by definition of~$\alpha$.
Further argue by contradiction that
\begin{equation*}
\limsup_{\substack{\gamma \to \eta \\ \gamma \in B^\dUpsilon_\infty(\eta)\cap \Xi}} \bar p >n \comma \qquad \bar p\eqdef \max\set{p:x_p\in E}\fstop
\end{equation*}
Then, we obtain the contradiction
\begin{align*}
0=&\ \limsup_{\gamma\to \eta} \mssd_\dUpsilon(\gamma,\eta)^2 \geq \limsup_{\substack{\gamma \to \eta \\ \gamma \in B^\dUpsilon_\infty(\eta)\cap \Xi}} \mssd(x_{\bar p}, y_{\bar p})^2 \geq \limsup_{\substack{\gamma \to \eta \\ \gamma \in B^\dUpsilon_\infty(\eta)\cap \Xi}} \mssd(x_{\bar p}, \supp\,\eta_{E^\complement})^2
\\
\geq&\ \mssd(E,\supp \eta_{E^\complement})^2>0\comma
\end{align*}
where the first inequality holds since~$(x_{\bar p},y_{\bar p})\in\supp \alpha$ and by optimality of~$\alpha$, the second inequality holds by the contradiction assumption~$\limsup \bar p> n$ and since~$y_p\notin E$ for~$p>n$ by definition of~$\mbfy$, the third inequality holds by definition of~$\bar p$, and the last one since~$E$ is $\T$-closed and $\supp \eta$ has no accumulation point.
As a consequence, we may conclude that
\begin{equation*}
\limsup_{\substack{\gamma \to \eta \\ \gamma \in B^\dUpsilon_\infty(\eta)\cap \Xi}} \bar p \leq n \comma \qquad \bar p\eqdef \max\set{p:x_p\in E}\comma
\end{equation*}
so that, eventually as~$\gamma \to \eta$, we have that~$x_p\in E$ for~$p\leq n$.
Thus, finally, continuing from~\eqref{eq:l:CylinderLocLip:1}
\begin{align*}
\slo[\mssd_\dUpsilon]{u}(\eta)=&\ \limsup_{\substack{\gamma \to \eta \\ \gamma \in B^\dUpsilon_\infty(\eta)\cap \Xi}} \frac{\abs{u(\gamma)-u(\eta)}}{\mssd_\dUpsilon(\gamma,\eta)}
\\
=&\ \limsup_{\substack{\gamma \to \eta\\ \gamma \in B^\dUpsilon_\infty(\eta)\cap \Xi}} \frac{\abs{u^\sym{n}(\mbfx)-u^\sym{n}(\mbfy)}}{\mssd_\tym{\infty}(\mbfx,\mbfy)}
\\
\leq& \ \limsup_{\substack{\gamma \to \eta\\ \gamma \in B^\dUpsilon_\infty(\eta)\cap \Xi}} \frac{\abs{u^\sym{n}(\mbfx^\asym{n})-u^\sym{n}(\mbfy^\asym{n})}}{\mssd_\tym{n}(\mbfx^\asym{n},\mbfy^\asym{n})}\leq \slo[\mssd_\tym{n}]{u^\sym{n}}(\mbfy^\asym{n}) \fstop
\end{align*}

Since~$\eta$ was arbitrary, $n\geq \eta E$ ranges in $\N_0$, and therefore passing to the limit superior in~$n$ in the previous inequality we obtain the conclusion.
\end{proof}
\end{lem}

\begin{lem}\label{l:CylinderSlopeLeqCdC}
Let~$(\mcX,\cdc,\mssd)$ be an \MLDS with $\Dz\subset \bLip(\T,\mssd)$, and satisfying~\ref{ass:T}.
Further let~$\QP$ be a probability measure on~$\dUpsilon$ satisfying Assumptions~\ref{ass:AC} and~\ref{ass:CWP}.
Then,
\begin{align*}
\slo[\mssd_\dUpsilon]{u}^2\leq \SF{\dUpsilon}{\QP}(u)\quad \as{\QP}\comma \qquad u\in\Cyl{\Dz} \fstop
\end{align*}
\begin{proof}
Let~$n\in\N$ and~$u^\sym{n}$ be defined as in~\eqref{eq:CylFReduction}.
Since~$u^\sym{n}$ is $\mssd_{\tym{n}}$-Lipschitz for each $n\in \N$, we have, by Assumption~\ref{ass:T}, that
\begin{align}\label{eq:p:TensorizationPoisson:1}
\cdc^\otym{n}\ttonde{u^\sym{n}}=\slo[\mssd_\tym{n}]{u^\sym{n}}^2 \qquad \as{\mssm^\otym{n}}\comma \qquad n\in\N \fstop
\end{align}
Similarly to the proof of Lemma~\ref{l:CdCCylinderTruncation}, let~$B^\asym{n}\in \Bo{\T}^\otym{n}$ be a set of full $\mssm^\otym{n}$-measure so that~\eqref{eq:p:TensorizationPoisson:1} holds everywhere on~$B^\asym{n}$.
For a fixed labeling map~$\lb$, set~$\Omega_n\eqdef (\tr^n\circ \lb)^{-1}(B^\asym{n})$ and note that~$\QP \Omega_n=1$ by Assumption~\ref{ass:AC}.
Thus,~$\Omega \eqdef \bigcap_{n\geq 1} \Omega_n$ has full~$\QP$-measure.

Now, a sequential application of Lemma~\ref{l:CylinderLocLip}, \eqref{eq:p:TensorizationPoisson:1}, and Lemma~\ref{l:CdCCylinderTruncation}, yields for every~$\eta\in\Omega$
\begin{align*}
\slo[\mssd_\dUpsilon]{u}^2(\eta)\leq&\ \nlimsup \ttonde{\slo[\mssd_\tym{n}]{u^\sym{n}}^2\circ \tr^{n}\circ \lb}(\eta)
\\
=&\ \nlimsup \ttonde{\cdc^\otym{n}(u^\sym{n})\circ \tr^{n}\circ \lb}(\eta) = \SF{\dUpsilon}{\QP}(u)(\eta) \comma
\end{align*}
which concludes the proof.
\end{proof}
\end{lem}

\begin{lem}\label{l:ChleE}
Let~$(\mcX,\cdc,\mssd)$ be an \MLDS with $\Dz\subset \bLip(\T,\mssd)$, and satisfying~\ref{ass:T}.
Further let~$\QP$ be a probability measure on~$\dUpsilon$ satisfying Assumptions~\ref{ass:AC} and~\ref{ass:CWP}.
Then,
\begin{align*}
\Ch[\mssd_\dUpsilon,\QP] \leq \EE{\dUpsilon}{\QP} \fstop
\end{align*}

\begin{proof}
Fix~$u=F\circ \mbff^\trid\in\Cyl{\Dz}$ with~$\mbff\eqdef\seq{f_i}_{i\leq k}$, and let~$E\eqdef \cup_{i\leq k}\supp f_i$ in~$\Ed$
Further fix a curve~$\gamma_\emparg\eqdef \seq{\gamma_t}_t\in \AC^1\ttonde{I; (\dUpsilon,\mssd_\dUpsilon)}$.
Since~$E$ is $\T$-closed, the function~$t\mapsto\gamma_t E$ is upper semi-continuous, and thus it has a maximum~$n= n(E,\gamma_{\emparg})\in\N_0$.

Now, let~$\Xi\eqdef \Xi_n^\sym{\infty}(E)$ be as in~\eqref{eq:t:CylinderLocLip:0}.
Note that the restriction of~$u$ to~$\Xi$ is $\mssd_\dUpsilon$-Lipschitz by Theorem~\ref{t:CylinderLocLip}, and that~$\seq{\gamma_t}_t\subset \Xi$ by definition of the latter.
As a consequence,~$u$ is $\mssd_\dUpsilon$-Lipschitz along~$\gamma_\emparg$, and
\begin{align*}
\abs{u(\gamma_1)-u(\gamma_0)}\leq \int_0^1 \slo[\mssd_\dUpsilon]{u}(\gamma_r) \abs{\dot\gamma}_r\diff r 
\end{align*}
by~\cite[Rmk.~2.8]{AmbGigSav14}.
Since~$\gamma_\emparg$ was arbitrary, the slope~$\slo[\mssd_\dUpsilon]{u}$ is an upper gradient for~$u$.

Now, since~$\slo[\mssd_\dUpsilon]{u}$ is $\Bo{\T_\mrmv(\Ed)}^*$-measurable by~\cite[Lem.~2.6]{AmbGigSav14}, it follows from Definition~\ref{d:CheegerW}, the definition of minimal weak upper gradient, and Lemma~\ref{l:CylinderSlopeLeqCdC} that
\begin{align*}
\Ch[w,\mssd_\dUpsilon,\QP] \eqdef \int_{\dUpsilon} \slo[w,\mssd_\dUpsilon]{u}^2 \diff \QP \leq \int_{\dUpsilon} \slo[\mssd_\dUpsilon]{u}^2 \diff \QP \leq \int_{\dUpsilon} \SF{\dUpsilon}{\QP}(u)\diff \QP \fstop
\end{align*}
By Proposition~\ref{p:ConsistencyCheeger}, this concludes the assertion for cylinder functions.

Finally, in order to extend this result to~$u\in\dom{\EE{\dUpsilon}{\QP}}$, let~$\seq{u_n}_n\subset \Cyl{\Dz}$ be $\dom{\EE{\dUpsilon}{\QP}}$-converging to~$u$.
Then, respectively by: lower semi-continuity of~$\Ch[\mssd_\dUpsilon]$, definition of~$\Ch[\mssd_\dUpsilon]$ (Lem.~\ref{l:CylinderSlopeLeqCdC}), and definition of~$u_n$,
\begin{align*}
\Ch[\mssd_\dUpsilon,\QP](u)\leq&\ \nliminf\Ch[\mssd_\dUpsilon,\QP](u_n)\leq \nliminf \int_\dUpsilon \slo{u_n}^2 \diff\QP
\\
\leq &\ \nliminf \int_\dUpsilon \SF{\dUpsilon}{\QP}(u_n) = \nliminf\EE{\dUpsilon}{\QP}(u_n) = \EE{\dUpsilon}{\QP}(u) \fstop \qedhere
\end{align*}
\end{proof}
\end{lem}

Combining Lemma~\ref{l:ChleE} with the opposite inequality, shown in Lemma~\ref{l:EleCh} by means of Theorem~\ref{t:Rademacher}, finally yields the following Identification Theorem.
\begin{thm}[Identification]\label{t:IdentificationCheeger}
Let~$(\mcX,\cdc,\mssd)$ be an \MLDS satisfying~$(\Rad{\mssd}{\mssm})$, $\Dz\subset \bLip(\T,\mssd)$, and~\ref{ass:T}.
Further let~$\QP$ be a probability measure on~$\dUpsilon$ satisfying Assumptions~\ref{ass:AC} and \ref{ass:CWP}.
Then,
\begin{align}\label{eq:Cheeger=EE}
\ttonde{\Ch[\mssd_\dUpsilon,\QP],\dom{\Ch[\mssd_\dUpsilon,\QP]}}=\ttonde{\EE{\dUpsilon}{\QP},\dom{\EE{\dUpsilon}{\QP}}} \fstop
\end{align}
\end{thm}

\begin{rem}[Comparison with~\cite{ErbHue15}]\label{r:ErbHue}
When~$(\mcX,\cdc,\mssd)$ is a Riemannian manifold with Ricci curvature bounded below and~$\QP=\PP_\mssm$, Theorem~\ref{t:IdentificationCheeger} is asserted in~\cite[Prop.~2.3]{ErbHue15}.
However, the proof of~\cite[Prop.~2.3]{ErbHue15} relies on the incorrect claim that cylinder functions be $\mssd_\dUpsilon$-Lipschitz, which is disproved by Example~\ref{e:NonLipCyl}.
\end{rem}

\paragraph{Some consequences}
As a consequence of Theorem~\ref{t:IdentificationCheeger}, whereas~$u\in\Cyl{\Dz}$ is not necessarily Lipschitz, there exists a sequence of functions~$\seq{u_n}_n\subset\bLip(\mssd_\dUpsilon,\A_\mrmv)$ converging to~$u$ in~$\dom{\EE{\dUpsilon}{\QP}}$.
Under a minor additional assumption on~$\QP$, one can show a constructive version of the above result providing an explicit sequence~$\seq{u_n}_n$ defined by means of the constrained McShane extensions in Lemma~\ref{l:McShane}.

\begin{ass}\label{d:ass:QT}
Let~$\mcX$ be a local structure. We say that a probability measure~$\QP$ on~$\dUpsilon$ is \emph{quantitatively tight} if
\begin{align}\tag*{$(\mathsf{QT})_{\ref{d:ass:QT}}$}\label{ass:QT}
\nlim n\cdot \QP \ttonde{\Xi_{\geq n}(E)}=0\comma \qquad E\in\msE \fstop
\end{align}
\end{ass}

\begin{rem}\label{r:ExamplesQT}
Assumption~\ref{ass:QT} is readily checked for several classes of measures, e.g.:
\begin{enumerate}[$(i)$]
\item\label{i:r:ExamplesQT:0} Poisson and mixed Poisson measures;
\item\label{i:r:ExamplesQT:1} Ruelle-type grand canonical Gibbs measures in the sense of~\cite{Rue70, AlbKonRoe98b};
\item\label{i:r:ExamplesQT:2} $\alpha$-determinantal point measures~\cite{ShiTak03}.
\end{enumerate}

\begin{proof}
\iref{i:r:ExamplesQT:0} By~\eqref{eq:AKR2.7} we may compute
\begin{align*}
n\cdot \PP_\mssm \ttonde{\Xi_{\geq n}(E)}=&\ n\tonde{1- \sum_{i=0}^n \PP_\mssm \ttonde{\Xi_{=i}(E)} }= n \tonde{1-\frac{\Gamma(1+n,\mssm E)}{n!}}
\end{align*}
where~$\Gamma(a,x)$ denotes the (upper) incomplete~$\Gamma$ function.
The conclusion follows since this latter term is vanishing as~$n$ goes to infinity as soon as $\mssm E<\infty$.

\iref{i:r:ExamplesQT:1} It suffices to note that~$\QP\ttonde{\Xi_{> n}(E)}$ decays super-exponentially by~\cite[Cor.~2.9]{Rue70}.

\iref{i:r:ExamplesQT:2} By~\cite[Lem.~4.2]{ShiTak03} there exist constants~$a_\alpha, b_\alpha>0$, and $0<c_\alpha<1$, so that
\begin{align*}
\QP\ttonde{\Xi_{> n}(E)} \leq 
\begin{cases}
\displaystyle\sum_{i=n+1}^\infty \frac{(a_\alpha)^i}{i!}= \frac{e^{a_\alpha} \ttonde{n!-\Gamma(n+1,a_\alpha)}}{n!} & \text{if}\quad \alpha<0
\\
b_\alpha \displaystyle\sum_{i=n}^\infty c_\alpha^i =\frac{b_\alpha \,c_\alpha^n}{1-c_\alpha}& \text{if}\quad \alpha>0
\end{cases}\comma
\end{align*}
whence~\ref{ass:QT} follows by standard facts about the incomplete Gamma function.
\end{proof}
\end{rem}

\begin{prop}\label{p:TensorizationPoisson}
Let~$(\mcX,\cdc,\mssd)$ be an \MLDS with $\Dz\subset \bLip(\T,\mssd)$, and satisfying~\ref{ass:T}.
Further let~$\QP$ be a probability measure on~$\dUpsilon$ satisfying Assumptions~\ref{ass:AC},~\ref{ass:CWP}, and~\ref{ass:QT}.
Then, for each~$u\in \Cyl{\Dz}$ there exists~$\seq{u_n}_n\subset \bLip(\mssd_\dUpsilon,\Bo{\T_\mrmv(\Ed)})$ such that
\begin{enumerate}[$(i)$]
\item\label{i:p:TensorizationPoisson:1} $L^2(\QP)$-$\nlim u_n=u$;
\item\label{i:p:TensorizationPoisson:2} $L^2(\QP)$-$\nlim \slo[\mssd_\dUpsilon]{u_n}=\slo[\mssd_\dUpsilon]{u}$.
\end{enumerate}
\begin{proof}
\iref{i:p:TensorizationPoisson:1} Let~$u=F\circ \mbff^\trid$ with~$\mbff\eqdef\seq{f_i}_{i\leq k}$, set~$M\eqdef \sup \abs{u}$, and let
\begin{align*}
E\eqdef \cup_{i\leq k}\supp f_i\in \Ed\fstop
\end{align*}
Further let~$\Xi_n^\sym{\infty}(E)$ be as in~\eqref{eq:t:CylinderLocLip:0}, note that it is $\T_\mrmv(\Ed)$-open by Remark~\ref{r:CylContinuity}\iref{i:r:CylContinuity:3}, and note that the restriction of~$u$ to~$\Xi_n^\sym{\infty}(E)$ is $\mssd_\dUpsilon$-Lipschitz by Theorem~\ref{t:CylinderLocLip}.
The corresponding (upper) constrained McShane extension defined as in~\eqref{eq:McShane}, say~$u_n$, is thus $\Bo{\T_\mrmv(\Ed)}$-measurable by Proposition~\ref{p:BorelMcShane}.
By construction,~$u_n$ is additionally: $\mssd_\dUpsilon$-Lipschitz, uniformly bounded by~$M$, and coinciding with~$u$ on~$\Xi_n^\sym{\infty}(E)$.
Thus,
\begin{align*}
\nlim \norm{u-u_n}_{L^2(\QP)}\leq \nlim 2M \,\QP\ttonde{\Xi_n^\sym{\infty}(E)^\complement} = 0 \fstop
\end{align*}

\iref{i:p:TensorizationPoisson:2} By $\T_\mrmv(\Ed)$-locality of the $\mssd_\dUpsilon$-slope as in Remark~\ref{r:Slopes}\iref{i:r:Slopes:1}, we have~$\slo[\mssd_\dUpsilon]{u}=\slo[\mssd_\dUpsilon]{u_n}$ on $\T_\mrmv(\Ed)$-$\inter\set{u=u_n}$.
In particular, 
\begin{equation}\label{eq:p:TensorizationPoisson:0}
\slo[\mssd_\dUpsilon]{u}(\eta)=\slo[\mssd_\dUpsilon]{u_n}(\eta) \comma \qquad \eta\in \Xi_n^\sym{\infty}(E) \fstop
\end{equation}
Thus, we have that
\begin{align*}
\int_\dUpsilon \ttonde{\slo[\mssd_\dUpsilon]{u_n}-\slo[\mssd_\dUpsilon]{u}}^2 \diff\QP = &\ \int_{\Xi_n^\sym{\infty}(E)^\complement} \ttonde{\slo[\mssd_\dUpsilon]{u_n}-\slo[\mssd_\dUpsilon]{u}}^2 \diff\QP
\\
\leq &\ 2 \int_{\Xi_n^\sym{\infty}(E)^\complement} \slo[\mssd_\dUpsilon]{u_n}^2 \diff\QP +  2\int_{\Xi_n^\sym{\infty}(E)^\complement} \slo[\mssd_\dUpsilon]{u}^2 \diff\QP
\intertext{hence, by Theorem~\ref{t:CylinderLocLip} and Lemma~\ref{l:CylinderSlopeLeqCdC},}
\int_\dUpsilon \ttonde{\slo[\mssd_\dUpsilon]{u_n}-\slo[\mssd_\dUpsilon]{u}}^2 \diff\QP\leq&\ 2\, C_u\, n\cdot \QP \ttonde{\Xi_n^\sym{\infty}(E)^\complement} + 2 \QP \ttonde{\Xi_n^\sym{\infty}(E)^\complement}\cdot \sqrt{\EE{\dUpsilon}{\QP}(u)} \fstop
\end{align*}
The first term is vanishing as~$n$ goes to infinity by Assumption~\ref{ass:QT}.
The second term is vanishing as~$n$ goes to infinity since~$\tseq{\Xi_n^\sym{\infty}(E)}_n$ is invading.
\end{proof}
\end{prop}

\paragraph{Some further consequences}
Under the assumptions of Theorem~\ref{t:IdentificationCheeger}, the form $\ttonde{\EE{\dUpsilon}{\QP},\dom{\EE{\dUpsilon}{\QP}}}$ is quasi-regular by Proposition~\ref{p:QRegSLoc}, hence there exists a Markov process~$\mbfM$ with state space~$\dUpsilon$ properly associated with it.
The Identification Theorem~\ref{t:IdentificationCheeger} has a number of applications to the study of these process which will be the object of forthcoming work in this series.
As an example, we list a simple yet very useful such application in the next Remark.

\begin{rem}[Point-separating families of uniformly bounded energy]
Under the assumptions of Theorem~\ref{t:IdentificationCheeger}, the family of functions~$\bLip\ttonde{\mssd_\dUpsilon,\T_\mrmv(\msE)}$:
\begin{enumerate*}[$(a)$]
\item\label{i:r:ProcessRad:1} separates points,
\item\label{i:r:ProcessRad:2} is dense in~$\dom{\EE{\dUpsilon}{\QP}}$, 
and 
\item\label{i:r:ProcessRad:3} each of its elements has uniformly bounded square field.
\end{enumerate*}
\begin{proof}
\iref{i:r:ProcessRad:1} is readily shown by choosing as a separating family all functions of the form~$\rho_{\gamma, U}$ as in~\eqref{eq:W2Upsilon:0}.
\iref{i:r:ProcessRad:2} is a consequence of the identification together with the density of~$\bLip\ttonde{\mssd_\dUpsilon,\T_\mrmv(\msE)}$ in~$\dom{\Ch[\mssd_\dUpsilon,\QP]}$ which is a standard fact.
\iref{i:r:ProcessRad:3} holds by the Rademacher property.
\end{proof}

Additionally, we will show in subsequent work that~$\bLip\ttonde{\mssd_\dUpsilon,\T_\mrmv(\msE)}$ is as well \emph{strongly point-separating} in the sense of e.g.~\cite[Dfn.~1]{BloKou10}, which will be of use in the study of tightness for the sample paths of the Markov process~$\mbfM$.
\end{rem}

\subsection{Sobolev-to-Lipschitz-type properties}\label{ss:VariousStoL}
In this section, we pre\-sent several proofs of various Sobolev-to-Lipschitz-type properties on configuration spaces.
In particular, whereas the continuous- and distance-Sobolev-to-Lipschitz property are mutually equivalent by~\eqref{eq:EquivalenceRadStoL}, we provide here two different proofs, under different sets of assumptions for both the base space and the reference measure~$\QP$.

We start by collecting a result independent of the metric structure which will be essential to the proof of both the continuous- and distance-constinuous-Sobolev-to-Lipschitz properties.
Namely, under the assumption of conditional closability~\ref{ass:ConditionalClos}, we can extend the representation formula~\eqref{eq:CdCRestrConditionalFormCylinderF} for the square field of~$\ttonde{\EE{\dUpsilon(E)}{\QP^\eta_E},\dom{\EE{\dUpsilon(E)}{\QP^\eta_E}}}$ from cylinder functions to the whole form domain, as shown in the following proposition.

Let~$\eta\in\dUpsilon$, and~$E\in\msE$, and recall the definition of
\begin{itemize}
\item the conditional probabilities~$\QP^{\eta_{E^\complement}}$ on~$\dUpsilon$, Definition~\ref{d:ConditionalQP};
\item the projected conditional probabilities~$\QP^\eta_E$ on~$\dUpsilon(E)$, Formula~\ref{eq:ProjectedConditionalQP};
\item the forms~$\ttonde{\EE{\dUpsilon}{\QP}_E,\dom{\EE{\dUpsilon}{\QP}_E}}$, Formula~\eqref{eq:VariousFormsA}.
\end{itemize}

\begin{prop}\label{p:MarginalFormDomains}
Let~$(\mcX,\cdc)$ be a \TLDS, and~$\QP$ be a probability measure on $\ttonde{\dUpsilon,\A_\mrmv(\msE)}$ satisfying Assumptions~\ref{ass:CAC} and~\ref{ass:ConditionalClos} for some localizing sequence~$\seq{E_h}_h$.
Then, for every~$E\eqdef E_h$,~$h\in\N$, and every $\A_\mrmv(\msE)$-measurable~$\rep u$ with~$u\eqdef \class[\QP]{\rep u}\in \dom{\EE{\dUpsilon}{\QP}_E}$,
\begin{enumerate}[$(i)$]
\item\label{i:p:MarginalFormDomains:1} $u_{E,\eta}\eqdef\class[\QP^\eta_E]{\rep u_{E,\eta}}\in\dom{\EE{\dUpsilon(E)}{\QP^\eta_E}}$ $\QP$-a.e.;
\item\label{i:p:MarginalFormDomains:2} for $\QP$-a.e.~$\eta\in\dUpsilon$, letting~$\Lambda_{\eta, E^\complement}\eqdef\set{\gamma\in\dUpsilon:\gamma_{E^\complement}=\eta_{E^\complement}}$ be as in~\eqref{eq:RoeSch99Set},
\begin{equation*}
\SF{\dUpsilon(E)}{\QP^\eta_E}\ttonde{u_{E,\eta}}\circ \pr^E=\SF{\dUpsilon}{\QP}_E(u)  \quad \as{\QP^{\eta_{E^\complement}}} \text{ on } \Lambda_{\eta, E^\complement} \semicolon
\end{equation*}

\item\label{i:p:MarginalFormDomains:3} if~$u\in\dom{\EE{\dUpsilon}{\QP}}$ satisfies~$\SF{\dUpsilon}{\QP}(u)\leq 1$ $\QP$-a.e., then
\begin{equation*}
\SF{\dUpsilon(E)}{\QP^\eta_E}(u_{E,\eta})\leq 1 \quad \as{\QP^\eta_E} \quad \forallae{\QP} \eta\in\dUpsilon \fstop
\end{equation*}
\end{enumerate}
\end{prop}

We need a preparatory lemma.
For notational simplicity, we denote indices of sequences by superscripts rather than subscripts.

\begin{lem}\label{l:DisintegationConvergence}
Let~$(\mcX,\cdc)$ be a \TLDS,~$\QP$ be a probability measure on~$\ttonde{\dUpsilon,\A_\mrmv(\msE)}$, and~$p\in[1,\infty)$.
Further let~$\seq{\rep u^n}_n\subset \mcL^p(\QP)$ and~$\rep u\in \mcL^p(\QP)$ be so that~$\nlim \norm{u^n-u}_{L^p(\QP)}=0$. 
Then, for every~$E\in\msE$ there exists~$\Omega_E\in\A_\mrmv(\msE)$ with~$\QP\Omega_E=1$ additionally so that for every~$\eta\in \Omega_E$ there exists a subsequence $\seq{\rep u^{n_k(\eta)}}_k\subset \seq{\rep u^n}_n$, depending on~$\eta$, satisfying
\begin{equation}\label{eq:l:DisintegationConvergence:0}
\klim {\rep u^{n_k(\eta)}}_{E,\eta} = \rep u_{E,\eta} \qquad \text{in~$L^p(\QP^\eta_E)$ \; and} \as{\QP^\eta_E} \comma\qquad \eta\in \Omega_E \fstop
\end{equation}
\begin{proof}
By Proposition~\ref{p:ConditionalIntegration} we have
\begin{align*}
0=\nlim \int \abs{\rep u^n-\rep u}^p\diff\QP= \nlim \int_{\dUpsilon} \quadre{\int_{\dUpsilon(E)} \abs{{\rep u^n}_{E,\eta}-\rep u_{E,\eta}}^p \diff\QP^\eta_E}\diff\QP(\eta) \fstop
\end{align*}
Letting~$G_E^n(\eta)\eqdef \int_{\dUpsilon(E)} \abs{\rep u^n-\rep u}^p \diff\QP^\eta_E$, we have~$L^p(\QP)$-$\nlim G_E^n=0$, hence it follows by standard arguments that there exists a subsequence~$\seq{G_E^{n_k}}_k$ so that~$\klim G_E^{n_k}(\eta)=0$ for every~$\eta$ in some set~$\Omega_E$ of full $\QP$-measure.
By the very same argument applied to the integrand of~$G_E^n$, for every~$\eta$ in~$\Omega_E$ there exists a further subsequence of indices~$\seq{n_k(\eta)}_k\subset \seq{n_k(\eta)}_k$ so that~$\klim \rep u^{n_k(\eta)}_{E,\eta}=\rep u_{E,\eta}$ $\QP^\eta_E$-a.e.\ on~$\dUpsilon(E)$, which concludes the proof.
\end{proof}
\end{lem}

\begin{proof}[Proof of Proposition~\ref{p:MarginalFormDomains}]
As usual, let~$E\eqdef E_h$ for some~$h\in\N$.
Let~$\seq{u^n}_n\subset \CylQP{\QP}{\Dz}$ be convergent to~$u\in\dom{\EE{\dUpsilon}{\QP}_E}\subset L^2(\QP)$ both in~$L^2(\QP)$ and w.r.t.~$(\EE{\dUpsilon}{\QP}_E)^{1/2}$.
Since~$\Cyl{\Dz}\subset \Cb\ttonde{\T_\mrmv(\msE)}$, we shall denote by~$u^n$ both a $\QP$-class and its continuous representative.
By Proposition~\ref{p:MarginalWP} we have
\begin{align*}
\lim_{m,n} \EE{\dUpsilon}{\QP}_E(u^n- u^m)=\int_\dUpsilon \EE{\dUpsilon(E)}{\QP^\eta_E}\ttonde{u^n_{E,\eta}-u^m_{E,\eta}}\diff\QP(\eta)=0\fstop
\end{align*}
It follows that there exists a $\A_\mrmv(\msE)$-measurable~$\Omega_E\subset \dUpsilon$ of full $\QP$-measure and so that, up to taking a non-relabeled subsequence,
\begin{equation}\label{eq:p:MarginalFormDomain:1}
\lim_{n,m} \EE{\dUpsilon(E)}{\QP^\eta_E}\ttonde{u^n_{E,\eta}-u^m_{E,\eta}} = 0\comma \qquad \eta\in \Omega_E\fstop
\end{equation}

For~$\seq{u^n}_n$ as above and fixed~$\eta\in\Omega_E$, let~$\seq{u^{n_k(\eta)}}_k$ be the subsequence constructed in Lemma~\ref{l:DisintegationConvergence} (with~$p=2$), thus satisfying
\begin{equation}\label{eq:p:MarginalFormDomain:1.1}
L^2(\QP^\eta_E)\text{-}\klim u^{n_k(\eta)}_{E,\eta} = u_{E,\eta}\fstop
\end{equation}
Combining~\eqref{eq:p:MarginalFormDomain:1} with~\eqref{eq:p:MarginalFormDomain:1.1} yields~\iref{i:p:MarginalFormDomains:1} by closability of~$\ttonde{\EE{\dUpsilon(E)}{\QP^\eta_E},\Cyl{\Dz}}$ for~$\QP$-a.e.~$\eta$, up to possibly changing~$\Omega_E$ by a $\QP$-negligible set.

\medskip

In order to show~\iref{i:p:MarginalFormDomains:3} let~$u\in\dom{\EE{\dUpsilon}{\QP}}$ with~$\SF{\dUpsilon}{\QP}(u)\leq 1$ $\QP$-a.e..
By virtue of the monotonicity in~\eqref{eq:MonotonicityE}, we have that~$u\in\dom{\EE{\dUpsilon}{\QP}_E}$ and~$\SF{\dUpsilon}{\QP}_E(u)\leq \SF{\dUpsilon}{\QP}$.
The conclusion follows from~\iref{i:p:MarginalFormDomains:2} and the assumption.

\medskip

It remains to show~\iref{i:p:MarginalFormDomains:2}. We divide the proof in several steps.

\paragraph{Step~1} Let~$\rep v$ be any fixed $\QP$-representative of~$\SF{\dUpsilon}{\QP}_E(u)$,
and, for every~$n\in\N$, let~$\rep v^n\eqdef \rep\cdc^\dUpsilon_E\ttonde{u^n}$ be the $\QP$-representative of~$\SF{\dUpsilon}{\QP}_E(u^n)$ defined in~\eqref{eq:RestrictedCdCUpsilon}.
Since~$\nlim \EE{\dUpsilon}{\QP}_E(u^n-u)=0$ by definition of~$\seq{u^n}_n$, by reverse triangle inequality (cf.~\eqref{eq:l:AbstractCompletion:0}) we have that
\begin{align*}
L^1(\QP)\text{-}\nlim v^n = v \fstop
\end{align*}
We may thus apply Lemma~\ref{l:DisintegationConvergence} (with~$p=1$) to the sequence~$\seq{\rep v^n}_n$. Up to possibly changing~$\Omega_E$ by a $\QP$-negligible set, we obtain a family of $\A_\mrmv(\msE)$-measurable sets~$\Omega^\eta_E\subset \dUpsilon(E)$ with~$\QP^\eta_E\Omega^\eta_E=1$ and a family of subsequences of indices~$\seq{n_k(\eta)}_k$ such that
\begin{align*}
\klim \rep v^{n_k(\eta)}_{E,\eta} = \rep v_{E,\eta} \qquad \text{pointwise on~$\Omega^\eta_E$}\comma \qquad \eta\in\Omega_E\fstop
\end{align*}
As a consequence, pre-composing with~$\pr^E$,
\begin{align}\label{eq:p:MarginalFormDomains:1}
\klim \rep v^{n_k(\eta)}_{E,\eta}\circ\pr^E = \rep v_{E,\eta}\circ \pr^E \qquad \text{pointwise on~$\pr_E^{-1}(\Omega^\eta_E)$}\comma \qquad \eta\in\Omega_E\fstop
\end{align}
Set~$\Omega_\eta'\eqdef \pr_E^{-1}(\Omega^\eta_E)\cap \Lambda_{\eta, E^\complement}\subset \dUpsilon$, and recall from the proof of Proposition~\ref{p:ConditionalIntegration} that~$\QP^{\eta_{E^\complement}} \Lambda_{\eta,E^\complement}=1$.
Since~$\QP^\eta_E\Omega^\eta_E=1$ as well, we conclude that~$\QP^{\eta_{E^\complement}} \Omega_\eta'=1$.
Choosing~$\rep v^{n_k(\eta)}$ in place of~$\rep u$ in~\eqref{eq:p:ConditionalIntegration:1}, and combining it with~\eqref{eq:p:MarginalFormDomains:1}, we obtain that
\begin{equation}\label{eq:p:MarginalFormDomains:2}
\rep\cdc^\dUpsilon_E\ttonde{u^{n_k(\eta)}} \defeq \rep v^{n_k(\eta)} = \rep v^{n_k(\eta)}_{E,\eta}\circ\pr^E \xrightarrow{k\to\infty} \rep v_{E,\eta}\circ \pr^E = \rep v \quad \text{pointwise on } \Omega_\eta'\comma \qquad \eta\in\Omega_E
\end{equation}
and thus in particular~$\QP^{\eta_{E^\complement}}$-a.e., for $\QP$-a.e.~$\eta$.

\paragraph{Step 2} Let~$\rep w^\eta$ be a $\QP^\eta_E$-representative of~$\SF{\dUpsilon(E)}{\QP^\eta_E}(u_{E,\eta})$, and, for each~$n\in\N$, let~$\rep w^{\eta,n}$ be the $\QP^\eta_E$-representative of~$\SF{\dUpsilon(E)}{\QP^\eta_E}\ttonde{u^n_{E,\eta}}$ defined by~$\rep w^{\eta,n}\eqdef \rep\cdc^{\dUpsilon(E)}\ttonde{u^n_{E,\eta}}$.
In light of~\eqref{eq:p:MarginalFormDomain:1}, agan by reverse triangle inequality we have that
\begin{align*}
L^1(\QP^\eta_E)\text{-}\nlim w^{\eta, n} = w^\eta \comma \qquad \eta\in\Omega_E\comma
\end{align*}
As a consequence,
\begin{align*}
L^1(\QP^\eta_E)\text{-}\klim w^{\eta, n_k(\eta)} = w^\eta \comma \qquad \eta\in\Omega_E\fstop
\end{align*}
By standard arguments, possibly up to extracting from~$\seq{n_k(\eta)}_k$ a further non-relabeled subsequence of indices, there exist sets~$\Xi^\eta_E$ with~$\QP^\eta_E \Xi^\eta_E=1$ such that
\begin{equation*}
\klim \rep w^{\eta, n_k(\eta)} = \rep w^\eta \quad \text{pointwise on } \Xi^\eta_E\comma \qquad \eta\in\Omega_E\fstop
\end{equation*}
As a consequence, pre-composing with~$\pr^E$,
\begin{equation}\label{eq:p:MarginalFormDomains:3}
\rep\cdc^{\dUpsilon(E)}\ttonde{u^{n_k(\eta)}_{E,\eta}}\circ \pr^E \defeq \rep w^{\eta, n_k(\eta)}\circ \pr^E \xrightarrow{k\to\infty} \rep w^\eta\circ \pr^E \quad \text{pointwise on } \pr_E^{-1}(\Xi^\eta_E)\comma \qquad \eta\in\Omega_E\fstop
\end{equation}

\paragraph{Step 3}
Set~$\Omega_\eta\eqdef \Omega_\eta'\cap \pr_E^{-1}(\Xi^\eta_E)$ and note that~$\QP^{\eta_{E^\complement}} \Omega_\eta=1$ since~$\QP^\eta_E \Xi^\eta_E=1$.
Finally, combining on the set of full $\QP^{\eta_{E^\complement}}$-measure~$\Omega_\eta\subset \Lambda_{\eta, E^\complement}$ the equalities~\eqref{eq:p:MarginalFormDomains:3},~\eqref{eq:CdCRestrConditionalFormCylinderF}, and~\eqref{eq:p:MarginalFormDomains:2},
\begin{equation*}
\rep w^\eta\circ \pr^E=\klim \rep\cdc^{\dUpsilon(E)}\ttonde{u^{n_k(\eta)}_{E,\eta}}\circ \pr^E=\klim \rep\cdc^\dUpsilon_E\ttonde{u^{n_k(\eta)}}=\rep v \quad \text{pointwise on~} \Omega_\eta \comma
\end{equation*}
which concludes the proof of~\iref{i:p:MarginalFormDomains:2} by recalling the definitions of~$\rep w^\eta$ and~$\rep v$.
\end{proof}

%%%%%%%%%%%%%%%%%%%%%%%%%%%%%%%%%%%%%%%%%%%%%%%%%%%%%%%%%%%%%%%%%%%%%%%%%%%%%%%%%%%%%%% cSL %%%%%%%%%%%%%%%%%%%%%%%%%%%%%%%%%%%%%%%%%%%%%%%%%%%%%%%%%%%%%%%%%%%%%%%%%%%%%%%%%%%%%%%%%%%%%%%%%

\subsubsection{The continuous-Sobolev-to-Lipschitz property}\label{sss:cSL}
Throughout this section, let us fix~$x_0\in X$.
For each~$r>0$, let us set~$B_r\eqdef B_r(x_0)$. Further let~$\mssm_r\eqdef \mssm_{B_r}$ be the restriction of~$\mssm$ to~$B_r(x_0)$, and~$\T_r$ be the subspace topology on~$B_r(x_0)$.

Recall the definition~\eqref{eq:ProjectedConditionalQP} of the projected conditional probabilities~$\QP^\eta_E$ on~$\dUpsilon(E)$, for~$\eta\in\dUpsilon$ and~$E\in\Ed$.
In the following, we will be interested in the case when~$E=B_r$ for some~$r>0$.
For simplicity of notation, let us set~$\QP^\eta_r\eqdef \QP^\eta_{B_r}$.

We introduce the following assumptions. The first one consists in fact of two variations of the Sobolev-to-Lipschitz-property on products.

\begin{ass}\label{d:ass:cSLTensor}
Let~$(\mcX,\cdc,\mssd)$ be an \MLDS.
We assume that for every~$n\in\N$ and~$r>0$, the standard product form
\begin{equation*}
\EE{\asym{n}}{\mssm_r}(f^\asym{n})\eqdef \int_{B_r(x_0)^\tym{n}} \cdc^\otym{n}(f^\asym{n}) \diff\mssm_r^\otym{n}\comma\qquad f^\asym{n}\in \Lip\ttonde{B_r(x_0)^\tym{n},\mssd_\tym{n}}
\end{equation*}
is closable on~$L^2(\mssm_r^\otym{n})$ and satisfies either
\begin{equation}\tag*{$(\mathsf{cSL}^\otym{}_r)_{\ref{d:ass:cSLTensor}}$}\label{ass:cSLTensor}
\text{the continuous-Sobolev--to--Lipschitz property} \quad (\cSL{\T_r^\tym{n}}{\mssm_r^\otym{n}}{\mssd_\tym{n}})\comma
\end{equation}
or its strengthtening
\begin{equation}\tag*{$(\mathsf{SL}^\otym{}_r)_{\ref{d:ass:cSLTensor}}$}\label{ass:SLTensor}
\text{the Sobolev-to-Lipschitz property}\quad (\SL{\mssm_r^\otym{n}}{\mssd_\tym{n}}) \fstop
\end{equation}
\end{ass}
Assumptions~\ref{ass:cSLTensor} and~\ref{ass:SLTensor} can be verified on every Riemannian manifold, and on $\MLDS$'s satisfying both~$\RCD(K,N)$ and~$\CAT(0)$.

The second assumption is a strengthening of~\ref{ass:CE}, and~\ref{ass:ConditionalClos}.

\begin{ass}\label{ass:dcSLConfig}
Let~$(\mcX,\cdc,\mssd)$ be an \MLDS, $\QP$ be a probability measure on~$\ttonde{\dUpsilon,\A_\mrmv(\Ed)}$. We further assume that, for $\QP$-a.e.~$\eta\in\dUpsilon$ and every~$r>0$,
\begin{align}
\tag*{$(\mathsf{CC}_r)_{\ref{ass:dcSLConfig}}$}\label{ass:CCr} 
&\text{Assumption~\ref{ass:ConditionalClos} holds with~$B_r$ in place of~$E$} \comma
\end{align}
and the measures~$\QP^\eta_r$ and~$\PP_{\mssm_r}$ are mutually absolutely continuous, viz.~$\QP^\eta_r\sim \PP_{\mssm_r}$, and there exists a well-separated $\EE{\dUpsilon(B_r)}{\PP_{\mssm_r}}$-nest~$\seq{F_k}_k$ and constants~$a_k>0$ so that
\begin{align}\tag*{$(\mathsf{CE}_r)_{\ref{ass:dcSLConfig}}$}\label{ass:CEr}
0<a_k\leq \displaystyle\frac{\diff \QP^\eta_r}{\diff \PP_{\mssm_r}} \leq a_k^{-1} <+\infty \as{\PP_{\mssm_r}} \quad \text{on} \quad \inter_{\T_\mrmv(\Ed)} F_k\comma \qquad k\in\N \fstop
\end{align}
\end{ass}

\begin{prop}\label{p:cSLProjPoisson}
Let~$(X,\cdc,\mssd)$ be an \MLDS, fix~$r>0$, and let~$\ttonde{\EE{\dUpsilon(B_r)}{\PP_{\mssm_r}},\dom{\EE{\dUpsilon(B_r)}{\PP_{\mssm_r}}}}$ be defined as in~\ref{ass:ConditionalClos} with~$\PP_{\mssm_r}$ in place of~$\QP^\eta_{B_r}$ and~$B_r$ in place of~$E$, cf.\ Remark~\ref{r:ConditionalAC}\iref{i:r:ConditionalAC:2}. Then,
\begin{enumerate}[$(i)$]
\item\label{i:p:cSLProjPoisson:1} if~$(X,\cdc,\mssd)$ satisfies~\ref{ass:cSLTensor}, resp.~\ref{ass:SLTensor},
then~$\ttonde{\EE{\dUpsilon(B_r)}{\PP_{\mssm_r}},\dom{\EE{\dUpsilon(B_r)}{\PP_{\mssm_r}}}}$ satisfies~$(\cSL{\T_\mrmv(\Ed)}{\PP_{\mssm_r}}{\mssd_\dUpsilon})$, resp.~$(\SL{\PP_{\mssm_r}}{\mssd_\dUpsilon})$;
\item\label{i:p:cSLProjPoisson:2} if~$(X,\cdc,\mssd)$ satisfies~$(\Rad{\mssd}{\mssm})$, then~$\ttonde{\EE{\dUpsilon(B_r)}{\PP_{\mssm_r}},\dom{\EE{\dUpsilon(B_r)}{\PP_{\mssm_r}}}}$ satisfies~$(\Rad{\mssd_\dUpsilon}{\PP_{\mssm_r}})$.
\end{enumerate}
\begin{proof}
A simple proof of~\ref{i:p:cSLProjPoisson:1} follows directly from the assumptions invoking the representation~\eqref{eq:PoissonLebesgue} and working on each component of the disjoint union~$\dUpsilon(B_r) = \bigsqcup_{n\in\N_0} B_r^\sym{n}$, with~$B_r^\sym{n}\eqdef B_r^\tym{n}/\mfS_n$.
A proof of~\ref{i:p:cSLProjPoisson:2} follows in a similar way and by tensorization of the Rademacher property Theorem~\ref{t:Tensor}.
We omit the details.
\end{proof}
\end{prop}

\begin{prop}\label{p:cSLProjQP}
Let~$(X,\cdc,\mssd)$ be an \MLDS, and fix~$r>0$ and~$\eta\in\dUpsilon$.
Further assume that $(X,\cdc,\mssd)$ satisfies~\ref{ass:cSLTensor}, resp.~\ref{ass:SLTensor}, and~$\QP$ be a probability measure on~$\ttonde{\dUpsilon,\A_\mrmv(\Ed)}$ satisfying Assumptions~\ref{ass:CEr} and~\ref{ass:CCr}.
Further assume that for every $\mssd_\dUpsilon$-well-separated $\T_{\mrmv(\Ed)}$-open sets~$U_1,U_2$ there exists~$\varrho\eqdef \varrho_{U_1,U_2}\in \dom{\EE{\dUpsilon(B_r)}{\PP_{\mssm_r}}}$ and a constant~$C=C_{U_1,U_2}$ such that
\begin{equation}\label{eq:p:cSLProjQP:0}
\varrho \equiv 1 \as{\PP_\mssm} \quad \text{on} \quad U_1\comma \qquad \varrho \equiv 0 \as{\PP_\mssm} \quad \text{on} \quad U_2\comma \qquad \SF{\dUpsilon(B_r)}{\PP_{\mssm_r}}(\varrho)\leq C\fstop
\end{equation}
Then, the form~$\ttonde{\EE{\dUpsilon(B_r)}{\QP^\eta_r},\dom{\EE{\dUpsilon(B_r)}{\QP^\eta_r}}}$ defined as in~\ref{ass:ConditionalClos} with~$B_r$ in place of~$E$, satisfies~$(\cSL{\T_\mrmv(\Ed)}{\QP^\eta_r}{\mssd_\dUpsilon})$, resp.~$(\SL{\QP^\eta_r}{\mssd_\dUpsilon})$, for every~$\eta\in\dUpsilon$.
\begin{proof}
Since~$\QP^\eta_r\sim \PP_{\mssm_r}$ we have that~$\QP^\eta_r$ has full support on~$\dUpsilon(B_r)$ and that
\begin{equation*}
\CylQP{\QP^\eta_r}{\Dz}=\CylQP{\PP_{\mssm_r}}{\Dz}=\Cyl{\Dz}
\end{equation*}
is a common core for both~$\ttonde{\EE{\dUpsilon(B_r)}{\PP_{\mssm_r}},\dom{\EE{\dUpsilon(B_r)}{\PP_{\mssm_r}}}}$ and~$\ttonde{\EE{\dUpsilon(B_r)}{\QP^\eta_r},\dom{\EE{\dUpsilon(B_r)}{\QP^\eta_r}}}$. 
In particular, the corresponding square field operators~$\SF{\dUpsilon(B_r)}{\PP_{\mssm_r}}$ and~$\SF{\dUpsilon(B_r)}{\QP^\eta_r}$ satisfy
\begin{align*}
\SF{\dUpsilon(B_r)}{\PP_{\mssm_r}}(u)=\cdc^{\dUpsilon(B_r)}(u)=\SF{\dUpsilon(B_r)}{\QP^\eta_r}(u)\comma \qquad u\in \Cyl{\Dz} \fstop
\end{align*}

The conclusion now follows from Proposition~\ref{p:Locality}, once we verify the assumptions in Proposition~\ref{p:LocalityProbab}.
Since~$\Cyl{\Dz}$ is a common core by definition of the forms, and since both~$\QP^\eta_r$ and~$\PP_{\mssm_r}$ are probability measures, it suffices to verify Proposition~\ref{p:LocalityProbab}\ref{i:p:LocalityProbab:2} and~\ref{i:p:LocalityProbab:3}, which are respectively~\eqref{eq:p:cSLProjQP:0} and Assumption~\ref{ass:CEr}.
\end{proof}
\end{prop}

Let us now comment on~\eqref{eq:p:cSLProjQP:0}.

\begin{rem}
It is not difficult to show that, if~$(X,\mssd,\mssm)$ is an \MLDS satisfying~$(\Rad{\mssd}{\mssm})$, then the form~$\ttonde{\EE{\dUpsilon(B_r)}{\PP_{\mssm_r}},\dom{\EE{\dUpsilon(B_r)}{\PP_{\mssm_r}}}}$ satisfies~$(\Rad{\mssd_\dUpsilon}{\PP_{\mssm_r}})$ as a consequence of Theorem~\ref{t:Tensor} and of the representation~\eqref{eq:PoissonLebesgue}.
As a consequence,~\eqref{eq:p:cSLProjQP:0} is satisfied, by Remark~\ref{r:CutOff}, as soon as~$(\Rad{\mssd}{\mssm})$ holds.
In particular,~\eqref{eq:p:cSLProjQP:0}  ought to be thought of as an assumption on the base space, rather than on the corresponding configuration space.
In the following therefore,~\eqref{eq:p:cSLProjQP:0} will be essentially immaterial, for in most cases we shall assume~$(\Rad{\mssd}{\mssm})$ for the base space.
We omit to repeat that~\eqref{eq:p:cSLProjQP:0} holds under this assumption for the base space.
\end{rem}

\begin{thm}[Continuous-Sobolev-to-Lipschitz]\label{t:cSLUpsilon}
Let~$(X,\cdc,\mssd)$ be an \MLDS satisfying Assumption \ref{ass:cSLTensor} and~\eqref{eq:p:cSLProjQP:0}, and~$\QP$ be a probability measure on~$\ttonde{\dUpsilon,\A_\mrmv(\Ed)}$ satisfying Assumptions~\ref{ass:CEr} and~\ref{ass:CCr}.
Then, the \EMLDS~$\ttonde{\dUpsilon,\SF{\dUpsilon}{\QP},\mssd_\dUpsilon}$ satisfies~$(\cSL{\T_\mrmv(\Ed)}{\QP}{\mssd_\dUpsilon})$.
\begin{proof}
Let~$u\in\dom{\EE{\dUpsilon}{\QP}}\cap \Cb(\T_\mrmv(\Ed))$ with~$\SF{\dUpsilon}{\QP}(u)\leq 1$ $\QP$-a.e., and fix~$\gamma, \xi\in\dUpsilon$ with~$\mssd_\dUpsilon(\gamma,\xi)<\infty$.
Further fix~$x_0\in X$, let~$\seq{r_n}_n$ be a sequence of radii with~$\nlim r_n=+\infty$ additionally so that, setting~$B_n\eqdef B_{r_n}(x_0)$, it holds that~$\gamma B_n=\xi B_n$.
Then,~$\mssd_\dUpsilon(\gamma_{B_n},\xi_{B_n})$ is increasing to~$\mssd_\dUpsilon(\gamma,\xi)$ as~$n\to\infty$.
Furthermore, for any~$\eta\in\dUpsilon$ we additionally have that~$\gamma_{B_n}+\eta_{B_n^\complement}$ $\Ed$-vaguely converges to~$\gamma$, and analogously for~$\xi$.
As a consequence, by $\T_\mrmv(\Ed)$-continuity of~$u$,
\begin{equation*}
\nlim \ttonde{u_{B_n,\eta}\circ \pr^{B_n}}(\gamma)=\nlim u(\gamma_{B_n}+\eta_{B_n^\complement}) =u(\gamma)\comma \qquad \gamma,\eta\in\dUpsilon\comma
\end{equation*}
and analogously for~$\xi$.

Now, for every~$n$, by Proposition~\ref{p:MarginalFormDomains}\iref{i:p:MarginalFormDomains:3} there exists a $\A_\mrmv(\Ed)$-measurable set~$\Omega_n$ with~$\QP\Omega_n=1$ and
\begin{equation}\label{eq:t:cSLUpsilon:1}
\SF{\dUpsilon(B_n)}{\QP^\eta_{B_n}}(u_{B_n,\eta})\leq 1 \quad \as{\QP^\eta_{B_n}}\comma\qquad \eta\in \Omega_n\fstop
\end{equation}
Set~$\Omega\eqdef \bigcap_n \Omega_n$ and fix~$\eta\in \Omega$.
By the assumption and Proposition~\ref{p:cSLProjQP}, the function~$u_{B_n,\eta}$ is~$\mssd_\dUpsilon$-Lipschitz on~$\dUpsilon(B_n)$ for every~$n$, and
\begin{align*}
\abs{u_{B_n,\eta}(\zeta)-u_{B_n,\eta}(\stigma)}\leq \mssd_\dUpsilon(\zeta,\stigma) \comma\qquad \zeta,\stigma\in\dUpsilon(B_n)\fstop
\end{align*}
As a consequence
\begin{align*}
\abs{u(\gamma)-u(\xi)}=\nlim \abs{u_{B_n,\eta}(\gamma_{B_n})-u_{B_n,\eta}(\xi_{B_n})}\leq \nlim \mssd_\dUpsilon(\gamma_{B_n}, \xi_{B_n})=\mssd_\dUpsilon(\gamma,\xi)\comma
\end{align*}
which concludes the proof.
\end{proof}
\end{thm}

As a consequence of the Theorem and of~\eqref{eq:EquivalenceRadStoL}, we immediately obtain the distance-Sobolev-to-Lipschitz property.
\begin{cor}[Distance-Sobolev-to-Lipschitz I]\label{c:dSLUpsilon-cSL}
Let~$(X,\cdc,\mssd)$ be an \MLDS satisfying Assumption~\ref{ass:cSLTensor} and~\eqref{eq:p:cSLProjQP:0}, and~$\QP$ be a probability measure on~$\ttonde{\dUpsilon,\A_\mrmv(\Ed)}$ satisfying Assumptions~\ref{ass:CEr} and~\ref{ass:CCr}.
Then,
\begin{equation*}
\mssd_\dUpsilon\geq \mssd_\QP \fstop
\end{equation*}
\end{cor}

%%%%%%%%%%%%%%%%%%%%%%%%%%%%%%%%%%%%%%%%%%%%%%%%%%%%%%%%%%%%%%%%%%%%%%%%%%%%%%%%%%%%%%%%%%% dCSL %%%%%%%%%%%%%%%%%%%%%%%%%%%%%%%%%%%%%%%%%%%%%%%%%%%%%%%%%%%%%%%%%%%%%%%%%%%%%%%%%%%%%%%%%%%%

\subsubsection{The \texorpdfstring{$\mssd_\dUpsilon$}{d}-continuous-Sobolev-to-Lipschitz property}\label{sss:cdSL}
In this section we discuss the distance-continuous-Sobolev-to-Lipschitz property on configuration spaces. 
In light of Remark~\ref{r:dcSL}, the proof of~$(\dcSL{\mssd_\dUpsilon}{\QP}{\mssd_\dUpsilon})$ does not void that of~$(\cSL{\T_\mrmv(\msE)}{\QP}{\mssd_\dUpsilon})$ presented in the previous section.
For all practical purposes however, the former property is `stronger', and will play an important role in subsequent work, in particular in relationship with the~$L^\infty(\QP)$-to-$\Cb(\mssd_\dUpsilon)$-regularization property of the semigroup~$\TT{\dUpsilon}{\QP}_\bullet$, e.g.~\cite[Thm.~7.1]{KonLytRoe02}. 

\begin{thm}[Distance-continuous Sobolev-to-Lipschitz]\label{t:dcSLUpsilon}
Let~$(X,\cdc,\mssd)$ be an \MLDS satisfying Assumption~\ref{ass:cSLTensor} , and~\eqref{eq:p:cSLProjQP:0}, and~$\QP$ be a probability measure on~$\ttonde{\dUpsilon,\A_\mrmv(\Ed)}$ satisfying Assumptions \ref{ass:CEr} and~\ref{ass:CCr}.
Then, $\ttonde{\EE{\dUpsilon}{\QP},\dom{\EE{\dUpsilon}{\QP}}}$ satisfies~$(\dcSL{\mssd_\dUpsilon}{\QP}{\mssd_\dUpsilon})$.
\begin{proof}
Let~$u\in\dom{\EE{\dUpsilon}{\QP}}$ be so that~$\SF{\dUpsilon}{\QP}(u)\leq 1$ $\QP$-a.e., and suppose that there exists a $\A_\mrmv(\Ed)$-measurable, bounded $\mssd_\dUpsilon$-continuous representative~$\rep u$ of~$u$.

Fix~$x_0\in X$ and set~$B_r\eqdef B_r(x_0)$ for~$r>0$.
By Proposition~\ref{p:MarginalFormDomains}\iref{i:p:MarginalFormDomains:3} for every~$r>0$ there exists a $\A_\mrmv(\Ed)$-measurable set~$\Omega_r$ with~$\QP\Omega_r=1$ and
\begin{equation}\label{eq:t:dcSLUpsilon:1}
\SF{\dUpsilon(B_r)}{\QP^\eta_{B_r}}(u_{B_r,\eta})\leq 1 \quad \as{\QP^\eta_{B_r}}\comma\qquad \eta\in \Omega_r\fstop
\end{equation}
Set~$\Omega\eqdef \bigcap_{r\in\Q^+} \Omega_r$ and note that~$\QP \Omega=1$.

Now, fix~$\eta\in\Omega$ and~$\gamma\in\dUpsilon$ with~$\mssd_\dUpsilon(\gamma,\eta)<\infty$. 
There exists a sequence of radii~$\seq{r_n}_n\subset \Q^+$ with~$\nlim r_n=+\infty$ additionally so that, setting~$B_n\eqdef B_{r_n}(x_0)$, it holds that~$\gamma B_n=\eta B_n$.
Then,
\begin{equation}\label{eq:t:dcSLUpsilon:1.5}
\nlim \mssd_\dUpsilon(\gamma_{B_n},\eta_{B_n})=\mssd_\dUpsilon(\gamma,\eta)\comma 
\end{equation}
which implies that~$\nlim \mssd_\dUpsilon(\gamma_{B_n^\complement},\eta_{B_n^\complement})=0$.
As a consequence,
\begin{align*}
\nlim \mssd_\dUpsilon\ttonde{\gamma,\gamma_{B_n}+\eta_{B_n^\complement}}=\nlim \mssd_\dUpsilon\ttonde{\gamma_{B_n}+\gamma_{B_n^\complement},\gamma_{B_n}+\eta_{B_n^\complement}} \leq \nlim \mssd_\dUpsilon(\gamma_{B_n^\complement},\eta_{B_n^\complement}) =0\fstop
\end{align*}
By $\mssd_\dUpsilon$-continuity of~$\rep u$, the latter convergence implies that
\begin{align}\label{eq:t:dcSLUpsilon:2}
\nlim \rep u_{B_n,\eta}(\gamma)=\nlim \rep u(\gamma_{B_n}+\eta_{B_n^\complement})=\rep u(\gamma) \fstop
\end{align}

Now, let us show that~$\rep u_{B_n,\eta}$ is $\mssd_\dUpsilon$-continuous on~$\dUpsilon(E)$.
Indeed let~$\seq{\zeta_k}_k\in\dUpsilon(E)$ be $\mssd_\dUpsilon$-convergent to~$\zeta\in\dUpsilon(E)$. Then,
\begin{equation*}
\klim \mssd_\dUpsilon(\zeta_k+\eta_{B_n^\complement},\zeta+\eta_{B_n^\complement})\leq \klim \mssd_\dUpsilon(\zeta_k,\zeta)=0\comma\qquad n\in \N_1\comma
\end{equation*}
which implies that
\begin{align*}
\klim \rep u_{B_n,\eta}(\zeta_k)=\nlim \rep u(\zeta_k+\eta_{B_n^\complement})= \rep u(\zeta+\eta_{B_n})=\rep u_{B_n,\eta}(\zeta)\comma \qquad n\in \N_1\comma
\end{align*}
that is, the sought $\mssd_\dUpsilon$-continuity of~$\rep u_{B_n,\eta}$.
By Proposition~\ref{p:DistTopE}\iref{i:p:DistTopE:2}, $\rep u_{B_n,\eta}$ is also $\T_\mrmv(\Ed)$-continuous, hence, by~\eqref{eq:t:dcSLUpsilon:1} and Proposition~\ref{p:cSLProjQP}, the function~$\rep u_{B_n,\eta}$ is~$\mssd_\dUpsilon$-Lipschitz on~$\dUpsilon(B_n)$ for every~$n$, and
\begin{align}\label{eq:t:dcSLUpsilon:3}
\abs{\rep u_{B_n,\eta}(\zeta)- \rep u_{B_n,\eta}(\stigma)}\leq \mssd_\dUpsilon(\zeta,\stigma) \comma\qquad \zeta,\stigma\in\dUpsilon(B_n)\fstop
\end{align}

Therefore, combining~\eqref{eq:t:dcSLUpsilon:2},~\eqref{eq:t:dcSLUpsilon:3} and~\eqref{eq:t:dcSLUpsilon:1.5},
\begin{align*}
\abs{\rep u(\gamma)-\rep u(\eta)}=\nlim \abs{\rep u_{B_n,\eta}(\gamma_{B_n})-\rep u_{B_n,\eta}(\eta_{B_n})}\leq \nlim \mssd_\dUpsilon(\gamma_{B_n},\eta_{B_n})=\mssd_\dUpsilon(\gamma,\eta) \comma
\end{align*}
which concludes that~$\rep u$ is $\mssd_\dUpsilon$-Lipschitz on~$\Omega\times \dUpsilon$.
In order to extend this property to the whole of~$\dUpsilon$, define~$\reptwo u\colon \dUpsilon\to\R$ as the (upper) constrained McShane extension of~$\rep u\restr_\Omega$ defined as in~\eqref{eq:McShane}.
Since~$\QP\Omega=1$, the function~$\reptwo u$ is $\QP$-measurable, and the conclusion follows by Lemma~\ref{l:McShane}\iref{i:l:McShane:2}.
\end{proof}
\end{thm}

\begin{rem}[Comparison with~{\cite{RoeSch99}}, part II]\label{r:RoeSch99SL}
The distance-continuous-Sobolev-to-Lipschitz property on configuration spaces was originally discussed by M.~R\"ockner and A.~Schied in~\cite{RoeSch99}.
Indeed, it was shown to holds on the configuration space over a Riemannian manifold in~\cite[Thm.~1.5(i)]{RoeSch99} under the assumptions~\cite[Ass.~1.1]{RoeSch99}. 
As discussed in Remark~\ref{r:RoeSch99SL0}, these assumption are usually difficult to verify outside the class of Ruelle-type Gibbs measures.

As already noted in~\cite[p.~331, Rmk.]{RoeSch99}, if~$u\in L^2(\QP)$ has a $\mssd_\dUpsilon$-Lipschitz $\QP$-representative, then it has in fact uncountably many such representatives, with arbitrarily large Lipschitz constant.
This shows the necessity to allow for different representatives in the formulation of~$(\dcSL{\mssd_\dUpsilon}{\QP}{\mssd_\dUpsilon})$.
\end{rem}

%%%%%%%%%%%%%%%%%%%%%%%%%%%%%%%%%%%%%%%%%%%%%%%%%%%%%%%%%%%%%%%%%%%%%%%%%%%%%%%%%%%%%%%%% dSL %%%%%%%%%%%%%%%%%%%%%%%%%%%%%%%%%%%%%%%%%%%%%%%%%%%%%%%%%%%%%%%%%%%%%%%%%%%%%%%%%%%%%%%%%%%%%%%

%%%%%%%%%%%%%%%%%%%%%%%%%%%%%%%%%%%%%%%%%%%%%%%%%%%%%%%%%%%%%%%%%%%%%%%%%%%%%%%%%%%%%%%%%%% d = d %%%%%%%%%%%%%%%%%%%%%%%%%%%%%%%%%%%%%%%%%%%%%%%%%%%%%%%%%%%%%%%%%%%%%%%%%%%%%%%%%%%%%%%%%%%%
\subsection{Identification of distances}
Let us now come to the identification of the extended distances~$\mssd_\QP$ and~$\mssd_\dUpsilon$.

\begin{thm}\label{t:StoL2}
Let~$(X,\cdc,\mssd)$ be an \MLDS with~$\Dz\subset \bLip(\mssd)$, satisfying both~$(\Rad{\mssd}{\mssm})$ and Assumption \ref{ass:cSLTensor}, and~$\QP$ be a probability measure on~$\ttonde{\dUpsilon,\A_\mrmv(\Ed)}$ satisfying Assumptions~\ref{ass:CEr} and~\ref{ass:CCr}.
Then,
\begin{equation*}
\mssd_\QP=\mssd_\dUpsilon \fstop
\end{equation*}
\begin{proof}
By Theorem~\ref{t:Rademacher}, the space~$\ttonde{\dUpsilon,\SF{\dUpsilon}{\QP},\mssd_\dUpsilon}$ satisfies~$(\Rad{\mssd_\dUpsilon}{\QP})$.
By Proposition~\ref{p:ConfigEMTS}, the space~$\ttonde{\dUpsilon,\T_\mrmv(\Ed),\mssd_\dUpsilon}$ is an extended metric-topological space in the sense of Definition~\ref{d:AES}.
By~\cite[Lem.~3.6]{LzDSSuz20}, the above two facts together imply that~$\ttonde{\dUpsilon,\SF{\dUpsilon}{\QP},\mssd_\dUpsilon}$ satisfies~$(\dRad{\mssd_\dUpsilon}{\QP})$, which is the inequality~`$\leq$'.
The converse inequality is Corollary~\ref{c:dSLUpsilon-cSL}.
\end{proof}
\end{thm}

\begin{cor}\label{c:PoissonEMTSp}
Under the same assumptions as in either Theorem~\ref{t:StoL2}, the space $\ttonde{\dUpsilon, \T_\mrmv(\Ed), \mssd_\QP}$ is a complete extended metric topological space. In particular:~$\mssd_\QP$ is non-trivial, i.e.~$\mssd_\QP\not\equiv 0,+\infty$.
\end{cor}

\begin{rem}[Relation with Varadhan short time asymptotic]
The identification $\mssd_\QP=\mssd_\dUpsilon$ does not necessarily imply the integral Varadhan short time asymptotic w.r.t.\ $\mssd_\dUpsilon$ since $\mssd_\dUpsilon$ does not generate the topology. This point will be investigated in \S\ref{s:Applications} and also in a forthcoming paper.
\end{rem}

%%%%%%%%%%%%%%%%%%%%%%%%%%%%%%%%%%%%%%%%%%%%%%%%%%%%%%%%%%%%%%%%%%%%%%%%%%%%%%%%%%%%%%%%%%%%%%%%%%       APPLICATIONS       %%%%%%%%%%%%%%%%%%%%%%%%%%%%%%%%%%%%%%%%%%%%%%%%%%%%%%%%%%%%%%%%%%%%%%%%%%%%%%%%%%%%%%%%%%%
\section{Applications}\label{s:Applications}
We collect here some applications of the results and ideas presented in~\S\ref{s:Interplay}.
As anticipated in~\S\ref{s:Intro}, further applications to the study of curvature of configuration spaces and to stochastic analysis of infinite particle systems will be the subject of other papers in this series.

\subsection{Applications of the Rademacher property}
Let us start by collecting applications of the Rademacher property, hence in particular of Theorem~\ref{t:Rademacher}.

\subsubsection{Topological properties: quasi-regularity and strong locality}
The following result shows that the Rademacher property implies the quasi-regularity of the form on~$\dUpsilon$.

\begin{prop}[Quasi-regularity]\label{p:QRegSLoc}
Let~$(\mcX,\cdc)$ be an \MLDS, and~$\QP$ be a probability measure on $\ttonde{\dUpsilon,\A_\mrmv(\Ed)}$ satisfying Assumption~\ref{ass:CWP}.
If~$(\dUpsilon,\SF{\dUpsilon}{\QP},\mssd_\dUpsilon)$ satisfies $(\Rad{\mssd_\dUpsilon}{\QP})$, then the form $\ttonde{\EE{\dUpsilon}{\QP},\dom{\EE{\dUpsilon}{\QP}}}$ is a conservative quasi-regular strongly local Dirichlet form on~$L^2(\QP)$.
\begin{proof}
Since~$\QP$ is a probability measure,~$L^2(\QP)$ contains constant functions.
By definition of~$\SF{\dUpsilon}{\QP}$ we have~$\SF{\dUpsilon}{\QP}(\car)\equiv 0$, and conservativeness follows.

In order to show the quasi-regularity, it suffices to show that the $\EE{\dUpsilon}{\QP}_1$-capacity is tight.
Since~$\QP$ is Radon by Remark~\ref{r:QPRadon}, there exists a sequence~$\seq{\Kappa_n}_n$ of $\T_\mrmv(\Ed)$-compact sets so that~$\nlim \QP \Kappa_n=1$.
For~$n\in \N$ let~$\rho_{\Kappa_n}$ be defined as in~\eqref{eq:W2Upsilon:0} and note that the sub-level sets~$\set{\rho_{\Kappa_n}\leq 1/2}$ are $\T_\mrmv(\Ed)$-compact by Proposition~\ref{p:W2Upsilon}\iref{i:p:W2Upsilon:8}.
By~$(\Rad{\mssd_\dUpsilon}{\QP})$ we have that~$\sup_n \EE{\dUpsilon}{\QP}(\rho_{\Kappa_n}\wedge 1)<\infty$, hence, by e.g.~\cite[Lem.~I.2.12]{MaRoe92}, there exists a subsequence~$\seq{k_n}_n$ so that~$u_n\eqdef \tfrac{1}{n}\sum_j^n \rho_{\Kappa_{k_j}}\wedge 1$ $\ttonde{\EE{\dUpsilon}{\QP}_1}^{1/2}$-converges to~$0$.
Furthermore,~$u_n\geq 1/2$ on~$\tset{\rho_{\Kappa_{k_j}}\geq 1/2}$ for all~$j\leq n$.
As a consequence,
\begin{align*}
\Cap_{\EE{\dUpsilon}{\QP}_1}\set{\rho_{\Kappa_{k_n}}> 1/2}\leq \Cap_{\EE{\dUpsilon}{\QP}_1}\set{u_n> 1/2} \leq \EE{\dUpsilon}{\QP}_1(u_n) \xrightarrow{n\rar\infty} 0\comma
\end{align*}
which shows that~$\Cap_{\EE{\dUpsilon}{\QP}_1}$ is tight.

Concerning strong locality, by conservativeness it suffices to show locality.
By quasi-regularity, this follows exactly as in the proof of~\cite[Prop.~4.12]{MaRoe00}.
\end{proof}
\end{prop}

\begin{rem}[Comparison with~{\cite[\S4]{MaRoe00}}, part II]
A proof of quasi-regularity can also be provided by showing that an~\MLDS satisfying the Rademacher property additionally satisfies Assumption~(Q) in~\cite[p.~298]{MaRoe00}; also cf.~Remark~\ref{r:ComparisonMR00-1} for a comparison of all other assumptions.
\end{rem}

As a corollary of Proposition~\ref{p:QRegSLoc} and Theorem~\ref{t:Rademacher} we obtain the following.
\begin{cor}\label{c:RadQReg}
Let~$(\mcX,\cdc,\mssd)$ be an \MLDS satisfying~$(\Rad{\mssd}{\mssm})$ and assume that $\Dz\subset \bLip(\T,\mssd)$. Further let~$\QP$ be a probability measure on~$\dUpsilon$ satisfying Assumptions~\ref{ass:AC} and~\ref{ass:CWP}.
Then the form $\ttonde{\EE{\dUpsilon}{\QP},\dom{\EE{\dUpsilon}{\QP}}}$ is a conservative quasi-regular strongly local Dirichlet form on~$L^2(\QP)$.
\end{cor}

%\purple{
When $X$ satisfies~$(\Rad{\mssd}{\mssm})$, the above corollary gives a very mild sufficient condition for the quasi-regularity. In particular, when~$\QP$ is a quasi-Gibbs measure such that~$\ttonde{\EE{\dUpsilon}{\QP},\CylQP{\QP}{\Dz}}$ is closable (i.e.,~\ref{ass:CWP} holds), then the corresponding closure~$\ttonde{\EE{\dUpsilon}{\QP},\dom{\EE{\dUpsilon}{\QP}}}$ is automatically quasi-regular.
For the definition of the quasi-Gibbs measures see Definition~\ref{d:QuasiGibbs} below.
\begin{cor}\label{c:EucQReg}
Let~$(\mcX,\cdc,\mssd)$ be an \MLDS satisfying~$(\Rad{\mssd}{\mssm})$ and assume that $\Dz\subset \bLip(\T,\mssd)$.
Let~$\QP$ be a quasi-Gibbs measure satisfying~\ref{ass:CWP}. Then, $\ttonde{\EE{\dUpsilon}{\QP},\dom{\EE{\dUpsilon}{\QP}}}$ is quasi-regular.
%Then the form $\ttonde{\EE{\dUpsilon}{\QP},\dom{\EE{\dUpsilon}{\QP}}}$ is a conservative quasi-regular strongly local Dirichlet form on~$L^2(\QP)$.
\end{cor}
%}

\subsubsection{Distance and heat-kernel estimates}\label{sss:DistHeatEstimates}
We collect here two results on estimates for the $L^2$-transporta\-tion distance and the heat kernel.

The first one is a type of integral Gaussian estimates for the heat semigroup in terms of the set-to-set $L^2$-transportation distance.
On general Dirichlet spaces, and in fact on most infinite-dimensional spaces, the existence of an absolutely continuous heat kernel measure should not be expected.
Gaussian estimates are therefore only available in integral form. In such form they may be addressed in terms of the maximal functions in Definition~\ref{d:MaximalFunction}.
However, as anticipated in~\S\ref{sss:DistHeatEstimates} and thoroughly discussed in~\cite[\S4]{LzDSSuz20}, maximal functions are hardly computable. 
The next corollary gives an explictly computable upper bound for the heat semigroup, as further detailed in Remark~\ref{r:ExplicitEstimates}.

\begin{cor}[Upper Gaussian estimates]\label{c:HeatKernelEstimateConfig}
Let~$(\mcX,\cdc)$ be an \MLDS, and~$\QP$ be a probability measure on~$\ttonde{\dUpsilon,\A_\mrmv(\Ed)}$ satisfying Assumption~\ref{ass:CWP}.
If~$(\dUpsilon,\SF{\dUpsilon}{\QP},\mssd_\dUpsilon)$ satisfies $(\Rad{\mssd_\dUpsilon}{\QP})$, e.g.\ under the assumptions of Theorem~\ref{t:Rademacher}, then
\begin{equation*}
\TT{\dUpsilon}{\QP}_t(\Lambda_1, \Lambda_2) \leq \sqrt{\QP\Lambda_1\cdot \QP\Lambda_2} \cdot \exp \tonde{-\frac{\mssd_\dUpsilon(\Lambda_1, \Lambda_2)^2}{2t}}\comma\qquad t>0\comma
\end{equation*}
for every~$\Lambda_1,\Lambda_2\in\A_\mrmv(\msE)$ so that~$\mssd_\dUpsilon(\emparg, \Lambda_i)$ is $\A_\mrmv(\msE)$-measurable for either~$i=1,2$.

\begin{proof}
By~\cite[Lem.~4.16]{LzDSSuz20}, we have that
\begin{align*}
-\tonde{\QP\text{-}\essinf_{\eta\in\Lambda_1} \hr{\QP,\Lambda_2}(\eta)}^2 \leq - \tonde{\QP\text{-}\essinf_{\eta\in\Lambda_1} \mssd_\dUpsilon(\eta, \Lambda_2)}^2\comma
\end{align*}
and analogously when exchanging the roles of~$\Lambda_1$ and~$\Lambda_2$. The conclusion now follows from~\cite[Thm.~4.2]{AriHin05}.
\end{proof}
\end{cor}

\begin{rem}[Explicit estimates]\label{r:ExplicitEstimates}
Under the same assumptions as in Corollary~\ref{c:HeatKernelEstimateConfig}, let~$U\in\Ed$ be $\T$-open and so that $X\setminus U$ has non-empty $\T$-interior, and fix~$\gamma$, $\eta\in\dUpsilon$.
Further suppose that~$\QP$ has full $\T_\mrmv(\Ed)$-support.
Then, the estimate in Corollary~\ref{c:HeatKernelEstimateConfig}, specialized to specific sets.
For sets~$\Lambda=\Lambda_{\gamma, U}$ of the form~\eqref{eq:W2Upsilon:0}, we have that
\begin{align}\label{eq:SemigroupEstimate}
\TT{\dUpsilon}{\QP}_t(\Lambda_{\gamma, U}, \Lambda_{\eta, U}) \leq \sqrt{\QP \Lambda_{\gamma, U} \cdot \QP\Lambda_{\eta, U}} \cdot \exp\tonde{-\frac{\mssd_\dUpsilon(\gamma_U,\eta_U)^2}{2t}} \semicolon
\end{align}
The above estimate can be made even more explicit in the case of Poisson resp.\ Ruelle-type Gibbs measures, in which case one may compute explicitly, resp.\ sharply estimate, the $\QP$-measure of the sets involved, e.g.\ by means of~\eqref{eq:PoissonRestriction} and~\eqref{eq:PoissonLebesgue}, resp.\ Remark~\ref{r:ExamplesQT}.

\begin{proof}[Proof of~\eqref{eq:SemigroupEstimate}]
Since~$\mssd_\dUpsilon(\emparg, \Lambda_{\gamma, U})$ is continuous (hence measurable) by Proposition~\ref{p:W2Upsilon2}\iref{i:p:W2Upsilon:12}, we may apply Corollary~\ref{c:HeatKernelEstimateConfig} to obtain 
\begin{align*}
\TT{\dUpsilon}{\QP}_t(\Lambda_{\gamma, U}, \Lambda_{\eta, U}) %\leq& \sqrt{\QP \Lambda_{\gamma, U} \cdot \QP\Lambda_{\eta, U}} \cdot \exp\tonde{-\frac{\QP\text{-}\essinf_{\Lambda_{\gamma,U}}\mssd_\dUpsilon(\emparg,\Lambda_{\eta, U})^2}{2t}}
%\\
\leq& \sqrt{\QP \Lambda_{\gamma, U} \cdot \QP\Lambda_{\eta, U}} \cdot \exp\tonde{-\frac{\mssd_\dUpsilon(\Lambda_{\gamma, U},\Lambda_{\eta, U})^2}{2t}}\comma
%\\
%=& \sqrt{\QP \Lambda_{\gamma, U} \cdot \QP\Lambda_{\eta, U}} \cdot \exp\tonde{-\frac{\mssd_\dUpsilon(\gamma_U,\eta_U)^2}{2t}}\fstop \qedhere
\end{align*}
and the conclusion follows by computing~$\mssd_\dUpsilon(\Lambda_{\gamma, U},\Lambda_{\eta, U})\leq\mssd_\dUpsilon(\gamma_U,\eta_U)$.
\end{proof}
\end{rem}

As a second result, we show that the \emph{analytic} properties of~$\ttonde{\EE{\dUpsilon}{\QP},\dom{\EE{\dUpsilon}{\QP}}}$ ---~or, equivalently, the \emph{stochastic} properties of the associated Markov process~--- allow us to make statements about the \emph{geometry} of the $L^2$-transportation distance~$\mssd_\dUpsilon$.

Since~$\mssd_\dUpsilon$ is an extended distance, it is natural to ask: When is the set-to-set distance~$\mssd_\dUpsilon(\Lambda_1,\Lambda_2)$ finite?
In general, it is extremely difficult to answer this question for arbitrary sets~$\Lambda_1$, $\Lambda_2$ only by means of the metric geometry of~$\ttonde{\dUpsilon,\mssd_\dUpsilon}$.
Indeed the answer may depend sensibly on the \emph{shape} of the sets~$\Lambda_i$, and thus on the asymptotic distribution of the points in the support of the elements~$\gamma_i\in\Lambda_i$ at the infiniy of~$X$ (i.e.\ outside $\mssd$-bounded sets).

The next Corollary is particularly relevant, since it shows that \emph{geometric} information on the shape of the sets~$\Lambda_i$ can be substituted by mere \emph{numerical} information on their $\QP$-measure.

\begin{cor}[Finiteness of set-to-set $L^2$-transportation distance]\label{c:DistanceIrreducibile}
Let~$(\mcX,\cdc)$ be an \MLDS, and~$\QP$ be a probability measure on~$\ttonde{\dUpsilon,\A_\mrmv(\Ed)}$ satisfying Assumption~\ref{ass:CWP}.
If~$(\dUpsilon,\SF{\dUpsilon}{\QP},\mssd_\dUpsilon)$ satisfies $(\Rad{\mssd_\dUpsilon}{\QP})$, e.g.\ under the assumptions of Theorem~\ref{t:Rademacher}, and if the form $\ttonde{\EE{\dUpsilon}{\QP},\dom{\EE{\dUpsilon}{\QP}}}$ is irreducible, then
\begin{align*}
\Xi,\Lambda\in \A_\mrmv(\Ed)\comma \QP\Xi, \QP\Lambda >0\comma \mssd_\dUpsilon(\emparg, \Xi) \text{ is $\A_\mrmv(\Ed)$-measurable} \implies \mssd_\dUpsilon(\Lambda, \Xi)<\infty \fstop
\end{align*}

\begin{proof}
Since~$\mssd_\dUpsilon(\emparg, \Xi)$ is $\A_\mrmv(\Ed)$-measurable, it follows from Theorem~\ref{t:Rademacher} and the Definition~\ref{d:MaximalFunction} of maximal function, that~$\mssd_\dUpsilon(\emparg, \Xi)\leq \hr{\QP,\Xi}$ $\QP$-a.e..
As a consequence,
\begin{align*}
\mssd_\dUpsilon(\Lambda, \Xi)\eqdef \inf_{\gamma\in\Lambda} \mssd_\dUpsilon(\gamma, \Xi)\leq \QP\textrm{-}\essinf_{\gamma\in\Lambda} \hr{\QP,\Xi}(\gamma)\defeq \hr{\QP}(\Lambda, \Xi) \fstop
\end{align*}
By irreducibility, and since~$\QP\Lambda,\QP\Xi>0$, we have~$\TT{\dUpsilon}{\QP}_t(\Lambda,\Xi)>0$ for every~$t>0$.
As a consequence,~$\hr{\QP}(\Lambda, \Xi)<\infty$ by combining~\cite[Prop.~5.1 and Prop.~3.11]{AriHin05}, and the conclusion follows.
\end{proof}
\end{cor}

\begin{rem}[About irreducibility]
The irreducibility of the form~$\ttonde{\EE{\dUpsilon}{\QP},\dom{\EE{\dUpsilon}{\QP}}}$ will be the subject of further work in this series.
It is well-studied on configuration spaces over Riemannian manifolds, in which case it is shown to hold for Poisson and grand-canonical Gibbs measures, under the assumption that the corresponding form~$\ttonde{\EE{X}{\mssm},\dom{\EE{X}{\mssm}}}$ be conservative, see~\cite[Thm.~4.3]{AlbKonRoe98} and~\cite[p.~289, Thm.s~6.5, 6.6, and Cor.~6.2]{AlbKonRoe98b}.
\end{rem}

\subsection{Varadhan-type short-time asymptotics}\label{sss:VSTA}
Profiting the ideas in the proof of the continuous Sobolev-to-Lipschitz property,~\S\ref{sss:cSL}, we present here some statements about Varadhan's short-time asymptotics for the heat semigroup of~$\tonde{\EE{\dUpsilon}{\QP},\dom{\EE{\dUpsilon}{\QP}}}$.

In general, it is not easy to pointwise compare the point-to-set distance function~$\mssd_\mssm(\emparg, A)$ and the maximal function~$\hr{\mssm,A}$.
The reason behind this fact lies in the potential lack of measurability of extended point-to-set distances, and in the fact that~$\hr{\mssm,A}$ only depends on the $\mssm$-class of~$A$, whereas~$\mssd(\emparg, A)$ depends quite sensibly on the representative set~$A$; see~\cite[\S4.2]{LzDSSuz20} for the interplay between point-to-set distances and maximal functions.

\subsubsection{Integral Varadhan-type estimates under~\texorpdfstring{\ref{ass:SLTensor}}{SLtensor}}
Everywhere in the following, let~$x_0\in X$ be fixed, and set~$B_r\eqdef B_r^\mssd(x_0)$ for each $r>0$. As customary, for every~$r>0$ we let
\begin{itemize}
\item $\QP^\eta_r\eqdef \QP^\eta_{B_r}$ be defined as in~\eqref{eq:ProjectedConditionalQP};

\item and $u_{r,\eta}\eqdef u_{B_r,\eta}$ be defined as in~\eqref{eq:ConditionalFunction}.
\end{itemize}

\medskip

The form~$\ttonde{\EE{\dUpsilon}{\QP},\dom{\EE{\dUpsilon}{\QP}}}$ satifies the following \emph{approximate} Sobolev-to-$\mssd_\dUpsilon$-Lipschitz property.

\begin{prop}[Approximate $\mssd_\dUpsilon$-Sobolev-to-Lipschitz]\label{p:ReducedSL}
Let~$(\mcX,\cdc,\mssd)$ be an \MLDS satisfying \linebreak \ref{ass:SLTensor} , and~\eqref{eq:p:cSLProjQP:0}, and $\QP$ be a probability measure on~$\ttonde{\dUpsilon,\A_\mrmv(\Ed)}$ satisfying Assumptions \ref{ass:CEr} and \ref{ass:CCr}.
Further assume that the form~$\ttonde{\EE{\dUpsilon}{\QP},\dom{\EE{\dUpsilon}{\QP}}}$ is quasi-regular and let~$u\in\dom{\EE{\dUpsilon}{\QP}}$ with~$\SF{\dUpsilon}{\QP}(u)\leq 1$ $\QP$-a.e..

Then, for every $\A_\mrmv(\Ed)^\QP$-measurable quasi-continuous $\QP$-representative~$\reptwo u$ of~$u$ there exists a $\QP$-conegligible set~$\Omega=\Omega_{\reptwo u}\in \A_\mrmv(\Ed)$ enjoying the following property.
For every~$\eps>0$, every~$\eta\in\Omega$ and every~$\gamma\in\dUpsilon$ with~$\mssd_\dUpsilon(\gamma,\eta)<\infty$ there exists~$\gamma_\eps\in\dUpsilon$ with
\begin{align}\label{eq:p:ReducedSL:0}
\mssd_\dUpsilon(\gamma,\gamma_\eps)<\eps \qquad \text{and} \qquad \abs{\reptwo u(\gamma_\eps)-\reptwo u(\eta)}\leq \mssd_\dUpsilon(\gamma,\eta)+\eps \fstop
\end{align}

\begin{proof}
Fix~$u\in\dom{\EE{\dUpsilon}{\QP}}$ with~$\SF{\dUpsilon}{\QP}(u)\leq 1$ $\QP$-a.e..
For a sequence of radii~$\seq{r_n}_n$ with~$r_n\nearrow_n \infty$ to be determined later, set~$B_n\eqdef B_{r_n}$.
By Proposition~\ref{p:MarginalFormDomains}\iref{i:p:MarginalFormDomains:3}, we may conclude for every~$n\in\N$ that~$\SF{\dUpsilon(B_n)}{\QP^\eta_n}(u_{n,\eta})\leq 1$ $\QP^\eta_n$-a.e.\ on~$\dUpsilon(B_n)$ for $\QP$-a.e.~$\eta\in\dUpsilon$ for every~$\eta$ in some $\QP$-conegligible set~$\Omega_n$.
Let us set~$\Omega_0\eqdef \cap_n \Omega_n$ and note that it is $\QP$-conegligible.
By Proposition~\ref{p:cSLProjQP} for~\ref{ass:SLTensor}, the function~$u_{n,\eta}$ on~$\dUpsilon(B_n)$ has a $\QP^\eta_n$-representative~$\rep{u}^{\eta,n}$ satisfying
\begin{enumerate}[$(a)$]
\item\label{i:p:ReducedSL:1} $\rep{u}^{\eta,n}$ is $\mssd_\dUpsilon$-Lipschitz, hence $\T_\mrmv(\Ed)$-continuous by Proposition~\ref{p:DistTopE}.
\end{enumerate}

By definition of quasi-regularity,~$\ttonde{\EE{\dUpsilon}{\QP},\dom{\EE{\dUpsilon}{\QP}}}$ admits a $\T_\mrmv(\Ed)$-compact nest~$\seq{\Kappa^n}_n$ witnessing the $\EE{\dUpsilon}{\QP}$-quasi-continuity of some $\QP$-representative $\reptwo u$ of~$u$.
Without loss of generality, we may and will assume that~$\Kappa^n\subset \Kappa^{n+1}$ for each~$n\in\N$.
Setting~$\Xi^n\eqdef \Kappa^n\cap \Omega_0$ we conclude that~$\Omega\eqdef \cup_n \Xi^n$ is $\QP$-conegligible.
Note that~$\Omega$ does not depend on the chosen sequence of radii~$\seq{r_n}_n$.

By definition of quasi-continuity,~$\reptwo u^n\eqdef \reptwo u\restr_{\Kappa^n}$ is $\T_\mrmv(\Ed)$-continuous.
Since~$B_n$ is $\T$-open, the projection~$\pr^{B_n}\colon \dUpsilon\to\dUpsilon(B_n)$ is $\T_\mrmv(\Ed)$-continuous.
Since~$\reptwo u^n$ is $\T_\mrmv(\Ed)$-continuous on~$\Kappa^n$, we conclude that
\begin{enumerate}[$(a)$]\setcounter{enumi}{1}
\item\label{i:p:ReducedSL:2} for every~$\eta\in\Omega_0$ and~$n\in\N$ the function $\reptwo u^n_{n,\eta}\colon\gamma \to \reptwo u^n(\gamma_{B_n}+\eta_{B_n^\complement})$ is $\T_\mrmv(\Ed)$-continuous on
\begin{equation*}
\Xi^n_{n,\eta}\eqdef \set{\gamma\in\dUpsilon(B_n) : \gamma+\eta_{B_n^\complement}\in \Xi^n} \fstop
\end{equation*}
\end{enumerate}

For every~$\eta\in\Omega_0$, it follows from the definition of~$\rep u^{\eta,n}$ that
\begin{align*}
\reptwo u_{n,\eta}= \rep u^{\eta,n} \quad \as{\QP^\eta_n} \quad \text{on } \dUpsilon(B_n) \comma
\end{align*}
whence
\begin{align}\label{eq:l:ReducedFormSL:1}
\reptwo u^n_{n,\eta} = \rep u^{\eta ,r} \quad \as{\QP^\eta_n} \quad \text{on } \Xi^n_{n,\eta} \fstop
\end{align}
Let us now show how to improve~\eqref{eq:l:ReducedFormSL:1} to an equality everywhere on~$\Xi^n_{n,\eta}$.
As a consequence of Assumption~\ref{ass:CEr},~$\QP^\eta_n$ has full $\T_\mrmv(\Ed)$-support on~$\dUpsilon(B_n)$ for~$\QP$-a.e.~$\eta\in\dUpsilon$, thus, up to possibly changing~$\Omega_0$ by a $\QP$-negligible set,~$\supp \QP^\eta_n=\dUpsilon(B_n)$ for all~$\eta\in\Omega_0$ for each~$n\in\N$.
By~\ref{i:p:ReducedSL:1} and~\ref{i:p:ReducedSL:2} above, we have that~$\reptwo u^n_{n,\eta}$ and~$\rep u^{n,\eta}\restr_{\Xi^n_{n,\eta}}$ are $\T_\mrmv(\Ed)$-continuous functions, equal a.e.\ w.r.t.\ the measure~$\QP^\eta_n$.
Since the latter has full $\T_\mrmv(\Ed)$-support on~$\dUpsilon$, we conclude that
\begin{align}\label{eq:l:ReducedFormSL:2}
\reptwo u^n_{n,\eta}\equiv \rep u^{\eta,n} \quad \text{everywhere on } \Xi^n_{n,\eta} \fstop
\end{align}

Now, fix~$\eta\in\Omega$ and~$\gamma\in\dUpsilon$ with~$\mssd_\dUpsilon(\gamma,\eta)<\infty$.
By definition of~$\Omega$, there exists~$n_\eta$ so that~$\eta\in \Xi^{n_\eta}$.
Since~$\mssd_\dUpsilon(\gamma,\eta)<\infty$, we may choose $\seq{r_n}_n$ so that~$\gamma B_n=\eta B_n$ for every~$n\in\N$, and, for every~$\eps>0$, there exists~$n_\eps\in\N$ so that
\begin{align}\label{eq:l:ReducedFormSL:3}
\mssd_\dUpsilon(\gamma_{B_n^\complement},\eta_{B_n^\complement})<\eps\qquad \text{and} \qquad \abs{\mssd_\dUpsilon(\gamma_{B_n},\eta_{B_n})-\mssd_\dUpsilon(\gamma,\eta)}<\eps\comma \qquad n\geq n_\eps \fstop
\end{align}
For every~$\eps>0$ and~$n\in \N$ set~$\gamma_\eps\eqdef \gamma_{B_n}+\eta_{B_n^\complement}$ and note that
\begin{align*}
\mssd_\dUpsilon(\gamma,\gamma_\eps)\leq \mssd_\dUpsilon(\gamma_{B_n^\complement},\eta_{B_n^\complement})<\eps\comma \qquad n\geq n_\eps\fstop
\end{align*}

Fix~$m\eqdef n_\eps\vee n_\eta$.
Repsectively by~\eqref{eq:l:ReducedFormSL:2}, the $\mssd_\dUpsilon$-Lipschitz continuity of~$\rep u^{\eta,n}$, and~\eqref{eq:l:ReducedFormSL:3}, we have that
\begin{align*}
\abs{\reptwo u(\gamma_\eps)-\reptwo u(\eta)}=&\ \abs{\reptwo u(\gamma_{B_m}+\eta_{B_m^\complement})- \reptwo u(\eta_{B_m}+\eta_{B_m^\complement})}=\abs{\reptwo u^m_{m,\eta}(\gamma_{B_m})-\reptwo u^m_{m,\eta}(\eta_{B_m})}
\\
=&\ \abs{\rep u^{\eta,m}(\gamma_{B_m})- \rep u^{\eta,m}(\eta_{B_m})}
%\\
\leq%&\ 
\mssd_\dUpsilon(\gamma_{B_m},\eta_{B_m})
\\
\leq&\ \mssd_\dUpsilon(\gamma,\eta)+\eps\fstop
\end{align*}
Combining the above inequality with~\eqref{eq:l:ReducedFormSL:3} concludes the assertion.
\end{proof}
\end{prop}

\begin{thm}[Integral Varadhan-type short-time asymptotics]\label{t:VaradhanSecond}
Let~$(\mcX,\cdc,\mssd)$ be an \MLDS satisfying $(\Rad{\mssd}{\mssm})$ and \ref{ass:SLTensor}, and $\QP$ be a probability measure on~$\ttonde{\dUpsilon,\A_\mrmv(\Ed)}$ satisfying Assumptions~\ref{ass:CEr} and~\ref{ass:CCr}.
Further let~$\Xi\in\Bo{\T_\mrmv(\Ed)}$ be $\mssd_\dUpsilon$-open. Then,
\begin{align}\label{eq:t:VaradhanSecond:0}
\hr{\QP,\Xi}\equiv \mssd_\dUpsilon(\emparg,\Xi) \quad \as{\QP} \fstop
\end{align}
As a consequence,~$\TT{\dUpsilon}{\QP}_t$ satisfies to the Varadhan-type short-time asymptotics
\begin{align}\label{eq:Varadhan}
\lim_{t\downarrow 0} -2t \log \TT{\dUpsilon}{\QP}_t (\Xi,\Lambda)= \QP\text{-}\essinf_{\Lambda}\mssd_\dUpsilon(\emparg,\Xi)^2
\end{align}
for every~$\mssd_\dUpsilon$-open $\Xi\in\Bo{\T_\mrmv(\Ed)}$ and every~$\Lambda\in\A_\mrmv(\Ed)^\QP$.

\begin{proof}
By~\cite[Thm.~1.1]{HinRam03} we have that~$\lim_{t \downarrow 0} -2t\, \log \TT{\dUpsilon}{\QP}_t (\Xi, \Lambda)=\QP\text{-}\essinf_\Lambda \hr{\QP,\Xi}^2$.
Therefore, the second assertion is a consequence of the first one.

Since~\ref{ass:CEr} implies~\ref{ass:CAC}, the form~$\ttonde{\EE{\dUpsilon}{\QP},\dom{\EE{\dUpsilon}{\QP}}}$ is well-defined, densely defined and closable by Theorem~\ref{t:ClosabilitySecond}.
Furthermore, it satisfies~$(\Rad{\mssd_\dUpsilon}{\QP})$ by $(\Rad{\mssd}{\mssm})$ and Theorem~\ref{t:Rademacher}, and it is therefore quasi-regular by Proposition~\ref{p:QRegSLoc}.

Now, let~$\Xi$ be satisfying the assumption.
Since~$\mssd_\dUpsilon(\emparg,\Xi)$ is $\A_\mrmv(\Ed)^*$-measurable by Corollary~\ref{c:MeasurabilitydU}, again by Theorem~\ref{t:Rademacher} we have that $\mssd_\dUpsilon(\emparg,\Xi) \wedge r \in \dom{\EE{\dUpsilon}{\QP}}$ and~$\SF{\dUpsilon}{\QP}\ttonde{\mssd_\dUpsilon(\emparg,\Xi) \wedge r}\leq 1$ for every~$r>0$.
Since~$\mssd_\dUpsilon(\emparg,\Xi) \wedge r$ vanishes everywhere on~$\Xi$, we conclude that
\begin{equation}\label{eq:t:VaradhanSecond:1}
\hr{\QP,\Xi}\geq \mssd_\dUpsilon(\emparg, \Xi) \quad \as{\QP}
\end{equation}
by Definition~\ref{d:MaximalFunction} of maximal function.

In order to show the reverse inequality, fix~$r>0$ and set~$u_r\eqdef \hr{\QP,\Xi}\wedge r\in\dom{\EE{\dUpsilon}{\QP}}$.
Since $\ttonde{\EE{\dUpsilon}{\QP},\dom{\EE{\dUpsilon}{\QP}}}$ is quasi-regular,~$u_r$ admits a ($\A_\mrmv(\Ed)^\QP$-measurable) quasi-continuous $\QP$-representa\-tive~$\reptwo u_r$, additionally satisfying~$\reptwo u_r\geq 0$ everywhere on~$\dUpsilon$ and~$\reptwo u_r\equiv 0$ everywhere on~$\Xi$.
By Proposition~\ref{p:ReducedSL}, for every~$\eps>0$, for every~$\eta$ in a $\QP$-conegligible set~$\Omega_r$ and every~$\gamma\in \Xi$ with~$\mssd_\dUpsilon(\gamma,\eta)<\infty$ there exists~$\gamma_\eps\in\dUpsilon$ so that~\eqref{eq:p:ReducedSL:0} holds with~$\reptwo u_r$ in place of~$\reptwo u$.
Since~$\Xi$ is $\mssd_\dUpsilon$-open, choosing $\eps>0$ sufficiently small we have that~$\gamma_\eps\in \Xi$ as well.
Since~$\reptwo u_r\equiv 0$ everywhere on~$\Xi$, we have that~$\reptwo u_r(\gamma_\eps)=0$, and~\eqref{eq:p:ReducedSL:0} reads
\begin{equation*}
\reptwo u_r(\eta) \leq \mssd_\dUpsilon(\gamma,\eta) +\eps \comma \qquad \gamma\in\Xi\comma \eta\in \Omega \fstop
\end{equation*}
Since~$\eps>0$ is arbitrary, extremizing over~$\gamma\in\Xi$ the above inequality we have that
\begin{equation*}
\reptwo u_r(\eta) \leq \mssd_\dUpsilon(\eta,\Xi) \qquad \eta\in\Omega_r \fstop
\end{equation*}
Since~$r>0$ is arbitrary, and since~$\QP\Omega_r=1$ for every~$r$,
\begin{equation}\label{eq:t:VaradhanSecond:2}
\hr{\QP,\Xi}\leq \mssd_\dUpsilon(\emparg, \Xi) \quad \as{\QP} \fstop
\end{equation}
Combining~\eqref{eq:t:VaradhanSecond:1} with~\eqref{eq:t:VaradhanSecond:2} concludes~\eqref{eq:t:VaradhanSecond:0} the assertion.
\end{proof}
\end{thm}

%\purple{
\begin{rem}[Wasserstein Universality] \normalfont
The assumptions in Theorem \ref{t:VaradhanSecond} can be verified for a wide class of quasi-Gibbs measures including ---~in the case of the Euclidean space $X=\R^n$: $\mathrm{sine}_\beta$, $\mathrm{Airy}_\beta$, $\mathrm{Bessel}_{\alpha, \beta}$, Ginibre. It is remarkable that the short-time asymptotics of all the aforementioned classes of invariant measures --- in general singular with respect to each other~---  are governed by the same distance function $\mssd_\dUpsilon$. In other words, all these measures belong to the same  \emph{$\mssd_\dUpsilon$-universality class} in the sense that the short-time asymptotic of the corresponding semigroups is governed universally by $\mssd_\dUpsilon$. This universality phenomenon for singular invariant measures was not known, even in the case when~$X=\R^n$. 
%This universality phenomena over singular measures 
This is a characteristic feature of infinite-dimensional spaces.
\end{rem}
%}
\section{Examples}\label{s:Examples}
In this section, we collect examples of base spaces and of probability measures on their configuration spaces satisfying, fully or in part, our assumptions.
As usual, we start from the base spaces, and subsequently move to the configuration spaces.

\subsection{Base spaces}\label{ss:ExamplesBase}
Our results apply to various classes of spaces satisfying  ---~fully or in part~--- our assumptions on~\MLDS's.
%We discuss here several classes of spaces satisfying ---~fully or in part~--- our assumptions on~\MLDS's.
Among them are:
\begin{itemize}
\item complete Riemannian manifolds, for which every property of our interest is virtually well-known;
\item ideal sub-Riemannian manifolds, see e.g.~\cite{HajKos00, BarRiz19, LzDSSuz21a};
%\item ideal sub-Riemannian manifold, see~\S\ref{sss:SubRiemannian};
\item spaces satisfying the Riemannian Curvature-Dimension condition, see~\S\ref{sss:RCD}, including e.g.\ Ricci-limit spaces, finite-dimensional Alexandrov spaces endowed with their Hausdorff measure, and Hilbert spaces endowed with log-concave measures;
\item spaces satisfying Measure Contraction Property, see~\S\ref{sss:RCD};
\item path/loop spaces over Riemannian manifolds, see~\cite{MaRoe00}.
%which will be discussed in a forthcoming paper \cite{LzDSSuz22a}.
\end{itemize}

In this section, we mainly discuss spaces satisfying synthetic Ricci-curvature lower bounds or the Measure Contraction Property.
However, let us firstly describe more precisely the correspondence betwenn the \MLDS structure and classical objects, at least in the case of Riemannian manifolds.
\begin{ese}[Riemannian manifolds]\label{e:Riemannian}
We say that~$(\mcX,\cdc,\mssd)$ is the \MLDS arising from a Riemannian manifold~$(X,g)$ if~$X$ is a metrizable connected $\mcC^2$-manifold, endowed with a complete Riemannian metric~$g$, the intrinsic distance~$\mssd$ induced by~$g$, the volume measure~$\mssm\eqdef \vol_g$, the localizing ring generated by relatively compact sets, and the square field~$(\cdc,\Dz)$ defined by~$\Dz\eqdef \mcC^\infty_c(X)$ and $\cdc(f)\eqdef \abs{\diff f}_g^2$ for~$f\in\Dz$.
\end{ese}

\subsubsection{Infinitesimally Hilbertian metric measure spaces}
Let~$(\mcX,\mssd)$ be a metric local structure, and denote by~$\bsLip(\mssd)$ the family of all $\mssd$-Lipschitz functions with bounded support.
Further let~$\slo[*]{\emparg}$ be the minimal relaxed slope and~$\Ch[\mssd,\mssm]$ be the Cheeger energy defined in~\S\ref{sss:CheegerE}.

\begin{defs}[Infinitesimal Hilbertianity]\label{d:IH}
A metric local structure~$(\mcX,\mssd)$ is \emph{infinitesimally Hil\-bertian} if $\Ch[\mssd,\mssm]$ is a quadratic functional, viz.
\begin{align*}%\tag{$\mathsf{IH}$}
2\, \Ch[\mssd,\mssm](f)+2\, \Ch[\mssd,\mssm](g) = \Ch[\mssd,\mssm](f+g)+ \Ch[\mssd,\mssm](f-g) \fstop
\end{align*}
\end{defs}

In this case, a square field operator on~$\mcX$ is defined by polarization via
\begin{align}\label{eq:SFIH}
\cdc(f)\eqdef \slo[*]{f}^2\comma\qquad f\in\Dz\eqdef\bsLip(\mssd)\fstop
\end{align}

\begin{prop}\label{p:IHtoMLDS}
Let~$(\mcX,\mssd)$ be a metric local structure, and~$\cdc$ be defined as~\eqref{eq:SFIH}. If~$(\mcX,\mssd)$ is infinitesimally Hilbertian, then~$(\mcX,\cdc)$ is an \LDS.
If additionally the associated Dirichlet form $\ttonde{\EE{X}{\mssm},\dom{\EE{X}{\mssm}}}$ is quasi-regular, then~$(\mcX,\cdc,\mssd)$ is an \MLDS.
\begin{proof}
Since~$\mssm$ is finite on~$\mssd$-bounded sets (by Definition~\ref{d:EMLS}\iref{i:d:EMLS:2}), we have that~$\Dz\subset L^2(\mssm)$.
In addition, the family of functions~$\Dz\eqdef\bsLip(\mssd)$ is an algebra, and thus straightforwardly satisfies Definition~\ref{d:SF}\iref{i:d:SF:1}.

Since~$(\mcX,\mssd)$ is infinitesimally Hilbertian, then~$\cdc$ is (symmetric non-negative definite) bi\emph{linear}, i.e.\ satisfying Definition~\ref{d:SF}\iref{i:d:SF:2}.
Furthermore, since~$\slo[*]{\emparg}$ satisfies Definition~\ref{d:SF}\iref{i:d:SF:3} by e.g.~\cite[Lem.~3.1.4]{Gig18}.
By e.g.~\cite[Cor.~I.6.1.3]{BouHir91}, we have that~$\cdc$ is a square field operator in the sense of Definition~\ref{d:SF}.

Finally, since~$\EE{X}{\mssm}\eqdef\Ch[\mssd,\mssm]$ is closed by construction, and since all other conditions in Definition~\ref{d:EMLS} are satisfied by definition, it suffices to verify that~$\Dz$ is dense in~$L^2(\mssm)$, which is e.g.~\cite[Prop.~4.1]{AmbGigSav14}, the condition~(4.5) there being readily verified on any~\MLDS.
\end{proof}
\end{prop}

\subsubsection{Synthetic lower Ricci curvature bounds}\label{sss:RCD}
We briefly recall some main facts about synthetic Ricci-curvature lower bounds for metric measure spaces after Lott--Villani~\cite{LotVil09} and Sturm~\cite{Stu06a,Stu06b}.
Since we are only interested in the case of \emph{Riemannian} curvature bounds, we refer to definitions and results in~\cite{AmbGigSav14b,ErbKuwStu15}.

\begin{prop}[Properties of $\RCD(K,\infty)$ spaces]\label{p:PropertiesRCD(Kinfty)}
Let~$(X,\mssd,\mssm)$ be an~$\RCD(K,\infty)$ space, $K\in\R$ (see e.g.~\cite[Thm.~5.1]{AmbGigSav14b}), and denote by~$\ttonde{\Ch[\mssd,\mssm],\dom{\Ch[\mssd,\mssm]}}$ its Cheeger energy.
Then,
\begin{enumerate}[$(i)$]
\item $\ttonde{\Ch[\mssd,\mssm],\dom{\Ch[\mssd,\mssm]}}$ is quadratic, and thus a Dirichlet form, additionally admitting square field operator~$\SF{X}{\mssm}=\slo[w,\mssd]{\emparg}$, see \cite[Thm.~4.18(iv)]{AmbGigSav14b}, and satisfying~$(\Rad{\mssd}{\mssm})$ by definition;

\item\label{i:p:PropertiesRCD(Kinfty):2} $\ttonde{\Ch[\mssd,\mssm],\dom{\Ch[\mssd,\mssm]}}$ is quasi-regular, see~\cite[Lem.~6.7]{AmbGigSav14b} or~\cite[Thm.~4.1]{Sav14};

\item\label{i:p:PropertiesRCD(Kinfty):3} $\ttonde{\Ch[\mssd,\mssm],\dom{\Ch[\mssd,\mssm]}}$ is irreducible (consequence of~\iref{i:p:PropertiesRCD(Kinfty):7} below);

\item\label{i:p:PropertiesRCD(Kinfty):7} $(X,\SF{X}{\mssm},\mssd)$ satisfies~$(\SL{\mssm}{\mssd})$, see~\cite[Thm.~7.2]{AmbGigMonRaj12} after~\cite[Thm.~6.2]{AmbGigSav14b};

\item\label{i:p:PropertiesRCD(Kinfty):8} the intrinsic distance~$\mssd_\mssm$ of~$\ttonde{\Ch[\mssd,\mssm],\dom{\Ch[\mssd,\mssm]}}$ coincides with~$\mssd$, see~\cite[Thm.~7.4]{AmbGigMonRaj12} after~\cite[Thm.~6.10]{AmbGigSav14b};

\item\label{i:p:PropertiesRCD(Kinfty):4} $\ttonde{\Ch[\mssd,\mssm],\dom{\Ch[\mssd,\mssm]}}$ is conservative, see~\cite[Thm.~4]{Stu94}, applicable by~\ref{i:p:PropertiesRCD(Kinfty):8}.
\end{enumerate}
\end{prop}

\begin{prop}[Properties of~$\RCD^*(K,N)$ spaces]\label{p:PropertiesRCD(K,N)}
Let~$(X,\mssd,\mssm)$ be an~$\RCD^*(K,N)$ spaces with $K\in\R$, $N\in [1,\infty)$ (see e.g.~\cite[Thm.~2]{ErbKuwStu15}), and denote by~$\ttonde{\Ch[\mssd,\mssm],\dom{\Ch[\mssd,\mssm]}}$ its Cheeger energy.
Then,
\begin{enumerate}[$(i)$]
\item $(X,\mssd,\mssm)$ is an $\RCD(K,\infty)$ space, hence all properties in Proposition~\ref{p:PropertiesRCD(Kinfty)} hold;

\item\label{i:p:PropertiesRCD(K,N):1} $(X,\mssd)$ is locally compact and all closed balls are compact (see~\cite[Cor.~2.4]{Stu06b});

%\item\label{i:p:PropertiesRCD(K,N):1.5} $(X,\mssd,\mssm)$ satisfies the volume growth condition~\eqref{eq:VolGrowth} (see~\cite[Thm.~2.3, Eqn.~(2.5)]{Stu06b})

\item\label{i:p:PropertiesRCD(K,N):2} $\ttonde{\Ch[\mssd,\mssm],\dom{\Ch[\mssd,\mssm]}}$ is a regular Dirichlet form;

\item\label{i:p:PropertiesRCD(K,N):3} $(X,\mssd,\mssm)$ is locally doubling and satisfies a local weak $2$-Poincar\'e inequality, see~\cite[Cor.~2.4, Cor.~6.6(i)]{Stu06b};

%\item\label{i:p:PropertiesRCD(K,N):4} $(\mcX,\SF{X}{\mssm},\mssd)$ satisfies the heat kernel estimate~\ref{ass:GE} (see~\cite{JiaLiZha16});

\item\label{i:p:PropertiesRCD(K,N):5} $(\mcX,\SF{X}{\mssm},\mssd)$ satisfies Assumption~\ref{ass:T};

\item\label{i:p:PropertiesRCD(K,N):6} if every $\mssd$-ball in $(X,\mssd)$ is geodesically convex (e.g.\ if~$(X,\mssd)$ is additionally a $\CAT(0)$ space), then $(\mcX,\SF{X}{\mssm},\mssd)$ satisfies Assumption~\ref{ass:SLTensor}.
\end{enumerate}

\begin{proof}
\iref{i:p:PropertiesRCD(K,N):2} follows from the density of~$\Lip(\mssd)\cap L^2(\mssm)$ in~$\dom{\Ch[\mssd,\mssm]}$ and in~$L^2(\mssm)$, e.g.~\cite[Prop.~4.1]{AmbGigSav14}.

\iref{i:p:PropertiesRCD(K,N):5}
It is noted in~\cite[Rmk.~4.20]{AmbGigSav14b} that whenever~$(X,\mssd,\mssm)$ is an $\RCD(K,\infty)$ space additionally satisfying local doubling and a weak Poincar\'e inequality, then~$\SF{X}{\mssm}(f)\eqdef \slo[\mssd,\mssm,*]{f}=\slo{f}$ for every~$f\in\Lip(\mssd)$, which shows~\ref{ass:T} for~$n=1$.
$\RCD^*(K,N)$ spaces satisfy these assumptions by~\iref{i:p:PropertiesRCD(K,N):3}.
In order to verify~\ref{ass:T} for~$n>1$, it suffices to recall the tensorization of the $\RCD^*(K,N)$ condition, proved in~\cite[Thm.~3.23]{ErbKuwStu15}.

In order to show~\ref{i:p:PropertiesRCD(K,N):6} it suffices to verify that the $n$-fold symmetric product~$B^\sym{n}$ of any $\mssd$-ball~$B$ is as well an~$\RCD^*(K',N')$ space for some appropriate~$N'$,~$K'$, and apply Proposition~\ref{p:PropertiesRCD(Kinfty)}\iref{i:p:PropertiesRCD(Kinfty):7}.
Since geodesically convex subsets of~$\RCD^*(K,N)$ spaces (endowed with the restriction of the reference measure) are again~$\RCD^*(K,N)$ (combine e.g.~\cite[Prop.~4.15]{Stu06a} with~\cite[Thm.~4.19]{AmbGigSav14} for the restriction of infinitesimal Hilbertianity), under the assumption in~\ref{i:p:PropertiesRCD(K,N):6} every ball in~$(X,\mssd)$ is an~$\RCD^*(K,N)$ space.
By tensorization of the $\RCD^*(K,N)$ condition (recalled above), $B_r^\tym{n}$ in an~$\RCD^*(K,nN)$ space for every~$n\in \N$.
Finally, since the action of the symmetric group~$\mfS_n$ on~$B^\tym{n}$ by permutation of coordinates is compact, isometric, and measure-preserving, the quotient space~$B^\sym{n}$ is again an~$\RCD^*(K,nN)$ space by~\cite[Thm.~6.2]{GalKelMonSos18}.
\end{proof}
\end{prop}

The authors are grateful to Professor Shin-ichi Ohta for having pointed to their attention that~$\MCP(K,N)$ spaces as well satisfy some of the present assumptions.

\begin{prop}[Properties of~$\MCP(K,N)$ spaces]
Let~$(X,\mssd,\mssm)$ be an~$\MCP(K,N)$ space, $K\in\R$, $N\in [1,\infty)$ with~$\supp\mssm = X$, see e.g.~\cite[Dfn.~5.1]{Stu06b} or~\cite[Dfn.~2.1]{Oht07b}.
Then,
\begin{enumerate}[$(i)$]
\item\label{i:p:PropertiesMCP:1} $(X,\mssd,\mssm)$ is locally compact and all closed balls are compact, see~\cite[Cor.~2.4]{Stu06b};
\item\label{i:p:PropertiesMCP:2} $(X,\mssd,\mssm,\Ed)$ is a metric local structure (consequence of~\iref{i:p:PropertiesMCP:1} and standard facts).
\end{enumerate}

Additionally assume that
\begin{equation*}
\omega_N(x)\eqdef \lim_{r\to 0} r^{-N} \mssm B^\mssd_r(x) \quad \text{is locally bounded on~$X$} \fstop
\end{equation*}
and denote by~$\ttonde{\EE{X}{\mssm}_N,\dom{\EE{X}{\mssm}_N}}$ the Dirichlet form on~$L^2(\mssm)$ defined in~\cite[p.~173]{Stu06b}.
Further set~$\ttonde{\EE{X}{\mssm},\dom{\EE{X}{\mssm}}}\eqdef \ttonde{\tfrac{1}{N}\EE{X}{\mssm}_N,\dom{\EE{X}{\mssm}_N}}$
Then,
\begin{enumerate}[$(i)$]\setcounter{enumi}{2}
\item\label{i:p:PropertiesMCP:3} $\ttonde{\EE{X}{\mssm},\dom{\EE{X}{\mssm}}}$ is a regular, strongly local Dirichlet form on~$L^2(\mssm)$ with core~$\bsLip(\mssd)$ (see~\cite[Cor.~6.6(i)]{Stu06b});

\end{enumerate}
Thus, if additionally~$\ttonde{\EE{X}{\mssm},\dom{\EE{X}{\mssm}}}$ admits square field operator~$\SF{X}{\mssm}$, then
\begin{enumerate}[$(i)$]\setcounter{enumi}{5}
\item\label{i:p:PropertiesMCP:6} $(X,\SF{X}{\mssm},\mssd)$ is an \MLDS (consequence of~\iref{i:p:PropertiesMCP:1} and~\iref{i:p:PropertiesMCP:2});

\item\label{i:p:PropertiesMCP:7} $(X,\SF{X}{\mssm},\mssd)$ satisfies~$(\Rad{\mssd}{\mssm})$ (consequence of~\iref{i:p:PropertiesMCP:3} and~\cite[Eqn.~6.2]{Stu06b}).
\end{enumerate}
\end{prop}

\subsection{Configuration Spaces}\label{ss:ExamplesConfig}
We collect here a variety of examples of probability measures on~$\dUpsilon$ satisfying, fully or in part, our previous assumptions.

\subsubsection{Quasi-Gibbs measures}\label{sss:ExamplesAC}
As a further very broad class of examples, we recall the definition of \emph{quasi-Gibbs measures}.
Several slightly different (possibly \emph{non}-equivalent) definitions for this concept were introduced by H.~Osada in the case $X=\R^n$, see e.g.~\cite[Dfn.~2.1]{Osa13},~\cite[Dfn.~5.1]{Osa19}, or~\cite[Dfn.~2.2]{OsaTan20}.
The definition we give here is closest to the one in~\cite{OsaTan20}.

Let~$(\mcX,\mssd)$ be a metric local structure.
Further let~$\Phi\colon X\rar\R$ be $\A$-measurable, and by~$\Psi\colon X^\tym{2}\rar\R$ be $\A^\otym{2}$-measurable and symmetric.
The function~$\Phi$ will be called the \emph{free potential}, and~$\Psi$ the \emph{interaction potential}.
For fixed~$E\in \Ed$, these potentials define a \emph{Hamiltonian}~$\msH_E\colon \dUpsilon\rar \R$ as
\begin{align*}
\msH_E\colon \gamma\longmapsto \Phi^\trid \gamma_{E}+\tfrac{1}{2}\Psi^\trid \ttonde{\gamma_{E}^\otym{2}}\comma\qquad \gamma\in \dUpsilon\fstop
\end{align*}

\begin{defs}[Quasi-Gibbs measures]\label{d:QuasiGibbs}
We say that a probability measure~$\QP$ on~$\ttonde{\dUpsilon,\A_{\mrmv}(\Ed)}$ is a \emph{$(\Phi,\Psi)$-quasi-Gibbs} measure if for every~$E\in\msE$, every~$k\in\N_0$, and $\QP$-a.e.~$\eta\in\dUpsilon$ there exists a constant~$c_{E,\eta,k}>0$ so that, for $\QP$-a.e.~$\eta\in\dUpsilon$, letting~$\QP^\eta_E$ denote the projected conditional probabilities~\eqref{eq:ProjectedConditionalQP}
\begin{align}\label{eq:localACquasiGibbs}
c_{E,\eta,k}^{-1}\, e^{-\msH_E} \cdot \PP_{\mssm_E}\mrestr{\dUpsilon^\sym{k}(E)} \ \leq \ \QP^\eta_E\mrestr{\dUpsilon^\sym{k}(E)} \ \leq \ c_{E,\eta,k}\, e^{-\msH_E} \cdot \PP_{\mssm_E}\mrestr{\dUpsilon^\sym{k}(E)} \fstop
\end{align}
\end{defs}

\begin{rem}\label{r:QuasiGibbs}
\begin{enumerate*}[$(a)$]
\item\label{i:r:QuasiGibbs:1} As usual, it suffices to assume~\eqref{eq:localACquasiGibbs} for each~$E_n$ in a localizing sequence; in particular, it is sufficient to choose~$E=B^\mssd_r(x_0)$ for every~$r>0$ and some fixed~$x_0\in X$.

\item\label{i:r:QuasiGibbs:2} Since the range of~$\ev_E$ is at most countable, our definition is equivalent to~\cite[Dfn.~2.2]{OsaTan20}, which prefers conditioning to~$\ev_E=k$ over restriction to~$\dUpsilon^\sym{k}(E)=\ev_E^{-1}(k)$.
\end{enumerate*}
\end{rem}

\begin{rem}\label{r:QuasiGibbsEx}
The class of Quasi-Gibbs measures includes all canonical Gibbs measures, and the laws of some determinantal/permnental point processes, as for instance:
\begin{enumerate}
\item\label{i:r:QuasiGibbsEx:1} mixed Poisson measures (Dfn.~\ref{d:MixedPoisson} below);
\item Ruelle type grand-canonical Gibbs measures~\cite{Rue70,AlbKonRoe98b} (cf.~\cite{Osa98}), by definition;
\item\label{i:r:QuasiGibbsEx:3} the laws of some determinantal/permanental point processes, as e.g.:
\begin{itemize}
\item the Dyson interacting Brownian motion, see~\cite[Thm.~2.2]{Osa13};
\item the Ginibre interacting Brownian motion, see~\cite[Thm.~2.3]{Osa13};
\item the $\mathrm{sine}_\beta$,~$\mathrm{Bessel}_{\alpha,\beta}$ and $\mathrm{Airy}_\beta$ processes~\cite{KatTan11}.
\end{itemize}
\end{enumerate}
For further examples of quasi-Gibbs measures when~$X=\R^n$ we refer to~\cite{Osa13, Osa19, OsaTan20}.
\end{rem}

We proceed now to briefly check our assumptions ---~\ref{ass:AC},~\ref{ass:CCr},~\ref{ass:CEr}~--- for the classes of measures in~\eqref{i:r:QuasiGibbsEx:1}--\eqref{i:r:QuasiGibbsEx:3} above.

Let~$\QP$ be a quasi-Gibbs measure.
By Proposition~\ref{p:CACtoAC}, Assumption~\ref{ass:CAC} implies Assumption~\ref{ass:AC}.
We will therefore concentrate on examples of measures satisfying~\ref{ass:CAC}.
It is clear that~\eqref{eq:localACquasiGibbs} implies~\ref{ass:CAC}, since it is in fact a constraint on the Radon--Nikod\'ym density.
Under some additional assumption on the pair of potentials~$(\Phi, \Psi)$ (e.g., in the case~$X=\R^n$, \emph{super-stability} and \emph{lower regularity} in the sense of Ruelle~\cite{Rue70, Osa98}, or the existence of upper semi-continuous bounds~$(\Phi_0,\Psi_0)$ such that~$c\,(\Phi_0,\Psi_0)\leq (\Phi,\Psi)\leq c^{-1}(\Phi_0,\Psi_0)$ for some constant~$c>0$, see~\cite[Eqn.~(A.3), p.~8]{Osa13}), one can show that~\eqref{eq:localACquasiGibbs} implies as well Assumption~\ref{ass:ConditionalClos}, so that the resulting form~$\ttonde{\EE{\dUpsilon}{\QP},\CylQP{\QP}{\Dz}}$ is closable by Theorem~\ref{t:ClosabilitySecond}.

Let~$(\mcX,\cdc,\mssd)$ be an~\TLDS.
For simplicity of notation, let~$\PP_{E,k}\eqdef \pr^E_\pfwd\PP_\mssm \mrestr{\dUpsilon^\sym{k}(E)}$ on~$\dUpsilon^\sym{k}(E)$.
We consider the pre-Dirichlet form~$\ttonde{\EE{\dUpsilon^\sym{k}(E)}{\PP_{E,k}},\CylQP{\PP_{E,k}}{\Dz}}$ on~$L^2(\PP_{E,k})$, defined analogously to~\eqref{eq:VariousFormsB} with~$\PP_{E,k} = (\PP_\mssm)^\eta_E\mrestr{\dUpsilon^\sym{k}(E)}$ (independently of~$\eta\in\dUpsilon$) in place of~$\QP^\eta_E$.
By making use of the techniques in~\S\ref{ss:FurtherClosability}, one can show that this form is well-defined, densely defined, and closable.
For simplicty of notation, we denote its closure by~$\ttonde{\EE{\sym{k}}{\pi},\dom{\EE{\sym{k}}{\pi}}}$.
We give a sufficient condition for~\ref{ass:ConditionalClos} to hold.

\begin{prop}[Sufficent condition for~\ref{ass:ConditionalClos}]
Let~$(\mcX,\cdc,\mssd)$ be an \MLDS, and~$\QP$ be a quasi-Gibbs measure on~$\ttonde{\dUpsilon,\A_\mrmv(\msE)}$ satisfying Assumption~\ref{ass:Mmu}.
Assume further that the density~$\varphi_{E,k}\eqdef e^{-\msH_E}\restr_{\dUpsilon^\sym{k}(E)}$ satisfies~$\varphi_{E,k}\in \dom{\EE{\sym{k}}{\PP}}$ for each~$k\in\N$ and~$E\in\msE$.
Then, the form~$\ttonde{\EE{\dUpsilon}{\QP},\dom{\EE{\dUpsilon}{\QP}}}$ is well-defined and a Dirichlet form on~$L^2(\QP)$.

\begin{proof}
For any constant~$c>0$, denote by~$\ttonde{\EE{\sym{k}}{\PP}_{c\varphi_{E,k}},\dom{\EE{\sym{k}}{\PP}_{c\varphi_{E,k}}}_*}$ the Girsanov-type transform of $\ttonde{\EE{\sym{k}}{\PP},\dom{\EE{\sym{k}}{\PP}}}$ by the potential~$c\varphi_{E,k}$.
Since~$c\varphi_{E,k}\in \dom{\EE{\sym{k}}{\PP}}$, by e.g.~\cite[Thm.~2.2]{CheSun06} this form is a well-defined Dirichlet form on~$L^2(c\varphi_{E,k} \cdot \PP_{E,k})$ with square field operator identical to the square field operator of~$\ttonde{\EE{\sym{k}}{\PP},\dom{\EE{\sym{k}}{\PP}}}$, and additionally satisfying~$\CylQP{\PP_{E,k}}{\Dz}\subset \dom{\EE{\sym{k}}{\PP}_{c\varphi_{E,k}}}_*$, since the former is a core of continuous functions for~$\ttonde{\EE{\sym{k}}{\PP},\dom{\EE{\sym{k}}{\PP}}}$.
It follows from the closability of~$\ttonde{\EE{\sym{k}}{\PP}_{c\varphi_{E,k}},\dom{\EE{\sym{k}}{\PP}_{c\varphi_{E,k}}}_*}$ that the pre-Dirichlet form~$\ttonde{\EE{\sym{k}}{\PP}_{c\varphi_{E,k}},\CylQP{\varphi_{E,k}\cdot \PP_{E,k}}{\Dz}}$ is finite (by Lemma~\ref{l:MmuL1}) and a closable form, with closure~$\ttonde{\EE{\sym{k}}{\PP}_{c\varphi_{E,k}},\dom{\EE{\sym{k}}{\PP}_{c\varphi_{E,k}}}}$ where, possibly,~$\dom{\EE{\sym{k}}{\PP}_{c\varphi_{E,k}}}\subsetneq \dom{\EE{\sym{k}}{\PP}_{c\varphi_{E,k}}}_*$.

By the two-sided bound on densities~\eqref{eq:localACquasiGibbs} and the standard fact~\cite[Prop.~I.3.5]{MaRoe92}, we may conclude that the pre-Dirichlet form~$\ttonde{\EE{\dUpsilon^\sym{k}}{\QP^\eta_{E,k}},\CylQP{\QP^\eta_{E,k}}{\Dz}}$, defined analogously to~\eqref{eq:VariousFormsB} with the measure~$\QP^\eta_{E,k} \eqdef \QP^\eta_E\mrestr{\ev^{-1}(k)}$ in place of~$\QP^\eta_E$, is densely defined and closable on~$L^2(\QP^\eta_{E,k})\cong L^2(c\varphi_{E,k}\PP_{E,k})$ for $c=c_{E,\eta,k}^{\pm 1}>0$ as in~\eqref{eq:localACquasiGibbs} and every~$\eta\in\dUpsilon$.

We may conclude that the form
\begin{equation}\label{eq:p:ClosabilityQGibbs:1}
\EE{\dUpsilon}{\QP^\eta_E}(u)\eqdef \sum_{k=0}^\infty \EE{\dUpsilon^\sym{k}}{\QP^\eta_{E,k}}(u)\comma \qquad u\in\Cyl{\Dz}\comma
\end{equation}
is well-defined, since so is~$\ttonde{\EE{\dUpsilon^\sym{k}}{\QP^\eta_{E,k}},\Cyl{\Dz}}$.
Its closability follows from the closability of the forms~$\ttonde{\EE{\dUpsilon^\sym{k}}{\QP^\eta_{E,k}},\CylQP{\QP^\eta_{E,k}}{\Dz}}$, $k\in\N$, and standard facts, e.g.~\cite[Prop.~I.3.7(i)]{MaRoe92}.
This proves that Assumption~\ref{ass:ConditionalClos} holds for~$\QP$, and concludes the proof of the closability of~$\ttonde{\EE{\dUpsilon}{\QP},\dom{\EE{\dUpsilon}{\QP}}}$ by Theorem~\ref{t:ClosabilitySecond}.
\end{proof}
\end{prop}

\begin{rem}
When~$X$ is a Riemannian manifold, the condition~$\varphi_{E,k}\in \dom{\EE{\sym{k}}{\PP}}$ may be further relaxed.
Indeed, it is instrumental only to the closedness of the forms~$\ttonde{\EE{\sym{k}}{\PP}_{c\varphi_{E,k}},\dom{\EE{\sym{k}}{\PP}_{c\varphi_{E,k}}}_*}$, $c>0$, which may alternatively be established on the orbifold~$X^\sym{k}\cong \dUpsilon^\asym{k}(E)$, e.g.\ by the Hamza condition, cf.~\cite{AlbBraRoe89}.
\end{rem}

In order to apply the theory developed in~\S\ref{s:Interplay}, we further need to verify either Assumption~\ref{ass:CEr} or directly the Sobolev-to-Lipschitz property~$(\SL{\QP^\eta_r}{\mssd_\dUpsilon})$.
Since all the quasi-Gibbs measures of our interest are studied in the case of Riemannian manifolds, we restrict ourselves to this generality.
Recall the setting of Example~\ref{e:Riemannian}.

Concerning~\ref{ass:CEr}, the mutual equivalence~$\QP^\eta_r\sim \PP_{\mssm_r}$ is already implicit in the definition of quasi-Gibbs measures.
The bounds on the density is not satisfied by arbitrary quasi-Gibbs measures.
However, in the following Proposition we provide sufficient conditions for Assumption~\ref{ass:CEr} in the case of manifolds of dimension~$d\geq 2$, which includes invariant measures of point processes describing particle systems with hard-core interactions.
In particular, this applies to all classes of measures listed in Remark~\ref{r:QuasiGibbsEx}, except for $\mathrm{sine}_\beta$, $\mathrm{Bessel}_{\alpha,\beta}$, and $\mathrm{Airy}_\beta$, which are measures on~$\dUpsilon(\R)$.

\begin{prop}[Conditional equivalence for quasi-Gibbs measures]\label{p:CondEquivQGibbs}
Let~$(\mcX,\cdc,\mssd)$ be the \MLDS arising from a Riemannian manifold~$(X,g)$ of dimension~$d\geq 2$.
Assume that~$\QP$ is a quasi-Gibbs measure satisfying Assumption~\ref{ass:Mmu} and that
\begin{enumerate}[$(a)$]
\item the free potential~$\Phi$ of~$\QP$ satisfies~$\Phi\in L^\infty_\loc(\mssm)$;
\item the interaction potential~$\Psi$ of~$\QP$ if of the form~$\Psi(x_1,x_2)=\psi\circ \mssd(x_1,x_2)$ for some measurable function~$\psi\colon [0,\infty) \to \R\cup\set{+\infty}$ locally bounded away from~$0$ and~$\infty$ on~$(0,\infty)$.
\end{enumerate}
Then Assumption~\ref{ass:CEr} holds for~$\QP$.

\begin{proof}
By the representation of~$\dUpsilon(B_r)$ as a disjoint union, it suffices to verify the statement for the restricted measure~$\QP^\eta_{r,n}\eqdef \QP^\eta_r\mrestr{B_r^\sym{n}}$.
Set
\begin{align*}
(B_r^\sym{n})_\circ\eqdef \set{(x_1,\dotsc, x_n)\in B_r^\sym{n}: x_i\neq x_j \text{ for } i\neq j} \comma \qquad \Delta_r^\sym{n}\eqdef B_r^\sym{n}\setminus (B_r^\sym{n})_\circ \fstop
\end{align*}
It is readily verified that~$\frac{\diff \QP^\eta_r}{\diff\PP_{\mssm_r}}$ is bounded away from~$0$ and~$\infty$ on
\begin{align*}
\ttonde{ \Delta_r^\sym{n}}_\eps\eqdef \set{\mbfx^\sym{n}\in B_r^\sym{n} : \mssd_\sym{n}(\mbfx,  \Delta_r^\sym{n})>\eps }
\end{align*}
for every~$\eps>0$.
Thus, it suffices to show that~$\Delta_r^\sym{n}$ is $\EE{B_r^\sym{n}}{\PP_{\mssm_r,n}}$-polar.
By the representation~\eqref{eq:PoissonLebesgue}, this is equivalent to say that~$\Delta_r^\sym{n}$ is $\EE{B_r^\sym{n}}{\mssm_r^\sym{n}}$-polar.
This is a standard consequence of the fact that each one-point set in~$X$ is $\EE{X}{\mssm}$-polar, which holds since~$\dim X\geq 2$, cf.\ e.g.~\cite[Lem.~7.17]{LzDS17+}.
\end{proof}
\end{prop}

Since singletons are not polar for the standard Brownian motion on~$\R$, the proof of Proposition~\ref{p:CondEquivQGibbs} fails on~$\dUpsilon(\R)$.
In fact, the validity of Assumption~\ref{ass:CEr} for a quasi-Gibbs measure~$\QP$ on~$\dUpsilon(\R)$ would imply that the corresponding Hamiltonian~$\msH_r$ is uniformly bounded for every~$r>0$.
In spite of this fact, it is still possible to establish the Sobolev-to-Lipschitz property by other means, as we now show.

\begin{prop}[Sobolev-to-Lipschitz property for quasi-Gibbs measures on manifolds]
Let $(\mcX,\cdc,\mssd)$ be the \MLDS arising from a Riemannian manifold~$(X,g)$ of dimension~$d\geq 1$.
Assume that~$\QP$ is a quasi-Gibbs measure satisfying Assumptions~\ref{ass:Mmu} and~\ref{ass:ConditionalClos}, and that
\begin{enumerate}[$(a)$]
\item there exists a closed $\mssm$-negligible set~$F\subset X$ so that the free potential~$\Phi$ of~$\QP$ satisfies~$\Phi\in L^\infty_\loc(X\setminus F,\mssm)$;
\item there exists a closed $\mssm^\otym{2}$-negligible set~$F^\asym{2}\subset X$ so that the interaction potential~$\Psi$ of~$\QP$ satisfies~$\Psi\in L^\infty_\loc\ttonde{X^\tym{2}\setminus F^\asym{2},\mssm^\otym{2}}$.
\end{enumerate}
Then, the form~$\ttonde{\EE{\dUpsilon(B_r)}{\QP^\eta_r},\dom{\EE{\dUpsilon(B_r)}{\QP^\eta_r}}}$ defined as in~\ref{ass:ConditionalClos} with~$B_r$ in place of~$E$, satisfies~$(\SL{\QP^\eta_r}{\mssd_\dUpsilon})$ for every~$\eta\in\dUpsilon$ and every~$r>0$.

\begin{proof}
Fix~$r>0$.
By the representation of~$\dUpsilon(B_r)$ as a disjoint union, it suffices to verify the statement for the restricted measure~$\QP^\eta_{r,n}\eqdef \QP^\eta_r\mrestr{B_r^\sym{n}}$.
By the assumptions there exists a closed $\mssm^\sym{n}$-negligible set~$F^\sym{n}_r\subset B_r^\sym{n}$ so that~$\msH_{r,n}\eqdef \msH_{B_r}\restr_{\dUpsilon^\sym{n}(B_r)}$ satisfies~$\msH_{r,n}\in L^\infty_\loc \ttonde{B_r^\sym{n}\setminus F^\sym{n}_r}$.

Since~$F\cap B_r$ and~$F^\asym{2}\cap B_r^\tym{2}$ are relatively compact,~$F^\sym{n}_r$ is so as well.
Thus, for every~$\eps>0$ there exists~$m_\eps\in \N$ and points~$\seq{\mbfx^\sym{n}_{i,\eps}}_{i\leq m_\eps}\subset B_r^\sym{n}$ so that~$B^i_\eps\eqdef B^{\mssd_\sym{n}}_\eps(\mbfx^\sym{n}_{i,\eps})$ is a finite open covering of~$F^\sym{n}_r$.
Since~$m_\eps$ is finite, the centers~$\mbfx^\sym{n}_{i,\eps}$ may be chosen in such a way that each connected component of~$Y^\sym{n}_\eps\eqdef B^\sym{n}_r\setminus \cl_{\sym{n}}\ttonde{ \cup_{i\leq n_\eps} B^i_\eps}$ is quasi-convex, with quasi-convexity constant~$C_\eps$ possibly depending on~$\eps>0$.
That is, for every~$\mbfx^\sym{n},\mbfy^\sym{n}$ in the same connected component of~$Y^\sym{n}_\eps$ there exists a curve~$\gamma$ joining them and additionally so that
\begin{align*}
\length(\gamma)\leq C_\eps\, \mssd_\sym{n}(\mbfx^\sym{n},\mbfy^\sym{n}) \fstop
\end{align*}

Since~$\msH_{r,n}\in L^\infty_\loc \ttonde{B_r^\sym{n}\setminus F^\sym{n}_r}$, we have that~$n!\, e^{\mssm B_r}\frac{\diff\QP^\eta_{r,n}}{\diff\PP_{\mssm_r,n}}=\frac{\diff\QP^\eta_{r,n}}{\diff\mssm^\sym{n}_r}$ is uniformly bounded away from~$0$ and infinity on~$Y_\eps^\sym{n}$.
Thus, the weighted Sobolev space~$W^{1,\infty}(Y_\eps^\sym{n},\QP^\eta_{r,n}\mrestr{Y_\eps^\sym{n}})$ coincides with the Sobolev space~$W^{1,\infty}(Y_\eps^\sym{n},\mssm_r^\sym{n}\mrestr{Y_\eps^\sym{n}})$.
Since~$Y_\eps^\sym{n}$ is quasi-convex, the standard characterization of Sobolev spaces by absolute continuity along segments, e.g.~\cite[Thm.~5.6.5]{KufJohFuc77}, implies that
\begin{align}\label{eq:p:SLExample:1}
W^{1,\infty}(Y_\eps^\sym{n},\QP^\eta_{r,n})=W^{1,\infty}(Y_\eps^\sym{n},\mssm_r^\sym{n}) \cong \bLip(Y_\eps^\sym{n},\mssd_\sym{n}) \comma
\end{align}
(here,~$\cong$ denotes an isometry of Banach spaces) and therefore that~$(\SL{\PP_{\mssm_r,n}}{\mssd_\sym{n}})$ holds.

Now, let~$f^\sym{n}\in \DzLocB{\QP^\eta_{r,n}}$.
By locality of~$\SF{\dUpsilon(B_r)}{\QP^\eta_{r,n}}$ and since~$Y_\eps^\sym{n}$ is open, we have that
\begin{align*}
\SF{\dUpsilon(B_r)}{\QP^\eta_{r,n}}(f^\sym{n})\restr_{Y_\eps^\sym{n}}\leq 1 \as{\mssm_r^\sym{n}} \quad \text{on } Y_\eps^\sym{n} \comma
\end{align*}
which shows that~$f^\sym{n}\restr_{Y_\eps^\sym{n}}$ is $1$-Lipschitz w.r.t.~$\mssd_\sym{n}$ by~\eqref{eq:p:SLExample:1}.
Let~$f^\sym{n}_\eps$ be the (upper) constrained McShane extension~\eqref{eq:McShane} of~$f^\sym{n}\restr_{Y_\eps^\sym{n}}$ on~$B_r^\sym{n}$.
Since~$(\mcX,\cdc,\mssd)$ is the local diffusion space arising from a Riemannian manifold, it satisfies~$(\Rad{\mssd}{\mssm})$.
Thus,~$\ttonde{\EE{\dUpsilon^\sym{n}(B_r)}{\PP_{\mssm_r,n}},\dom{\EE{\dUpsilon^\sym{n}(B_r)}{\PP_{\mssm_r,n}}}}$ satisfies~$(\Rad{\mssd_\sym{n}}{\PP_{\mssm_r,n}})$ by Proposition~\ref{p:cSLProjPoisson}\iref{i:p:cSLProjPoisson:2}.
As a consequence,~$f^\sym{n}_\eps$ is as well an element of~$\DzLocB{\PP_{\mssm_r,n}}$.
Since~$F$ is $\mssm$-negligible and~$F^\asym{2}$ is $\mssm^\otym{2}$-negligible, we have that~$\cap_{\eps>0} Y^\sym{n}_\eps$ is $\mssm_r^\sym{n}$-negligible as well.
Therefore,~$L^1(\mssm_r^\sym{n})$-$\lim_{\eps\to 0} f^\sym{n}_\eps=f^\sym{n}$ by Dominated Convergence with dominating function~$\norm{f^\sym{n}}_{L^\infty(\mssm_r^\sym{n})}\car_{B_r^\sym{n}}$.

Summarizing, we have in particular that
\begin{align*}
\sup_\eps\norm{f^\sym{n}_\eps}_{L^\infty(\PP_{\mssm_r,n})}\leq \norm{f^\sym{n}}_{L^\infty(\PP_{\mssm_r,n})}\comma \qquad f^\sym{n}_\eps \in \DzLocB{\PP_{\mssm_r,n}}\comma \qquad f^\sym{n}_\eps\xrightarrow[ \text{w*-$L^\infty(\PP_{\mssm_r,n})$}]{\ \eps\to 0 \ } f^\sym{n} \fstop
\end{align*}
By~\cite[Prop.~2.35]{LzDSSuz21a} we conclude that~$f^\sym{n}\in \DzLocB{\PP_{\mssm_r,n}}$, which verifies the assumptions in Proposition~\ref{p:Locality}\iref{i:p:Locality:1}.
The conclusion then follows from Proposition~\ref{p:Locality}\iref{i:p:Locality:1} and the Sobolev-to-Lipschitz property~$(\SL{\PP_{\mssm_r,n}}{\mssd_\sym{n}})$ shown above.
\end{proof}
\end{prop}

\subsubsection{Integral measures}\label{sss:MixedPoisson}
In this short section, we provide a general abstract method for enlarging the class of measures satisfying our assumptions.
Let us start by recalling the definition of mixed Poisson measures, as an example.
We say that~$\lambda\in\msP(\R^+_0)$ is a \emph{L\'evy measure} if~$\lambda\set{0}=0$ and~$\lambda (1\wedge t)<\infty$.

\begin{defs}[Mixed Poisson measures]\label{d:MixedPoisson}
Let~$\mcX$ be a local structure, and~$\lambda\in\msP(\R^+_0)$ be a L\'evy measure.
The \emph{mixed Poisson measure} with intensity measure~$\mssm$ and L\'evy measure~$\lambda$ is the probability measure on~$\ttonde{\dUpsilon,\A_\mrmv(\msE)}$ defined as
\begin{align*}
\QP_{\lambda,\mssm}=\int_{\R^+} \PP_{s\cdot\mssm}\diff\lambda(s) \fstop
\end{align*}
\end{defs}

It readily verified that \emph{mixed Poisson measures} satisfy most of the previously stated assumptions, and in fact all previously established conclusions.
We note that mixed Poisson measures do not satisfy Assumption~\ref{ass:Closability}, for they do not satisfy any Campbell identity~\ref{ass:GNZ}, which shows the necessity to discuss other assumptions on closability, as in the previous sections.
However, it is not difficult to show that the form~$\ttonde{\EE{\dUpsilon}{\QP_{\lambda,\mssm}},\dom{\EE{\dUpsilon}{\QP_{\lambda,\mssm}}}}$ is closable (for every L\'evy measure~$\lambda$), and that it satisfies both~$(\Rad{\mssd_\dUpsilon}{\QP_{\lambda,\mssm}})$ and~$(\dSL{\mssd_\dUpsilon}{\QP_{\lambda,\mssm}})$.

Mixed Poisson measures are the prototipical example of \emph{integral measures} on~$\dUpsilon$, i.e.\ measures constructed by integration.
The corresponding Dirichlet forms can be understood as direct integrals of Dirichlet forms in the sense of e.g.~\cite{LzDS20}.
In the next proposition we show how to transfer properties of invariant measures to any integral measures constructed from them.

\begin{prop}\label{p:Linearity}
Let~$(\mcX,\cdc,\mssd)$ be an \MLDS, and~$(Z,\Tau,\nu)$ be a standard probability space.
For each~$\zeta\in Z$, let~$\QP_\zeta$ be a probability measures on~$\ttonde{\dUpsilon,\A_\mrmv(\Ed)}$ so that
\begin{enumerate}[$(a)$]
\item\label{i:p:Linearity:0.1} $\QP_\zeta$ has full $\T_\mrmv(\Ed)$-support for $\nu$-a.e.~$\zeta\in Z$;

\item\label{i:p:Linearity:0.2} there exists a family~$\set{\Omega_\zeta}_{\zeta\in Z}\subset \A_\mrmv(\Ed)$ of pairwise disjoint sets so that~$A_\zeta$ is $\QP_\zeta$-conegligible for $\nu$-a.e.~$\zeta\in Z$;

\item\label{i:p:Linearity:0.3} $\ttonde{\EE{\dUpsilon}{\QP_\zeta},\CylQP{\QP_\zeta}{\Dz}}$ is a closable densely defined pre-Dirichlet form.
\end{enumerate}

Finally, define a measure~$\QP$ on~$\ttonde{\dUpsilon,\A_\mrmv(\Ed)}$ as
\begin{equation*}
\QP\eqdef \int_Z \QP_\zeta\diff\nu(\zeta) \fstop
\end{equation*}

Then,
\begin{enumerate}[$(i)$]
\item\label{i:p:Linearity:0} the superposition~$\ttonde{\EE{\dUpsilon}{\QP},\CylQP{\QP}{\Dz}}$ of the forms~$\seq{\EE{\dUpsilon}{\QP_\zeta},\dom{\EE{\dUpsilon}{\QP_\zeta}}}_{\zeta\in Z}$ defined by
\begin{equation}\label{eq:p:Linearity:0}
\EE{\dUpsilon}{\QP}(u)\eqdef\int_Z \EE{\dUpsilon}{\QP_\zeta}(u) \diff\nu(\zeta)
\end{equation}
is a closable densely defined pre-Dirichlet form on~$L^2(\QP)$;
 
\item\label{i:p:Linearity:1} if~$(\Rad{\mssd_\dUpsilon}{\QP_\zeta})$ holds for~$\nu$-a.e.~$\zeta\in Z$, then~$(\Rad{\mssd_\dUpsilon}{\QP})$ holds as well;

\item\label{i:p:Linearity:2} if~$(\cSL{\T_\mrmv}{\QP_\zeta}{\mssd_\dUpsilon})$ ---~or, equivalently,~$(\dSL{\mssd_\dUpsilon}{\QP_\zeta})$%, cf.~\eqref{eq:EquivalenceRadStoL}
~--- holds for~$\nu$-a.e.~$\zeta\in Z$, then $(\cSL{\T_\mrmv}{\QP}{\mssd_\dUpsilon})$ holds for~$\QP$.
\end{enumerate}
\begin{proof}
\iref{i:p:Linearity:0} is the content of~\cite[\S{V.3.1}]{BouHir91}.
Assumption~\iref{i:p:Linearity:0.2} is necessary to show that, if~$u\in\dom{\EE{\dUpsilon}{\QP}}$, then there exists a suitable representative~$\rep u$ of~$u$ so that~$\class[\QP_\zeta]{\rep u}\in\dom{\EE{\dUpsilon}{\QP_\zeta}}$ for $\nu$-a.e.~$\zeta\in Z$, as discussed in~\cite[\S2.7]{LzDS20}.
Additionally, it ensures the direct integral representation of~$L^2$-spaces,
\begin{equation}\label{eq:p:Linearity:0.5}
L^2(\QP_{\lambda,\mssm})=\dint{\R^+_0} L^2(\PP_{s\cdot \mssm})\diff\lambda(s)\fstop
\end{equation}

\iref{i:p:Linearity:1} Fix~$\rep u\in \Lip(\mssd_\dUpsilon)$.
By assumption,
\begin{align*}
\class[\QP_\zeta]{u}\in\dom{\EE{\dUpsilon}{\QP_\zeta}}\comma \qquad \norm{\SF{\dUpsilon}{\QP_\zeta}(u)}_{L^\infty(\QP_\zeta)}\leq \Li[\mssd_\dUpsilon]{\rep u} \quad \forallae{\nu} \zeta\in Z\comma
\end{align*}
which, together with the equality~\eqref{eq:p:Linearity:0} established in~\iref{i:p:Linearity:0}, proves that~$u\in\dom{\EE{\dUpsilon}{\QP}}$
Furthermore, everywhere on~$\dUpsilon$,
\begin{equation*}
\widehat{\SF{\dUpsilon}{\QP}(v)}=\cdc^\dUpsilon(v)=\int_Z\cdc^\dUpsilon(v)\diff\nu(\zeta)=\int_Z\widehat{\SF{\dUpsilon}{\QP_\zeta}(v)}\diff\nu(\zeta)\comma \qquad v\in\Cyl{\Dz}\comma
\end{equation*}
which extends to
\begin{equation}\label{eq:p:Linearity:1}
\SF{\dUpsilon}{\QP}(v)=\int_Z\SF{\dUpsilon}{\QP_\zeta}(v)\diff\nu(\zeta)\as{\QP}\comma \qquad v\in\dom{\EE{\dUpsilon}{\QP}}\comma
\end{equation}
by density of~$\Cyl{\Dz}$ in~$\dom{\EE{\dUpsilon}{\QP}}$.
As a consequence
\begin{equation}\label{eq:p:Linearity:1.5}
\norm{\SF{\dUpsilon}{\QP}(u)}_{L^\infty(\QP)}= \nu\text{-}\esssup_\zeta \norm{\SF{\dUpsilon}{\QP_\zeta}(u)}_{L^\infty(\QP_\zeta)} \leq  \nu\text{-}\esssup_\zeta \Li[\mssd_\dUpsilon]{u}=\Li[\mssd_\dUpsilon]{u}
\end{equation}
by~\eqref{eq:p:Linearity:1}, which concludes the proof.

\iref{i:p:Linearity:2}
Let~$u\in\dom{\EE{\dUpsilon}{\QP}}\cap \mcC\ttonde{\T_\mrmv(\Ed)}$ be so that~$\norm{\SF{\dUpsilon}{\QP}}_{L^\infty(\QP)}<\infty$.
As a consequence of~\iref{i:p:Linearity:0.1}, $\QP$~has full $\T_\mrmv(\Ed)$-support.
Therefore there exists a unique continuous representative of~$u$, again denoted by~$u$.
Since~$\QP_\zeta$ is a probability with full $\T_\mrmv(\Ed)$-support,~$u$ is as well the unique $\T_\mrmv(\Ed)$-continuous representative of~$\class[\QP_\zeta]{u}$.
Letting~$\zeta_*\in Z$ be such that~$(\cSL{\T_\mrmv(\Ed)}{\mssd_\dUpsilon}{\QP_{\zeta_*}})$ holds, we immediately conclude that~$u$ is $\mssd_\dUpsilon$-Lipschitz.
Furthermore, by the first equality in~\eqref{eq:p:Linearity:1.5} and by the assumption of~$(\cSL{\T_\mrmv(\Ed)}{\mssd_\dUpsilon}{\QP_{\zeta}})$ for $\nu$-a.e.~$\zeta\in Z$,
\begin{equation*}
\Li[\mssd]{u}\leq \nu\text{-}\esssup_\zeta \norm{\SF{\dUpsilon}{\QP_\zeta}}_{L^\infty(\QP_\zeta)}  = \norm{\SF{\dUpsilon}{\QP}}_{L^\infty(\QP)}\fstop \qedhere
\end{equation*}
\end{proof}
\end{prop}

\begin{cor}
Let~$(\mcX,\cdc,\mssd)$ be an \MLDS, $\QP_{\lambda,\mssm}$ be a mixed Poisson measure on~$\ttonde{\dUpsilon,\A_\mrmv(\Ed)}$.
If either~$(\Rad{\mssd_\dUpsilon}{\PP_\mssm})$,~$(\cSL{\T_\mrmv(\msE)}{\PP_\mssm}{\mssd_\dUpsilon})$, or~\eqref{eq:Cheeger=EE} with~$\QP=\PP_\mssm$ holds, then the same property holds for~$\QP_{\lambda,\mssm}$ in place of~$\PP_\mssm$.
\begin{proof}
It suffices to verify the assumptions of Proposition~\ref{p:Linearity}:
\iref{i:p:Linearity:0.1} is Lemma~\ref{l:SupportPoisson}; 
\iref{i:p:Linearity:0.2} holds since~$\PP_{s\cdot\mssm}\perp \PP_{s'\cdot\mssm}$ for every~$s,s'>0$ with~$s\neq s'$.
This latter fact is a consequence of the celebrated Shorokhod Theorem~\cite{Sko57}.
\iref{i:p:Linearity:0.3} is a consequence of Proposition~\ref{p:MRLifting} together with Remark~\ref{r:PoissonClosability}.
\end{proof}
\end{cor}
%\purple{Can the Varadhan also hold in the same argument?}

\afterpage{%
    \clearpage% Flush earlier floats (otherwise order might not be correct)
    \thispagestyle{empty}% empty page style (?)
    \begin{landscape}% Landscape page
        \centering
\begin{table}[htb!]
\centerfloat
\def\arraystretch{1.3}
\begin{tabular}{c | c | c c c c c  c c c}
\multirow{6}{*}{\rotatebox[origin=c]{90}{Assumptions}} & \multirow{3}{*}{on~$\mcX$} & \LDS & \TLDS & \TLDS & \TLDS & \MLDS & $\text{\MLDS}$ & $\text{\MLDS}$ & $\text{\MLDS}$
\\
&&&&&& $(\Rad{\mssd}{\mssm})$& & $(\Rad{\mssd}{\mssm})$ & $(\Rad{\mssd}{\mssm})$
\\
&&&&&&& \ref{ass:cSLTensor} & \ref{ass:cSLTensor} & \ref{ass:T}
\\
\cline{2-10}
&\multirow{3}{*}{on~$\QP$} &&&\ref{ass:ConditionalClos} & \ref{ass:CWP} & \ref{ass:CWP} &&& \ref{ass:CWP}
\\
&&&& \ref{ass:CAC} & \ref{ass:AC} & \ref{ass:AC} & & & \ref{ass:AC}
\\
&& \ref{ass:Closability}&\ref{ass:Closability} &&&& \ref{ass:dcSLConfig} & \ref{ass:dcSLConfig} &
\\
\hline
\multirow{7}{*}{\rotatebox[origin=c]{90}{Conclusions on $\dUpsilon$}} & Closability & \ref{t:Closability} & \ref{p:MRLifting} & \ref{t:ClosabilitySecond} & $*$ & $*$ & \ref{t:ClosabilitySecond} & \ref{t:ClosabilitySecond} & $*$
\\
& $\EE{\dUpsilon}{\QP}=\EE{\infty}{\QP^\asym{\infty}}$ &&&& \ref{c:Isometry} & \ref{c:Isometry}& \ref{c:Isometry}& \ref{c:Isometry}& \ref{c:Isometry}
\\
& $(\Rad{\mssd_\dUpsilon}{\QP})$ &&&&& \ref{t:Rademacher} & & \ref{t:Rademacher} & \ref{t:Rademacher}
\\
& $(\cSL{\T_\mrmv(\Ed)}{\QP}{\mssd_\dUpsilon})$ &&&&& & \ref{t:cSLUpsilon}
\\
& $(\dcSL{\mssd_\dUpsilon}{\QP}{\mssd_\dUpsilon})$ &&&&&& \ref{t:dcSLUpsilon} & \ref{t:dcSLUpsilon}
\\
& $\mssd_\dUpsilon=\mssd_\QP$ &&&&&&& \ref{t:StoL2}&
\\
& $\EE{\dUpsilon}{\QP}=\Ch[\mssd_\dUpsilon,\QP]$ &&&&&&&& \ref{t:IdentificationCheeger}
\\
\hline
\end{tabular}

\bigskip
\caption{ \label{tbl:1}
A summary of results and assumptions on configuration spaces.
The first line block indicates assumptions on the base space.
The second line block indicates assumptions on the reference measure~$\QP$ on~$\dUpsilon$.
The first column indicates the sought conclusion.
For instance, in order to locate the proof of~$(\Rad{\mssd_\dUpsilon}{\QP})$ it suffices to consult the corresponding row, and find out that the assertion is shown in (Theorem)~\ref{t:Rademacher}.
The necessary assumptions ---~\MLDS satisfying~$(\Rad{\mssd}{\mssm})$, and~\ref{ass:AC}~--- are listed on top of the corresponding column.
An asterisk~`$*$' indicates that a conclusion is rather assumed.
}
\end{table}
\end{landscape}
    %\clearpage% Flush page
}

\begin{table}[b!]
%\centerfloat
\def\arraystretch{1.3}
\begin{tabular}{c | c | c c c}
\multirow{13}{*}{\rotatebox[origin=c]{90}{Conclusions}} && \multicolumn{3}{c}{Assumptions}
\\
\cline{3-5}
%\multirow{15}{*}{\rotatebox[origin=c]{90}{Conclusions}} &&&&%& \MLDS
%\\
&&compl.\ Riem.\ mfd. & $\RCD^*(K,N)$ & $\RCD(K,\infty)$ %& $(\mathsf{D}_\loc)_{\ref{d:LocalDoubling}}$+\ref{ass:GE}+
%\\
%&&&&%&+\ref{ass:Poincare}+\eqref{eq:VolGrowth}
\\
\hline
&\LDS & \checkmark & \ref{p:IHtoMLDS} & \ref{p:IHtoMLDS} %& *
\\
%&$(X,\mssd)$ proper & & \ref{p:PropertiesRCD(K,N)} & $\times$ & *
%\\
&(quasi-/)regular & \checkmark &\ref{p:PropertiesRCD(K,N)}& \ref{p:PropertiesRCD(Kinfty)} %& *
\\
& \MLDS & \checkmark & regular+\ref{p:IHtoMLDS} & q.-regular+\ref{p:IHtoMLDS} %& *
\\
&$(\Rad{\mssd}{\mssm})$ & \checkmark & \ref{p:PropertiesRCD(Kinfty)} & \ref{p:PropertiesRCD(Kinfty)} %& $\EE{X}{\mssm}=\Ch[\mssd,\mssm]$
\\
&$(\cSL{\T}{\mssm}{\mssd})$ & \eqref{eq:EquivalenceRadStoL} & \eqref{eq:EquivalenceRadStoL} & \eqref{eq:EquivalenceRadStoL}
\\
&$(\SL{\mssm}{\mssd})$ &\checkmark & \ref{p:PropertiesRCD(Kinfty)} & \ref{p:PropertiesRCD(Kinfty)}
\\
&\ref{ass:SLTensor} &\checkmark & see~\ref{p:PropertiesRCD(K,N)}\ref{i:p:PropertiesRCD(K,N):6} & $?$
\\
&$\mssd=\mssd_\mssm$ &\checkmark & \ref{p:PropertiesRCD(Kinfty)} & \ref{p:PropertiesRCD(Kinfty)}
\\
&$\EE{X}{\mssm}=\Ch[\mssd,\mssm]$ &\checkmark & * & * %& *
\\
&\ref{ass:T} &\checkmark & \ref{p:PropertiesRCD(K,N)} & $?$ 
\\
\hline
\end{tabular}
\bigskip
\caption{ \label{tbl:2}
A summary of results and assumptions on base spaces.
The first line indicates assumptions on the space. The first column indicates the sought conclusion.
For instance, in order to locate the proof `irreducibility' on `$\RCD(K,\infty)$ spaces' it suffices to consult the corresponding crossing, and find out that the assertion is shown in (Proposition)~\ref{p:PropertiesRCD(Kinfty)}.
A checkmark~`\checkmark' indicates that a conclusion is well-known.
An asterisk~`$*$' indicates that a conclusion is rather assumed.
A question mark~`$?$' indicates that the validity of the assertion in the corresponding row on spaces in the corresponding column is not known.
}
\end{table}
%\clearpage

\afterpage{%
    \clearpage% Flush earlier floats (otherwise order might not be correct)
    \thispagestyle{empty}% empty page style (?)
    \begin{landscape}% Landscape page
        \centering
\begin{table}[htb!] 
%\centerfloat
\def\arraystretch{1.3}
\begin{tabular}{c | c | c c c c }
\multirow{13}{*}{\rotatebox[origin=c]{90}{\hspace{-5mm} Measure $\QP$}} && \multicolumn{4}{c}{Base space $X$}
\\
\cline{3-6}
%\multirow{15}{*}{\rotatebox[origin=c]{90}{Conclusions}} &&&&%& \MLDS
%\\
&& $\R^n$ & compl.\ Riem.\ mfd. & $\RCD^*(K,N) \cap \CAT(0)$ & path/loop spaces %& $(\mathsf{D}_\loc)_{\ref{d:LocalDoubling}}$+\ref{ass:GE}+
%\\
%&&&&%&+\ref{ass:Poincare}+\eqref{eq:VolGrowth}
\\
\hline
&Poisson & \doublecheck  & \doublecheck & \doublecheck & \ref{p:MRLifting} %& *
\\
& mixed Poisson  & \doublecheck & \doublecheck & \doublecheck&  \ref{p:MRLifting}
%\\
%& Ruelle & \checkmark & $\times$ & $\times$ & $\times$
%\\
%&grand canonical Gibbs (\ref{ass:dcSLConfig}) &\checkmark & \checkmark& \checkmark & %\ref{p:MRLifting}
\\
& canonical Gibbs & \checkmark & \checkmark & \checkmark & %\ref{p:MRLifting} %& *
\\
& canonical Gibbs + \ref{ass:dcSLConfig} & \doublecheck & \doublecheck & \doublecheck & %\ref{p:MRLifting} %& *
\\
& $\mathrm{sine}_\beta$/$\mathrm{Bessel}_{\alpha,\beta}$/$\mathrm{Airy}_\beta$ & \doublecheck \quad $n=1$ & $\times$ & $\times$ & $\times$
\\
& Ginibre & \doublecheck \quad $n=2$ & $\times$ & $\times$ & $\times$
\\
&determinantal/permanental & \checkmark & \checkmark & \checkmark  & %\ref{p:MRLifting}%& $\EE{X}{\mssm}=\Ch[\mssd,\mssm]$
\\
&determinantal/permanental + \ref{ass:dcSLConfig} &  \doublecheck & \doublecheck & \doublecheck & %\ref{p:MRLifting}%& $\EE{X}{\mssm}=\Ch[\mssd,\mssm]$
\\
&quasi Gibbs &  \checkmark & \checkmark & \checkmark  &%\ref{p:MRLifting}
\\
&quasi Gibbs + \ref{ass:dcSLConfig} &  \doublecheck & \doublecheck & \doublecheck  &%\ref{p:MRLifting}
\\
\hline
\end{tabular}
\bigskip
\caption[Table3]{\label{tbl:3}
A summary of the main examples to which our main results apply.
The first line indicates assumptions on the base space.
The first column indicates the assumptions on the invariant measure.
We write~`+ \ref{ass:dcSLConfig}' to indicate that~$\QP$ is assumed to satisfy as well Assumption~\ref{ass:dcSLConfig}.
The symbol `$\times$' indicates that the family of measures in the corresponding row is not defined on spaces in the corresponding column.
The number/symbol in each cell indicates the results that hold true given the assumptions on the corresponding row and column.
We write~`\checkmark', resp.~`\doublecheck', whenever the following results hold:
\begin{itemize}
\item[\checkmark] the closability (Thm.s~\ref{p:MRLifting}, or~\ref{t:ClosabilitySecond}), $\Ch[\mssd_\dUpsilon,\QP]=\EE{\dUpsilon}{\QP}$ (Thm.~\ref{t:IdentificationCheeger}), and quasi-regularity (Cor.~\ref{c:RadQReg});
\item[\doublecheck] all the results listed in~\checkmark and additionally: $\mssd_\QP=\mssd_\dUpsilon$ (Thm.~\ref{t:StoL2}), and the Varadhan short-time asymptotics (Thm.~\ref{t:VaradhanSecond}).
\end{itemize}
}%
% The blanc means that no result applies. 
%For instance, in order to locate the proof `irreducibility' on `$\RCD(K,\infty)$ spaces' it suffices to consult the corresponding crossing, and find out that the assertion is shown in (Proposition)~\ref{p:PropertiesRCD(Kinfty)}.
%A checkmark~`\checkmark' indicates that a conclusion is well-known.
%An asterisk~`$*$' indicates that a conclusion is rather assumed.
%
%A question mark~`$?$' indicates that the validity of the assertion in the corresponding row on spaces in the corresponding column is not known.
%
\end{table}
\end{landscape}
    %\clearpage% Flush page
}

\end{document}